\documentclass[12pt]{report}
\usepackage{hyperref}
\usepackage{hyperxmp}
\usepackage{braket}
\usepackage{pgfplots} 
\usepackage[a4paper]{geometry}
\usepackage{epsfig}
\usepackage{pstricks}
\usepackage{tcolorbox} 
\usepackage{tikz}
\usepackage{nomencl}
\usepackage{graphicx}
\usepackage{epstopdf}
\author{{\bf \large Principles, Algorithms and Case Studies}}
\usepackage{url}
\usepackage{graphicx}
\usepackage{amsmath}
\usepackage{amssymb}
\usepackage{theorem}
\usepackage{algorithm}
\usepackage{algorithmic}
\usepackage{psfrag}
\usepackage{tikz}
\usetikzlibrary{calc,intersections,patterns}
\usepackage{hyperref}
\usepackage{xcolor}
\usepackage{fullpage}
\usepackage{subcaption}
\usepackage[
    type={CC},
    modifier={by},
    version={4.0},
]{doclicense}

\usepackage{enumitem}

\newcounter{exercise}[chapter]
\renewcommand{\theexercise}{\thechapter.\arabic{exercise}}
\newcommand{\exercise}[3]{%
    \refstepcounter{exercise} 
    \item[\textbf{\theexercise}] \textbf{#1.} #2
    \label{#3}
}

\usetikzlibrary{arrows, intersections, calc, through, positioning}

\title{\bf \huge Multicriteria Optimization and Decision Making}
\date{Michael Emmerich and Andr\'e Deutz\\ 
Lecture Notes \\
	 MSc Course 2012-2023, LIACS, Leiden University\\ The Netherlands\\
        {\scriptsize Updated and Revised Edition, February 2025}
	\begin{center}
		\includegraphics[height=9cm]{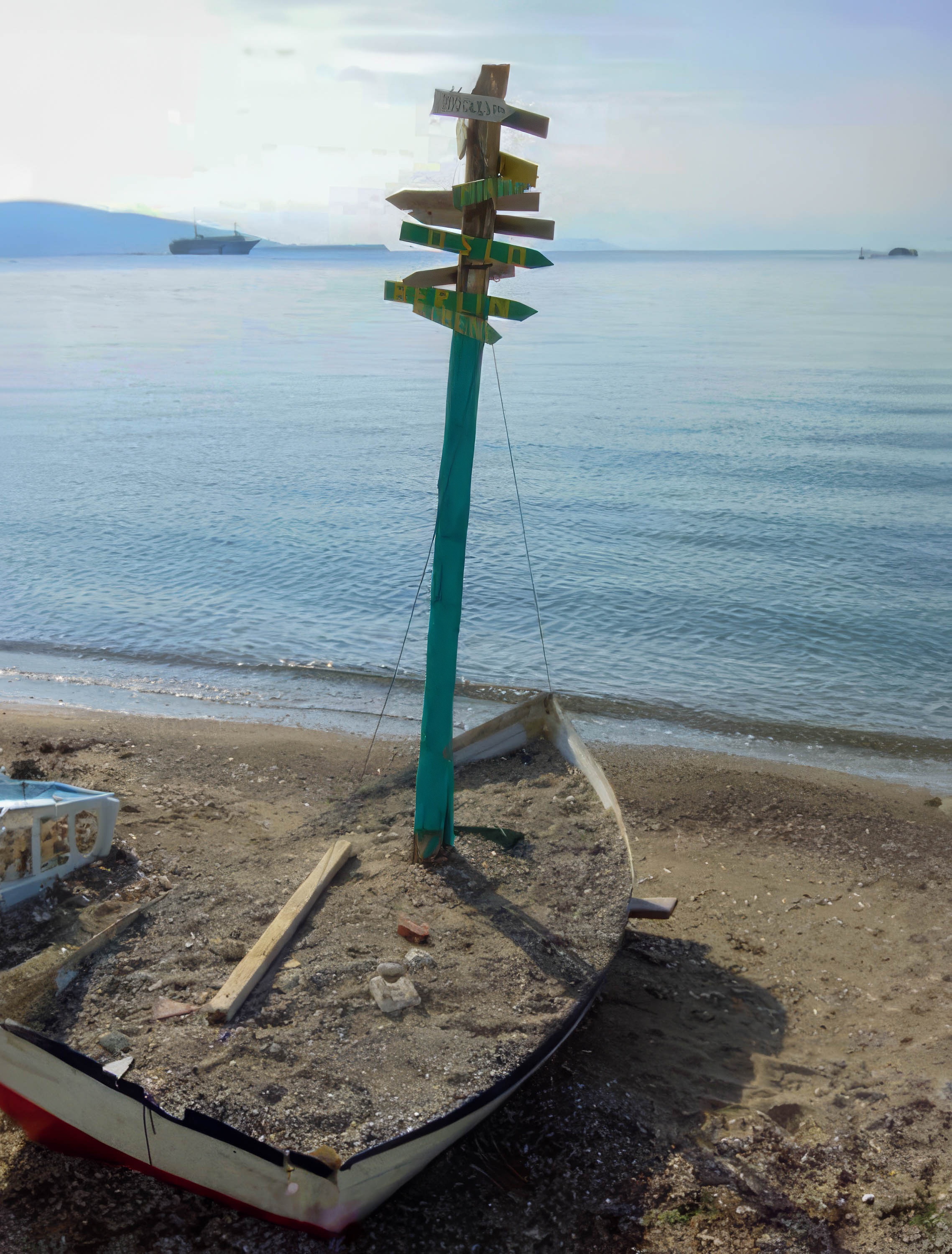}\\\tiny{Photo by Michael Emmerich (Author) - Aegina, Greece}
	\end{center}
}

\newtheorem{theorem}{Theorem}
\newtheorem{lemma}[theorem]{Lemma}
\newtheorem{proposition}[theorem]{Proposition}
\newtheorem{corollary}[theorem]{Corollary}
\newtheorem{definition}[theorem]{Definition}

\newenvironment{proof}[1][Proof:]{\begin{trivlist}
\item[\hskip \labelsep {\bfseries #1}]}{\end{trivlist}}
\newenvironment{example}[1][Example]{\begin{trivlist}
\item[\hskip \labelsep {\bfseries #1}]}{\end{trivlist}}
\newenvironment{remark}[1][Remark]{\begin{trivlist}
\item[\hskip \labelsep {\bfseries #1}]}{\end{trivlist}}

\newcommand{\acs}[1]{\emph{#1}}
\newcommand{\qed}{\nobreak\ifvmode\relax\else
  \ifdim\lastskip<1.5em \hskip-\lastskip
  \hskip1.5em plus0em minus0.5em\fi\nobreak
  \hbox{$\square$}\fi}
\newcommand{\ybf}{\mathbf{y}}

\pgfplotsset{compat=1.18}
\begin{document}
\maketitle 

\noindent

\tableofcontents

\abstract{
Real-world decision and optimization problems, often involve constraints and conflicting criteria. For example, choosing a travel method must balance speed, cost, environmental footprint, and convenience. Similarly, designing an industrial process must consider safety, environmental impact, and cost efficiency. Ideal solutions where all objectives are optimally met are rare; instead, we seek good compromises and aim to avoid lose-lose scenarios. Multicriteria optimization offers computational techniques to compute Pareto optimal solutions, aiding decision analysis and decision making. 
This reader offers an introduction to this topic and is based on the revised edition of the MSc computer science course lecture "Multicriteria Optimization and Decision Analysis" at the Leiden Institute of Advanced Computer Science, Leiden University, The Netherlands, conducted from 2007 to 2023. The introduction is organized in a unique didactic manner developed by the authors, starting from more simple concepts such as linear programming and single-point methods, and advancing from these to more difficult concepts such as optimality conditions for nonlinear optimization and set-oriented solution algorithms. In addition, we focus on the mathematical modeling and foundations rather than on specific algorithms, though we do not exclude the discussion of some representative examples of solution algorithms. Our aim was to make the material accessible to MSc students who do not study mathematics as their core discipline by introducing basic numerical analysis concepts when necessary and providing numerical examples for interesting cases.
}
\chapter*{Revised Version}
The latest release introduces several key updates and improvements:

\begin{itemize}
    \item \textbf{Linearization in Mathematical Programming:} A new section provides practical advice on linearizing mathematical programming models, particularly for integer linear programming (ILP) solvers. Additionally, we reference commonly used solvers for mathematical programming and multiobjective ptimization.
    
    \item \textbf{Karush-Kuhn-Tucker (KKT) Conditions:} Revisions in the KKT chapter correct a sign error and expand the discussion on constraint qualifications. The latest edition further clarifies these conditions and includes exercises to enhance understanding, inspired by course assignments and exams.
    
    \item \textbf{Combinatorial Optimization and the NP $=$ P Problem:} A new paragraph offers a historical perspective on the NP $=$ P problem in the context of combinatorial optimization.

    \item \textbf{Modeling Utility Functions}: Preference elicitation methods and value-focused thinking workflows have been introduced as methods to find utility functions in practise.
    
    \item \textbf{Evolutionary Algorithms and Non-Dominated Set Computation:} We provide a more detailed overview of different types of evolutionary algorithms and introduce the KLP algorithm for computing non-dominated sets.
    
    \item \textbf{References and Software:} Additional references have been incorporated, along with a discussion on relevant open-source software libraries for optimization.
    \item {\bf Analytical Example:} An illustrative example for multi-objective analytical optimization of a rectangular region has been incorporated.
    \item {\bf Goal programming:} The discussion and historical context of goal programming is now included in the section on scalarization methods.
\end{itemize}

\chapter*{Preface}
Identifying optimal solutions within extensive and constrained search spaces has long been a central focus of operations research and engineering optimization. These problems are generally algorithmically challenging. Multicriteria optimization is a contemporary subdiscipline of optimization, considering that real-world issues also involve decision making in the presence of multiple conflicting objectives. The quest for searching optimal solutions should be integrated with elements of multicriteria decision analysis (MCDA), which is the discipline of making well-informed choices through systematic evaluation of alternatives.

Real world decision and optimization problems usually involve
conflicting criteria. Think of chosing a means to travel from one country to
another. It should be fast, cheap or convenient, but you probably cannot have
it all. Or you would like to design an industrial process, that should be
safe, environmental friendly and cost efficient. Ideal solutions, where all objectives are at their optimal level, are rather the exception than
the rule. Rather we need to find good compromises, and avoid lose-lose situations. 

These lecture notes deal with Multiobjective Optimization and Decision Analysis (MODA). We define this field, based on some other scientific disciplines:
\begin{itemize}
\item \emph{Multicriteria Decision Aiding (MCDA)} (or: Multiattribute Decision Analysis)  is a scientific field that studies evaluation of a finite number of alternatives based on multiple criteria. It provides methods to compare, evaluate, and rank solutions.
\item \emph{Multicriteria Optimization (MCO)} (or: Multicriteria Design, Multicriteria Mathematical Programming) is a scientific field that studies search for optimal solutions given multiple criteria and constraints. Here, usually, the search space is very large and not all solutions can be inspected.
\item \emph{Multicriteria Decision Making (MCDM)}  deals with MCDA and MCO or combinations of these.
\end{itemize}
We use here the title: \textbf{Multicriteria Optimization and Decision Analysis = MODA} as a synonym of MCDM in order to focus more on the algorithmically challenging optimization aspect.

In this course we will deal with algorithmic methods for
solving (constrained) multi-objective optimization and decision making problems.
The rich mathematical structure of such problems as well as their
high relevance in various application fields led recently to a
significant increase of research activities. In particular
algorithms that make use of fast, parallel computing technologies
are envisaged for tackling hard combinatorial and/or nonlinear
application problems. In the course we will discuss the theoretical
foundations of multi-objective optimization problems and their
solution methods, including order and decision theory, analytical,
interactive and meta-heuristic solution methods as well as
state-of-the-art tools for their performance-assessment. Also an
overview on decision aid tools and formal ways to reason about
conflicts will be provided. All theoretical concepts will be
accompanied by illustrative hand calculations and graphical
visualizations during the course. In the second part of the course,
the discussed approaches will be exemplified by the presentation of
case studies from the literature, including various application
domains of decision making, e.g. economy, engineering, medicine or
social science.

This reader is covering the topic of Multicriteria Optimization and
Decision Making.
Our aim is to give a broad introduction to the field, rather than to
specialize on  certain types of algorithms and applications. Exact
algorithms for solving optimization algorithms are discussed as well
as selected techniques from the field of metaheuristic optimization,
which received growing popularity in recent years. The lecture notes provides
a detailed introduction into the foundations and a starting point
into the methods and applications for this exciting field of
interdisciplinary science. Besides orienting the reader about
state-of-the-art techniques and terminology, references are given
that invite the reader to further reading and point to specialized
topics.

\chapter{Introduction}
\medskip

For several reasons multicriteria optimization and decision making is an exciting field of computer science and operations research. Part of its fascination stems from the fact that in MCO
and MCDM different scientific fields are addressed. Firstly, to
develop the general foundations and methods of the field, one has to
deal with structural sciences, such as algorithmics, relational
logic, operations research, and numerical analysis.

\begin{itemize}
\item How can we state a decision/optimization problem in a formal way?
\item What are the essential differences between single-objective and multi-objective optimization?
\item How can we classify the solutions? What different types of order relations
are used in decision theory, and how are they related to each other?
\item Given a decision model or optimization problem, which formal
conditions must be satisfied for the solutions to be optimal?
\item How can we construct algorithms that obtain optimal solutions
or approximations to them in an efficient way?
\item What is the geometrical structure of solution sets for problems
with more than one optimal solution?
\end{itemize}

Whenever it comes to decision making in the real world, these
decisions will be made by people responsible for it. In order to
understand how people come to decisions, and how the psychology of
individuals (cognition, individual decision making) and
organizations (group decision making) needs to be studied. Questions
like the following may arise:
\begin{itemize}
\item What are our goals? What makes it difficult to state goals?
How do people define goals? Can the process of identifying goals be
supported?
\item What different strategies are used by people to make
decisions? How can satisfaction be measured? What strategies are
promising in making satisfactory decisions?
\item What are the cognitive aspects in decision making? How can
decision support systems be build in a way that takes care of
cognitive capabilities and limits of humans?
\item How do groups of people come to decisions? What are conflicts
and how can they be avoided? How do we deal with minority interests in
a democratic decision-making process? Can these aspects be integrated into
formal decision models?
\end{itemize}

Moreover, decisions are always related to a real-world problem.
Given an application field, we may find very specific answers to the
following questions.

\begin{itemize}
\item What is the set of alternatives?
\item By which means can we retrieve the values for the criteria (experiments, surveys, function evaluations)?
Are there any particular problems with these measurements (dangers and costs) and how to deal with them? What are the uncertainties in
these measurements?
\item What are the problem-specific objectives and constraints?
\item What are typical decision processes in the field and what
implications do they have for the design of decision support
systems?
\item Are there existing problem-specific procedures for decision support
and optimization, and what about the acceptance and performance of
these procedures in practice?
\end{itemize}

In summary, this list of questions gives some kind of bird's eye
view of the field. However, in these lecture notes we will mainly focus on the
structural aspects of multi-objective optimization and decision
making. On the other hand, we also devote one chapter to
human-centric aspects of decision making and one chapter to the
problem of selecting, adapting, and evaluating MOO tools for
application problems.

\section{Viewing mulicriteria optimization as a task in system design and analysis}

The discussion above can be seen as a rough sketch of questions that
define the scope of multicriteria optimization and decision-making.
However, it needs to be clarified more precisely what is going to be
the focus of these lecture notes. For this reason, we want to approach the
problem class from the point of view of system design and analysis.
Here, with system analysis, we denote the interdisciplinary research
field that deals with the modeling, simulation, and synthesis of
complex systems.

In addition to experimentation with a physical system, often a system model
is used. Today, system models are typically implemented as
computer programs that solve (differential) equation systems,
simulate interacting automata, or stochastic models. We will also
refer to them as {\em simulation models}. An example of a
simulation model based on differential equations would be the
simulation of the fluid flow around an airfoil based on the Navier-Stokes equations. An example of a stochastic system model could be
the simulation of a system of elevators, based on some agent-based
stochastic model.
\begin{figure}[ht]
\begin{center}
\includegraphics[width=10cm]{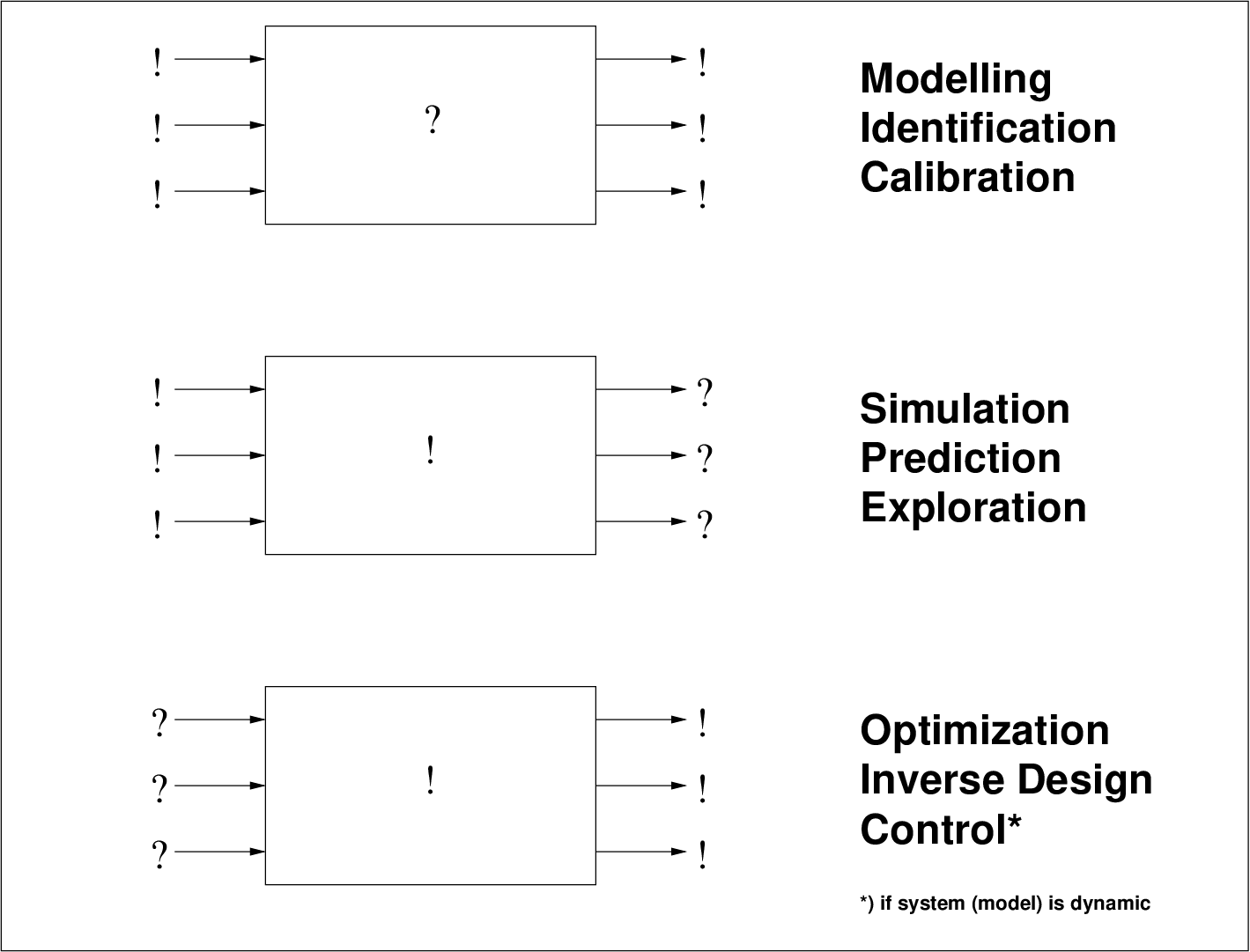}
\end{center}
\caption{\label{fig:tasks}Different tasks in systems analysis.}
\end{figure}

In Figure \ref{fig:tasks} different tasks of system analysis based
on simulation models are displayed schematically. {\em
Modeling} means identifying the internal structure of the simulation
model. This is done by looking at the relationship between the known
inputs and the outputs of the system. In many cases, the internal
structure of the system is already known up to a certain granularity
and only some parameters need to be identified. In this case we
usually speak of {\em calibration} of the simulation model instead
of modeling. In control theory, the term {\em identification}
is also common.

Once a simulation-model of a system is given, we can simulate the
system, i.e. predict the state of the output variables for different
input vectors. Simulation can be used to predict the output for
unmeasured input vectors. Usually such model-based predictions
are much cheaper than doing the experiment in the real world.
Consider, for example, crash test simulations or wind channel simulation. In many cases, such as for future predictions, where
time is the input variable, it is even impossible to do the
experiments in the physical world. Often the purpose of simulation
is also to learn more about the behavior of the systems. In this
case systematic experimenting is often used to study effects of
different input variables and combinations of them. The field of
Design and Analysis of Computer Experiments (DACE) is devoted to
such systematic explorations of the behavior of a system.

Finally, we may want to optimize a system: In that case we basically
specify what the output of the system should be. We also are given a
simulation-model to do experiments with, or even the physical system
itself. The relevant open question is how to choose the input
variables in order to achieve the desired output. In optimization, we
typically want to maximize (or minimize) the value of an output
variable.

On the other hand, a very common situation in practice is the task
of adjusting the value of an output variable in a way that it is as
close as possible to a desired output value. In that case we speak
about inverse design, or if the system is dynamically changing, it
may be classified as a optimal control task. An example for an
inverse design problem is given in airfoil design, where a specified
pressure profile around an airfoil should be achieved for a given
flight condition. An example for an optimal control task would be to
keep a process temperature of a chemical reactor as close to a
specified temperature as possible in a dynamically changing
environment.

Note, that the inverse design problem can be reformulated as
optimization problem, as it aims at minimizing the deviation between
the current state of the output variables and the desired state.

In multi-objective optimization we look at the optimization of
systems w.r.t. more than one output variables. Single-objective
optimization can be considered as a special case of multi-objective
optimization with only one output variable.

Moreover, classically, multi-objective optimization problems are
most of the time reduced to single-objective optimization problems.
We refer to these reduction techniques as {\em scalarization}
techniques. A chapter in these lecture notes is devoted to this topic. Modern
techniques, however, often aim at obtaining a set of 'interesting'
solutions by means of so-called Pareto optimization techniques. What
is meant by this will be discussed in the remainder of this chapter.

\section[Formal Problem Definitions]{Formal Problem Definitions in Mathema\-tical Programming}

Researchers in the field of {\em operations research} use an elegant and standardized notation for the classification and formalization of optimization and decision problems, the so-called {\em mathematical
programming problems}, among which linear programs (LP) are certainly the most prominent representant.
Using this notion a {\em generic definition of optimization problems} is as follows:
\begin{eqnarray}
f(\mathbf{x}) &\rightarrow& \min  \mbox{\textcolor{blue}{\ \ \ \  \ \        (* Objectives *)}}\\
g_1(\mathbf{x}) &\leq&  0 \mbox{\textcolor{blue}{ \ \ \ \ \ \ \ \ \
\ (* Inequality
constraints *)}}\\
&\vdots&\\
g_{n_g}(\mathbf{x}) &\leq& 0\\
h_1(\mathbf{x}) &=&  0 \mbox{\textcolor{blue}{\ \ \ \  \ \ \ \ \ \          (* Equality Constraints *)}}\\
&\vdots&\\
h_{n_h}(\mathbf{x}) &=& 0 \\
\mathbf{x} &\in& \mathcal{X} = [\mathbf{x}^{\min},
\mathbf{x}^{\max}] \subset \mathbb{R}^{n_x} \times \mathbb{Z}^{n_z}
\mbox{\textcolor{blue}{\  \ \     (* Box constraints *)}} \label{eq:boxcon}\\
\end{eqnarray}

The objective function $f$ is a function to be minimized (or maximized\footnote{Maximization can be rewritten as minimization by changing the sign of the objective function, that is, replacing $f(\mathbf{x}) \rightarrow \max$ with $-f(\mathbf{x}) \rightarrow \min$}). This is the goal of the optimization. The function $f$ can be evaluated for every point
$\mathbf{x}$ in the search space (or decision space). Here the {\em search space} is defined by a set of intervals that restrict the range of variables,
so-called bounds or box constraints. In addition to this, variables can be integer variables, that is, they are chosen from $\mathbb{Z}$ or subsets of it, or continuous variables (from $\mathbb{R}$). An important special case of integer variables are binary variables which are often used to model binary decisions in mathematical programming.

Whenever inequality and equality constraints are stated
explicitly, the search space $\mathcal{X}$ can be partitioned in a {\em feasible search space} $\mathcal{X}_f \subseteq \mathcal{X}$ and an {\em infeasible subspace} $\mathcal{X} - \mathcal{X}_f$. In the feasible subspace, all conditions stated in the mathematical programming problem are satisfied. The conditions in the mathematical program are used to avoid constraint violations in the system under design, e.g., the
excess of a critical temperature or pressure in a chemical reactor (an example for an inequality constraint) or the keeping of conservation of mass formulated as an equation (an example for an equality constraint). The conditions are called constraints. Due to a convention in the field of operations research, constraints are typically written in a {\em
standardized form} such that $0$ appears on the right-hand side. Equations can easily be transformed into the standard form by means of algebraic equivalence transformations.

Based on this very general definition of problems, we can define several
classes of optimization problems by looking at the characteristics
of the functions $f$, $g_i, i = 1, \dots, n_g$, and $h_i, i = 1,
\dots, n_h$. Some important classes are listed in Table \ref{tab:classificationmp}.
\begin{table}[h!]
\begin{tabular}{l|l|l|l}
Name & Abbreviation & Search Space & Functions \\
\hline
Linear Programming & LP & $\mathbb{R}^{n_r}$ & linear \\
Quadratic Programming & QP & $\mathbb{R}^{n_r}$ & quadratic \\
Integer Linear Progamming & ILP &$\mathbb{Z}^{n_z}$ & linear \\
Integer Progamming & IP &$\mathbb{Z}^{n_z}$ & arbitrary \\
Mixed Integer Linear Programming & MILP &$\mathbb{Z}^{n_z} \times \mathbb{R}^{n_r}$  & linear \\
Mixed Integer Nonlinear Programming & MINLP &$\mathbb{Z}^{n_z} \times
\mathbb{R}^{n_r}$ & nonlinear
\end{tabular}
\caption{\label{tab:classificationmp}Classification of mathematical programming problems.}
\end{table}

\subsection{Other Problem Classes in Optimization}
There are also other types of mathematical programming problem. These are, for instance, based on:
\begin{itemize}
\item The handling of uncertainties and noise of the input variables and of the parameters of the objective function: Such problems fall into the class of \emph{robust optimization problems} and \emph{stochastic programming}. If some of the constants in the objective functions are modeled as stochastic variables the corresponding problems are also called a \emph{parametric optimization problem}.
\item Non-standard search spaces: Non-standard search spaces are for instance the search spaces of trees (e.g. representing regular expressions), network configurations (e.g. representing flowsheet designs) or searching for 3-D structures (e.g. representing bridge constructions). Such problem are referred to as {\em topological}, {\em grammatical}, or {\em structure optimization}.
\item A finer distinction between different mathematical programming problems based on the characteristics of the functions: Often subclasses of mathematical programming problems have certain mathematical properties that can be exploited for faster solving them. For instance {\em convex quadratic programming problems} form a special class of relatively easy to solve quadratic programming problems. Moreover, {\em geometrical programming problems} are an important subclass of nonlinear programming tasks with polynomials that are allowed to h negative numbers or fractions as exponents.
\end{itemize}

In some cases, the demand that a solution consists of a vector of variables is too restrictive, and instead we can define the search space as some set $\mathcal{X}$. In order to capture also these kind of problems, a more general
definition of a general optimization problem can be used:

\begin{equation}
\label{eq:genopt} f_1(\mathbf{x}) \rightarrow \min, \ \ \ \mathbf{x}
\in \mathcal{X}
\end{equation}
$\mathbf{x} \in \mathcal{X}$  is called the {\em search point} or
{\em solution candidate} and $\cal X$ is the search space or
decision space. Finally, $f: \mathcal{X} \rightarrow \mathbb{R}$
denotes the objective function. Only in cases where $\mathcal{X}$ is
a vector space, we may talk of a {\em decision vector}.

Another important special case is given, if $\mathcal{X} = \mathbb{R}^n$.
Such problems are defined as {\em continuous unconstrained
optimization problems} or simply as unconstrained optimization
problems.

For notational convenience, in the following we will refer
mainly to the generic definition of an optimization problem given in
Equation \ref{eq:genopt}, whenever the constraint treatment is not
particularly addressed. In such cases, we assume that $\mathcal{X}$
already contains only feasible solutions.

\subsection{Linear Programming}
Researchers in the field of {\em operations research} use an elegant and standardized notation for the classification and formalization of optimization and decision problems, the so-called {\em mathematical
programming problems}, among which linear programs (LP) are certainly the most prominent representant.
Using this notion, a {\em generic definition of optimization problems} is as follows:
\begin{eqnarray}
\sum_{j=1}^{n} c_j x_j &\rightarrow& \min  \mbox{\textcolor{blue}{\ \ \ \  \ \        (* Objective function *)}}\\
\sum_{j=1}^{n} a_{i j} x_j &\leq& b_i, \quad i = 1, \dots, m \mbox{\textcolor{blue}{ \ \ \ \ \ (* Inequality constraints *)}}\\
x_j^{\min} \leq x_j &\leq& x_j^{\max}, \quad j = 1, \dots, n
\mbox{\textcolor{blue}{\  \ \     (* Box constraints *)}}
\end{eqnarray}

where:
- \( x_j \) are the decision variables for \( j = 1, \dots, n \).
- \( c_j \) are the cost coefficients for \( j = 1, \dots, n \).
- \( a_{i j} \) are the coefficients in the constraints for \( i = 1, \dots, m \) and \( j = 1, \dots, n \).
- \( b_i \) are the right-hand side values for \( i = 1, \dots, m \).
- \( x_j^{\min} \) and \( x_j^{\max} \) define the lower and upper bounds on \( x_j \).

This problem can be rewritten in a compact matrix notation:

\[
\begin{aligned}
\text{Minimize} \quad & c^\top \mathbf{x} \\
\text{subject to} \quad & A \mathbf{x} \leq \mathbf{b}, \\
& \mathbf{x}^{\min} \leq \mathbf{x} \leq \mathbf{x}^{\max}.
\end{aligned}
\]

where:

\[
A \mathbf{x} \leq \mathbf{b} \quad \text{is explicitly:}
\]

\[
\begin{bmatrix}
a_{11} & a_{12} & \cdots & a_{1n} \\
a_{21} & a_{22} & \cdots & a_{2n} \\
\vdots & \vdots & \ddots & \vdots \\
a_{m1} & a_{m2} & \cdots & a_{mn}
\end{bmatrix}
\begin{bmatrix}
x_1 \\
x_2 \\
\vdots \\
x_n
\end{bmatrix}
\leq
\begin{bmatrix}
b_1 \\
b_2 \\
\vdots \\
b_m
\end{bmatrix}
\]

and the box constraints:

\[
\begin{bmatrix}
x_1^{\min} \\
x_2^{\min} \\
\vdots \\
x_n^{\min}
\end{bmatrix}
\leq
\begin{bmatrix}
x_1 \\
x_2 \\
\vdots \\
x_n
\end{bmatrix}
\leq
\begin{bmatrix}
x_1^{\max} \\
x_2^{\max} \\
\vdots \\
x_n^{\max}
\end{bmatrix}
\]

\section*{Time Complexity of Solving Linear Programs}

The time complexity of solving linear programs depends on the algorithm used:

\begin{itemize}
    \item The \emph{Simplex Method} is widely used in practice and often performs efficiently, but its worst-case time complexity is \(\mathcal{O}(2^n)\) for some pathological cases.
    \item \emph{Interior Point Methods} (such as the Karmarkar algorithm) solve LPs in polynomial time, specifically in \(\mathcal{O}(n^{3.5} L)\), where \( L \) is the input size.
    \item \emph{Ellipsoid Method} also provides polynomial-time complexity, but it is generally outperformed by interior point methods in practical applications.
\end{itemize}

For large-scale linear programs, modern solvers such as CPLEX \cite{Bonami22}, Gurobi\cite{Gurobi2024},  CBC \cite{Forrest2005}, and LPSOLVE \cite{Berkelaar21} use highly optimized implementations combining simplex and interior-point methods.

\subsection{Geometrical Aspects of Linear Programming}

The geometric characteristics of linear programming serve as a foundational tool for comprehending non-linear programming and multiobjective optimization, the details of which will be addressed in subsequent discussions.

A set \( A \subseteq \mathbb{R}^n \) is called \emph{convex} if, for every pair of points \( x, y \in A \), the entire line segment between them is also in \( A \). Formally, this means:

\[
\forall x, y \in A, \quad \forall \lambda \in [0,1], \quad (1 - \lambda) x + \lambda y \in A.
\]

\begin{lemma}
The feasible subset in linear programming is a convex subset of $\mathbb{R}^n$.
\end{lemma}

\begin{proof}
Each inequality constraint of a linear program (LP) defines a convex set as its feasible region. It is straightforward to show that the intersection of two convex sets is also convex. Specifically, let \( A \) and \( B \) be two convex sets. By definition, for any two points in \( A \), the line segment connecting them is fully contained in \( A \), and similarly, for any two points in \( B \), the connecting line segment lies entirely within \( B \). Now, consider two points in the intersection \( A \cap B \). Since these points belong to both \( A \) and \( B \), the entire line segment between them must be contained in both \( A \) and \( B \), and therefore, in their intersection \( A \cap B \). This confirms that \( A \cap B \) is also convex.
\end{proof}

\begin{lemma}
Given that the constraints and the objective function of a linear program (LP) are linearly independent, the optimal solution is unique and occurs at the boundary defined by the active constraints.
\end{lemma}

\begin{proof}
Consider a standard linear program in \( n \) variables:
\[
\min c^\top x \quad \text{subject to} \quad Ax \leq b, \quad x \geq 0.
\]
where \( A \in \mathbb{R}^{m \times n} \) and \( c \in \mathbb{R}^{n} \).

\textbf{Step 1: Uniqueness of the Optimizer}

Since the constraints and the objective function are assumed to be linearly independent and all variables have a bounded range, the feasible region is a convex polyhedron, and the objective function is a linear function over this polyhedron. Linear independence of constraints means that at most \( n \) constraints can be simultaneously active at a given vertex of the feasible region.

If the LP has an optimal solution, it must occur at a vertex (extreme point) of the feasible region. Suppose there were two distinct optimal solutions \( x^* \) and \( y^* \) such that \( c^\top x^* = c^\top \mathbf{y}^* \). Then, any convex combination of these points,
\[
z^\lambda = \lambda x^* + (1 - \lambda) \mathbf{y}^*, \quad \lambda \in (0,1),
\]
would also be optimal since
\[
c^\top z^\lambda = \lambda c^\top x^* + (1 - \lambda) c^\top \mathbf{y}^* = c^\top x^*.
\]

However, this contradicts the assumption that the constraints and the objective function are linearly independent. If multiple points satisfy the optimality condition, then the system would be underdetermined, violating the assumption that at most \( n \) independent active constraints define a unique solution. Thus, the optimizer must be unique.

\textbf{Step 2: Occurrence at the Boundary of Active Constraints}

Since the LP is defined over a convex polyhedron, the optimal solution must lie on the boundary of the feasible region, specifically at a vertex where a set of constraints becomes active. If the optimal solution were in the interior of the feasible region, there would exist a direction in which the objective function could be further decreased, contradicting optimality. 

By the linear independence assumption, the number of active constraints at the optimizer must be exactly equal to the number of variables \( n \), ensuring a unique solution.

\end{proof}

\subsection{Graphical Solution of LP}

We show by means of a small example how to graphically solve an LP.
We solve the following linear programming problem:

\begin{align*}
\text{Maximize } & Z = x_1 + x_2 \\
\text{subject to} \quad
& x_2 \leq 4 - 2x_1, \\
& x_2 \leq 2 - \frac{1}{2}x_1, \\
& x_1, x_2 \geq 0.
\end{align*}

The feasible region is illustrated in Figure~\ref{fig:lp_graph}.

\begin{figure}
    \centering
    \begin{tikzpicture}[scale=1]

    \draw[->] (0,0) -- (5.5,0) node[right] {$x_1$};
    \draw[->] (0,0) -- (0,5.5) node[above] {$x_2$};

    \draw[blue, thick] (2,0) -- (0,4) node[above left] {$x_2 = 4 - 2x_1$};
    \draw[green, thick] (0,2) -- (4,0) node[above right] {$x_2 = 2 - \frac{1}{2}x_1$};

    \fill[yellow!30, opacity=0.5] (0,0) -- (0,2) -- (4/3,4/3) -- (2,0) -- (0,0);

    \fill[red] (0,2) circle (2pt) node[left] {$(0,2)$};
    \fill[red] (4/3,4/3) circle (2pt) node[above right] {$(\frac{4}{3},\frac{4}{3})$};
    \fill[red] (2,0) circle (2pt) node[below] {$(2,0)$};
    \fill[red] (4,0) circle (2pt) node[below] {$(4,0)$};

    \draw[dotted, gray] (0,1) -- (5,1);
    \draw[dotted, gray] (0,2) -- (5,2);
    \draw[dotted, gray] (0,3) -- (5,3);
    \draw[dotted, gray] (0,4) -- (5,4);
    \draw[dotted, gray] (1,0) -- (1,5);
    \draw[dotted, gray] (2,0) -- (2,5);
    \draw[dotted, gray] (3,0) -- (3,5);
    \draw[dotted, gray] (4,0) -- (4,5);
    \draw[dotted, gray] (5,0) -- (5,5);
    
    \draw[dashed, red] (0,1) -- (1,0);
    \draw[dashed, red] (0,2) -- (2,0);
    \draw[dashed, red] (0,3) -- (3,0);

\end{tikzpicture}
    \caption{Graphical solution of an LP. The feasible region is indicated in yellow. The isoheightlines (levels) of the objective funtion are indicated in red dashed lines.}
    \label{fig:lp_graph}
\end{figure}

\subsection{Linearization Techniques}
As in linear programming, also in mathematical programming, it is common to use existing solvers, such as CPLEX \cite{Bonami22}, GUROBI \cite{Gurobi2024}, LPSOLVE (for C/C$++$) \cite{Berkelaar21}, CBC \cite{Forrest2005}\footnote{An open-source MILP solver that is used by default in the PuLP library for Python.}, or SHOT, to find the solution of a problem that is presented to the solver software in a standardized form \cite{Lundell19}. There are also front-ends to these solvers, such as Google OR-Tools, PuLP (Python), and GAMS. Furthermore, multicriteria optimization libraries, such as DESDEO \cite{Misitano21}, use solvers as their backbone.
The optimization model must be prepared in a form easily solved by these models.

Even though MILP and ILP problems belong to the class of NP-hard problems and thus, in the worst case, require extensive computational resources to be solved exactly, there exist efficient solvers for them. These solvers make use of branch-and-bound or branch-and-cut techniques that exploit linear relaxations of integer problems to produce lower bounds and/or prune the search space of discrete variables. A key principle here is that linear programming problems with continuous variables can be solved efficiently.

\textit{Linearization techniques} are widely used in the practice of solving mathematical programming problems. They are defined as techniques to reformulate nonlinear mixed-integer problem formulations as ILP problems, often by introducing additional auxiliary variables. Sometimes these techniques are already integrated (hard-coded) in the solvers; see, e.g., \cite{Bonami22}.

Let us list some of the most common linearization techniques in the following.

\paragraph{Min-max problems.} Problems of the type 
\[
\min \max(f_1(\mathbf{x}), f_2(\mathbf{x}))
\]
use the nonlinear \(\max(\cdot)\) operation. They can be linearized by introducing an additional decision variable \(\alpha\) and stating the equivalent linear problem:
\[
\min \alpha, \quad \text{subject to} \quad f_1(\mathbf{x}) \leq \alpha, \quad f_2(\mathbf{x}) \leq \alpha.
\]

\paragraph{Products of binary variables.} If a product of two binary variables, say \( x_i x_j \), occurs in one of the terms of a QP, it can be replaced by a single auxiliary variable \( q_{ij} \) and we need to add the constraints:
\begin{align}
    q_{ij} &= x_i x_j, \quad x_i, x_j \in \{0,1\}
\end{align}
To linearize this product, introduce an auxiliary binary variable \( q_{ij} \) and enforce the following constraints:
\begin{align}
    q_{ij} &\leq x_i, \\
    q_{ij} &\leq x_j, \\
    q_{ij} &\geq x_i + x_j - 1.
\end{align}
These constraints ensure that \( q_{ij} = 1 \) if and only if both \( x_i \) and \( x_j \) are 1 while keeping the formulation linear.

\paragraph{Disjunctive constraints.} If we have two disjunctive constraints, meaning we require that  \( g_a(\mathbf{x}) \leq 0 \) or \( g_b(\mathbf{x}) \leq 0 \) (but not necessarily both), we can use the so-called \textit{large-constant trick} by introducing a binary decision variable \( x_{ab} \) and relaxing one of the two constraints by adding a large constant \( L \). The reformulated problem is:
\[
g_a(\mathbf{x}) - x_{ab} L \leq 0, \quad g_b(\mathbf{x}) - (1 - x_{ab}) L \leq 0.
\]
The constant \( L \) should however be chosen wisely to avoid ill-conditioning.

\paragraph{No-subtour constraints.} Formulating route planning problems (with time windows) with a road network given as a distance matrix using linear constraints is complex, requiring the entire route's connectivity. However, linearization techniques for this challenge have been proposed. Here we refer to \cite{Guelmez2024} for a detailed model solution.

\subsection{Multiobjective Optimization}
All optimization problems and mathematical programming problem classes can be generalized to multiobjective optimization problem classes by stating multiple objective functions:

\begin{equation}
\label{eq:genmulopt} f_1(\mathbf{x}) \rightarrow \min, \dots,
f_{m}(\mathbf{x}) \rightarrow \min, \ \ \mathbf{x} \in \mathcal{X}
\end{equation}
At this point in time it is not clear, how to deal with situations
with conflicting objectives, e.g. when it is not possible to minimize all objective functions simultaneously.
Note that the
problem definition does not yet prescribe how to compare different
solutions. To discuss this we will introduce concepts from the
theory of ordered sets, such as the Pareto dominance relation.
A major part of these lecture notes will then be devoted to the treatise of multiobjective optimization.

Before proceeding in this direction, it is, however, important to note that many difficulties of solving single objective optimization problems are inherited by the more general class of multiobjective optimization problems. We will therefore first summarize these.

\section{Problem Difficulty in Optimization}

The way in which problem difficulty is defined in continuous unconstrained optimization differs widely from the concepts typically referred to in discrete optimization. This is why we look at these two classes separately. Thereafter we will show that discrete optimization problems can be formulated as constrained continuous optimization problems, or, referring to the classification scheme in Table \ref{tab:classificationmp}, as non-linear programming problems.

\subsection{Problem Difficulty in Continuous Optimization}
In continuous optimization, the metaphor of a optimization landscape is often used in order to define problem difficulty. Unlike talking about a function, when using the term (search) {\it landscapes} one explicitly requires that the search space be equipped with a neighborhood structure, which could be a metric or a topology. This topology is typically the standard topology in $\mathbb{R}^n$ and as a metric  typically the Euclidean metric is used.


As we shall discuss with more rigor in Chapter \ref{chap:opticond}, this gives rise to definitions such as {\em local optima}, which are points that cannot be improved by replacing them with neighboring points. For many optimum-seeking algorithms, it is difficult to escape from such points or find a good direction (in case of plateaus). If local optima are not also global optima, the algorithm might return suboptimal solutions.

Problems with multiple local optima are called \emph{multimodal optimization problems}, whereas a problem with only a single local optimum is called \emph{unimodal optimization problem}. Multimodal optimization problems are, in general, considered more difficult to solve than unimodal optimization problems. However, in some cases, unimodal optimization problems can also be very difficult. For instance, in the case of large neighborhoods, it can be hard to find the neighbor that improves a point.

The examples of continuous optimization problems are given in Figure. The problem TP2 is difficult due to discontinuities. The TP3 function has only one local optimum (unimodal) and no discontinuities; therefore, it can be expected that local optimization can easily solve this problem. The highly multimodal problems are given in TP5 and TP6.
\begin{figure}
\centerline{\includegraphics[width=0.7\textwidth]{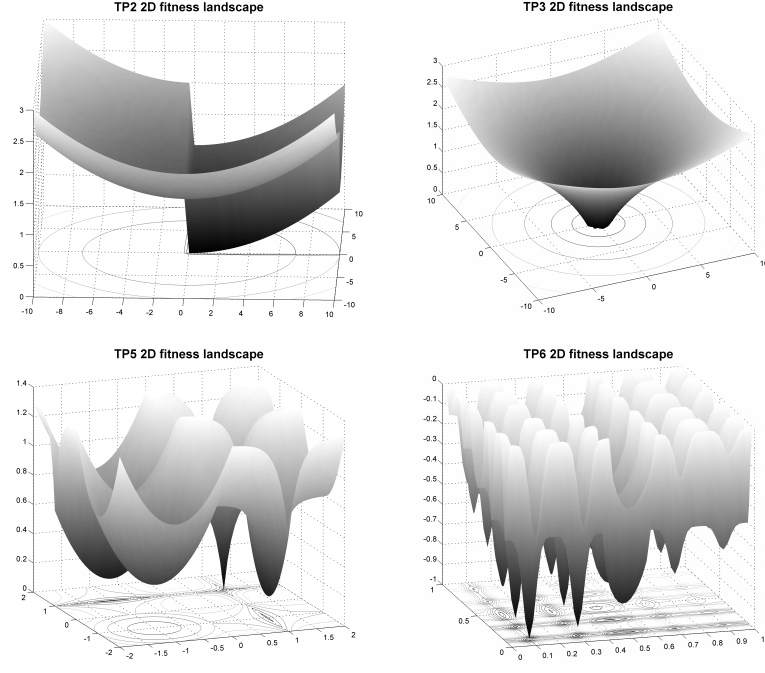}}
\caption{\label{fig:multimodal} Examples of continuous optimization problems.}
\end{figure}

Another difficulty is imposed by constraints. In constrained optimization problems, optima can be located at the boundary of the search space and they can give rise to disconnected feasible subspaces.
Again, connectedness is a property that requires the definition of neighborhoods. The definition of a continuous path can be based on this, which is used again to define connectivity.
The reason why disconnected subspaces make problems hard to solve is, similar to the multimodal case, that barriers are introduced in these problems that might prevent optimum seeking algorithms that use strategies of gradual improvement to find the global optimum.

Finally, discontinuities and ruggedness of the landscape make problems difficult to solve for many solvers. Discontinuities are abrupt changes of the function value in some neighborhood. In particular, these cause difficulties for optimization methods that assume the objective function to be continuous, that is, they assume that similar inputs cause similar outputs.
A common definition of continuity is that of Lipschitz continuity.
\begin{definition}
Let $d(x,y)$ denote the Euclidean distance between two points in the search space. Then function $f$ is Lipschitz continuous, if and only if
$$| f(x) - f(y)| < k d(x, y) \mbox{ for some } k>0.$$
\end{definition}
For instance the work by Ritter and Novak \cite{RN96} has clarified that Lipschitz continuity alone is not sufficient to guarantee that a problem is easy to solve. However, continuity can be exploited to guarantee that a region has been explored sufficiently and therefore a small Lipschitz constant has a damping effect on the worst-case time complexity for continuous global optimization, which, even given Lipschitz continuity, grows exponentially with the number of variables involved \cite{RN96}. In cases where we omit continuity assumptions, the time complexity might even grow super-exponentially. Here complexity is defined as the number of function evaluations it takes to get a solution that has a distance of $\epsilon$ to the global optimum, and it is assumed that the variables are restricted in a closed interval range.

As indicated above, the number of optimization variables is another source of difficulty in continuous optimization problems. In particular, if $f$ is a black box function it is known that even for Lipschitz continuous problems the number of required function evaluations for finding a good approximation to the global optimum grows exponentially with the number of decision variables. This result is also referred to as the {\em curse of dimensionality}.

Again, a word of caution is in order: The fact that a problem is low-dimensional or even one-dimensional in isolation does not say something about its complexity. \emph{Kolmogorov$\,'$s superposition theorem} shows that every continuous multivariate  function can be represented by a one dimensional function, and it is therefore often possible to re-write optimization problems with multiple variables as one-dimensional optimization problems.

 Besides continuity assumptions, also differentiability of the objective function and constraints, convexity, and mild forms of nonlinearity (as given in convex quadratic optimization), as well as limited interaction between variables can make a continuous problem easier to solve. The degree of interaction between variables is given by the number of variables in the term of the objective function: Assume that it is possible to (re)write the optimization problem in the form
$\sum_{i=1}^{n} f_i(x_{i_1}, \dots, x_{i_{k(i)}}) \rightarrow \max$, then the value of $k(i)$ is the degree of interaction in the $i$-th component of the objective function. In case of continuous objective functions it can be shown that problems with a low degree of interaction can be solved more efficiently in terms of worst-case time complexity\cite{RN96}. One of the reasons why convex quadratic problems can be solved efficiently is that, given the Hessian matrix, the coordinate system can be transformed by simple rotation in such a way that these problems become decomposable, i.e., $k(i)$ is bounded by $1$.


\subsection{Problem Difficulty in Combinatorial Optimization}
Many optimization problems in practice, such as scheduling problems, subset selection problems, and routing problems, belong to the class of \emph{combinatorial optimization problems} and, as the name suggests, they look in some sense for the best combination of parts in a solution (e.g. selected elements of a set, traveled edges in a road network, switch positions in a Boolean network).
Combinatorial optimization problems are problems formulated on (large) finite search spaces. In the classification scheme in Table \ref{tab:classificationmp} they belong to the IP and ILP classes. Although combinatorial optimization problems are originally not always formulated in search spaces with integer decision variables, most combinatorial optimization problems can be transformed to equivalent IP and ILP formulations with binary decision variables. For the sake of brevity, the following discussion will focus on binary unconstrained problems.
Most constrained optimization problems can be transformed to equivalent unconstrained optimization problems by simply assigning a sufficiently large ('bad') objective function value to all infeasible solutions.

A common characteristic of many combinatorial optimization problems is that they have a concise (closed-form) formulation of the objective function and the objective function (and the constraint functions) can be computed efficiently.

Having said this, a combinatorial optimization problem can be defined by means of a pseudo-boolean objective function, i.e.
$f: \{0, 1\}^n \rightarrow \mathbb{R}$ and stating the goal $f(\mathbf{x}) \rightarrow \min$.
Theoretical computer science has developed a rich theory on the complexity of decision problems. A \emph{decision problem} is the problem of answering a query on input of size $n$ with the answer being either \texttt{yes} or \texttt{no}. In order to relate the difficulty of optimization problems to the difficulty of decision problems, it is beneficial to formulate the so-called decision versions of optimization problems.
\begin{definition}
Given an combinatorial optimization problem of the form $f(\mathbf{x}) \rightarrow \max$ for $\mathbf{x} \in \{0,1\}^n$ its \emph{decision version} is defined as the query:
\begin{equation}
\exists \mathbf{x} \in \{0,1\}^n: f(\mathbf{x}) \leq k
\label{eq:Exists}
\end{equation}
 for a given value of $k \in \mathbb{R}$.
 \end{definition}

\subsubsection{NP hard combinatorial optimization problems}

The Ukrainian mathematician Leonid Levin and the American computer scientist Stephen Cook independently developed the foundation of computational complexity theory in the 1970s \cite{Shasha98}. Their work provided a classification framework for combinatorial optimization problems based on computational feasibility. They introduced the concept of NP-completeness, distinguishing between problems that require an approach close to complete enumeration (known as "perebor" in Russian) and those that can be solved efficiently, such as in polynomial time. Cook’s 1971 theorem established Boolean Satisfiability (SAT) as the first NP-complete problem \cite{Cook71}, while Levin independently identified a set of NP-complete problems in his 1973 paper on universal search problems\cite{Levin73}. Their contributions laid the groundwork for modern complexity theory and the fundamental P vs NP question.

A decision problem is said to belong to the class P if there exists an algorithm on a Turing machine\footnote{or any in any common programming language operating on infinite memory and not using parallel processing and not assuming constant time complexity for certain infinite precision floating point operations such as the $\mbox{floor}$ function.} that solves it with a time complexity that grows at most polynomially with the size $n$ of the input. It belongs to the class NP if a candidate solution $\mathbf{x}$ of size $n$ can be verified ('checked') with polynomial-time complexity in $n$ (e.g., does it satisfy the formula $f(\mathbf{x}) \leq k$ or not). Obviously, the class NP subsumes the class P, but P does not necessarily subsume NP. In fact, the question whether P subsumes NP is the open problem often discussed in theoretical computer science, known as the problem 'P$=$NP?'. Under the assumption 'P $\neq$ NP', that is that P does not include NP, it is meaningful to introduce the complexity class of NP complete problems:

\begin{definition}
A decision problem $D$ is \emph{NP-complete} if it belongs to NP, and every instance of another problem in NP can be transformed into an instance of $D$ through polynomial-time reduction.
\end{definition}

If any NP-complete problem could be solved with polynomial-time complexity, then all problems in NP have polynomial time complexity.
Numerous decision forms of optimization issues are considered NP-complete. The class of NP-hard problems is closely associated with NP-complete problems.

\begin{definition}[NP hard]
A problem is considered NP-hard if every problem in NP problem can be transformed into it within polynomial time.
\end{definition}

To demonstrate NP hardness, it is enough to reduce any single NP-complete problem to the problem in question.
Moreover, that a problem is NP-hard does not imply that it is in NP. Moreover, given that any NP hard problem could be solved in polynomial time, then all problems in NP could be solved in polynomial time, but not vice versa.

Numerous combinatorial optimization problems are categorized as NP hard due to their decision problems being part of the NP complete class. Some examples of NP hard optimization problems include the knapsack problem, the traveling salesman problem, and integer linear programming (ILP). An integer programming reduction to the familiar problem of 3SAT (boolean satisfiability with disjunctive clauses containing up to 3 boolean variables) is simple, considering that for $x_1, x_2 \in \{0,1\}^2$, the logical OR can be expressed as the constraint $x_1 + x_2 \geq 1$, the logical AND as the constraint $x_1 + x_2 \geq 2$, and the logical NOT as $x_2 = 1 - x_1$.

In continuous mathematical programming, it is established that linear programming can be solved in polynomial time, while quadratic programming may already be NP-hard. The distinction is made between strictly convex quadratic programming (which can be solved in polynomial time) and non-convex quadratic programming with just one negative eigenvalue of the quadratic form matrix \cite{PV91}.

At this point in time, despite considerable efforts of researchers, no polynomial time algorithms are known for NP complete problems, and thus also not for NP hard problems. As a consequence, relying on currently known algorithms, the computational effort to solve NP complete (NP hard) problems grows (at least) exponentially with the size $n$ of the input.

The fact that a given problem instance belongs to a class of complete problems NP does not mean that this instance itself is difficult to solve. Firstly, exponential growth is a statement about \emph{worst case} time complexity and thus gives an upper bound for the time complexity that holds for all instances of the class. It might well be the case that for a given instance the worst case is not binding. Often certain structural features such as \emph{bounded tree width} reveal that an instance belongs to an easier to solve subclass of an NP complete problem. Moreover, exponential growth might occur with a small growth rate, and problem sizes relevant in practice might still be solvable in an acceptable time. The area of \emph{Parameterized Complexity Theory} focuses on deriving findings along these lines, such as identifying the parameters and configurations of the problem that allow it to be solved by a polynomial time or rapid algorithm.

\subsubsection{Continuous vs. discrete optimization}

Given that some continuous versions of mathematical programming problems belong to easier to solve problem classes than their discrete counterparts one might ask the question whether integer problems are essentially more difficult to solve than continuous problems.

Optimization problems on binary input spaces can indeed be transformed into quadratic optimization problems through the following method: Given an integer programming problem with binary decision variables $b_i \in \{0, 1\}$, $i = 1, \dots, n$, this can be rephrased as a quadratic programming problem with continuous decision variables $x_i \in \mathbb{R}$ by adding the constraints $(x_i) (1-x_i) = 1$ for $i = 1, \dots, n$. 
It is clear that continuous optimization issues cannot always be represented as discrete optimization issues. Despite this, some maintain that all problems solved by digital computers are fundamentally discrete and that infinite precision is rarely required in practice. Assuming that operations with infinite precision can be performed in constant time could produce unusual results. For instance, polynomial time algorithms could be formulated for NP-complete problems if the floor function could be computed with infinite precision in polynomial time. Nevertheless, such algorithms are not feasible on a von Neumann architecture with finite precision arithmetic. This scenario underlines the need to consider the computational model alongside complexity outcomes. A noteworthy non-traditional computational model beyond the Turing machine or von Neumann architecture is quantum computing. Certain algorithms, like prime factorization, can be solved in polynomial time on quantum computers, whereas the existence of polynomial-time algorithms  is not yet known on classical von Neumann computers or Turing machines \cite{Shor99}.

Finally, in times of growing amounts of decision data, one should not forget that even guarantees of polynomial-time complexity can be insufficient in practice. Accordingly, there is a growing interest for problem solvers that require only subquadratic running time. Similarly to the construction of the class of complete NP problems, theoretical computer scientists have constructed a definition of the class of 3SUM-complete problems.  For this class up to date only slightly better than quadratic running time algorithms are known, and until very recently it was believed that the quadratic time complexity barrier cannot be surpassed \cite{Freund17}. 
A prominent problem from the domain of mathematical programming that belongs to this group is the {\em linear satisfiability problem}, i.e. the problem of whether a set of $r$ linear inequality constraints formulated on $n$ continuous variables can be satisfied \cite{Jef95}.

\section[Pareto dominance and incomparability]{Pareto dominance}
A fundamental problem in multicriteria optimization and decision
making is to compare solutions w.r.t. different, possibly
conflicting, goals. Before we lay out the theory of orders in a more
rigorous manner, we will introduce some fundamental concepts by
means of a simple example.

Consider the following decision problem: We have to select one car
from the following set of cars:
\begin{table}[ht]
\begin{tabular}{l|l|l|l|l}
 Criterion &  Price [kEuro] & Maximum Speed [km/h]& length [m] & color\\
 \hline
 VW Beetle & 3      & 120 &    3.5  &  red  \\
 Ferrari   & 100    & 232 &    5    &  red  \\
 BMW       & 50     & 210 &    3.5  &  silver \\
 Lincoln   & 60     & 130 &     8    &  white
\end{tabular}
\end{table}\\
For the moment, let us assume, that our goal is to minimize the
price and maximize speed and we do not care about other components.

In that case we can clearly say that the BMW outperforms the Lincoln
stretch limousine, which is at the same time more expensive and
slower then the BMW. In such a situation we can decide clearly for
the BMW. We say that the first solution {\em (Pareto) dominates} the
second solution. Note, that the concept of Pareto 
dominance is
named after Vilfredo Pareto, an italian economist and engineer who
lived from 1848-1923 and who introduced this concept for
multi-objective comparisons.

Consider now the case, that you have to compare the BMW to the VW
Beetle. In this case it is not clear how to make a decision, as the
beetle outperforms the BMW in the cost objective, while the BMW
outperforms the VW Beetle in the speed objective. We say that the
two solutions are {\em incomparable}. Incomparability is a very
common characteristic that occurs in so-called {\em partial ordered}
sets.

We can also observe, that the BMW is incomparable to the Ferrari,
and the Ferrari is incomparable to the VW Beetle. We say these three
cars form a set of mutually incomparable solutions. Moreover, we may
state that the Ferrari is incomparable to the Lincoln, and the VW
Beetle is incomparable to the Lincoln. Accordingly, also the VW
Beetle, the Lincoln and the Ferrari form a mutually incomparable
set.

Another characteristic of a solution in a set can be that it is {\em
non-dominated} or {Pareto optimal}. This means that there is no
other solution in the set which dominates it. The set of all
non-dominated solutions is called the {\em Pareto front}. It might
exist of only one solution (in case of non-conflicting objectives)
or it can even include no solution at all (this holds only for some
infinite sets). Moreover, the Pareto set is always a mutually
incomparable set. In the example this set is given by the VW Beetle,
the Ferrari, and the BMW.

An important task in multi-objective optimization is to identify the
Pareto front. Usually, if the number of objective is small and there
are many alternatives, this reduces the set of alternatives already
significantly. However, once the Pareto front has been obtained, a
final decision has to be made. This decision is usually made by
interactive procedures where the decision maker assesses trade-offs
and sharpens constraints on the range of the objectives. In the
subsequent chapters we will discuss these procedures in more detail.

Turning back to the example, we will now play a little with the
definitions and thereby get a first impression about the rich
structure of partially ordered sets in Pareto optimization:
 What happens if we add a further objective to the set of
objectives in the car-example? For example let us assume, we also
would like to have a very big car and the size of the car is
measured by its length! It is easy to verify that the size of the
non-dominated set increases, as now the Lincoln is also incomparable
to all other cars and thus belongs to the non-dominated set. Later
we will prove that introducing new objectives will always increase
the size of the Pareto front. On the other hand we may define a
constraint that we do not want a silver car. In this case the
Lincoln enters the Pareto front, since the only solution that
dominates it leaves the set of feasible alternatives. In general,
the introduction of constraints may increase or decrease Pareto
optimal solutions or its size remains the same.
\subsection{Formal Definition of Pareto Dominance}
\label{sec:pardom}
Next we want to define Pareto dominance and Pareto optimality for an arbitrary number of criteria.
To begin, let's first define Pareto dominance informally, followed by a formal definition, as it serves as\textbf{ the key principle for ranking solutions} in multiobjective optimization.
\begin{tcolorbox}[colback=gray!10,colframe=black!70,title={\textbf{Definition of Pareto Dominance (informal)}}]
    \textit{A solution A is said to Pareto dominate a solution B if A is better in at least one criterion and not worse in any other criterion.}
\end{tcolorbox}
From this definition follows the definition of Pareto Optimality.
\begin{tcolorbox}[colback=gray!10,colframe=black!70,title={\textbf{Definition of Pareto Optimality (informal)}}]
    \textit{A solution is said to be Pareto optimal if it is not Pareto dominated by any other solution, i.e. if the solution cannot be improved in some criterion without worsening another.}
\end{tcolorbox}

\bigskip

A formal and precise definition of Pareto dominance and Pareto optimality using mathematical notation is given as follows.  

We define a \emph{partial order}\footnote{Partial orders will be defined in detail in Chapter \ref{chap:orders}. For now, we can assume that it is an order where not all elements can be compared.} 
on the \emph{solution space} \( \mathcal{Y} = f(\mathcal{X}) \) by means of the Pareto dominance concept for vectors in \( \mathbb{R}^m \):

\bigskip
For any \( \mathbf{y}^{(1)} \in \mathbb{R}^m \) and \( \mathbf{y}^{(2)} \in \mathbb{R}^m \),  
\( \mathbf{y}^{(1)} \) \emph{dominates} \( \mathbf{y}^{(2)} \) (denoted as \( \mathbf{y}^{(1)} \prec_{\text{Pareto}} \mathbf{y}^{(2)} \))  
if and only if:

\begin{equation}
    \forall i \in \{1, \dots, m\}: \quad \mathbf{y}^{(1)}_i \leq \mathbf{y}^{(2)}_i \quad \text{and} \quad \exists i \in \{1, \dots, m\}: \mathbf{y}^{(1)}_i < \mathbf{y}^{(2)}_i.
\end{equation}

\bigskip
In the bi-criteria case, this definition simplifies to:

\begin{equation}
    \mathbf{y}^{(1)} \prec_{\text{Pareto}} \mathbf{y}^{(2)} \quad \Leftrightarrow \quad 
    \big( y_1^{(1)} < y_1^{(2)} \wedge y_2^{(1)} \leq y_2^{(2)} \big) \vee 
    \big( y_1^{(1)} \leq y_1^{(2)} \wedge y_2^{(1)} < y_2^{(2)} \big).
\end{equation}

\bigskip
In addition to the Pareto dominance relation \( \prec_{\text{Pareto}} \), we define further comparison operators:

\begin{equation}
    \mathbf{y}^{(1)} \preceq_{\text{Pareto}} \mathbf{y}^{(2)} \quad \Leftrightarrow \quad 
    \mathbf{y}^{(1)} \prec_{\text{Pareto}} \mathbf{y}^{(2)} \vee \mathbf{y}^{(1)} = \mathbf{y}^{(2)}.
\end{equation}

\bigskip
Moreover, we say that \( \mathbf{y}^{(1)} \) is \emph{incomparable} to \( \mathbf{y}^{(2)} \)  
(denoted as \( \mathbf{y}^{(1)} || \mathbf{y}^{(2)} \)), if and only if:

\begin{equation}
    \mathbf{y}^{(1)} \npreceq_{\text{Pareto}} \mathbf{y}^{(2)} \wedge \mathbf{y}^{(2)} \npreceq_{\text{Pareto}} \mathbf{y}^{(1)}.
\end{equation}

\bigskip
For technical reasons, we also define \emph{strict} Pareto domination:  
\( \mathbf{y}^{(1)} \) \emph{strictly dominates} \( \mathbf{y}^{(2)} \) if:

\begin{equation}
    \forall i \in \{1, \dots, m\}: \quad y_i^{(1)} < y_i^{(2)}.
\end{equation}

For any compact subset of $R^m$, say $\cal Y$, there exists a
non-empty set of minimal elements w.r.t. the partial order $\preceq$
(cf. \cite{Ehr05}, page 29).  Minimal elements of this partial order are
called non-dominated points. Formally, we can define a non-dominated
set via: ${\cal{Y}}_N = \{ \bf y \in {\cal Y} | \nexists {\bf y}'
\in {Y}: {\bf y}' \prec_{Pareto}  {\bf y}\}$. Following a convention
by Ehrgott \cite{Ehr05} we use the index $N$ to distinguish between the
original set and its non-dominated subset.

Having defined the non-dominated set and the concept of Pareto
domination for general sets of vectors in $\mathbb{R}^m$, we can now
relate it to the optimization task:
The aim of Pareto optimization is to find the non-dominated set
${\cal Y}_N$ for ${\cal Y} = f({\cal X})$ the image of $\cal X$
under $f$, the so-called {\em Pareto front} of the multi-objective
optimization problem.

We define ${\cal X}_E$ as the inverse image of ${\cal Y}_N$, i.\,e.
${\cal X}_E = f^{-1}({\cal Y}_N)$ . This set will be called the {\em
efficient set} of the optimization problem. Its members are called
{\em efficient solutions}.

For notational convenience, we will also introduce an order (which
we call prePareto) on the decision space via ${\bf x}^{(1)}
\prec_{prePareto} {\bf x}^{(2)}$ $\Leftrightarrow$ $f({\bf x}^{(1)})
\prec_{Pareto} f({\bf x}^{(2)})$. Accordingly, we define ${\bf
x}^{(1)} \preceq_{prePareto} {\bf x}^{(2)}$ $\Leftrightarrow$
$f({\bf x}^{(1)}) \preceq_{Pareto} f({\bf x}^{2})$. Note, the
minimal elements of this order are the efficient solutions, and the
set of all minimal elements is equal to ${\cal X}_E$.

It is easy to derive some basic principles about the set of Pareto optimal solutions:
\begin{itemize}
\item The dimensionality of the Pareto front can at most be $m-1$. In other words, when viewing the Pareto front as a relation (table with rows for the solutions and columns for the objective function values), any particular column is functionally dependent on the set of the other columns, because otherwise if in all columns but one the objective function values would be the same for two particular solutions (rows), dominance of one solution by the other would follow.
    \item If two solutions, say A and B, are mutually non-dominated, they stay mutually non-dominated when an additional objective is introduced, because already solution A is better in some objective as compared to B, and solution B is better in some  other objective and this does not change when an additional objective is introduced.
    Hence: Given a multi-objective optimization problem, a Pareto optimal solution stays Pareto optimal if additional objective functions are added. However, additional solutions might become non-dominated. 
    \item Incorporating additional constraints into a multi-objective optimization problem can lead to a reduction or an increase of the size (cardinality) of the set of Pareto optimal solutions. The non-dominated set set has the potential to increase in size, as solutions previously dominated by certain feasible solutions may become non-dominated once those solutions become infeasible.
\end{itemize}

\section*{Exercises}
\begin{enumerate}

    \exercise{Effect of New and Deleted Solutions}{
        How does the introduction of a new solution influence the size of the Pareto set? What happens if solutions are deleted? Prove your results!}
        {ex:pareto-change}

    \exercise{Differences Between Objectives and Constraints}{
        Why are objective functions and constraint functions essentially different? Give examples of typical constraints and typical objectives in real-world problems!} Find examples of equality constraints in real-world problems? (Hint: think of optimization over the surface of geometrical objects, or using the laws of physics)
        {ex:obj-vs-constraints}

    \exercise{Examples of Multiobjective Decision Problems}{
        Find examples of decision problems with multiple, conflicting objectives! How is the search space defined? What are the constraints, what are the objectives? How do these problems classify with respect to the classification scheme of mathematical programming? Name some human-centric aspects of solving multiobjective optimization  problems? Find examples of objectives that are difficult to quantify and where subjectivity plays a role in decision making.}
        {ex:multiobj-examples}

\end{enumerate}

\part{Foundations}

\chapter{Orders and Pareto dominance}

\label{chap:orders}
The theory of ordered sets is an essential
analytical tool in multi-objective optimization and decision analysis.
Different types of order relations can be defined by means of axioms on binary relations and, if we restrict ourselves to vector spaces, also geometrically.
Next, we will first show how different types of orders are defined as binary relations that satisfy a certain axioms\footnote{Using here the term 'axiom' to refer to an elementary statement that is used to define a class of objects (as promoted by, for instance, Rudolf Carnap\cite{Car58}) rather than viewing them as self-evident laws that do not require proof (Euclid's classical view).}. We will highlight the key differences among common families of ordered sets: preorders, partial orders, linear orders, and cone orders.

This chapter begins with a review of binary relations, then defines axiomatic properties of pre-ordered sets, a broad category of ordered sets. We introduce partial and linear orders as specific types of pre-orders, highlighting their differences in terms of incomparability and optimization criteria. We explore compact visualization methods for finite ordered sets using Hasse diagrams. The chapter concludes with defining orders on vector spaces via cones, offering an intuitive visualization through Pareto domination.

\section{Preorders}

Orders can be introduced and compared in an elegant manner as binary
relations that obey certain axioms. Let us first review the
definition of a binary relation and some common axioms  that can be introduced to specify special subclasses of binary relations and  that are relevant in the context of ordered sets.

\begin{definition}
A {\em binary relation} $\mathcal{R}$ on some set $\mathcal {S}$ is
defined as a set of pairs of elements of $\mathcal{S}$, that is, a subset of $\mathcal{S} \times \mathcal{S} = \{ (\mathbf{x}^1, \mathbf{x}^2) | \, \mathbf{x}^1 \in \mathcal{S} \mbox{ and } \mathbf{x}^2 \in \mathcal{S} \}$. We write
$\mathbf{x}^1\mathcal{R}\mathbf{x}^2 \Leftrightarrow (\mathbf{x}^1,
\mathbf{x}^2) \in \mathcal{R}$.
\end{definition}

\begin{definition} Properties of binary relations

\label{def:binprop} $\mathcal{R}$ is {\em reflexive}
$\Leftrightarrow \forall \mathbf{x} \in \mathcal{S}:  \mathbf{x}
\mathcal{R} \mathbf{x}$

$\mathcal{R}$ is {\em irreflexive} $\Leftrightarrow \forall
\mathbf{x} \in \mathcal{S}:  \neg \mathbf{x} \mathcal{R} \mathbf{x}$

$\mathcal{R}$ is {\em symmetric} $\Leftrightarrow \forall
\mathbf{x}^1, \mathbf{x}^2 \in \mathcal{S}:  \mathbf{x}^1
\mathcal{R} \mathbf{x}^2 \Leftrightarrow \mathbf{x}^2 \mathcal{R}
\mathbf{x}^1$

$\mathcal{R}$ is {\em antisymmetric} $\Leftrightarrow \forall
\mathbf{x}^1, \mathbf{x}^2 \in \mathcal{S}:  \mathbf{x}^1
\mathcal{R} \mathbf{x}^2 \wedge \mathbf{x}^2 \mathcal{R}
\mathbf{x}^1 \Rightarrow \mathbf{x}^1 = \mathbf{x}^2$

$\mathcal{R}$ is {\em asymmetric} $\Leftrightarrow \forall
\mathbf{x}^1, \mathbf{x}^2 \in \mathcal{S}:  \mathbf{x}^1
\mathcal{R} \mathbf{x}^2 \Rightarrow \neg (\mathbf{x}^2 \mathcal{R}
\mathbf{x}^1)$

$\mathcal{R}$ is {\em transitive} $\Leftrightarrow \forall
\mathbf{x}^1, \mathbf{x}^2, \mathbf{x}^3 \in \mathcal{S}:
\mathbf{x}^1 \mathcal{R} \mathbf{x}^2 \wedge \mathbf{x}^2
\mathcal{R} \mathbf{x}^3 \Rightarrow \mathbf{x}^1 \mathcal{R}
\mathbf{x}^3$
\end{definition}

\begin{example}
It is worthwhile to practise these definitions by finding examples
for structures that satisfy the aforementioned axioms. An example
for a reflexive relation is the equality relation on $\mathbb{R}$,
but also the relation $\leq$ on $\mathbb{R}$. A classical example
for a irreflexive binary relation would be marriage between two
persons. This relation is also symmetric. Symmetry is also typically
a characteristic of neighborhood relations -- if A is neighbor to B
then B is also neighbor to A.

Antisymmetry is exhibited by $\leq$, the standard order on
$\mathbb{R}$, as $x \leq y$ and $y \leq x$ entails $x=y$. Relations can be at the same time symmetric and antisymmetric: An example is the equality relation. Antisymmetry will also occur in the axiomatic definition of a partial order,
discussed later. Asymmetry, not to be confused with antisymmetry, is
somehow the counterpart of symmetry. It is also a typical
characteristic of strictly ordered sets -- for instance $<$ on
$\mathbb{R}$.

An example of a binary relation (which is not an order) that obeys
the transitivity axiom is the path-accessibility relation in
directed graphs. If node B can be reached from node A via a path,
and node C can reached from node B via a path, then also node C can
be reached from node A via a path.
\end{example}

\section{Preorders}
Next we will introduce preorders and some properties on them.
Preorders are a very general type of orders. Partial orders and
linear orders are preorders that obey additional axioms.  Beside
other reasons these types of orders are important, because the
Pareto order used in optimization defines a partial order on the
objective space and a pre-order on the search space.

\begin{definition} Preorder

\noindent A {\em preorder} (quasi-order) is a binary relation that
is both transitive and reflexive. We write $\mathbf{x}^1
\preceq_{pre} \mathbf{x}^2$ as shorthand for $\mathbf{x}^1
\mathcal{R} \mathbf{x}^2$. We call $(\mathcal{S}, \preceq_{pre})$ a
{\em preordered set}.
\end{definition}
In the sequel we use the terms preorder and order interchangeably.
%
Closely related to this definition are the following derived
notions:

\begin{definition} Strict preference

\noindent $\mathbf{x}^1 \prec_{pre} \mathbf{x}^2 :\Leftrightarrow
\mathbf{x}^1 \preceq_{pre} \mathbf{x}^2 \wedge \neg (\mathbf{x}^2
\preceq_{pre} \mathbf{x}^1) $
\end{definition}

\begin{definition} Indifference

\noindent $\mathbf{x}^1 \sim_{pre} \mathbf{x}^2 :\Leftrightarrow
\mathbf{x}^1 \preceq_{pre} \mathbf{x}^2 \wedge \mathbf{x}^2
\preceq_{pre} \mathbf{x}^1 $
\end{definition}

\begin{definition} Incomparability

\label{def:incompar} \noindent A pair of solutions $\mathbf{x}^1,
\mathbf{x}^2 \in \mathcal{S}$ is said to be incomparable, iff
neither $\mathbf{x}^1 \preceq_{pre} \mathbf{x}^2$ nor $\mathbf{x}^2
\preceq_{pre} \mathbf{x}^1$. We write $\mathbf{x}^1
\textcolor{blue}{||} \mathbf{x}^2$.
\end{definition}

Strict preference is irreflexive and transitive, and, as a
consequence asymmetric. Indifference is reflexive, transitive, and
symmetric. The properties of the incomparability relation we leave
for exercise.

Having discussed binary relations in the context of pre-orders, let
us now turn to characteristics of pre-ordered sets. One important characteristic of pre-orders in the context of optimization is that they are elementary structures on which minimal and maximal elements can be defined.
Minimal elements of a pre-ordered set are elements that are not
preceded by any other element.

\begin{definition} Minimal and maximal elements of an pre-ordered
set $\mathcal{S}$

\noindent $\mathbf{x}^1 \in \mathcal{S}$ is minimal, if and only if  $\mbox{ not exists } \mathbf{x}^2 \in \mathcal{S}$ such that $\mathbf{x}^2
\prec_{pre} \mathbf{x}^1$

\noindent $\mathbf{x}^1 \in \mathcal{S}$ is maximal, if and only if $\mbox{ not exists }
 \mathbf{x}^2 \in \mathcal{S}$ such that $\mathbf{x}^1
\prec_{pre} \mathbf{x}^2$
\end{definition}

\begin{proposition}
For every finite set (excluding here the empty set $\emptyset$) there exists
at least one minimal element and at least one maximal element.
\end{proposition}

For infinite sets, pre-orders with infinite many minimal (maximal) elements can be
defined and also sets with no minimal (maximal) elements at all, such as the natural numbers with the order $<$ defined on them, for which there exists no maximal element. Turning the argument around, one could elegantly define an infinite set as a non-empty set on which there exists a pre-order that has no maximal element.

In absence of any additional information the number of pairwise comparisons required to find all minimal (or maximal) elements of a finite pre-ordered set of size $|\mathcal{X}| = n$ is $\binom{n}{2} = \frac{(n-1) n}{2}$. This follows from the effort required to find the minima in the special case where all elements are mutually incomparable.

\section{Partial orders}
Pareto domination imposes a partial order on a set of criterion
vectors. The definition of a partial order is more strict than that
of a pre-order:

\begin{definition} Partial order

\noindent A {\em \bf partial order} is a preorder that is also
\emph{antisymmetric}. We call $(\mathcal{S},
\preceq_{partial})$ a {\em {\em partially ordered set} or {\em \bf
poset}}.
\end{definition}

As partial orders are a specialization of preorders, we can define
{\em strict preference} and {\em indifference} as before.
%
%
%
%
\noindent Note, that for partial orders two elements that are
indifferent to each other are always equal: $\mathbf{x}^1 \sim
\mathbf{x}^2 \Rightarrow \mathbf{x}^1 = \mathbf{x}^2$

To better understand the difference between pre-ordered sets and
posets let us illustrate it by means of two examples:

\begin{example}\

\noindent A pre-ordered set that is not a partially ordered set is
the set of complex numbers $\mathbb{C}$ with the following
precedence relation:
$$\forall (z_1, z_2) \in \mathbb{C}^2: z_1 \preceq z_2  :\Leftrightarrow |z_1| \leq
|z_2|.$$ It is easy to verify reflexivity and transitivity of this
relation. Hence, $\preceq$ defines a pre-order on $\mathbb{C}$.
However, we can easily find an example that proves that antisymmetry
does not hold. Consider two distinct complex numbers $z = -1$ and
$z' = 1$ on the unit sphere (i.e. with $|z| = |z'| = 1$. In this
case $z \preceq z'$ and $z' \preceq z$ but  $z \neq z'$ \qed
\end{example}

\begin{example}\

\noindent An example for a partially ordered set is the subset
relation $\subseteq$ on the power set\footnote{the power set of a
set is the set of all subsets including the empty set} $\wp(S)$ of
some finite set $S$. Reflexivity is given as $A \subseteq A$ for all
$A \in \wp(S)$. Transitivity is fulfilled, because $A \subseteq B$
and $B \subseteq C$ implies $A \subseteq C$, for all triples $(A, B,
C)$ in $\wp(S)\times \wp(S) \times \wp(S)$. Finally, antisymmetry is
fulfilled, since $A \subseteq B$ and $B\subseteq A$ implies $A=B$
for all pairs $(A,B) \in \wp(S)\times \wp(S)$ \qed
\end{example}

\begin{remark} In general the Pareto order on the search space is a
preorder which is not always a partial order in contrast to the
Pareto order defined on the objective space (that is, the Pareto
order is always a partial order).
\end{remark}

\section{Linear orders and anti-chains}

Perhaps the most well-known specializations of a partially ordered
sets are linear orders. Examples for linear orders are the $\leq$
relations on the set of real numbers or integers.
These types of orders play an important role in single criterion
optimization, while in the more general case of multiobjective
optimization we deal typically with partial orders that are not
linear orders.

\begin{definition}[Linear order]
\noindent A {\bf linear} (or:{\bf total) order} is a partial order
that satisfies also the {\em comparability} or {\em totality} axiom:
$\forall \mathbf{x}^1, \mathbf{x}^2: \in \mathcal{X}: \mathbf{x}^1
\preceq \mathbf{x}^2 \vee \mathbf{x}^2 \preceq \mathbf{x}^1$
\end{definition}

Totality is only axiom that distinguishes partial orders from linear orders. This also explains the name 'partial' order. The 'partiality'  essentially refers to
the fact that not all elements in a set can be compared, and thus,
as opposed to linear orders, there are incomparable pairs.

A linearly ordered set is also called a (also called \emph{chain}). The counterpart of the chain is the anti-chain:

\begin{definition}[Anti-chain]
\noindent A poset $(\mathcal{S},$ $\preceq_{partial})$ is said to be
an {\bf antichain}, iff: $\forall \mathbf{x}^1,  \mathbf{x}^2 \in
\mathcal{S}: \mathbf{x}^1 || \mathbf{x}^2$
\end{definition}

When looking at sets on which a Pareto dominance relation $\preceq$
is defined,  we encounter subsets that can be classified as
anti-chains and subsets that can be classified as linear orders, or
non of these two.
Examples of anti-chains are \emph{Pareto fronts}.

Subsets of ordered sets that form anti-chain play an important role in characterizing
the time complexity when searching for minimal elements, as the following recent result shows \cite{DKM09}:

\begin{theorem}[Finding minima of bounded width posets]
Given a poset\\ $(\mathcal{X},$ $\preceq_{partial})$, then its \emph{width} $w$ is defined the maximal size of a mutually non-dominated subset. Finding the minimal elements of a poset of size $n$ and width of size $w$ has a time complexity in $\Theta(w n)$ and an algorithm has been specified that has this time complexity.
\end{theorem}

In \cite{DKM09} a proof for this theorem is provided and efficient algorithms.

\section{Hasse diagrams}
One of the most attractive features of pre-ordered sets, and thus
also for partially ordered sets is that they can be graphically
represented. This is commonly done by so-called Hasse diagrams, named after
the mathematician Helmut Hasse (1898 - 1979). The advantage of these
diagrams, as compared to the graph representation of binary
relations is essentially that edges that can be deduced by transitivity are omitted.

For the purpose of description we need to introduce the {\bf covers}
relation:
\begin{definition}[Covers relation]
\noindent  Given two elements $\mathbf{x}^1$ and $\mathbf{x}^1$ from a poset
$(\mathcal{X}, \prec_{partial})$. Then $\mathbf{x}^2$ \emph{covers} $\mathbf{x}^1$, in
symbols $\mathbf{x}^1 \lhd \,\mathbf{x}^2$ :$\Leftrightarrow$
$\mathbf{x}^1 \prec_{partial} \mathbf{x}^2$ and $\mathbf{x}^1
\preceq_{partial} \mathbf{x}^3 \prec_{partial} \mathbf{x}^2$ implies
$\mathbf{x}^1 = \mathbf{x}^3$.
\end{definition}

One may also define the covers relation in more informal terms as:
${\bf x^2}$ covers ${\bf x^1}$ if and only if no element lies strictly between
${\bf x^1}$ and ${\bf x^2}$.

As an example, consider the covers relation on the linearly ordered
set $(\mathbb{N}, \leq)$. Here $\mathbf{x}^1$ $\lhd$ $\mathbf{x}^2$,
iff $\mathbf{x}^2 = \mathbf{x}^1 + 1$. Note, that for $(\mathbb{R}, \leq)$ the covers relation is the empty set.

Another example where the covers relation has a simple interpretation is the subset relation $\subseteq$.  In this example a set $A$ is covered by a set $B$, if and only if $B$ contains one additional element.
\begin{figure}
\newlength{\fullsetwidth}
\settowidth{\fullsetwidth}{$\{1,2,\}$}
\begin{center}
\begin{tikzpicture}[scale=1]
\node[draw, shape=circle] (v1234) at (0,0)  {$\begin{matrix}\{1, 2,\\3,4\}\end{matrix}$};
\node[draw, shape=circle] (v123) at (-3,-2) {\makebox[\fullsetwidth]{$\{1,2,3\}$}};
\node[draw, shape=circle] (v124) at (-1,-2) {\makebox[\fullsetwidth]{$\{1,2,4\}$}};
\node[draw, shape=circle] (v134) at (1,-2)  {\makebox[\fullsetwidth]{$\{1,3,4\}$}};
\node[draw, shape=circle] (v234) at (3,-2)  {\makebox[\fullsetwidth]{$\{2,3,4\}$}};
\node[draw, shape=circle] (v12) at (-5,-4)  {\makebox[\fullsetwidth]{$\{1,2\}$}};
\node[draw, shape=circle] (v13) at (-3,-4)  {\makebox[\fullsetwidth]{$\{1,3\}$}};
\node[draw, shape=circle] (v14) at (-1,-4)  {\makebox[\fullsetwidth]{$\{1,4\}$}};
\node[draw, shape=circle] (v23) at (1,-4)   {\makebox[\fullsetwidth]{$\{2,3\}$}};
\node[draw, shape=circle] (v24) at (3,-4)   {\makebox[\fullsetwidth]{$\{2,4\}$}};
\node[draw, shape=circle] (v34) at (5,-4)   {\makebox[\fullsetwidth]{$\{3,4\}$}};
\node[draw, shape=circle] (v1) at (-3,-6)   {\makebox[\fullsetwidth]{$\{1\}$}};
\node[draw, shape=circle] (v2) at (-1,-6)   {\makebox[\fullsetwidth]{$\{2\}$}};
\node[draw, shape=circle] (v3) at (1,-6)    {\makebox[\fullsetwidth]{$\{3\}$}};
\node[draw, shape=circle] (v4) at (3,-6)    {\makebox[\fullsetwidth]{$\{4\}$}};
\node[draw, shape=circle] (v0) at (0,-8)    {\makebox[\fullsetwidth]{$\emptyset$}};
\path[draw, ultra thick] (v1234) -- (v123) (v1234) -- (v124) (v1234) -- (v134) (v1234) -- (v234)
(v123) -- (v12) (v123) -- (v13) (v123) -- (v23)
(v124) -- (v12) (v124) -- (v14) (v124) -- (v24)
(v134) -- (v13) (v134) -- (v14) (v134) -- (v34)
(v234) -- (v23) (v234) -- (v24) (v234) -- (v34)
(v12) -- (v1) (v12) -- (v2) (v13) -- (v1) (v13) -- (v3)
(v14) -- (v1) (v14) -- (v4) (v23) -- (v2) (v23) -- (v3)
(v24) -- (v2) (v24) -- (v4) (v34) -- (v3) (v34) -- (v4)
(v1) -- (v0) (v2) -- (v0) (v3) -- (v0) (v4) -- (v0);
\end{tikzpicture}
\end{center}
\caption{\label{fig:subset}The Hasse Diagram for the set of all non-empty subsets partially ordered by means of $\subseteq$.}
\end{figure}
In Fig. \ref{fig:subset} the subset relation is summarized in a Hasse diagram. In
this diagram the cover relation defines the arcs.
A good description of the algorithm to draw a Hasse diagram has been
provided by Davey and Priestly (\cite{DP90}, page 11):

\begin{algorithm}[ht]
\begin{algorithmic}[1]
\STATE To each point $\mathbf{x} \in S$ assign a point
$p(\mathbf{x})$, depicted by a small circle with centre
$p(\mathbf{x})$
\STATE For each covering pair $\mathbf{x}^1$ and $\mathbf{x}^2$ draw
a line segment $\ell(\mathbf{x}^1, \mathbf{x}^2)$.
\STATE Choose the center of circles in a way such that:
\STATE \ \ whenever $\mathbf{x}^1 \lhd \mathbf{x}^2$, then
$p(\mathbf{x}^1)$ is positioned below $p(\mathbf{x}^2)$.
\STATE \ \ if $\mathbf{x}^3 \neq \mathbf{x}^1$ and $\mathbf{x}^3
\neq \mathbf{x}^2$, then the circle of $\mathbf{x}^3$ does not
intersect the line segment $\ell(\mathbf{x}^1, \mathbf{x}^2)$
\end{algorithmic}
\caption{Drawing the Hasse Diagram}
\end{algorithm}
There are many ways of how to draw a Hasse diagram for a given
order. Davey and Priestly \cite{DP90} note that diagram-drawing is 'as much an
science as an art'. Good diagrams should provide an intuition for
symmetries and regularities, and avoid crossing edges.

\section{Comparing ordered sets}
(Pre)ordered sets can be compared directly and on a structural
level. Consider the four orderings depicted in the Hasse diagrams of
Fig. \ref{fig:orders}.
It should be immediately clear, that the first two orders
($\preceq_1, \preceq_2$) on $X$ have the same structure, but they
arrange elements in a different way, while orders $\preceq_1$ and
$\preceq_3$ also differ in their structure. Moreover, it is evident
that all comparisons defined in $\prec_1$ are also defined in $\prec_3$,
but not vice versa (e.g. $c$ and $b$ are incomparable in
$\preceq_1$). The ordered set on $\preceq_3$ is an \emph{extension}
of the ordered set $\preceq_1$. Another extension of $\preceq_1$ is
given with $\preceq_4$.

Let us now define these concepts formally:
\begin{definition}[Order equality]
An ordered set $(X, \preceq)$ is said to be {\em equal} to an
ordered set $(X, \preceq')$, iff $\forall x, y \in X:
x \preceq y \Leftrightarrow x \preceq' y$.
\end{definition}

\begin{definition}[Order isomorphism]
An ordered set $(X', \prec')$ is said to be an {\em isomorphic} to
an ordered set $(X, \preceq)$, iff there exists a mapping $\phi: X
\rightarrow X'$ such that $\forall x, x' \in X: x \preceq x'
\Leftrightarrow \phi(x) \preceq' \phi(x')$. In case of two
isomorphic orders, a mapping $\phi$ is said to be an {\em order
embedding map} or {\em order isomorphism}.
\end{definition}

\begin{definition}[Order extension]
An ordered set $(X, \prec')$ is said to be an {\em extension} of an
ordered set $(X, \prec)$, iff $\forall x, x' \in X:  x \prec x'
\Rightarrow x \prec' x'$. In the latter case, $\prec'$ is said to be
compatible with $\prec$. A {\em linear extension} is an extension
that is totally ordered.
\end{definition}

Linear extensions play a vital role in the theory of multi-objective
optimization. For Pareto orders on continuous vector spaces linear
extensions can be easily obtained by means of any weighted sum
scalarization with positive weights.  In general, topological
sorting can serve as a means to obtain linear extensions. Both
topics will be dealt with in more detail later in this work. For
now, it should be clear that there can be many extensions of the
same order, as in the example of Fig. \ref{fig:orders}, where $(X,
\preceq_3)$ and $(X, \preceq_4)$ are both (linear) extensions of
$(X, \preceq_1)$.

Apart from extensions, one may also ask if the structure of an
ordered set is contained as a substructure of another ordered set.
\begin{definition}
Given two ordered sets $(X, \preceq)$ and $(X', \preceq')$.  A map
$\phi: X \rightarrow X'$ is called {\em order preserving}, iff
$\forall x, x' \in X: x \preceq x' \Rightarrow \phi(x) \preceq
\phi(x')$.
\end{definition}
Whenever $(X, \preceq)$ is an extension of $(X, \preceq')$ the
identity map serves as an order preserving map. An order embedding
map is always order preserving, but not vice versa.

There is a rich theory on the topic of partial orders and it is
still rapidly growing. Despite the simple axioms that define the
structure of the poset, there is a remarkably deep theory even on
finite, partially ordered sets.
%
%
 The number of ordered sets that can be
defined on a finite set with $n$ members, denoted with $s_n$,
evolves as
\begin{equation}
\{s_n\}_1^\infty = \{1, 3, 19, 219, 4231, 130023, 6129859,
431723379, \dots \}
\end{equation}
and the number of equivalence classes, i.e. classes that contain
only isomorphic structures, denoted with $S_n$, evolves as:
\begin{equation}
 \{S_n\}_1^\infty = \{1, 2, 5, 16, 63, 318, 2045, 16999, ...\}
\end{equation}. See Finch \cite{Fin03} for both of these results.
This indicates how rapidly the structural variety of orders grows
with increasing $n$. Up to now, no closed form expressions for the
growth of the number of partial orders are known \cite{Fin03}.
%

\begin{figure}[t]
\psfrag{O1:}{$(X, \preceq_1)$} \psfrag{O2:}{$(X, \preceq_2)$}
\psfrag{O3:}{$(X, \preceq_3)$} \psfrag{O4:}{$(X, \preceq_4)$}
\begin{center}
\begin{subfigure}{0.24\textwidth}
	\begin{tikzpicture}[scale=0.5]
	\tikzstyle{every node} = [draw, black, shape=circle];
	\path (0,2) node[draw=none] {$\leq_1$};
	\path (0,0) node (A) {$a$};
	\path (-1,-2) node (B) {$b$};
	\path (-1,-4) node (C) {$c$};
	\path (1,-2) node (D) {$d$};
	\path[draw, ultra thick] (D) -- (A) -- (B) -- (C);
	\end{tikzpicture}
\end{subfigure}
\begin{subfigure}{0.24\textwidth}
	\centering
	\begin{tikzpicture}[scale=0.5]
	\tikzstyle{every node} = [draw, black, shape=circle];
	\path (0,2) node[draw=none] {$\leq_2$};
	\path (0,0) node (A) {$c$};
	\path (-1,-2) node (B) {$b$};
	\path (-1,-4) node (C) {$a$};
	\path (1,-2) node (D) {$d$};
	\path[draw, ultra thick] (D) -- (A) -- (B) -- (C);
	\end{tikzpicture}
\end{subfigure}
\begin{subfigure}{0.24\textwidth}
	\centering
	\begin{tikzpicture}[scale=0.5]
	\tikzstyle{every node} = [draw, black, shape=circle];
	\path (0,2) node[draw=none] {$\leq_3$};
	\path (0,0) node (A) {$a$};
	\path (0,-2) node (B) {$b$};
	\path (0,-4) node (C) {$c$};
	\path (0,-6) node (D) {$d$};
	\path[draw, ultra thick] (A) -- (B) -- (C) -- (D);
	\end{tikzpicture}
\end{subfigure}
\begin{subfigure}{0.24\textwidth}
	\centering
	\begin{tikzpicture}[scale=0.5]
	\tikzstyle{every node} = [draw, black, shape=circle];
	\path (0,2) node[draw=none] {$\leq_4$};
	\path (0,0)  node (A) {$a$};
	\path (0,-2) node (B) {$c$};
	\path (0,-4) node (C) {$b$};
	\path (0,-6) node (D) {$d$};
	\path[draw, ultra thick] (A) -- (B) -- (C) -- (D);
	\end{tikzpicture}
\end{subfigure}	
\end{center}
\caption{\label{fig:orders} Different orders over the set $X=\{a,b,c,d\}$}
\end{figure}

%
\section{Cone orders}
There is a large class of partial orders on $\mathbb{R}^m$ that can be defined geometrically by means of cones. In particular the so-called \emph{cone orders} belong to this class. Cone
orders satisfy two additional axioms. These are:

\begin{definition}[Translation invariance]
Let $\mathcal{R} \in \mathbb{R}^m \times \mathbb{R}^m$ denote a binary relation on $\mathbb{R}^m$.
Then $\mathbb{R}$ is translation invariant, if and only if for all $\mathbf{t} \in \mathbf{R}^m$, $\mathbf{x}^1 \in \mathbb{R}^m$ and $\mathbf{x}^2 \in \mathbb{R}^m$: $\mathbf{x}^1 \mathcal{R} \mathbf{x}^2$, if and only if $(\mathbf{x}^1 + \mathbf{t}) \mathcal{R} (\mathbf{x}^2 + \mathbf{t})$.
\end{definition}

\begin{definition}[Multiplication invariance]
Let $\mathcal{R} \in \mathbb{R}^m \times \mathbb{R}^m$ denote a binary relation on $\mathbb{R}^m$.
Then $\mathbb{R}$ is multiplication invariant, if and only if for all $\alpha \in \mathbf{R}$, $\mathbf{x}^1 \in \mathbb{R}^m$ and $\mathbf{x}^2 \in \mathbb{R}^m$: $\mathbf{x}^1 \mathcal{R} \mathbf{x}^2$, if and only if $(\alpha \mathbf{x}^1) \mathcal{R} (\alpha \mathbf{x}^2)$.
\end{definition}
We may also define these axioms on some other (vector) space on which translation and scalar multiplication is defined, but restrict ourselves to $\mathbb{R}^m$ as our interest is mainly to compare vectors of objective function values.

It has been found by V. Noghin \cite{Nog97} that the only partial orders on $\mathbb{R}^m$ that satisfy these two additional axioms are the cone orders on $\mathbb{R}^m$ defined by polyhedral cones.
The Pareto dominance order is a special case of a strict cone order. Here the definition of strictness is inherited from the pre-order.

Cone orders can be defined geometrically and doing so provides a good intuition about their properties and minimal sets.

\begin{definition}[Cone]
A subset $\mathcal{C} \subseteq \mathbb{R}^m$ is called a cone, iff
$\alpha \mathbf{d} \in \mathcal{C}$ for all $\mathbf{d} \in
\mathcal{C}$ and for all $\alpha \in \mathbb{R}, \alpha > 0$.
\end{definition}

In order to deal with cones it is useful to introduce notations for
set-based calculus by Minkowski:

\begin{definition}[Minkowski sum]
The Minkowski sum of two subsets $S^1$ and $S^2$ of $\mathbb{R}^m$
is defined as $S^1 + S^2 := \{s^1 + s^2| s^1 \in
S^1, s^2 \in S^2 \}$. If $S^1$ is a
singleton $\{x\}$, we may write $s + S^2$ instead of
$\{s\} + S^2$.
\end{definition}

\begin{definition}[Minkowski product]
The Minkowski product of a scalar $\alpha \in \mathbb{R}^n$ and a
set $S \subset \mathbb{R}^n$ is defined as $\alpha S
:= \{\alpha s| s \in S\}$.
\end{definition}

Among the many properties that may be defined for a cone, we
highlight the following two:

\begin{definition}[Properties of cones]
A cone $\mathcal{C} \in \mathbb{R}^{m}$ is called:
\begin{itemize}
\item nontrivial or proper, iff  $\mathcal{C} \neq \emptyset$.
\item convex, iff  $\alpha \mathbf{d}^1 + (1-\alpha) \mathbf{d}^2 \in \mathcal{C}$ for
all $\mathbf{d}^1$ and $\mathbf{d}^2 \in \mathcal{C}$ for all $0 <
\alpha < 1$
\item pointed, iff for $\mathbf{d} \in \mathcal{C}, \mathbf{d} \neq 0, -\mathbf{d} \not\in \mathcal{C}$, i\,e. $\mathcal{C} \cap -\mathcal{C} \subseteq\{0\}$
\end{itemize}
\end{definition}

\begin{example}
As an example of a cone consider the possible futures of a particle
in a 2-D world that can move with a maximal speed of $c$ in all
directions: This cone is defined as $ \mathcal{C}^+ = \{
\mathcal{D}(t) | t
 \in \mathbb{R}^+  \},$ where $\mathcal{D}(t) = \{\mathbf{x} \in \mathbb{R}^3 | (x_1)^2 +
(x_2)^2  \leq (c t)^2, x_3 = t \}$. Here time is measured by
negative and positive values of $t$, where $t =0$ represents the
current time. We may ask now, whether given the current position
${\bf x}_0$ of a particle, a locus ${\bf x} \in \mathbb{R}^3$ is a
possible future of the particle. The answer is in the affirmative,
iff ${\bf x}_0$ if ${\bf x} \in {\bf x}_0 + \mathcal{C}^+$.
\end{example}
We will now can define Pareto
dominance and the weak (strict) componentwise order by means of dominance cones. For this we have to define special convex cones in
$\mathbb{R}$:
\begin{definition}[Orthants]
We define
\begin{itemize}
\item the {\em positive orthant} $\mathbb{R}_\geq^n$ $:=$ $\{ {\bf x}
\in \mathbb{R}^n |  x_1 \geq 0 , \dots, x_n \geq 0 \}$.
\item the {\em null-dominated orthant} $\mathbb{R}_{\prec_{pareto}}^n$ $:=$ $\{
{\bf x} \in \mathbb{R}^n |  0 \prec_{pareto} {\bf x} \}$.
\item the {\em strictly positive orthant} $\mathbb{R}_>^n$  $:=$
$\{ {\bf x} \in \mathbb{R}^n |  x_1 > 0 , \dots, x_n > 0 \}$.
\end{itemize}
\end{definition}

Now, let us introduce the alternative definitions for Pareto
dominance:
\begin{definition}[Pareto dominance]
Given two vectors $\bf x \in \mathbb{R}^n$ and $\bf x' \in
\mathbb{R}^n$:
\begin{itemize}
\item ${\bf x} < {\bf x}'$ (in symbols: ${\bf x}$ dominates ${\bf
x}'$) in the strict componentwise order $\Leftrightarrow$ ${\bf x}' \in {\bf x} + \mathbb{R}^n_{>}$
\item ${\bf x} \prec {\bf x^{\prime}}$ (in symbols: ${\bf x}$ dominates ${\bf
x^{\prime}}$) $\Leftrightarrow$ ${\bf x^{\prime}} \in {\bf x} +
\mathbb{R}^n_{\prec_{pareto}}$
\item ${\bf x} \geq {\bf x}'$ (in symbols: ${\bf x}$ dominates ${\bf
x}'$) in the weak componentwise order $\Leftrightarrow$ ${\bf x^{\prime}} \in {\bf x} -
\mathbb{R}^n_{\geq}$
\end{itemize}
\end{definition}

It is often easier to assess graphically whether a point dominates
another point by looking at cones (cf. Fig. \ref{fig:cone} (l)).
This holds also for a region that is dominated by a {\em set of
points}, such that at least one point from the set dominates it (cf.
Fig. \ref{fig:cone} (r)).
\begin{figure}
\begin{center}
\includegraphics[width=\textwidth]{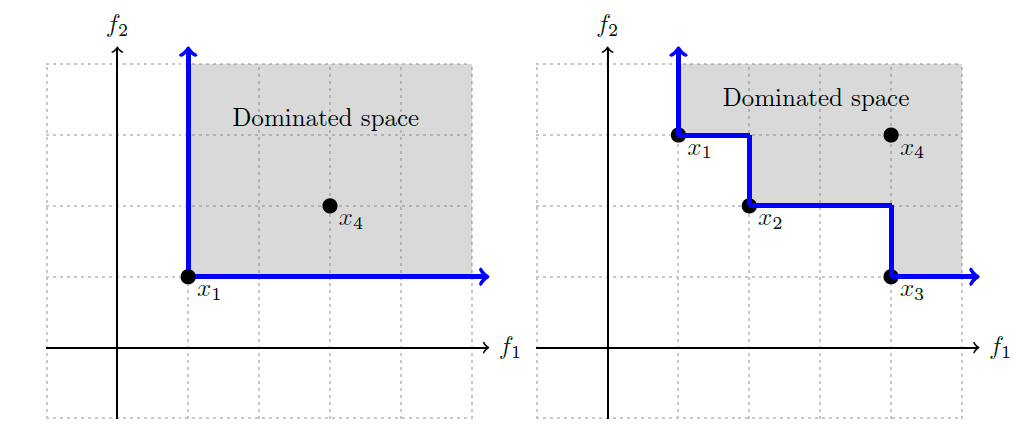}
\end{center}
\caption{\label{fig:cone} Pareto domination in $\mathbb{R}^2$
defined by means of cones. In the left hand side of the figure the
points inside the dominated region are dominated by ${\bf x}$. In
the figure on the right side the set of points dominated by the set
$A = \{{\bf x}_1, {\bf x}_2, {\bf x}_3, {\bf x}_4\}$ is depicted.}
\end{figure}

\begin{definition}[Dominance by a set of points]

A point ${\bf x}$ is said to be dominated by a set of points $A$
(notation: $A \prec {\bf x}$, iff ${\bf x} \in A +
\mathbb{R}^n_\prec$, i.\,e. iff there exists a point $\bf x' \in A$,
such that ${\bf x}' \prec_{Pareto} {\bf x}$.
\end{definition}
%


In the theory of multiobjective and constrained optimization, so-called polyhedral cones play a crucial role.
\begin{definition}
A cone $\mathcal{C}$ is a polyhedral cone with a finite basis, if and only if there is a set of vectors $D = \{\mathbf{d}_1, \dots, \mathbf{d}_k\}  \subset \mathbb{R}^m$ and
$\mathcal{C} = \{\lambda_1 \mathbf{d}_1 + \dots + \lambda_k \mathbf{d}_k | \lambda \in \mathbb{R}^+_0, i=1, \dots, k\}$. \end{definition}


\begin{example}
\label{ex:cone}
In figure \ref{fig:acutecone} an example of a polyhedral cone is depicted with finite basis $D = \{\mathbf{d}_1, \mathbf{d}_2\}$ and $\mathbf{d}_1 = (2,1)^\top, \mathbf{d}_2 =(1,2)^\top$. It is defined as
$$C := \left\{ \lambda_1 \mathbf{d}_1 + \lambda_2 \mathbf{d}_2 |\, \lambda_1 \in [0,\infty], \lambda_2 \in [0,\infty]  \right\}.$$
This cone is pointed, because $C \cap -C =\emptyset$. Moreover, $C$ is a convex cone. This is because two points in $C$, say $\mathbf{p}_1$ and $\mathbf{p}_2$ can be expressed by  $\mathbf{p}_1 =$ $\lambda_{11} \mathbf{d}_1 + \lambda_{21} \mathbf{d}_2$ and $\mathbf{p}_2 =$ $\lambda_{12} \mathbf{d}_1 + \lambda_{22} \mathbf{d}_2$, $\lambda_{ij} \in [0,\infty),$ $i=1,2; j=1,2$. Now, for a given $\lambda\in[0,1]$ $\lambda \mathbf{p}_1 + (1-\lambda) \mathbf{p}_2$ equals $\lambda \lambda_{11} \mathbf{d}_1 + (1-\lambda) \lambda_{12} \mathbf{d}_1$ $+ \lambda \lambda_{21} \mathbf{d}_2 + (1-\lambda) \lambda_{22} \mathbf{d}_2$ $=: c_1 \mathbf{d}_1 + c_2 \mathbf{d}_2$, where it holds that $c_1 \in [0,\infty)$ and $c_2 \in [0,\infty)$. According to the definition of  $C$ the cone therefore the point $\lambda \mathbf{p}_1 + (1-\lambda) \mathbf{p}_2$ is part of the cone $C$.
\end{example}
\begin{figure}
\begin{center}
\begin{tikzpicture}
	\draw[thick, color=black!22, dotted, step=1cm] (-1,-1) grid (5,4);
	\draw[thick, ->] (-1,0) -- (5.25,0) node[right] {$f_1$};
	\draw[thick, ->] (0,-1) -- (0,4.25) node [above] {$f_2$};
	\foreach \i in {1, 2} {
		\node[anchor=north] at (\i, 0) {$\i$};
		\node[anchor=east] at (0, \i) {$\i$};
	}
	\node[anchor=north east] at (0,0) {$(0,0)$};
	\fill [opacity=0.3, gray]
	(0,0) -- (4,2) -- (4,4) -- (2,4) -- cycle;
	\draw[ultra thick, blue, ->] (0,0) -- (2,1);
	\draw[ultra thick, dashed, blue] (2.2,1.1) -- (4,2);
	
	\draw[ultra thick, blue, ->] (0,0) -- (1,2);
	\draw[ultra thick, dashed, blue] (1.1,2.2) -- (2,4);
	
	\draw[ultra thick, dashed, gray] (2,4) -- (4,4);
	\draw[ultra thick, dashed, gray] (4,2) -- (4,4);
	
\end{tikzpicture}

\end{center}
\caption{\label{fig:acutecone}Dominance cone for cone order in Example \ref{ex:cone}.}
\end{figure}

By choosing the coordinate vectors $\mathbf{e}_i$, it is possible to create polyhedral cones that define the weak componentwise order as a cone-order, which is equivalent to defining the non-negative orthant.

Further topics related to cone orders are addressed, for instance, in \cite{Ehr05}.
\section*{Exercises}
\begin{enumerate}

    \exercise{Binary Relations in Real Life}{%
        In definition \ref{def:binprop}, some common properties of binary relations are defined, along with examples. Find further real-life examples of binary relations! Which axioms from definition \ref{def:binprop} do they obey?}%
        {ex:binary-relations}

    \exercise{Axiomatic Characterization of Incomparability}{%
        Characterize incomparability (definition \ref{def:incompar}) axiomatically! What are the essential differences between incomparability and indifference?}%
        {ex:incomparability}

    \exercise{Pareto Order on the 3D Hypercube Edges}{%
        Describe the Pareto order on the set of 3-D hypercube edges
        \[
        \{(0, 1, 0)^T, (0, 0, 1)^T, (1, 0, 0)^T, (0, 0, 0)^T, (0, 1, 1)^T, (1, 0, 1)^T, (1, 1, 0)^T, (1, 1, 1)^T\}
        \]
        by representing it as the graph of a binary relation and via a Hasse diagram.}%
        {ex:pareto-hypercube}

    \exercise{Partial Order on Natural Numbers with Divisibility}{%
        Prove that \((\mathbb{N} \setminus \{1\}, \preceq)\), where
        \[
        a \preceq b \Leftrightarrow a \mod b \equiv 0,
        \]
        is a partially ordered set. What are the minimal and maximal elements of this set?}%
        {ex:partial-order-numbers}

    \exercise{Polyhedral cone}{%
        Let $\mathbf{d_1} = (0,1)^T$ and $\mathbf{d_2} = (0.5,0.5)^T$ denote the generators of the polyhedral cone $\mathcal{C}$. Assume the cone is pointed. Show, how for this simple cone it can be checked by means of an equation whether or not a point is included in $\mathbf{y} \oplus \mathcal{C}$, or not. 
        Draw the Hasse diagram for the points in the set $\{(0,1)^T,  (0,2)^T, (2,3)^T, (1,1)^T, (2,0)^T\}$ w.r.t. this cone order}
        {ex:polyhedral-cone}

    \exercise{Convexity of the Time Cone}{%
        Prove that the Minkowski time cone \(\mathcal{C}^+\) is convex! Compare the Pareto order with the order defined by time cones. Assume a space-time where space has only one dimension. How can we define a cone order that determines whether or not two photons could potentially share the same position in space within a given time. (Hint: consider cone orders using polyhedral cones with generators $(-c,0)^T$, $(c,0)^T$, where $c$ is the speed of light)}%
        {ex:time-cone}

\end{enumerate}
\chapter{Landscape Analysis}
%
%
In this chapter we will come back to optimization problems, as
defined in the first chapter. We will introduce different notions of
Pareto optimality and discuss necessary and sufficient conditions
for (Pareto) optimality and efficiency in the constrained and
unconstrained case. In many cases, optimality conditions directly
point to solution methods for optimization problems. As in Pareto
optimization there is rather a set of optimal solutions then a
single optimal solution, we will also look at possible structures of
optimal sets.

\section{Search Space vs. Objective Space}
In Pareto optimization we are considering two spaces - the {\em
decision space} or {\em search space} $\mathbb{S}$ and the {\em
objective space} ${\mathbb{Y}}$. The vector valued objective
function $\mathbf{f}: \mathbb{S} \rightarrow \mathbb{Y}$ provides
the mapping from the decision space to the objective space. The set
of feasible solutions $\mathcal{X}$ can be considered as a subset of
the decision space, i.\,e. $\mathcal{X} \subseteq \mathbb{S}$. Given
a set $\mathcal{X}$ of feasible solutions, we can define
$\mathcal{Y}$ as the image of $\mathcal{X}$ under $\mathbf{f}$.

The sets $\mathbb{S}$ and $\mathbb{Y}$ are usually not arbitrary
sets. If we want to define optimization tasks, it is mandatory that
an order structure is defined on $\mathbb{Y}$. The space
$\mathbb{S}$ is usually equipped with a neighborhood structure. This
neighborhood structure is not needed for defining global optima, but
it is exploited, however, by optimization algorithms that gradually
approach optima and in the formulation of local optimality
conditions. Note, that the choice of neighborhood system may
influence the difficulty of an optimization problem significantly.
Moreover, we note that the definition of neighborhood gives rise to
many characterizations of functions, such as local optimality and
barriers. Especially in discrete spaces the neighborhood structure
needs to be mentioned then, while in continuous optimization
locality mostly refers to the Euclidean metric.

The definition of landscape is useful to distinguish the general
concept of a function from the concept of a function with a
neighborhood defined on the search space and a (partial) order
defined on the objective space. We define (poset valued) landscapes
as follows:
%

\begin{figure}
\centering
\begin{tikzpicture}[scale=1]
\node[draw, shape=circle] (v1234) at (0,0)  {$1111$};
\node[draw, shape=circle] (v123) at (-3,-2) {$1110$};
\node[draw, shape=circle] (v124) at (-1,-2) {$1101$};
\node[draw, shape=circle] (v134) at (1,-2)  {$1011$};
\node[draw, shape=circle] (v234) at (3,-2)  {$0111$};
\node[draw, shape=circle] (v12) at (-5,-4)  {$1100$};
\node[draw, shape=circle] (v13) at (-3,-4)  {$1010$};
\node[draw, shape=circle] (v14) at (-1,-4)  {$1001$};
\node[draw, shape=circle] (v23) at (1,-4)   {$0110$};
\node[draw, shape=circle] (v24) at (3,-4)   {$0101$};
\node[draw, shape=circle] (v34) at (5,-4)   {$0011$};
\node[draw, shape=circle] (v1) at (-3,-6)   {$1000$};
\node[draw, shape=circle] (v2) at (-1,-6)   {$0100$};
\node[draw, shape=circle] (v3) at (1,-6)    {$0010$};
\node[draw, shape=circle] (v4) at (3,-6)    {$0001$};
\node[draw, shape=circle] (v0) at (0,-8)    {$0000$};
\path[draw, ultra thick] (v1234) -> (v123) (v1234) -- (v124) (v1234) -- (v134) (v1234) -- (v234)
(v123) -- (v12) (v123) -- (v13) (v123) -- (v23)
(v124) -- (v12) (v124) -- (v14) (v124) -- (v24)
(v134) -- (v13) (v134) -- (v14) (v134) -- (v34)
(v234) -- (v23) (v234) -- (v24) (v234) -- (v34)
(v12) -- (v1) (v12) -- (v2) (v13) -- (v1) (v13) -- (v3)
(v14) -- (v1) (v14) -- (v4) (v23) -- (v2) (v23) -- (v3)
(v24) -- (v2) (v24) -- (v4) (v34) -- (v3) (v34) -- (v4)
(v1) -- (v0) (v2) -- (v0) (v3) -- (v0) (v4) -- (v0);
\end{tikzpicture}
\caption{\label{fig:binland}The 'binland' landscape of the bitstring $\{0,1\}^4$, with edges representing a Hamming distance of $1$, is an example of a discrete, partially ordered landscape.}
\end{figure}
Hasse diagram of the Pareto order for the
leading ones trailing zeros (LOTZ) problem. The first objective is
to maximize the number of leading ones in the bitstring, while the
second objective is to maximize the number of trailing zeros. The
preorder on $\{0,1\}$ is then defined by the Pareto dominance
relation. In this example all local minima are also global minima.
\begin{figure}
	\centering
\begin{tikzpicture}
	\tikzstyle{every node} = [draw, shape=rectangle];
	\path
		(0,0) node (n00) {$0,0 \enskip (0100, 0011, 0110, 0001, 0101, 0111)$}
		(0,-1) coordinate (level1)
		node[left = of level1] (n10) {$1,0 \enskip (1001, 1011)$}
		node[right = of level1] (n01) {$0,1 \enskip (0110, 0010)$}
		(0,-2)   node (n11) {$1,1 \enskip (1010)$}
		 node[anchor=east, left = of n11] (n40) {$4,0 \enskip (1111)$}
		 node[right = of n11] (n04) {$0,4 \enskip (0000)$}
		++(0,-1)   node (n22) {$2,2 \enskip (1100)$}
		 node[left = of n22] (n31) {$3,1 \enskip (1110)$}
		 node[right = of n22] (n13) {$1,3 \enskip (1000)$};
	\path[draw, thick]
		(n00) -- (n10) -- (n11) -- (n31)
		(n00) -- (n01) -- (n11) -- (n13)
		(n10) -- (n40)    (n01) -- (n04)
		(n11) -- (n22);
\end{tikzpicture}
\caption{\label{fig:lotz}(Figure 3.2) Hasse diagram of the Pareto order for the leading ones trailing zeros (LOTZ) problem. The first objective is to maximize the number of leading ones in the bitstring, while the second objective is to maximize the number of trailing zeros. The preorder on $\{0,4\}^2$ is then defined by the Pareto dominance relation. In this example all local minima are also global minima (cf. Fig. \ref{fig:binland}).}
\end{figure}
\begin{definition}
A poset valued landscape is a quadruple $\mathcal{L} = (\mathcal{X},
N, {\bf f}, \preceq)$ with $\mathcal{X}$ being a set and $N$ a
neighborhood system defined on it (e.g. a metric). ${\bf f}:
\mathcal{X} \rightarrow \mathbb{R}^{m}$ is a vector function and
$\preceq$ a partial order defined on $\mathbb{R}^{m}$. The function
${\bf f}: \mathcal{X} \rightarrow \mathbb{R}^{m}$ will be called
{\em height function}.
\end{definition}

An example for a poset-valued landscape is given in the Figure \ref{fig:binland}
and Figure \ref{fig:lotz}.
Here the neighborhood system is defined by the Hamming distance. It
gets obvious that in order to define a landscape in finite spaces we
need two essential structures. A neighborhood graph in search space
(where edges connect nearest neighbors) the Hasse diagram on the
objective space.

Note, that for many definitions related to optimization we do not
have to specify a height function and it suffices to define an order
on the search space. For concepts like global minima the
neighborhood system is not relevant either. Therefore, this
definition should be understood as a kind of superset of the structure
we may refer to in multicriteria optimization.

\section{Global Pareto Fronts and Efficient Sets}
Given $\mathbf{f}: \mathbb{S} \rightarrow \mathbb{R}^m$. Here we
write $\mathbf{f}$ instead of $(f_1, \dots, f_m)^\top$. Consider an
optimization problem:
\begin{equation}
{\bf f}({\bf x}) \rightarrow \min, {\bf x} \in \mathcal{X}
\end{equation}
 Recall that the Pareto front and the efficient set
are defined as follows (Section \ref{sec:pardom}):
\begin{definition} Pareto front and efficient set

The Pareto front $\mathcal{Y}_N$ is defined as the set of
non-dominated solutions in $\mathcal{Y} = {\bf f}(\mathcal{X})$,
i.\,e. $\mathcal{Y}_N = \{{\bf y} \in \mathcal{Y}\ |\  \nexists
\mathbf{y}' \in \mathcal{Y}: \mathbf{y}' \prec \mathbf{y}\}$. The
efficient set is defined as the pre-image of the Pareto-front,
$\mathcal{X}_E = f^{-1}(\mathcal{Y}_N)$.
\end{definition}

Note, that the cardinality $\mathcal{X}_E$ is at least as big as
$\mathcal{Y}_N$, but not vice versa, because there can be more than
one point in $\mathcal{X}_E$ with the same image in $\mathcal{Y}_N$.
The elements of $\mathcal{X}_E$ are termed efficient points.

In some cases it is more convenient to look at a direct definition
of efficient points:

\begin{definition}
A point ${\bf x}^{(1)} \in \mathcal{X}$ is {\em efficient}, iff
$\not\exists {\bf x}^{(2)} \in \mathcal{X}: {\bf x}^{(2)} \prec {\bf
x}^{(1)}$.
\end{definition}

Again, the set of all efficient solutions in $\mathcal{X}$ is
denoted as $\mathcal{X}_E$.

\begin{remark}
Efficiency is always relative to a set of solutions. In future, we
will not always consider this set to be the entire search space of
an optimization problem, but we will also consider the efficient set
of a subset of the search space. For example the efficient set for a
finite sample of solutions from the search space that has been
produced so far by an algorithm may be considered as a temporary
approximation to the efficient set of the entire search space.
\end{remark}

\section{Weak efficiency}
Besides the concept of {\em efficiency} also the concept of  {\em
weak efficiency}, for technical reasons, is important in the field
of multicriteria optimization. For example points on the boundary of
the dominated subspace are often characterized as weakly efficient
solutions though they may be not efficient.

Recall the definition of strict domination (Section
\ref{sec:pardom}):

\begin{definition} Strict dominance

Let ${\bf y}^{(1)}, {\bf y}^{(2)} \in \mathbb{R}^m$ denote two
vectors in the objective space. Then ${\bf y}^{(1)}$ strictly
dominates ${\bf y}^{(2)}$ (in symbols: ${\bf y}^{(1)} < {\bf
y}^{(2)}$), iff $\forall i = 1, \dots, m:$ $y_i^{(1)} < y_i^{(2)}$.
\end{definition}

\begin{definition} Weakly efficient solution

A solution ${\bf x}^{(1)} \in {\mathcal{X}}$ is weakly efficient,
iff $\not \exists \mathbf{x}^{(2)} \in \mathcal{X}: {\bf f}({\bf
x}^{(2)}) < {\bf f}({\bf x}^{(1)})$. The set of all weakly efficient
solutions in $\mathcal{X}$ is called $\mathcal{X}_{wE}$.
\end{definition}

\begin{figure}
\centering
\begin{tikzpicture}
	\tikzstyle{every node} = [anchor = north west];
	\tikzstyle{circnode} = [draw, black, fill=white, circle, inner sep=0pt, minimum size=7pt, anchor=center];
	
	\path coordinate (origin) at (0,0) coordinate (origin2) at (7,0);
	\draw[thick, color=black!22, dotted, step=1cm] (-1,-1) grid (5,4);
	\draw[thick, ->] (-1,0) -- (5.25,0) node[right] {$x_1$};
	\draw[thick, ->] (0,-1) -- (0,4.25) node[above] {$x_2$};
	\foreach \i in {1, 2} {
		\node[anchor=north] at (\i, 0) {$\i$};
		\node[anchor=east] at (0, \i) {$\i$};
	}
	\node[anchor=north east] at (0,0) {$(0,0)$};
	
	\draw[thick, color=black!22, dotted, step=1cm] (origin2)+(-1,-1) grid (12,4);
	\draw[thick, ->] (origin2)+(-1,0) -- +(5.25,0) node[right] {$f_1$};
	\draw[thick, ->] (origin2)+(0,-1) -- +(0,4.25) node[above] {$f_2$};
	\foreach \i in {1, 2} {
		\path (origin2) node[anchor=north] at +(\i, 0) {$\i$}
		(origin2) node[anchor=east] at +(0, \i) {$\i$};
	}
	\node[anchor=north east] at (origin2) {$(0,0)$};
	
	\fill[opacity=0.3, gray]
		(origin) -- (2,0) -- (2,2) -- (0,2) -- (origin) -- cycle
		(origin2) ++(2,1) -- ++(1,0) -- ++(0,2) -- ++(-2,0) -- ++(0,-1) -- ++(1,0) -- cycle;
	\path[draw, black, ultra thin]
		(origin) ++(2,0) -- ++(0,2) -- ++(-2,0)
		(origin2) ++(2,1) -- ++(1,0) -- ++(0,2) -- ++(-2,0);
	\path[draw, blue, ultra thick]
		(origin) ++(0,2) -- ++(0,-2) -- ++(2,0)
		(origin2) ++(1,3) -- ++(0,-1) -- ++(1,0) -- ++(0,-1);
	\path[draw, fill, white]
		(origin) circle(3pt)
		++(0,1)  coordinate (x1) circle (3pt);
	\path
		(origin) node[circnode] (x1) {}
		++(0,1)  node[circnode] (x2) {}
		(origin2)+(1,2) node[circnode] (f1) {}
		+(2,1) node[circnode] (f2) {};
	\path
		(-1,0.5) node[anchor=east] (efficient) {\color{black}{efficient}}
		(2.5,1.5) node[anchor=west, fill=white] (weff) {\color{black}{weakly efficient}}
		(origin2)+(0.5,1.5) node[anchor=east, fill=white] (efficient2) {\color{black}{efficient}}
             (origin2)+(3.5,2.1) node[anchor=west, fill=white] (weff3) {\color{black}{dominated}}
		(origin2)+(3.5,2.5) node[anchor=west, fill=white] (weff2) {\color{black}{weakly}};
	\path[draw, dashed, blue, thick, ->]
		(efficient) edge (x1) (efficient) edge (x2)
		(weff) edge (0,1.5) (weff) edge (0.5,0)
		(efficient2) edge (f1) (efficient2) edge (f2)
		(weff2) edge (8,2.5) (weff2) edge (9, 1.5);
	\draw[thick, ->] (1.5,1.75) .. controls (3,3)  and (6,4) .. (8.5,2.75);
\end{tikzpicture}
\caption{\label{fig:weakeff} Example of a solution set containing efficient solutions (open points) and weakly efficient solutions (thick blue line).}

\end{figure}

\begin{example}
In Fig. \ref{fig:weakeff} we graphically represent the efficient and
weakly efficient set of the following problem: $f=(f_1,f_2)
\rightarrow \min, \mathbb{S}= \mathcal{X} = [0,2] \times [0,2]$,
where $f_1$ and $f_2$ are as follows:
$$f_1(x_1, x_2) =
\left\{
\begin{array}{ll}
2+x_1 & \mbox{ if } 0 \leq x_2 < 1 \\
1+0.5 x_1 & \mbox{ otherwise}
\end{array}
\right., f_2(x_1, x_2) = 1+x_1, x_1 \in [0,2], x_2 \in [0,2].
$$.
The solutions $(x_1, x_2) = (0,0)$ and $(x_1, x_2) = (0,1)$ are
efficient solutions of this problem, while the solutions on the line
segments indicated by the bold line segments in the figure denote
weakly efficient solutions. Note, that both efficient solutions are
also weakly efficient, as efficiency implies weak efficiency.
\end{example}

\section{Characteristics of Pareto Sets}

\begin{figure}
\centering
\begin{tikzpicture}
	\tikzstyle{every node} = [anchor = north west];
	\draw[thick, color=black!22, dotted, step=1cm] (-1,-1) grid (5,4);
	\draw[thick, ->] (-1,0) -- (5.25,0) node[right] {$f_1$};
	\draw[thick, ->] (0,-1) -- (0,4.25) node[above] {$f_2$};
	\path[ultra thick, draw=blue, fill=gray, fill opacity=0.3] (0.5,2.5) -- (2,2) -- (2.5,0.5)	-- (3.5,1) arc (-60:0:2cm) -- (2,3.5) -- cycle;
	
	\path
		(0.5,0.5) coordinate (minimum)
		(2.5,2.5) coordinate (nadir)
		(4.5,3.5) coordinate (maximum)
		(2.6,-0.4) coordinate (min2text)
		(-1.6,3.2) coordinate (min1text);
	\path[fill]
		(minimum) node {$\underbar{y}$} circle (3pt)
		(nadir) node {$N$} circle (3pt)
		(maximum) node {$\overline{y}$} circle (3pt);
	\path[draw=blue, dashed]
		(0.5,2.5) -- (minimum) -- (2.5,0.5);
		(2,3.5) -- (maximum) -- (4.5,2.75);
		(2.5,0.5) -- (nadir) -- (0.5,2.5);
	\path
		node (min2textnode) at (min2text) {minimum of $f_2$}
		node[anchor=south west] (min1textnode) at (min1text) {minimum of $f_1$}
		node[anchor=west] (fstext) at (5,1.5) {Feasible objective space};
	\draw[blue, thick, dashed, ->]
		(min2textnode) -- (2.6,0.4);
	\draw[blue, thick, dashed, ->]
		(min1textnode) -- (0.5,2.6);
	\draw[blue, thick, dashed, ->]
		(fstext) -- (3,1.5);
\end{tikzpicture}
\caption{\label{fig:ideal}The shaded region indicates the feasible objective space of some function. Its \textit{ideal point}, $\underbar{y}$, its \textit{Nadir point}, $N$ and its \textit{maximal point}, $\overline{y}$, are visible.}
\end{figure}

There are some characteristic points on a Pareto front:
\begin{definition}
Given an multi-objective optimization problem with $m$ objective
functions and image set ${\cal Y}$: The ideal solution is defined as
$$\underline{\bf y} = (\min_{{\bf y} \in \mathcal{Y}} y_1, \dots,
\min_{{\bf y} \in \mathcal{Y}} y_m).$$ Accordingly we define the
maximal solution:
$$\overline{\bf y} = (\max_{{\bf y} \in \mathcal{Y}} y_1, \dots,
\max_{{\bf y} \in \mathcal{Y}}  y_m).$$
The Nadir point is defined:
$${\bf y}^N = (\max_{{\bf y} \in \mathcal{Y}_N} y_1, \dots,
\max_{{\bf y} \in \mathcal{Y}_N}  y_m).$$
\end{definition}
For the Nadir only points from the Pareto front ${\cal Y}_N$ are
considered, while for the maximal point all points in $\mathcal{Y}$
are considered. The latter property makes it, for dimensions higher
than two ($m > 2$), more difficult to compute the Nadir point. In
that case the computation of the Nadir point cannot be reduced to
$m$ single criterion optimizations.

A visualization of these entities in a 2-D space is given in figure
\ref{fig:ideal}.

\section{Optimality conditions based on level sets}

Level sets can be used to visualize $\mathcal{X}_E$,
$\mathcal{X}_{wE}$ and $\mathcal{X}_{sE}$ for continuous spaces and
obtain these sets graphically in the low-dimensional case:
Let in the following definitions $f$ be a function $f: \mathbb{S}
\rightarrow \mathbb{R}$, for instance one of the objective
functions:
\begin{definition} Level sets
\begin{equation}
\mathcal{L}_{\leq}(f(\hat{\bf x})) = \{\mathbf{x} \in \mathcal{X}:
f(\mathbf{x}) \leq f(\mathbf{\hat{\mathbf{x}}})\}
\end{equation}
\end{definition}
\begin{definition} Level curves
\begin{equation}
\mathcal{L}_{=}(f(\hat{\bf x})) = \{{\bf x} \in \mathcal{X}:
f(\mathbf{x}) = f(\mathbf{\hat{\mathbf{x}}})\}
\end{equation}
\end{definition}
\begin{definition} Strict level set
\begin{equation}
\mathcal{L}_{<}(f(\hat {\bf x})) = \{\mathbf{x} \in \mathcal{X}:
f(\mathbf{x}) < f(\mathbf{\hat{\mathbf{x}}})\}
\end{equation}
\end{definition}
Level sets can be used to determine whether $\hat {\bf{x}} \in
\mathcal{X}$ is (strictly, weakly) non-dominated or not.

The point $\hat{\mathbf{x}}$ can only be efficient if its level sets
intersect in level curves.

\begin{theorem}
$\mathbf{x}$ is  efficient $\Leftrightarrow$ $\bigcap^m_{k=1}
\mathcal{L}_{\leq}(f_k({\bf x})) = \bigcap_{k=1}^m
\mathcal{L}_{=}(f_k({\bf x})) $
\end{theorem}

\begin{proof} $\hat{\mathbf{x}} \mbox{ is efficient}$ $\Leftrightarrow$ there
is no $\mathbf{x}$ such that both $f_k(\mathbf{x}) \leq
f_k(\hat{\bf{x}})$ for all $k = 1, \dots, m$ and $f_k(\mathbf{x}) <
f(\hat{\mathbf{x}})$ for at least one $k = 1, \dots, m$
$\Leftrightarrow$ there is no ${\bf{x}} \in \mathcal{X}$ such that
both ${\bf{x}} \in \cap_{k=1}^m \mathcal{L}_\leq(f(\hat {\bf{x}}))$
and ${\bf{x}} \in \mathcal{L}_{<}(f_j(\hat {\bf{x}}))$ for some $j$
$\Leftrightarrow$ $\bigcap^m_{k=1} \mathcal{L}_{\leq}(f_k(\hat{x}))
= \bigcap^m_{k=1} \mathcal{L}_{=}(f_k(\hat{x})) $

\end{proof}

\begin{theorem}
The point $\hat{\mathbf{x}}$ can only be weakly efficient if its
strict level sets do not intersect. $\mathbf{x}$ is weakly efficient
$\Leftrightarrow$ $\bigcap^m_{k=1} \mathcal{L}_{<}(f_k({\bf x})) =
\emptyset$
\end{theorem}

\begin{theorem}
The point $\hat{\mathbf{x}}$ can only be strictly efficient if its
level sets intersect in exactly one point. $\mathbf{x}$ is strictly
efficient $\Leftrightarrow$ $\bigcap^m_{k=1}
\mathcal{L}_{\leq}(f_k({\bf x})) = \{\mathbf{x}\}$
\end{theorem}

Level sets have a graphical interpretation that helps to
geometrically understand optimality conditions and landscape
characteristics. Though this intuitive geometrical interpretation
may only be viable for lower dimensional spaces, it can help to
develop intuition about problems in higher dimensional spaces. The
visualization of level sets can be combined with the visualization
of constraints, by partitioning the search space into a feasible and
infeasible part.

The following examples will illustrate the use of level sets for
visualization:

\begin{example}
Consider the problem $f_1(x_1,x_2) = (x_1-1.75)^2 + 4 (x_2-1)^2
\rightarrow \min$, $f_2(x_1, x_2) = (x_1-3)^2 + (x_2-1)^2
\rightarrow \min, (x_1, x_2)^\top \in \mathbb{R}^2$. The level
curves of this problem are depicted in Figure \ref{fig:levelmeet}
together with the two marked points ${\bf p}_1$ and ${\bf p}_2$ that
we will study now. For ${\bf p}_1$ it gets clear from Figure
\ref{fig:levelp1} that it is an efficient point as it cannot be
improved in both objective function values at the same time. On the
other hand ${\bf p}_2$ is no level point as by moving it to the
region directly left of it can be improved in all objective function
values at the same time. Formally, the existence of such a region
follows from the non-empty intersection of $\mathcal{L}_{<}(f_1({\bf
p}_2))$ and $\mathcal{L}_{<}(f_2({\bf p}_2))$.
\end{example}

\begin{figure}
\centering
\begin{tikzpicture}
	\tikzstyle{every node} = [anchor = north west];
	\tikzstyle{circnode} = [draw, black, fill=black, circle, inner sep=0pt, minimum size=3pt, anchor=center];
	\draw[thick, color=black!22, dotted, step=1cm] (-1,-1) grid (5,4);
	\draw[thick, ->] (-1,0) -- (5.25,0) node[right] {$x_1$};
	\draw[thick, ->] (0,-1) -- (0,4.25) node[above] {$x_2$};
	\foreach \i in {1, 2} {
		\node[anchor=north] at (\i, 0) {$\i$};
		\node[anchor=east] at (0, \i) {$\i$};
	}
	\node[anchor=north east] at (0,0) {$(0,0)$};
	
	\path[draw, red]
		(2,2) ellipse (4 and 2)
		(2,2) ellipse (2 and 1)
		(2,2) ellipse (1 and 0.5)
		+(0,-0.5) node {$f_1=1$}
		+(0,-1) node {$f_1=4$}
		+(0,-2) node {$f_1=16$};
	\path[draw, blue]
		(5,2) ellipse (2 and 2)
		(5,2) ellipse (1 and 1)
		(5,2) ellipse (0.5 and 0.5)
		+(0,-0.5) node[black] {$f_2=0.25$}
		+(0,-1) node[black] {$f_2=1$}
		+(0,-2) node[black] {$f_2=4$}
		+(-1,0) node[black, opacity=1, circnode, anchor=center] {}
		node[black, opacity=1]  {$\textbf{p}_1$}
		+(1,0) node[black, circnode, anchor=center] {}
		node[black] {$\textbf{p}_2$};
\end{tikzpicture}
\caption{\label{fig:levelmeet} This graph depicts the level curves of $f_1(\textbf{x})=(x_1-2)+2(x_2-2) \rightarrow \min$ (red curves) and $f_2(\textbf{x})=(x_1-5) + (x_2-2)\rightarrow \min$ (blue curves).}
\end{figure}

\begin{figure}
\centering
\begin{tikzpicture}
	\tikzstyle{every node} = [anchor = north west];
	\tikzstyle{circnode} = [draw, black, fill=black, circle, inner sep=0pt, minimum size=3pt, anchor=center];
	\draw[thick, color=black!22, dotted, step=1cm] (-1,-1) grid (5,4);
	\draw[thick, ->] (-1,0) -- (5.25,0) node[right] {$x_1$};
	\draw[thick, ->] (0,-1) -- (0,4.25) node[above] {$x_2$};
	\foreach \i in {1, 2} {
		\node[anchor=north] at (\i, 0) {$\i$};
		\node[anchor=east] at (0, \i) {$\i$};
	}
	\node[anchor=north east] at (0,0) {$(0,0)$};
	
	\path[draw, black, ultra thick, fill=red, opacity=0.3]
		(2,2) ellipse (2 and 1)
		node[black, opacity=1] {$L_1$}
		+(0,-1) node[black, opacity=1] {$f_1=1$};
	\path[draw, black, ultra thick, fill=blue, opacity=0.3]
		(5,2) ellipse(1 and 1)
		node[black, opacity=1] {$L_2$}
		+(0,-1) node[black, opacity=1] {$f_2=1$}
		+(-1,0) node[black, opacity=1, circnode, anchor=center] {}
		node[black, opacity=1]  {$\mathbf{p}_1$};
\end{tikzpicture}
\caption{\label{fig:levelp1} The situation for ${\bf p}_1$: In order
to improve $f_1$ the point ${\bf p_1}$ has to move into the set
$\mathcal{L}_\leq(f_1({\bf p}_1))$ and in order to improve $f_2$ it
needs to move into $\mathcal{L}_\leq(f_1({\bf p}_1))$. Since these
sets only meet in $\mathbf{p}_1$, it is not possible to improve
$f_1$ and $f_2$ at the same time.}
\end{figure}

\begin{example}
Consider the search space $\mathcal{S} = [0,2]\times[0,3]$.
Two objectives $f_1(x_1, x_2) = 2 + \frac{1}{3} x_2 - x_1
\rightarrow \min$,$\quad$ $f_2(x_1, x_2) = \frac{1}{2} x_2 +
\frac{1}{2} x_1 \rightarrow \max$. In addition, the constraint
$g(x_1,x_2) = 2 -\frac{2}{3} x_1 - x_2 \geq 0$ needs to be satisfied.
To solve this problem, we mark the constrained region graphically (see Figure \ref{fig:bilinprog})
Now, we can check different points for efficiency. For ${\bf p}_1$
the region where both objectives improve is in the upper triangle
bounded by the level curves. As this set is partly feasible, it is
possible to find a dominating feasible point and ${\bf p}_1$ is not
efficient. In contrast, for ${\bf p}_2$ the set of dominating
solutions is completely in the infeasible domain, why this point
belongs to the efficient set. The complete efficient set in this
example lies on the constraint boundary. Generally, it can be found
that for linear problems with level curves intersecting in a single
point there exists no efficient solutions in the unconstrained case
whereas efficient solutions may lie on the constraint boundary in
the constrained case.
\end{example}
\begin{figure}
\begin{center}
%
\includegraphics[width=10cm]{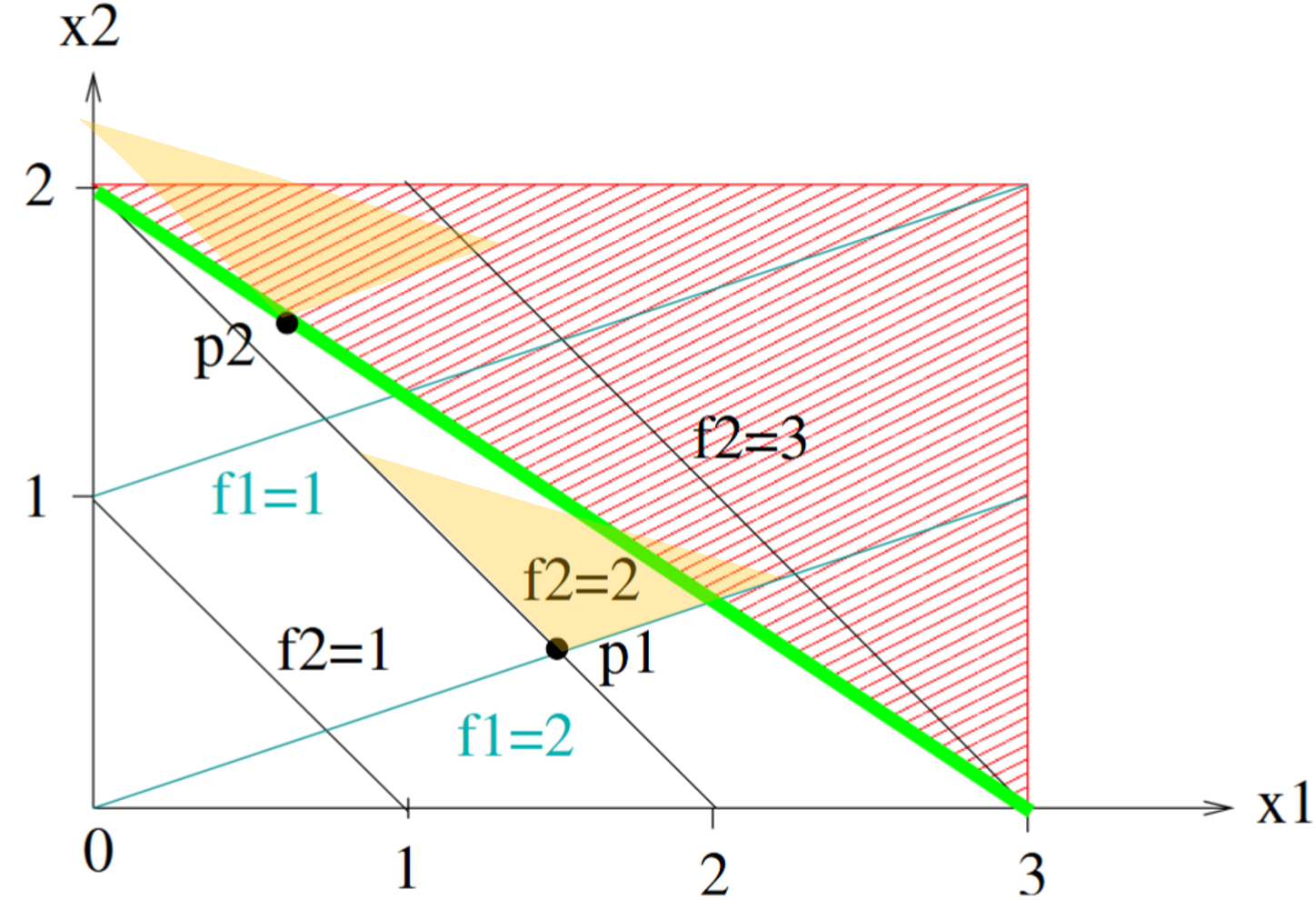}
\end{center}
\caption{\label{fig:bilinprog}Example for a linear programming problem with two objective functions. Here, f1 is to be minimized and f2 is to be maximized. The contour indicates the indifference curves for f1 and f2. The orange shaded cones indicate the regions where solutions dominate the point p1, or, respectively, p2, when considering only the objective functions and not the constraints.}
\end{figure}

\section{Local Pareto Optimality}
As opposed to global Pareto optimality we may also define local
Pareto optimality. Roughly speaking, a solution is a local optimum,
if there is no better solution in its neighborhood. In order to put
things into more concrete terms let us distinguish continuous and
discrete search spaces:

In finite discrete search spaces for each point in the search space
$\mathbb{X}$ a set of nearest neighbors can be defined by means of
some {\emph neighborhood function} $N: \mathbb{X} \rightarrow
\wp(\mathbb{X})$ with $\forall x \in \mathbb{X}: x \not\in N(x)$. As
an example consider the space $\{0,1\}^n$ of bit-strings of length
$n$ with the nearest neigbors of a bit-string $x$ being the elements
that differ only in a single bit, i.e. that have a Hamming distance
of 1.

\begin{definition} Locally efficient point (finite search spaces)

Given a neighborhood function $N: \mathbb{X} \rightarrow
\wp(\mathbb{X})$, a locally efficient solution is a point $x \in
\mathbb{X}$ such that $\not\exists x' \in N(x): x' \prec x$.
\end{definition}
\begin{definition} Strictly locally efficient point (finite search spaces)

Given a neighborhood function $N: \mathbb{X} \rightarrow
\wp(\mathbb{X})$, a strictly locally efficient solution is a point
$x \in \mathbb{X}$ such that $\not\exists x' \in N(x): x' \preceq
x$.
\end{definition}

\textbf{Remark:} The comparison of two elements in the search space
is done in the objective space. Therefore, for two elements $x$ and
$x'$ with $x \preceq x'$ and $x \preceq x'$ it can happen that $x
\neq x'$ (see also the discussion of the antisymmetry property in
chapter
 \ref{chap:orders}).

\begin{figure}
\begin{tikzpicture}[scale=2.5,
    x={(1.5cm,0cm)},
    y={(0cm,1.5cm)},
    z={(0.9cm,0.45cm)}
]

\tikzset{vertex/.style={circle,fill=black,inner sep=1.5pt}}

\newcommand{\vertexWithLabels}[6]{%
    \node[vertex] (#1) at (#2,#3,#4) {};
    \node[anchor=south west, inner sep=0.5pt, font=\small] at (#2,#3,#4) {#5};
    \node[anchor=north east, inner sep=0.5pt, font=\small, blue] at ($(#2,#3,#4)+(-0.05,-0.05)$) {#6};
}

\newcommand{\drawcube}[1]{%
    \vertexWithLabels{#1A000}{0}{0}{0}{000}{-1}
    \vertexWithLabels{#1A100}{1}{0}{0}{100}{5}
    \vertexWithLabels{#1A010}{0}{1}{0}{010}{1}
    \vertexWithLabels{#1A110}{1}{1}{0}{110}{0}
    
    \vertexWithLabels{#1A001}{0}{0}{1}{001}{6}
    \vertexWithLabels{#1A101}{1}{0}{1}{101}{2}
    \vertexWithLabels{#1A011}{0}{1}{1}{011}{1}
    \vertexWithLabels{#1A111}{1}{1}{1}{111}{4}
    
    \draw (#1A000) -- (#1A100) -- (#1A110) -- (#1A010) -- cycle;
    \draw (#1A010) -- (#1A000);
    
    \draw (#1A001) -- (#1A101) -- (#1A111) -- (#1A011) -- cycle;
    \draw (#1A011) -- (#1A001);
    
    \draw (#1A000) -- (#1A001);
    \draw (#1A100) -- (#1A101);
    \draw (#1A110) -- (#1A111);
    \draw (#1A010) -- (#1A011);
}

\drawcube{}

\begin{scope}[xshift=3cm]
    \vertexWithLabels{B000}{0}{0}{0}{000}{2,2}
    \vertexWithLabels{B100}{1}{0}{0}{100}{1,5}
    \vertexWithLabels{B010}{0}{1}{0}{010}{0,3}
    \vertexWithLabels{B110}{1}{1}{0}{110}{1,2}
    
    \vertexWithLabels{B001}{0}{0}{1}{001}{1,2}
    \vertexWithLabels{B101}{1}{0}{1}{101}{5,4}
    \vertexWithLabels{B011}{0}{1}{1}{011}{2,1}
    \vertexWithLabels{B111}{1}{1}{1}{111}{2,2}
    
    \draw (B000) -- (B100) -- (B110) -- (B010) -- cycle;
    \draw (B010) -- (B000);
    \draw (B001) -- (B101) -- (B111) -- (B011) -- cycle;
    \draw (B011) -- (B001);
    \draw (B000) -- (B001);
    \draw (B100) -- (B101);
    \draw (B110) -- (B111);
    \draw (B010) -- (B011);
\end{scope}

\end{tikzpicture}
\caption{\label{fig:bincube} Pseudoboolean landscapes with search
space $\{0,1\}^3$ and the Hamming neighborhood defined on it. The
linearly ordered landscape on the right hand side has four local
optima. These are $x^{(0)} = 000, x^{(5)} = 101, x^{(6)} = 110$, and $x^{(3)}= 011$.
$x^{(0)}$ is also a global minimum and $x^{(4)}$ a global maximum.
The partially ordered landscape on the right hand side has locally
efficient solutions are $x^{(1)} = 001, x^{(2)} = 010, x^{(3)} =
011, x^{(6)} = 110$. The globally efficient solutions are $x^{(1)},
x^{(2)}$ and $x^{(3)}$ }
\end{figure}

This definition can also be extended for countable infinite sets,
though we must be cautious with the definition of the neighborhood
function there.

For the Euclidean space $\mathbb{R}^n$, the notion of nearest
neighbors does not make sense, as for every point different from
some point $x$ there exists another point different from $x$ that is
closer to $x$. Here, the following criterion can be used to classify
local optima:

\begin{definition} Open $\epsilon$-ball

An open $\epsilon$-ball $B_\epsilon(x)$ around a point $x \in
\mathbb{R}^n$ is defined as: $B_\epsilon(x) = \{x' \in \mathbb{X}|
d(x,x') < \epsilon\}$.
\end{definition}

\begin{definition} Locally efficient point (Euclidean search spaces)

A point $x \in \mathbb{R}^n$ in a metric space is said to be a
locally efficient solution, iff $\exists \epsilon > 0: \not\exists
x' \in B_\epsilon(x): x' \prec x$.
\end{definition}

\begin{definition} Strictly locally efficient point (Euclidean search spaces)

A point $x \in \mathbb{R}^n$ in a metric space is said to be a
strictly locally efficient solution, iff $\exists \epsilon > 0:
\not\exists x' \in B_\epsilon(x) - \{x\}: x' \preceq x$.
\end{definition}

The extension of the concept of local optimality can be done also
for subspaces of the Euclidean space, such as box constrained spaces
as in definition \ref{eq:boxcon}.

%


\section{Barrier Structures}

Local optima are just one of the many characteristics we may discuss
for landscapes, i.e. functions with a neighborhood structure defined
on the search space. Looking at different local optimal of a
landscape we may ask ourselves how these local optimal are separated
from each other. Surely there is some kind of barrier in between,
i.e. in order to reach one local optimum from the other following a
path of neighbors in the search space we need to put up with
encountering worsening of solutions along the path. We will next
develop a formal framework on defining barriers and their
characteristics and highlight an interesting hierarchical structure
that can be obtained for all landscapes - the so-called barrier tree  
of totally ordered landscapes, which generalizes to a barrier forest
in partially ordered landscapes.

For the sake of clarity, let us introduce formal definitions first
for landscapes with a one-dimensional height function as they are
discussed in single-objective optimization.

\begin{definition} Path in discrete spaces

Let $N: \mathbb{X} \rightarrow \wp(\mathbb{X})$ be a neighborhood
function. A sequence ${\bf p}_1, \dots, {\bf p}_l$ for some $l \in
\mathbb{N}$ and ${\bf p}_1, \dots, \mathbf{p}_l \in \mathbb{S}$ is
called a {\em path connecting} ${\bf x}_1 $ \em{and} ${\bf x}_2$,
iff ${\bf p}_1= {\bf x}_1$, ${\bf p}_{i+1} \in N({\bf p}_i)$, for $i
= 1, \dots, l-1$, and ${\bf p}_l = {\bf x}_2$.
\end{definition}

\begin{definition} Path in continuous spaces

For continuous spaces, a path is a continuous mapping
$\mathbf{p}[0,1] \rightarrow \mathbb{X}$ with $\mathbf{p}(0) =
\mathbf{x}_1$ and $\mathbf{p}(1)= \mathbf{x}_2$.
\end{definition}

\begin{definition}
Let $\mathbb{P}_{\mathbf{x}_1, \mathbf{x}_2}$ denote the set of all
paths between $\mathbf{x}_1$ and $\mathbf{x}_2$.
\end{definition}

\begin{definition}
Let the function value of the {\em lowest  point} separating two
local minima $\mathbf{x}_1$ and $\mathbf{x}_2$ be defined as
$\hat f(\mathbf{x}_1,\mathbf{x}_2) = \min_{\mathbf{p} \in
\mathbb{P}_{\mathbf{x}_1, \mathbf{x}_2}} \max_{\mathbf{x}_3 \in
\mathbf{p}} f(\mathbf{x}_3)$. Points ${\bf s}$ on some path
$\mathbf{p} \in \mathbb{P}_{\mathbf{x}_1, \mathbf{x}_2}$ for which
$f({\bf s}) = \hat f(\mathbf{x}_1,\mathbf{x}_2)$ are called saddle
points between ${\bf x}_1$ and ${\bf x}_2$.
\end{definition}

\begin{example}
In the example given in Figure \ref{fig:binland} the search points
are labeled by their heights, i.e. $x_1$ has height $1$ and $x_4$
has height $4$. The saddle point between the local minima $x_1$ and
$x_2$ is $x_{12}$. The saddle point $x_3$ and $x_5$ is $x_{18}$.
\end{example}

\begin{figure}
\centering
\begin{tikzpicture}[scale=1]
	\tikzstyle{every node} = [draw, circle, minimum size = 21pt];
	\path
		(0,0) node (n00) {$15$} ++(1,0) node (n10) {$2$} ++(1,0) node (n20) {$10$} ++(1,0) node (n30) {$18$} ++(1,0) node (n40) {$3$}
		(0,1) node (n01) {$6$} ++(1,0) node (n11) {$12$} ++(1,0) node (n21) {$17$} ++(1,0) node (n31) {$5$} ++(1,0) node (n41) {$19$}
		(0,2) node (n02) {$16$} ++(1,0) node (n12) {$7$} ++(1,0) node (n22) {$9$} ++(1,0) node (n32) {$14$} ++(1,0) node (n42) {$13$}
		(0,3) node (n03) {$1$} ++(1,0) node (n13) {$11$} ++(1,0) node (n23) {$4$} ++(1,0) node (n33) {$20$} ++(1,0) node (n43) {$8$};
		
	\path[draw, ultra thick, black]
		(n00) -- (n10) -- (n20) -- (n30) -- (n40)
		(n01) -- (n11) -- (n21) -- (n31) -- (n41)
		(n02) -- (n12) -- (n22) -- (n32) -- (n42)
		(n03) -- (n13) -- (n23) -- (n33) -- (n43)
		(n00) -- (n01) -- (n02) -- (n03)
		(n10) -- (n11) -- (n12) -- (n13)
		(n20) -- (n21) -- (n22) -- (n23)
		(n30) -- (n31) -- (n32) -- (n33)
		(n40) -- (n41) -- (n42) -- (n43);
\end{tikzpicture}
\caption{\label{fig:binland2}Example of a discrete landscape. The
height of points is given by the numbers and their neighborhood is
expressed by the edges.}
\end{figure}

\begin{lemma}
For {\em non-degenerate landscapes}, i.e. landscapes where for all
${\bf x}_1$ and ${\bf x}_2$:  $f({\bf x}_1) \neq f({\bf x}_2)$,
saddle points between two given local optima are unique.
\end{lemma}

Note. that in case of degenerate landscapes, i.e. landscapes where
there are at least two different points which share the same value
of the height function, saddle points between two given local optima
are not necessarily unique anymore, which, as we will see later,
influences the uniqueness of barrier trees characterizing the
overall landscape.

\begin{definition}
The valley (or: basin) below a point  $\mathbf{s}$ is called
$B(\mathbf{s}):$ $B({\bf s}) = \{ \mathbf{x} \in \mathbb{S} |
\exists \mathbf{p} \in \mathbb{P}_{\mathbf{x}, \mathbf{s}}:
        \max_{{z} \in \mathbf{p}}
       f(\mathbf z) \leq f(\mathbf{s}) \}$
\end{definition}

\begin{example}
In the aforementioned example given in Figure \ref{fig:binland},
Again, search points $x_1, \dots, x_{20}$ are labeled by their
heights, i.e. $x_4$ is the point with height 4, etc.. The basin
below $x_1$ is given by the empty set, and the basin below $x_{14}$
is $\{x_1,$ $x_{11},$ $x_4,$ $x_9,$ $x_7, x_{13},x_5, x_8, x_{14}, x_{12}, x_2, x_6,x_{10}\}$.
\end{example}

Points in $B(\mathbf{s})$ are mutually connected by paths that never
exceed $f(\mathbf{s})$. At this point it is interesting to compare
the level set $\mathcal{L}_\leq(f({\bf x}))$ with the basin $B({\bf
x})$. The connection between both concepts is: Let $\mathcal{B}$ be
the set of connected components of the level set
$\mathcal{L}_\leq(f({\bf x}))$ with regard to the neighborhood graph
of the search space $\mathcal{X}$, then $B({\bf x})$ is the
connected
component in which ${\bf x}$ resides.

\begin{theorem}
\label{the:setinclusion} Suppose for two points $\mathbf{x}_1$ and
$\mathbf{x}_2$ that $f(\mathbf{x}_1) \leq f(\mathbf{x}_2)$. Then,
either $B(\mathbf{\mathbf{x}_1}) \subseteq B(\mathbf{x}_2)$ or
$B(\mathbf{x}_1) \cap B(\mathbf{x}_2) = 0.$ \ \ $\Box$
\end{theorem}

Theorem \ref{the:setinclusion} implies that the barrier structure of
a landscape can be represented as a tree where the saddle points
are the branching points and the local optima are the leaves. The
{\em flooding algorithm} (see Algorithm \ref{alg:flood}) can be used
for the construction of the barrier tree in discrete landscapes with
finite search space $\mathcal{X}$ and linearly ordered search points
(e.g. by means of the objective function values). Note that if the
height function is not injective, the flooding algorithm can still be
used but the barrier tree may not be uniquely defined. The reason
for this is that there are different possibilities of how to sort
elements with equal heights in line 1 of algorithm \ref{alg:flood}.

Finally, let us look whether concepts such as saddle points, basins,
and barrier trees can be generalized in a meaningful way for
partially ordered landscapes. Flamm and Stadler \cite{FS03} recently
proposed one way of generalizing these concepts. We will review
their approach briefly and refer to the paper for details.

Adjacent points in linearly ordered landscapes are always
comparable. This does not hold in general for partially ordered
landscapes. We have to modify the paths $\mathbf{p}$ that enter the
definition.

\begin{definition} Maximal points on a path\\
The set of maximal points on a path $\mathbf{p}$ is defined as
 $\sigma(\mathbf{p}) =
     \{ \mathbf{x} \in \mathbf{p} |
     \nexists \mathbf{x}' \in \mathbf{p}: f(\mathbf{x}) \prec f(\mathbf{x}') \}$
\end{definition}

\begin{definition} Poset saddle-points\\
$\Sigma_{\mathbf{x}_1, \mathbf{x}_2} = \bigcup_{\mathbf{p} \in
\mathbb{P}_{\mathbf{x}_1, \mathbf{x}_2}} \sigma(\mathbf{p})$ is the
set of maximal elements along all possible paths. {\em Poset-saddle
points} are defined as the Pareto optima\footnote{here we think of
minimization of the objectives.} of $\Sigma_{\mathbf{x}_1,
\mathbf{x}_2}$: $S(\mathbf{x}_1, \mathbf{x}_2) := \{\mathbf{z} \in
\Sigma_{\mathbf{x}_1, \mathbf{x}_2} | \nexists \mathbf{u} \in
\Sigma_{\mathbf{x}_1, \mathbf{x}_2}:  f(\mathbf{u}) \prec
f(\mathbf{z}) \}$
\end{definition}

The flooding algorithm can be modified in a way that incomparable
elements are not considered as neighbors ('moves to incomparable
solutions are disallowed'). The flooding algorithm may then lead to
a forest instead of a tree. For examples (multicriteria knapsack
problem, RNA folding) and further discussion of how to efficiently
implement this algorithm, the reader is referred to \cite{FS03}.

A barrier tree for a continuous landscape is drawn in Fig.
\ref{fig:barriertree1d}. In this case the saddle points correspond
to local maxima. For continuous landscapes the concept of barrier
trees can be generalized, but the implementation of flooding
algorithms is more challenging due to the infinite number of points
that need to be considered. Discretization could be used to get a
rough impression of the landscape's geometry.

\begin{figure}
\begin{center}
\includegraphics[width=10cm]{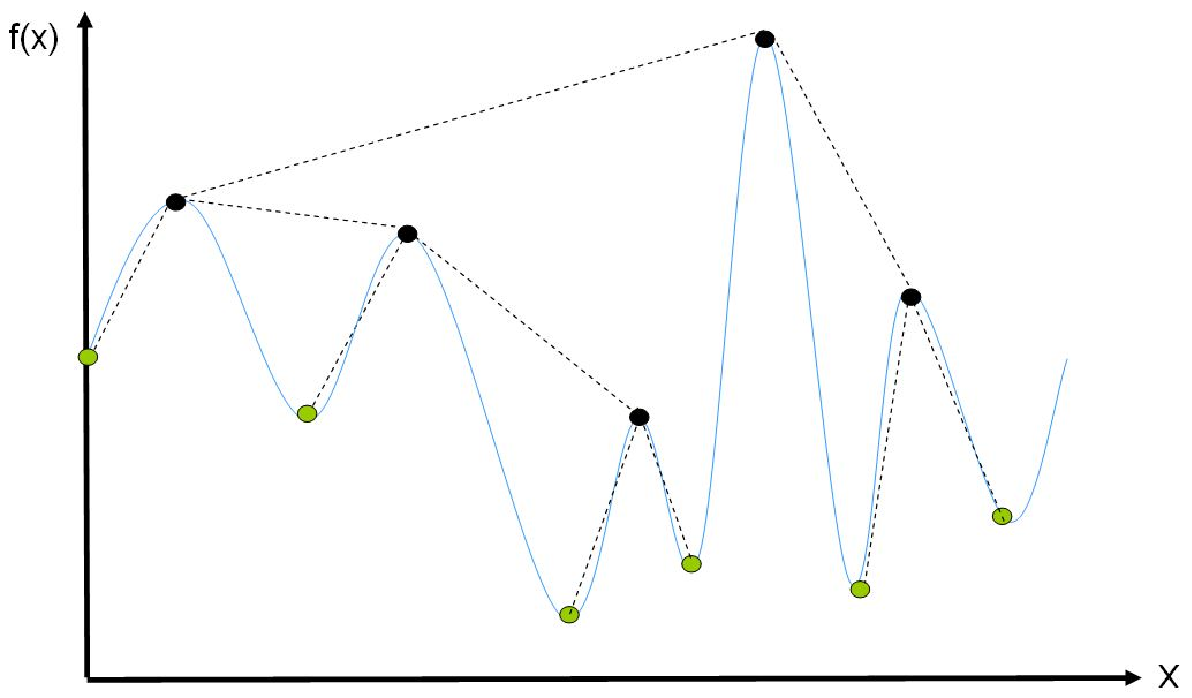}
\end{center}
\caption{\label{fig:barriertree1d} A barrier tree of a 1-D
continuous function.}
\end{figure}

\begin{algorithm}
\caption{Flooding algorithm} \label{alg:flood}
\begin{algorithmic}[1]
\STATE Let $x^{(1)}, \dots, x^{(N)}$ denote the elements of the
search space sorted in ascending order.
\STATE $i \rightarrow 1; \mathcal{B} = \emptyset$
\WHILE{$i \leq N$}
\IF{$N(x_i) \cap \{x^{(1)}, \dots, x^{(i-1)}\} = \emptyset$ [i.\,e.,
$x^{(i)}$ has no neighbour that has been processed.]}
\STATE \COMMENT{$x^{(i)}$ is local minimum}
\STATE {\bf Draw} $x^{(i)}$ as a new leaf representing basin
$B(x^{(i)})$ located at the height of $f$ in the 2-D diagram
\STATE $\mathcal{B} \leftarrow \mathcal{B} \cup \{B(x^{(i)})\}$
\qquad \COMMENT{Update set of basins}
\ELSE
\STATE Let $\mathcal{T}(x^{(i)}) = \{B(x^{(i_1)}), \dots
,B(x^{(i_N)})\}$ be the set of basins $B \in \mathcal{B}$ with
$N(x^{(i)}) \cap B \neq \emptyset$.
\IF{$|\mathcal{T}(x^{(i)})|$ = 1}
\STATE $B(x^{(i_1)}) \leftarrow B(x^{(i_1)}) \cup \{x^{(i)}\}$
\ELSE
\STATE \COMMENT{$x^{(i)}$ is a saddle point}
\STATE {\bf Draw} $x^{(i)}$ as a new branching point connecting the
nodes for $B(x^{(i_1)}), \dots ,B(x^{(i_N)})$.  Annotate saddle
point node with $B(x^{(i)})$ and locate it at the height of $f$ in
the 2-D diagram
\STATE \COMMENT{Update set of basins}
\STATE $B(x^{(i)}) = B(x^{(i_1)}) \cup \cdots \cup B(x^{(i_N)}) \cup
\{x^{(i)}\}$
\STATE Remove $B(x^{(i_1)}), \dots, B(x^{(i_N)})$ from $\mathcal{B}$
\STATE $\mathcal{B} \leftarrow \mathcal{B} \cup \{B(x^{(i)})\}$
\qquad
\ENDIF \ENDIF
%
%
%
%
\ENDWHILE
\end{algorithmic}
\end{algorithm}

An alternative generalization could be achieved by using the concept of $\epsilon$-dominance of a solution relative to the Pareto front; see \cite{Emmerich2025eps}.
$\epsilon$-dominance measures how much improvement is needed (epsilon added to both objective function value) so that the vector becomes non-dominated. As each decision vector has a unique level of epsilon, we can use this level to assess the proximity of a decision vector to a global optimum. The $\epsilon$-landscape can then be analyzed with landscape methods for single-objective optimization.

\section{Shapes of Pareto Fronts}

An interesting, since very general, question could be: How can the
geometrical shapes of Pareto fronts be classified? We will first
look at some general descriptions used in literature on how to
define the Pareto front w.r.t. convexity and connectedness. To state
definitions in an unambiguous way we will make use of Minkowski sums
and cones as defined in chapter \ref{chap:orders}.
\begin{figure}
\begin{center}
\includegraphics[width=12cm]{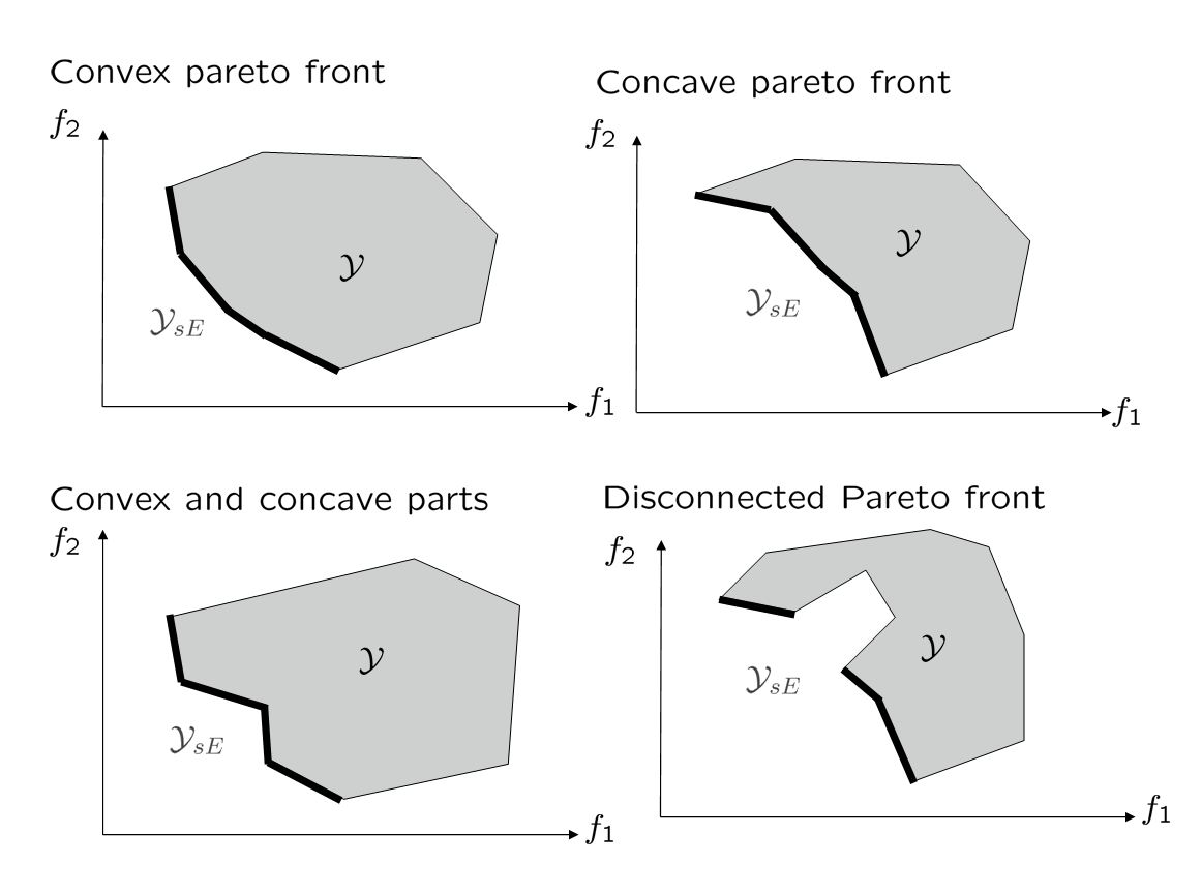}
\end{center}
\caption{\label{fig:paretoshapes}Different shapes of Pareto fronts
for bi-criteria problems.}
\end{figure}

\begin{definition}
A set $Y \subseteq \mathbb{R}^m$ is said to be {\em cone convex}
w.r.t. the positive orthant, iff $\mathcal{Y}_{N} +
\mathbb{R}_\geq^m$ is a convex set.
\end{definition}

\begin{definition}
A set $Y \subseteq \mathbb{R}^m$ is said to be {\em cone concave}
w.r.t. the positive orthant, iff $\mathcal{Y}_{N} -
\mathbb{R}_\geq^m$ is a convex set.
\end{definition}

\begin{definition}
A Pareto front $\mathcal{Y}_N$ is said to be convex (concave), iff
it is cone convex (concave) w.r.t. the positive orthant.
\end{definition}

Note, that Pareto fronts can be convex, concave, or may consist of
cone convex and cone concave parts w.r.t. the positive orthant.

Convex Pareto fronts allow for better compromise solutions than
concave Pareto fronts. In the ideal case of a convex Pareto front,
the Pareto front consists only of a single point which is optimal
for all objectives. In this situation, the decision maker can choose
a solution that satisfies all objectives at the same time. The most
difficult situation for the decision maker arises when the Pareto  
front consists of a separate set of points, one point for each  
single objective, and these points are separate and very distant
from each other. In such a case, the decision maker needs to make an  
either-or decision.

Another classifying criterion of Pareto fronts is connectedness.

\begin{definition}
A Pareto front $\mathcal{Y}_N$ is said to be connected, if and only if for all
${\bf y}_1, {\bf y}_2 \in \mathcal{Y}_N$ there exists a continuous
mapping  $\phi: [0,1] \rightarrow \mathcal{Y}_N$ with $\phi(0)= {\bf
y}_1$ and $\phi(1)= {\bf y}_2$.
\end{definition}

For the frequently occurring case of two objectives, examples of
convex, concave, connected and disconnected Pareto fronts are given
in Fig. \ref{fig:paretoshapes}.

Two further corollaries highlight general characteristics of
Pareto-fronts:

\begin{lemma} Dimension of the Pareto front

 Pareto fronts for problems with $m$-objectives are subsets or equal
 to $m-1$-dimensional manifolds.
\end{lemma}

\begin{lemma} Functional dependencies in the Pareto front
\label{lem:unique}

Let $\mathcal{Y}_N$ denote a Pareto front for some multiobjective
problem. Then for any sub-vector in the projection to the
coordinates in $\{1, \dots, m\}$ without $i$, the value of the
$i$-th coordinate is uniquely determined.
\end{lemma}

\begin{example}
For a problem with three objective functions, the Pareto front is a subset of
a 2-D manifold that can be represented as a function from the values
of the
\begin{itemize}
\item first two objectives to the third objective.
\item the first and third objective to the second objective
\item the last two objectives to the first objective
\end{itemize}
\end{example}

\section{Conclusions}
This chapter focused on types of landscapes and optimal shapes in multiobjective optimization. A method using level sets to visualize and analyze multicriteria landscapes was provided, aiding in understanding linear programming problems. We examined the structure of discrete landscapes through barrier trees, which help identify local optima, valleys, and hills and their separations. Additionally, we classified Pareto front shapes and highlighted their characteristics. Current definitions do not use differentiability, but future chapters will explore theorems supporting the search for Pareto optima with differentiability. 

\section*{Exercises}
\begin{enumerate}
    \exercise{Proper efficiency and Pareto optimality}{
        Identify all (proper, strictly) efficient points for the problem
        \[
        f_1(x)=x^2 \rightarrow \min, \quad f_2(x)=(x-1)^2 \rightarrow \min, \quad x \in [0,2]
        \]
        and show that the theorems on level sets apply for them. Additionally, draw the attainable set and the Pareto front of the problem in a coordinate diagram with axes $f_1$ and $f_2$.}
        {ex:pareto}

    \exercise{Level set theorems}{
        Consider the 2-D problem
        \[
        f_1(x) = |x_1| + |x_2| \rightarrow \min, \quad f_2(x) = |x_1-1| + |x_2-1| \rightarrow \min.
        \]
        with $(x_1,x_2) \in [0,2] \times [0,2]$.
        Draw the contour lines of the functions in two different colors in a 2-D diagram and identify all efficient points using the theorems on level sets.}
        {ex:levelsets}

    \exercise{Multiobjective discrete landscapes}{
        Consider the following integer knapsack problem with a penalty for excess items:
        \[
        f_1(x) = 2x_1 + 3x_2 \rightarrow \max, \quad f_2(x) = x_1 + 2x_2 \rightarrow \min, \quad x_1, x_2 \in \{0,1,2,3\}.
        \]
        Identify all locally efficient points. Draw the Pareto front of the problem in a 2-D coordinate diagram with axes $f_1$ and $f_2$ (mind that one of the objectives is a maximization objective).
        
        \textit{Bonus:} Use the flooding algorithm to identify the Barrier tree of the function 
        \[
        f_1(\mathbf{x}) - \max\{0, 20 + (f_2(\mathbf{x}) - \text{MAXWEIGHT})\},
        \]
        with $\text{MAXWEIGHT} = 10$.}
        {ex:knapsack}

    \exercise{Multiobjective Linear Programming}{
        Consider the search space $\mathcal{S} = [0,2] \times [0,3]$ and the objectives
        \[
        f_1(x_1, x_2) = 2 + \frac{1}{3} x_2 - x_1 \rightarrow \min, \quad
        f_2(x_1, x_2) = \frac{1}{2} x_2 + \frac{1}{2} x_1 \rightarrow \max.
        \]
        The following constraints must be satisfied:
        \[
        g_1(x_1,x_2) = -2 + \frac{2}{3} x_1 + x_2 \leq 0, \quad g_2(x_1,x_2) = -4 + 2 x_1 + x_2 \leq 0.
        \]
        To solve this problem, mark the constrained region graphically (as in Figure \ref{fig:bilinprog}). First, check whether the points $(1,1)$ and $(1.5,1.0)$ are Pareto efficient. Now determine all efficient points or regions.
        \textit{Hint:} Visualize the intersections of level sets of improvement for both $f_1$ and $f_2$ at the given points.}
        {ex:multiobjectivelp}
\end{enumerate}

\chapter{Optimality conditions for dif\-fe\-ren\-tiable pro\-blems}
\label{chap:opticond}
In the finite discrete case, local optimality of a point $x \in
\mathcal{X}$ can be done by comparing it to all neighboring
solutions. In the continuous case, this is not possible. For
differentiable problems, we can state conditions for local
optimality. We will start with looking at unconstrained
optimization, then provide conditions for optimization with equality
and  inequality constraints and, thereafter, their extensions for
the multiobjective case.

\section{Linear approximations}

A general observation we should keep in mind when understanding
optimization conditions for differentiable problems is that
continuously differentiable functions $f: \mathbb{R}^n \rightarrow
\mathbb{R}$ can be locally approximated at any point ${\bf x}^{(0)}$
by means of linear approximations $f({\bf x}^{(0)}) + \nabla f({\bf
x}^{(0)}) ({\bf x} - {\bf x}_0)$ with $\nabla f = (\frac{\partial
f}{\partial x_1},  \dots, \frac{\partial f}{\partial x_n})^\top$. In
other words:
\begin{equation}
\lim_{{\bf x} \rightarrow {\bf x}_0}  f({\bf x}) - [{f({\bf x}_0) +
\nabla f({\bf x}) ({\bf x} - {\bf x}_0)}] = 0
\end{equation}

The gradient $\nabla f({\bf x}^{(0)})$ points in the direction of
steepest ascent and is orthogonal to the level curves ${\mathcal
L}_=(f(\hat{\bf x}))$ at the point $\hat{\bf x}$. This has been
visualized in Fig. \ref{fig:locallinear}.

\begin{figure}
\begin{center}
\includegraphics[width=10cm]{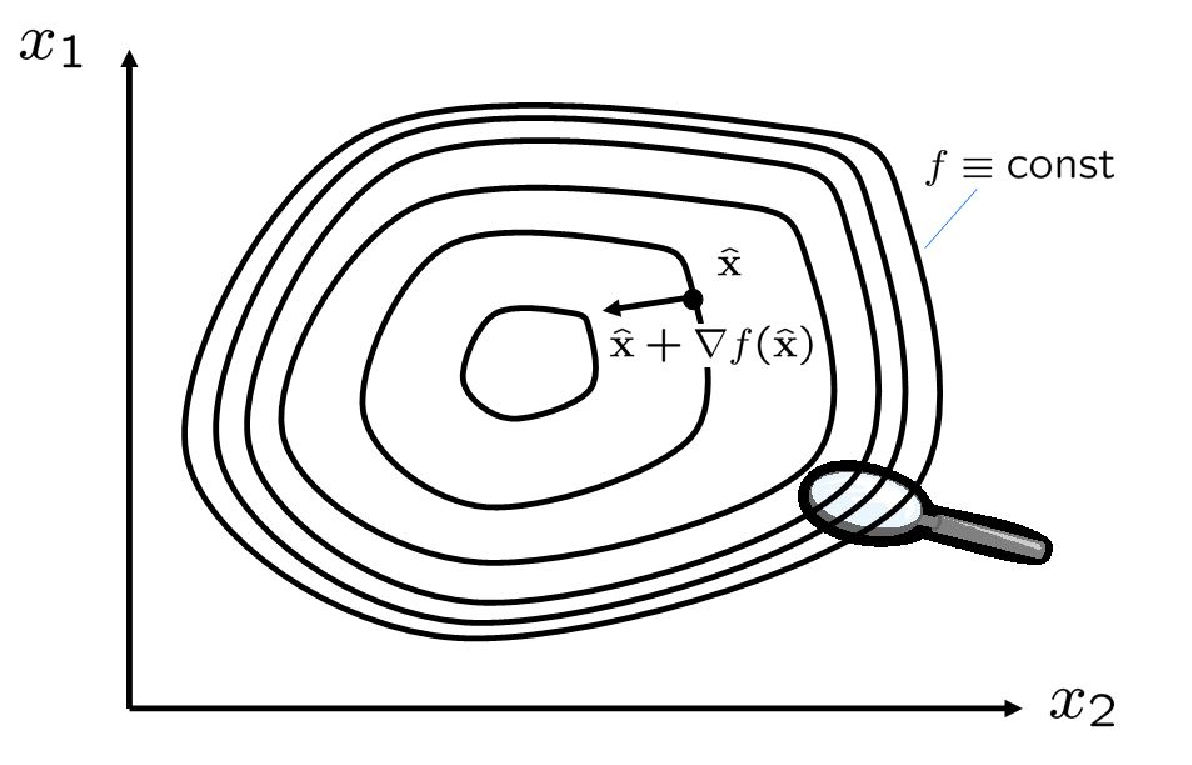}
\end{center}
\caption{\label{fig:locallinear} Level curves of a continuously
differentiable function. Locally the function 'appears' to be a
linear function with parallel level curves. The gradient vector
$\nabla f(\hat {\bf x})$ is perpendicular to the local direction of
the level curves at $\hat{\bf x}$.}
\end{figure}

\section{Unconstrained Optimization}
For the unconstrained minimization
\begin{equation}
f({\bf x}) \rightarrow \min
\end{equation}
problem, a well known result from calculus is:
\begin{theorem} Fermat's condition\\
\label{theo:fermat}
Given a differentiable function $f$. Then $\nabla f({\bf x}^*) = 0$
is a necessary condition for ${\bf x}^*$ to be a local extremum.
Points with $\nabla f({\bf x}^*) = 0$ are called \mbox{stationary
points}. A sufficient condition for ${\bf x}^*$ to be a (strict)
local minimum is given, if in addition the Hessian matrix $\nabla^2
f({\bf x}^*)$ is positive (semi)definite.
\end{theorem}
The following theorem can be used to test whether a matrix is
positive (semi)definite:

\begin{theorem}
A matrix is positive (semi-)definite, iff all eigenvalues are
positive (non-negative).
\end{theorem}

Alternatively, we may use local bounds to decide whether we have
obtained a local or global optimum. For instance, for the problem
$\min_{(x,y) \in \mathbb{R}^2} (x-3)^2 + y^2 + \exp y$ the bound of
the function is zero and every argument for which the function
reaches the value of zero must be a global optimum. As the function
is differentiable the global optimum will be also be one of the
stationary points. Therefore we can find the global optimum in this
case by looking at all stationary points. A more general way of
looking at boundaries in the context of optimum seeking is given by
the Theorem of Weierstrass discussed in \cite{BS95}. This theorem is
also useful for proving the existence of an optimum. This is
discussed in detail in \cite{Bri05}.

\begin{theorem} Theorem of Weierstrass\\
Let $\mathcal{X}$ be some closed\footnote{Roughly speaking, a closed
set is a set which includes all points at its boundary.} and bounded
subset of $\mathbb{R}^n$, let $f: \mathcal{X} \rightarrow
\mathbb{R}$ denote a continuous function. Then $f$ attains a global
maximum and minimum in  $\mathcal{X}$, i.~e. $\exists {\bf x}_{\min}
\in \mathcal{X} : \forall {\bf x}' \in \mathcal{X} : f({\bf
x}_{\min}) \leq f({\bf x}')$ and $\exists {\bf x}_{\max} \in
\mathcal{X} : \forall {\bf x}' \in \mathcal{X} : f({\bf x}_{\max})
\geq  f({\bf x}')$.
\end{theorem}

\section{Equality Constraints}
By introducing Lagrange multipliers, theorem \ref{theo:fermat} can
be extended to problems with {\em equality} constraints, i.\,e.:
\begin{equation}
f({\bf x}) \rightarrow \min, \mbox{ s.t. }g_1({\bf x}) = 0, \dots,
g_m({\bf x}) = 0
\end{equation}
In this case the following theorem holds:
\begin{theorem} \label{theo:Lagrange}
Let $f$ and $g_1$, $\dots$, $g_m$ denote differentiable functions.
Then a necessary condition for ${\bf x}^*$ to be a local extremum is
given, if there exist multipliers $\lambda_1$, $\dots$,
$\lambda_{m+1}$ with at least one $\lambda_i \neq 0$ for $i = 1,
\dots, m$ such that $\lambda_1 \nabla f({\bf x}^*) +
\sum_{i=2}^{m+1} \lambda_i \nabla g({\bf x}^*) = {\bf 0}$.
\end{theorem}
For a rigorous proof of this theorem we refer to \cite{Bri05}. Let
us remark, that the discovery of this theorem by Lagrange preceded
its proof by one hundred years \cite{Bri05}.

Next, by means of an example we will provide some geometric
intuition for this theorem. In Fig. \ref{fig:lagrange1} a problem
with a search space of dimension two is given. A single objective
function $f$ has to be maximized, and the sole constraint function
$g_1({\bf x})$ is set to $0$.

Let us first look at the level curve $f \equiv -13$. This curve does
not intersect with the level curve $g \equiv 0$ and thus there is no
feasible solution on this curve. Next, we look at $f \equiv -15$. In
this case the two curve intersects in two points with $g \equiv 0$.
However, these solutions are not optimal. We can do better by moving
to the point, where the level curve of $f \equiv c$ 'just'
intersects with $g \equiv 0$. This is the tangent point ${\bf x}^*$
with $c = f({\bf x}^*) = -14$.
\begin{figure}
\begin{center}
\psfrag{x*}{${\bf x}^*$}
\psfrag{x1}{$x_1$}
\psfrag{x2}{$x_2$}
\psfrag{Dg1}{$\nabla g_1({\bf x}^*)$}
\psfrag{Df1}{$\nabla f({\bf x}^*)$}
\psfrag{f(x)=const}{$f \equiv \mbox{const}$}
\psfrag{g(x)=0}{$g \equiv \mbox{const}$}
\psfrag{-14}{$-14$}
\psfrag{-15}{$-15$}
\psfrag{-13}{$-13$}
\includegraphics[width=10cm]{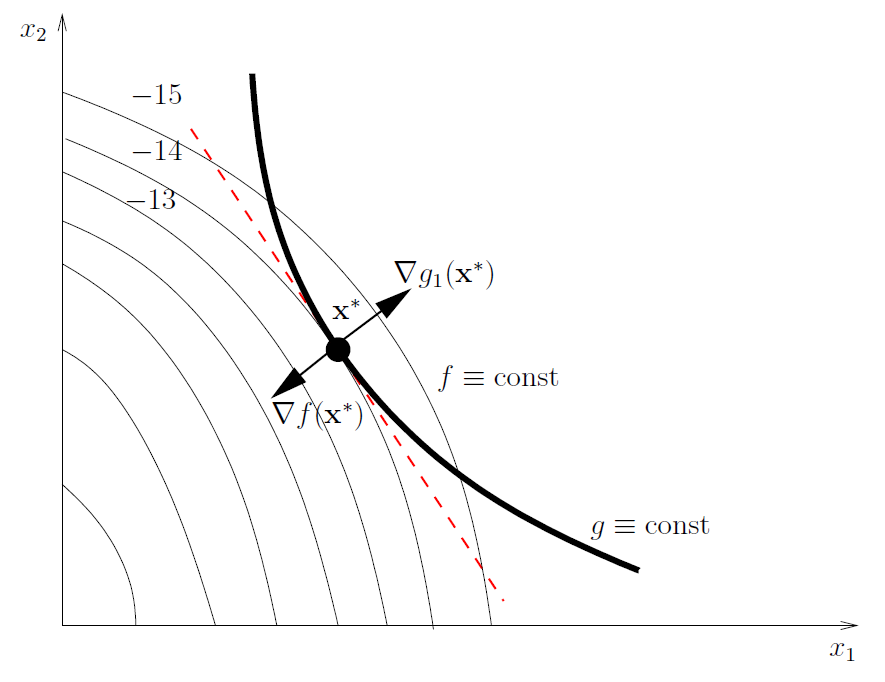}
\end{center}
\caption{\label{fig:lagrange1}Lagrange multipliers: Level-sets for a
single objective and single active constraint and search space
$\mathbb{R}^2$.}
\end{figure}

The tangential point satisfies the condition that the gradient
vectors are collinear to each other, i.e. $\exists \lambda \neq 0:
\lambda \nabla g({\bf x}^*) = \nabla f({\bf x}^*)$. In other words,
the tangent line to the $f$ level curve at a touching point is equal
to the tangent line to the $g \equiv 0$ level curve. Equality of
tangent lines is equivalent to the fact that the gradient vectors
are collinear.

Another way to reason about the location of optima is to check for
each point on the constraint curve whether it can be locally
improved or not. For points where the level curve of the objective
function intersects with the constraint function, we consider the
local linear approximation of the objective function. In case of
non-zero gradients, we can always improve the point further. In case
of zero gradients we already fulfill conditions of the theorem by
setting $\lambda_1 = 1$ and  $\lambda_i = 0$ for $i = 2, \dots,
m+1$. This way we can exclude all points but the tangential points
and local minima of the objective function (unconstrained) from
consideration.

In practical optimization often $\lambda_1$ is set to $1$. Then the
equations in the lagrange multiplier theorem boil down to an
equation system with $m+n$ unknowns and $m+n$ equations and this
gives rise to a set of candidate solutions for the problem. This way
of solving an optimization problem is called the {\em Lagrange
multiplier rule}.
\begin{example}
Consider the following problem:
\begin{equation}
f(x_1,x_2) = x^2_1+x^2_2 \rightarrow \min
\end{equation},
with equality constraint
\begin{equation}\label{eq:exLagrange2}
g(x_1,x_2) = x_1+x_2-1 = 0
\end{equation}
Due to the theorem of \ref{theo:Lagrange}, iff $(x_1, x_2)^\top$ is
a local optimum, then there exist $\lambda_1$ and $\lambda_2$ with
$(\lambda_1, \lambda_2) \neq (0,0)$ such that the constraint in
equation \ref{eq:exLagrange2} is fulfilled and
\begin{equation}
\lambda_1 \frac{\partial f}{\partial x_1} + \lambda_2 \frac{\partial
g}{\partial x_1} =  2 \lambda_1 x_1 + \lambda_2 = 0
\end{equation}
\begin{equation}
\lambda_1 \frac{\partial f}{\partial x_2} + \lambda_2 \frac{\partial
g}{\partial x_2} =   2 \lambda_1 x_2 + \lambda_2  = 0
\end{equation}
Let us first examine the case $\lambda_1 =0$. This entails:
\begin{equation}
\lambda_2 = 0
\end{equation}
 This contradicts the condition that $(\lambda_1, \lambda_2) \neq
 (0,0)$.

We did not yet prove that the solution we found is also a {\em
global} optimum. In order to do so we can invoke Weierstrass
theorem, by first reducing the problem to a problem with a reduced
search space, say:
\begin{equation}
f_{|A} \rightarrow \min
\end{equation}
\begin{equation}
 A=\{(x_1,x_2) | |x_1| \leq 10 \mbox{ and } |x_2|
\leq 10 \mbox{ and } x_1+x_2-1=0 \}
\end{equation}

For this problem a global minimum exists, due to the Weierstrass
theorem (the set $A$ is bounded and closed and $f$ is continuous).
Therefore, the original problem also has a global minimum in $A$, as
for points outside $A$ the function value is  bigger than $199$ and
in $A$ there are points $x \in A$ where $f(x_1,x_2) < 199$. The
(necessary) Lagrange conditions, however, are only satisfied for one
point in $\mathbb{R}^2$ which consequently must be the only local
minimum and thus it is the global minimum.

 Now we consider the case
$\lambda_1 = 1$. This leads to the conditions:
\begin{equation}
2 x_1 + \lambda_2 = 0
\end{equation}
\begin{equation}
2 x_2 + \lambda_2 = 0
\end{equation}
and hence $x_1=x_2$. From the equality condition we get:
From the constraint it follows $x_1+x_1 = 1$, which entails
$x_1=x_2=\frac{1}{2}$.

Another possibility to solve this problem is by means of
substitution: $x_1 = 1- x_2$ and the objective function can then be
written as $f(1-x_2,x_2) = (1-x_2)^2 + x_2^2$. Now minimize the
unconstrained 'substitute' function $h(x_2) = (1-x_2)^2 + x_2^2$.
$\frac{\partial h}{x_2} = - 2 (1 - x_2) + 2 x_2 = 0$. This yields
$x_2 = \frac{1}{2}$. The second derivative $\frac{\partial^2
f}{\partial^2 x_2}= 4$. This means that the point is a local
minimum.

\end{example}
 However, not always all candidate solutions for
local optima are captured this way as the case $\lambda_1 = 0$ may
well be relevant. Brinkhuis and Tikhomirov \cite{Bri05} give an
example of such a 'bad' case:

\begin{example}
Apply the multiplier rule to  $f_0(x) \rightarrow \min,
x_1^2-x_2^3=0$: The Lagrange equations hold at $\hat x$ with
$\lambda_0 = 0$ and $\lambda_1 =1$. An interesting observation is
that the level curves are cusps in this case at $\hat x$, as
visualized in Fig. \ref{fig:tikhomirov}.
\end{example}

\begin{figure}
\begin{center}
\includegraphics[width=10cm]{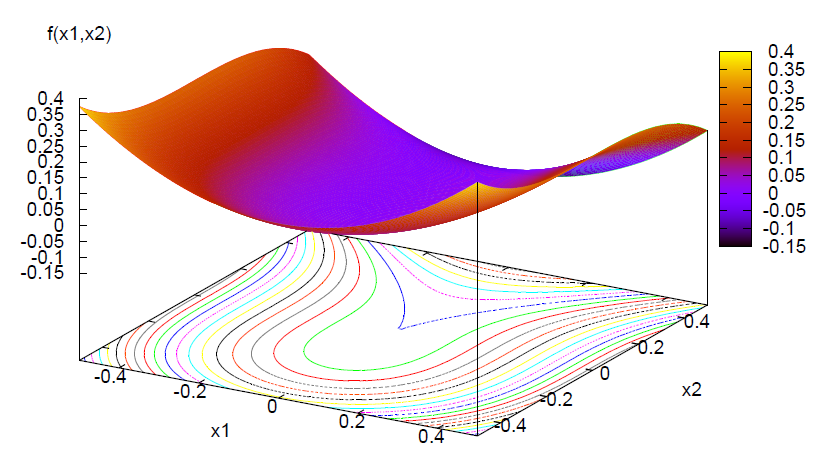}
\end{center}
\caption{\label{fig:tikhomirov} The level curves of $x_1^2-x_2^3$.
The level curve through $(0,0)^T$ is cusp.}
\end{figure}

\section{Inequality Constraints}

For {\em inequality} constraints, the Karush-Kuhn-Tucker (KKT) conditions \cite{Karush39,KuhnTucker51}
are used as an optimality criterion\footnote{The conditions for single-objective optimization with inequality constraints were independently discovered by Karush \cite{Karush39} and by Kuhn and Tucker \cite{KuhnTucker51}, who also included a multicriteria optimization formulation.}:

\begin{theorem}
The Karush-Kuhn-Tucker conditions are said to hold for ${\bf x}^*$,
if there exist multipliers $\lambda_1 \geq 0$, $\dots$,
$\lambda_{m+1} \geq 0$ and at least one $\lambda_i > 0$ for $i = 1,
\dots, m+1$, such that:
\begin{equation}
\lambda_1 \nabla f({\bf x}^*) + \sum_{i=1}^{m} \lambda_{i+1} \nabla
g_i({\bf x}^*) = {\bf 0} \label{eq:kktcone}
\end{equation}
\begin{equation}
 \lambda_{i+1} g_i({\bf x}^*) = 0, i = 1, \dots, m \label{eq:nonslack}
\end{equation}
\end{theorem}

\begin{theorem} \label{the:necessaryKKT} Karush-Kuhn-Tucker Theorem (Necessary conditions for smooth, convex programming:)\\ Assume the
objective and all constraint functions are convex in some
$\epsilon$-neighborhood of ${\bf x}^*$, If ${\bf x}^*$. Then there is a local
minimum; then there exist $\lambda_1, \dots, \lambda_{m+1}$ such
that KKT conditions are fulfilled.
\end{theorem}

In order to state sufficient conditions for optimality, we need to introduce so-called constraint qualifications, which are conditions about the local behavior of constraint functions in $\mathbf{x}^*$. 
The original constraint qualification used in the KKT theorem is somewhat difficult to check, and there exist stricter versions that are easier to check. The often-applied Linear Independence Constraint Qualification (LICQ) states that the gradients of the active constraints are linearly independent \cite{Peterson73}.

\begin{theorem} \label{the:suffKKT}
The KKT conditions are sufficient for optimality, provided
$\lambda_1 = 1$ and the constraint qualifications are satisfied and constraint qualifications are satisfied for active constraints in $\mathbf{x}^*$. In this case ${\bf x}^*$ is a local minimum.
\end{theorem}

Note that if ${\bf x}^*$ is in the interior of the feasible region
(a Slater point), all $g_i({\bf x}) < 0$ and thus $\lambda_1 > 0$.
This follows directly from the so-called non-slackness conditions in equation \ref{eq:nonslack}. Furthermore, let us observe that this equation makes sure that only active constraints in ${\bf x}^*$ enter the equation \ref{eq:kktcone}. Equation \ref{eq:kktcone} essentially states that the negative gradient of the objective function must be contained in the polyhedral cone generated by the gradients of the active constraint's gradients at ${\bf x}^*$. See, Section \ref{sec:pardom} for an introduction of the concept of a polyhedral cone.

\begin{example} KKT Conditions and Active Constraints.
To gain intuition into the Karush-Kuhn-Tucker (KKT) conditions, consider the constrained optimization problem where the objective function \( f(x) \) is minimized subject to the constraints \( g_{a_1}(x) \leq 0 \) and \( g_{a_2}(x) \leq 0 \). The point \( x^* \) in the figure represents an optimal solution where both constraints are active, meaning \( g_{a_1}(x^*) = 0 \) and \( g_{a_2}(x^*) = 0 \).

At \( x^* \), the gradients of the active constraints, \( \nabla g_{a_1}(x^*) \) (blue) and \( \nabla g_{a_2}(x^*) \) (red), define the feasible direction space. The negative gradient of the objective function, \( -\nabla f(x^*) \), must lie within the cone formed by the active constraint gradients. This aligns with the KKT stationarity condition, which states that at an optimal solution, \( \nabla f(x^*) \) can be expressed as a linear combination of the gradients of the active constraints.
The shaded regions in the figure illustrate the feasible regions satisfying \( g_{a_1}(x) \leq 0 \) (green) and \( g_{a_2}(x) \leq 0 \) (blue). The contours of \( f(x) \) (dashed) indicate that \( x^* \) is at a local optimum where the level sets of \( f(x) \) are tangent to the feasible region boundary.
\begin{figure}
    \centering
    \includegraphics[width=0.75\linewidth]{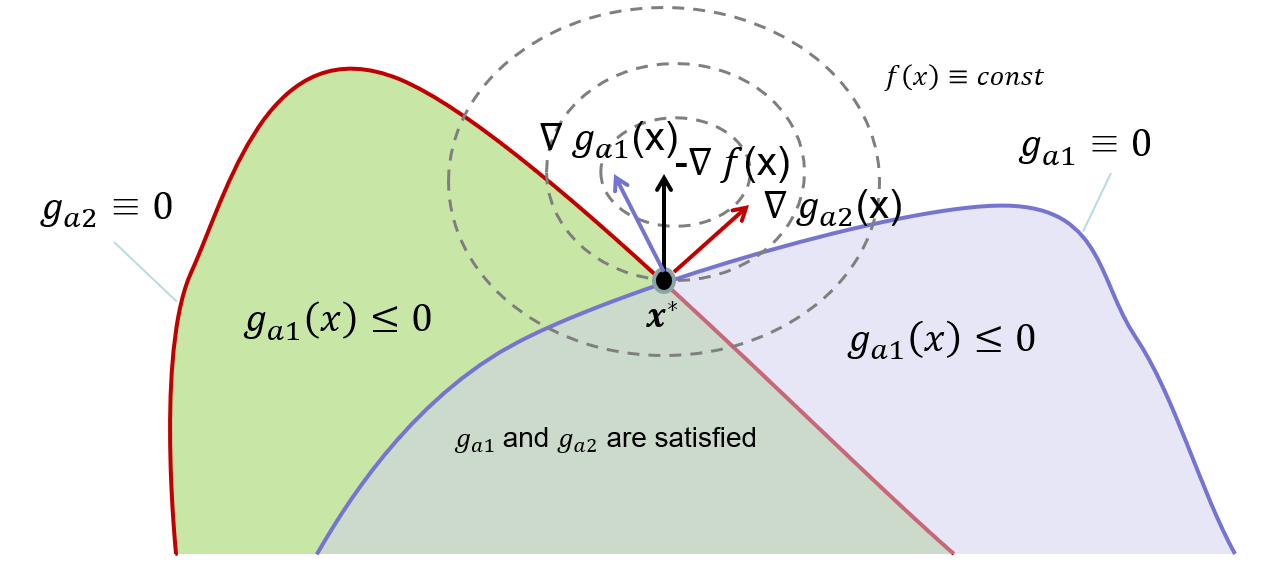}
    \caption{Example for KKT Conditions with two active constraints.}
    \label{fig:enter-label}
\end{figure}
\end{example}


The next examples discuss the usage of the Karush-Kuhn-Tucker conditions:
\begin{example}
 In order to get familiar with using the KKT theorem we apply it to a very
simple situation (solvable also with high school mathematics).   The
task is:
\begin{equation}
1 - x^2 \rightarrow \min,  x \in [-1,3]^2
\end{equation}

First, write the task in its standard form:
\begin{equation}
f(x) = 1- x^2 \rightarrow \min
\end{equation}
subject to constraints
\begin{equation}
g_1(x) = -x-1 \leq 0
\end{equation}
\begin{equation}
 g_2(x) = x-3 \leq 0
\end{equation}

The existence of the optimum follows from Weierstrass theorem, as
(1) the feasible subspace [-1,3] is bounded and closed and (2) the
objective function is continuous.

The KKT conditions in this case boil down to:
There exists $\lambda_1 \in \mathbb{R}, \lambda_2 \in
\mathbb{R}^+_0$ and $\lambda_3 \in \mathbb{R}^+_0$ and $(\lambda_1,
\lambda_2, \lambda_3) \neq (0,0,0)$ such that

\begin{equation} \label{eq:example1Kkt1}
\lambda_1 \frac{\partial f}{\partial x} + \lambda_2 \frac{\partial
g_1}{\partial x} + \lambda_3 \frac{\partial g_1}{\partial x} = -2
\lambda_1 x -\lambda_2 +\lambda_3 =0
\end{equation}
\begin{equation}\label{eq:example1Kkt2}
\lambda_2(-x-1)=0
\end{equation}
\begin{equation}\label{eq:example1Kkt3}
\lambda_3(x-3)=0
\end{equation}.

First, let us check whether $\lambda_1=0$ can occur:

In this case the three equations (\ref{eq:example1Kkt1},
\ref{eq:example1Kkt2}, and \ref{eq:example1Kkt3}) will be:
\begin{equation}\label{eq:example1lambdaZeroKkt1}
 -\lambda_2 +\lambda_3 =0
\end{equation}
\begin{equation} \label{eq:example1lambdaZeroKkt2}
\lambda_2(-x-1)=0
\end{equation}
\begin{equation} \label{eq:example1lambdaZeroKkt3}
\lambda_3(x-3)=0
\end{equation}.

and $(\lambda_2, \lambda_3) \neq (0,0) $, and $\lambda_i \geq 0,
i=2,3$. From \ref{eq:example1lambdaZeroKkt1} we see that
$\lambda_2=\lambda_3$. By setting $\lambda = \lambda_2$ we can write
\begin{equation}
\lambda(-x-1)=0
\end{equation} and
\begin{equation}
\lambda(x-3)=0
\end{equation} for the equations \ref{eq:example1lambdaZeroKkt2} and
\ref{eq:example1lambdaZeroKkt3}. Moreover $\lambda \neq 0$, for
$(\lambda, \lambda)=(\lambda_2, \lambda_3) \neq (0,0)$. From this we
get that $-x-1=0$ and $x-3=0$. Which is a contradiction so the case
$\lambda_1=0$ cannot occur -- later we shall see that this could
have derived by using a theorem on Slater points
\ref{the:suffKKT}.\\
Next we consider the case $\lambda_1 \neq 0$ (or equivalently
$\lambda_1=1$): In this case the three equations
(\ref{eq:example1Kkt1}, \ref{eq:example1Kkt2}, and
\ref{eq:example1Kkt3}) will be:
\begin{equation}
-2x - \lambda_2 + \lambda_3 =0
\end{equation},
\begin{equation}
\lambda_2(-x-1) =0
\end{equation}, and
\begin{equation}
\lambda_3(x-3) =0
\end{equation}
We consider four subcases:
\begin{description}
\item[case 1:] $\lambda_2=\lambda_3=0$. This gives rise to $x=0$
\item[case 2:] $\lambda_2 = 0$ and $\lambda_3 \neq 0$. In this case
we get as a condition on $x$: $2x(x-3)=0$ {\em and} $x \neq 0$ or
equivalently $x=3$
\item[case 3:] $\lambda_2 \neq 0$ and $\lambda_3 = 0$. We get from
this: $-2x(-x-1)=0$ {\em and} $x \neq 0$ or equivalently $x=-1$.
\item[case 4:] $\lambda_2 \neq 0$ and $\lambda_3 \neq 0$. This
cannot occur as this gives rise to $-x-1=0$ {\em and} $x-3=0$
(contradictory conditions).
\end{description}
In summary we see that a maximum can possibly only occur in $x=-1$,
$x=0$ or $x=3$. By evaluating $f$ on these three candidates, we see
that $f$ attains its global minimum in $x=3$ and the value of the
global minimum is $-8$. Note that we invoked also the Weierstrass
theorem in the last conclusion: the Weierstrass theorem tells us
that the function $f$ has a global minimum in the feasible region
([-1.3]) and KKT (necessary conditions) tell us that it must be one
of the three above mentioned candidates.
\end{example}

\section{Multiple Objectives}

For a generalization of the Lagrange multiplier rule to multiobjective optimization we refer to \cite{GJ99}.

For multicriterion optimisation the KKT conditions can be
generalized as follows:
\begin{theorem} Fritz John necessary conditions\\
A neccessary condition for ${\bf x}^*$ to be a locally efficient
point is that there exist vectors $\lambda_1, \dots, \lambda_k$ and
$\upsilon_1, \dots, \upsilon_m$ such that
\begin{equation}
{\bf \lambda} \succ {\bf 0}, {\bf \upsilon} \succ {\bf 0}
\end{equation}
\begin{equation}
\sum_{i=1}^k \lambda_i \nabla f_i({\bf x}^*) + \sum_{i=1}^{m}
\upsilon_{i} \nabla g_i({\bf x}^*) = {\bf 0}.
\end{equation}
\begin{equation}
\upsilon_{i} g_i({\bf x}^*) = 0, i = 1, \dots, m
\end{equation}
\end{theorem}

To make this condition sufficient we define regular points to be (locally) differentiable points $\mathbf{x}^*$ that satisfy constraint qualifications w.r.t. active constraints. See Maeda \cite{Maeda94} for constraint qualifications for multiobjective optimization.
A sufficient condition for points to be (locally) Pareto optima

\begin{theorem} Karush Kuhn Tucker sufficient conditions for a solution to be Pareto
optimal: 
Let ${\bf x}^*$ be feasible, with (locally) convex objectives and differentiable convex constraints, satisfying the Fritz John conditions in ${\bf x}^*$; then ${\bf x}^*$ is (locally) efficient.
\end{theorem}
\begin{figure}
\begin{center}
\includegraphics[width=10cm]{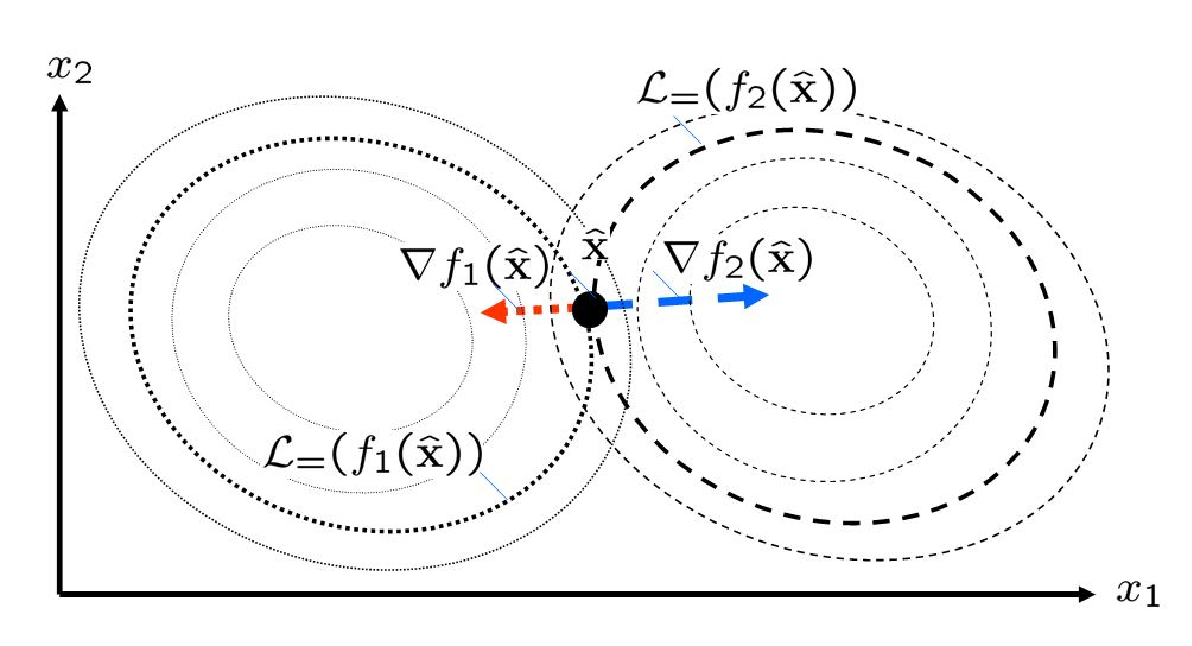}
\end{center}
\caption{\label{fig:touchinglevel} Level curves of the two
objectives touching in one point indicate locally Pareto optimal
points in the bi-criterion case, provided the functions are
differentiable.}
\end{figure}
A simple way to understand these conditions is to view the multipliers $\lambda_i$ as weights of a weighted sum scalarization with non-negative weights for the objective functions $\sum_{i=1}^n \lambda_i f_i(\mathbf{x})$, which we aim to minimize. 
In the unconstrained case we get the simple condition:
\begin{corollary}
In the unconstrained case Fritz John necessary conditions reduce to
\begin{equation}
{\bf \lambda} \succ {\bf 0}
\end{equation}
\begin{equation}
\sum_{i=1}^k \lambda_i  \nabla f_i({\bf x}^*) = {\bf 0}.
\end{equation}
\end{corollary}
In 2-dimensional spaces, this criterion reduces to the observation that either one of the objectives has a zero gradient (necessary condition for ideal points) or the gradients are collinear, as depicted in Fig. \ref{fig:touchinglevel}. We may, in this special case, use the angle between the gradients of the objective function as an indicator of the closeness to a locally efficient point, as it is, for instance, done in more recently proposed visualization methods for the landscape of multimodal multi-objective optimization problems (so-called gradient heat maps)
\cite{Kerschke19}.

This brief overview of the KKT condition for multiobjective optimization lacks rigorous proofs and instead relies on plausibility arguments. A more detailed mathematical treatment would necessitate introducing concepts that extend beyond the scope of this document. For a comprehensive treatment and detailed explanation of the conditions for multiobjective optimization, we suggest referring to \cite{KuhnTucker51},
\cite{Mie99}.

\section{Example for Analytical Solution of Multi-objective Problem}

This section presents an analytical example with a practical background: designing a rectangular area under the constraint of using a minimal amount of fencing. In other words, the goal is to design a rectangle that maximizes the enclosed area while keeping the perimeter (fencing) as short as possible. Although deliberately kept simple, the example illustrates key aspects of multiobjective optimization and the subtle relationship between the Karush-Kuhn-Tucker (KKT) conditions and linear weighting scalarization. In this example, although the KKT conditions for the multiobjective problem are satisfied (the gradients point in opposite directions), the candidate point is a saddle point in the scalarized formulation. The outline is as follows:
\begin{enumerate}
    \item \textbf{Problem Statement and Classical Methods:} We describe the design problem and review the substitution and Lagrange multiplier methods for a fixed-area case.
    \item \textbf{Weighted Sum Scalarization:} We form the weighted scalarization function using equal weights \(w_1=w_2=0.5\).
    \item \textbf{KKT and Hessian Analysis:} We analyze the candidate stationary point, showing that the gradients satisfy the KKT conditions though the Hessian is indefinite (indicating a saddle point).
    \item \textbf{Visualization:} A  surface plot is provided to illustrate the structure of the weighted scalarization function.
\end{enumerate}

\subsection{Problem Statement and Classical Methods}
Consider a rectangle with side lengths \(x_1\) and \(x_2\). The objectives are:
\begin{itemize}
    \item \textbf{Maximize the Area:}
    \[
    f_1(x_1,x_2)= x_1x_2,
    \]
    \item \textbf{Minimize the Perimeter:}
    \[
    f_2(x_1,x_2)= 2x_1+2x_2.
    \]
\end{itemize}
A practical interpretation is to maximize the enclosed area while limiting the use of fencing.
A illustration of the rectangle is shown in Figure~\ref{fig:rectangle}.

\begin{figure}[htbp]
\centering
\begin{tikzpicture}[scale=1.2]
    \draw[thick] (0,0) rectangle (4,3);
    \draw[<->] (0,-0.5) -- (4,-0.5);
    \node at (2,-0.8) {\(x_1\)};
    \draw[<->] (-0.5,0) -- (-0.5,3);
    \node[rotate=90] at (-0.8,1.5) {\(x_2\)};
\end{tikzpicture}
\caption{Illustration of a rectangle with side lengths \(x_1\) and \(x_2\).}
\label{fig:rectangle}
\end{figure}

\textbf{Fixed-Area Formulation:}  
For a fixed area \(x_1x_2=10\), one can find the rectangle that minimizes the perimeter.
\begin{itemize}
    \item \emph{Substitution Method:} Express \(x_2=\frac{10}{x_1}\) and minimize
    \[
    f_2(x_1)=2x_1+\frac{20}{x_1}.
    \]
    Differentiation yields \(x_1^2=10\) or \(x_1=\sqrt{10}\) (and \(x_2=\sqrt{10}\)).
    \item \emph{Lagrange Multiplier Method:} Form the Lagrangian
    \[
    \mathcal{L}(x_1,x_2,\lambda)=2x_1+2x_2+\lambda(10-x_1x_2),
    \]
    leading similarly to \(x_1=x_2=\sqrt{10}\).
\end{itemize}

\paragraph{Lagrange Multiplier Method:}
Introduce the Lagrangian
\[
\mathcal{L}(x_1,x_2,\lambda)=2x_1 + 2x_2 + \lambda(10 - x_1x_2).
\]
According to the Fritz John necessary conditions, there exist multipliers \(\lambda_0\) (for the objective) and \(\lambda_1\) (for the constraint) (not both zero) such that
\[
\lambda_0\,\nabla (2x_1+2x_2) - \lambda_1\,\nabla(x_1x_2) = 0,\qquad\text{and}\qquad \lambda_1 (10-x_1x_2)=0.
\]
\begin{itemize} 
\item \textbf{Case 1 (Degenerate):} \(\lambda_0 =0\)  
If \(\lambda_0=0\), the stationarity condition becomes
\[
-\lambda_1\,\nabla(x_1x_2)=0 \quad\Longrightarrow\quad -\lambda_1\,(x_2, x_1)=(0,0).
\]
Assuming \(\lambda_1\neq0\), we must have \(x_1=x_2=0\), but then the complementary slackness
\[
\lambda_1\,(10-x_1x_2)=10\lambda_1=0
\]
forces \(\lambda_1=0\), a contradiction. Hence, no valid solution arises in this case.

\item \textbf{Case 2 (Regular):} \(\lambda_0>0\). Then one may normalize (e.g. set \(\lambda_0=1\)) and solve
\[
\nabla (2x_1+2x_2) - \lambda_1\,\nabla(x_1x_2) = 0.
\]
In our example, this yields
\[
2 - \lambda_1 x_2 = 0,\quad 2 - \lambda_1 x_1 = 0,
\]
implying \(x_1=x_2\) and with the constraint \(x_1^2=10\), we recover \(x_1=x_2=\sqrt{10}\) and \(\lambda_1=\frac{2}{\sqrt{10}}\).
\end{itemize}

\subsection{Weighted Sum Scalarization}
To handle both objectives simultaneously, we scalarize the problem by converting the maximization of area into minimization (by taking its negative) and then forming:
\[
\min_{x_1,x_2>0}\; F(x_1,x_2)=w_1\Bigl[-x_1x_2\Bigr]+w_2\,(2x_1+2x_2),
\]
with weights \(w_1,w_2 \ge 0\). Choosing equal weights \(w_1=w_2=0.5\), the function becomes
\[
F(x_1,x_2)=0.5\Bigl[-x_1x_2\Bigr]+0.5\Bigl[2x_1+2x_2\Bigr]= -0.5\,x_1x_2+x_1+x_2.
\]
Taking first-order partial derivatives,
\[
\frac{\partial F}{\partial x_1}=-0.5\,x_2+1,\quad \frac{\partial F}{\partial x_2}=-0.5\,x_1+1,
\]
we obtain the candidate stationary point by setting them to zero:
\[
-0.5\,x_2+1=0 \;\Rightarrow\; x_2=2,\qquad -0.5\,x_1+1=0 \;\Rightarrow\; x_1=2.
\]
Thus, the candidate is \((x_1,x_2)=(2,2)\).

To evaluate the sufficient condition for optimality (for interior points of the feasible region), we examine the Hessian of
\[
F(x_1,x_2)= -0.5\,x_1x_2+x_1+x_2.
\]
The second derivatives are:
\[
\frac{\partial^2 F}{\partial x_1^2}=0,\quad \frac{\partial^2 F}{\partial x_2^2}=0,\quad \frac{\partial^2 F}{\partial x_1\partial x_2}=\frac{\partial^2 F}{\partial x_2\partial x_1}=-0.5.
\]
Thus, the Hessian matrix is:
\[
H=\begin{pmatrix}
0 & -0.5\\[1mm]
-0.5 & 0
\end{pmatrix}.
\]
Its eigenvalues are found by solving
\[
\det(H-\lambda I)=\lambda^2-0.25=0 \quad\Rightarrow\quad \lambda_{1,2}=\pm0.5.
\]
Since \(H\) is indefinite (one positive and one negative eigenvalue), \((2,2)\) is not a local minimum—it is, in fact, a saddle point. While the KKT conditions for the multiobjective problem are met (with gradients opposing each other), the classical scalarization does not achieve an optimum at \((2,2)\).

\subsubsection{Visualization of the Weighted Scalarization Function}
The plot of the function
\[
F(x_1,x_2)= -0.5\,x_1x_2+x_1+x_2,
\]
illustrates its saddle-point behavior near \((2,2)\). See Figure \ref{fig:sad}.

\begin{figure}
\begin{center}
\begin{tikzpicture}
    \begin{axis}[
        view={60}{30},
        xlabel={$x_1$},
        ylabel={$x_2$},
        zlabel={$F(x_1,x_2)$},
        domain=0:4,
        y domain=0:4,
        samples=30,
        mesh/ordering=x varies,
        colormap/cool,
        title={Surface Plot of $F(x_1,x_2)=-0.5x_1x_2+x_1+x_2$}
    ]
    \addplot3[surf,shader=interp] { -0.5*x*y + x + y };
    \end{axis}
\end{tikzpicture}
\end{center}
\caption{\label{fig:sad} Weighted scalarization function has no minimizer for weight combination (0.5, 0.5).}
\end{figure}

\subsection{KKT and Hessian Analysis}
The candidate \((2,2)\) satisfies the KKT conditions as the gradients are correctly opposing each other. We show this by means of a graphical analysis in Fig. \ref{fig:levelsets}. Also by means of level set analysis we can confirm that the Pareto efficient solutions occur where $x_1 = x_2$ (on the diagonal), as the level curves meet there in a single point.

We can also show this analytically: the negative area gradient is given by $\nabla(-x_1 x_2) = (-x_2, -x_1)$ and the perimeter gradient is $\nabla(x_1, x_2)=(1,1)^T$.
The gradient of the area objective for $x_1 = x_2$ is $(-x_1, -x_1) = (-x_2,-x_2)$, and hence the gradients point exactly in the opposite direction, as required by the KKT conditions.  We can choose  positive multipliers $\lambda_1 = 1/x_1$ and $\lambda_2 = 1$ to show that 
$\lambda_1\nabla (x_1,x_1) + \lambda_2(x_1,x_2) = (0,0)$.

We can see that all vectors $(x,x), x\in \mathbb{R}_{>0}$ are Pareto efficient. Further analysis, we leave this as an exercise for the reader, shows that no other candidate solutions can be found for satisfying the KKT condition; thus we conclude that the set of squares are all Pareto optimum solutions and we can now compute the Pareto front, see Figure \ref{fig:paretosquare}. Unsurprisingly, the Pareto front is concave, which we will see is an indication of the weighted sum scalarization not working properly.

\begin{figure}
    \centering
    \includegraphics[width=0.75\linewidth]{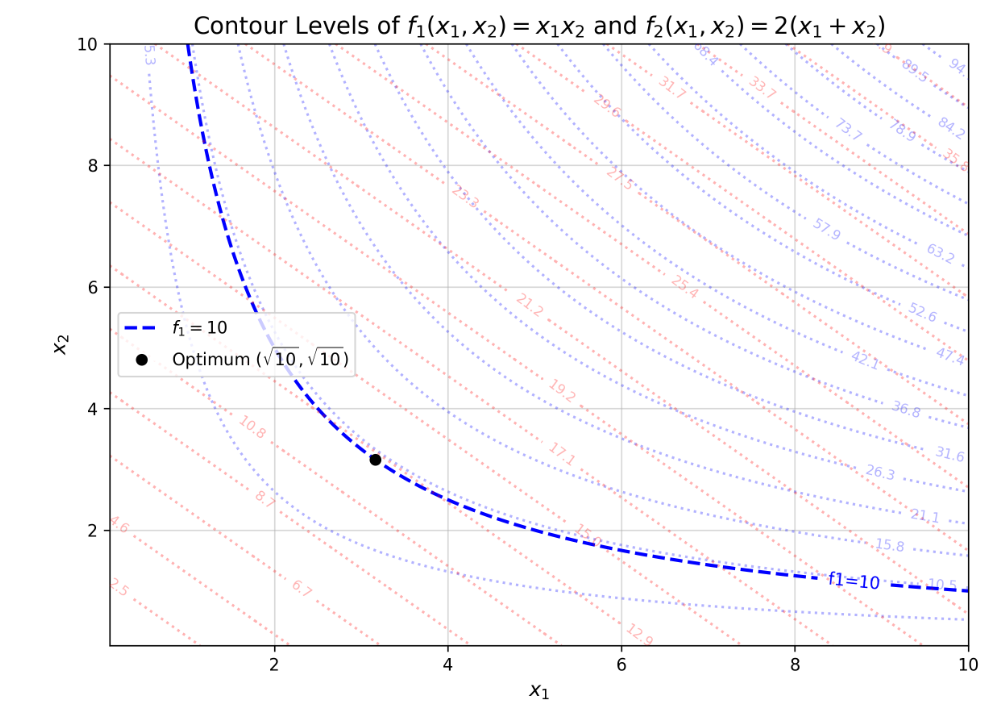}
    \caption{Plot of the level sets of the area (hyperbolic level curves) and of the circumference (linear level curves).}
\label{fig:levelsets}
\end{figure}

\begin{figure}
    \centering
    \includegraphics[width=0.75\linewidth]{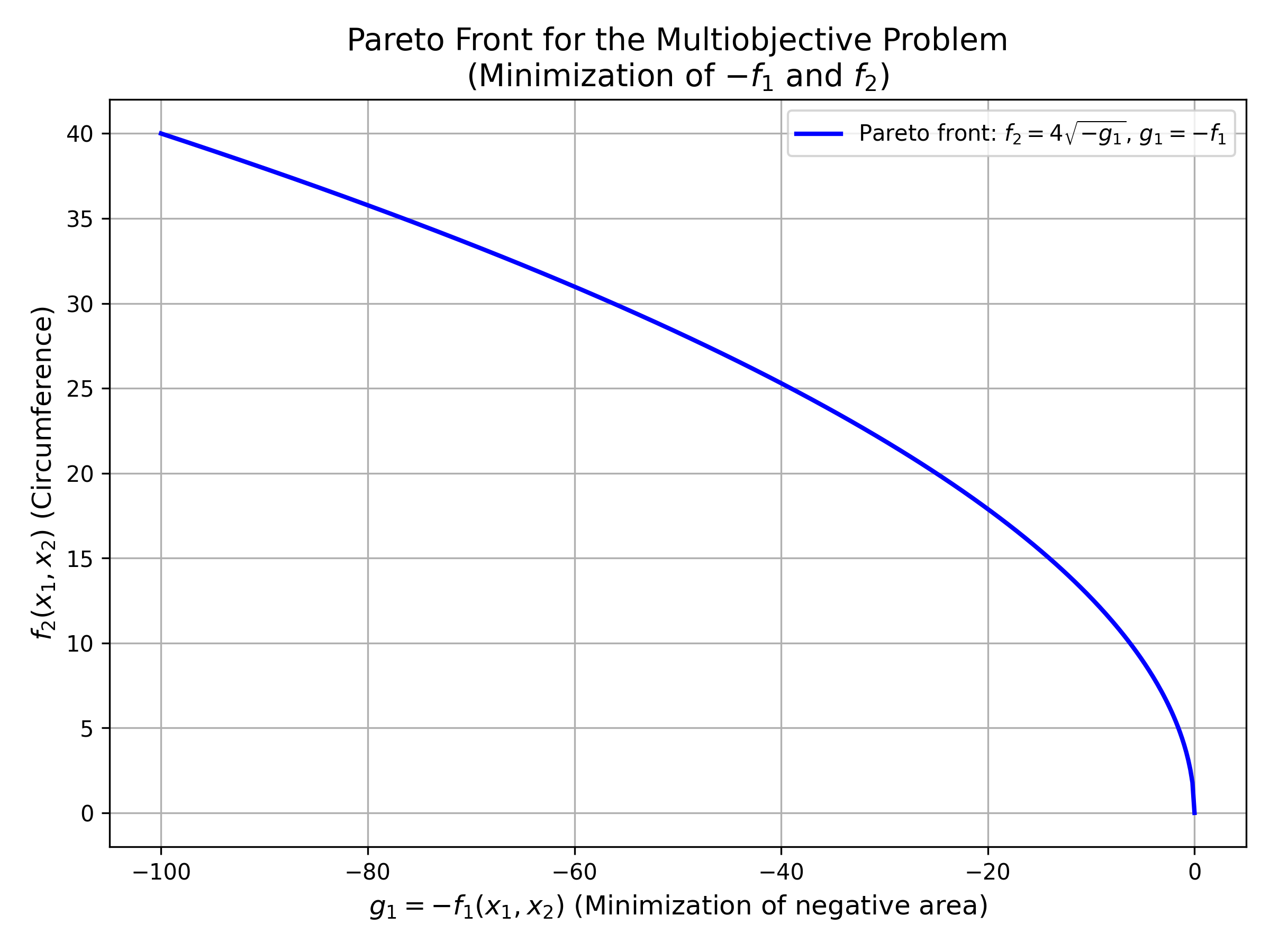}
    \caption{The Pareto front of the example optimization problem with negative area minimization and perimeter (circumference) minimization of a rectangular area with side lengths $x_1$ and $x_2$.}
    \label{fig:paretosquare}
\end{figure}

\subsection{Discussion and Practical Implications}
The analysis above leads to the following insights:
\begin{itemize}
    \item \textbf{Stationarity vs. Optimality:} Although the candidate \((2,2)\) satisfies the KKT conditions for the weighted scalarization (gradients point in opposing directions), the Hessian analysis reveals that it is a saddle point. Hence, it is not a local minimizer of the scalarized function when using linear weighting with equal weights.
    \item \textbf{Non-existence of solution for weighted sum scalarization:} In multiobjective optimization, a point may be Pareto optimal even if it does not correspond to a strict local minimum in the scalarized problem. In this example, which has a practical background, we observe a concave Pareto front. If we use the weighted sum method, there exists no optimizer, since the only solution that satisfies the necessary condition for local optimality is a saddle point and thus no local optimizer.
\end{itemize}

\section*{Exercises}
\begin{enumerate}

    \exercise{Unconstrained Optimization}{%
    }%
    {ex:unconstrained-opt}
    Let us consider the optimization problem of finding the height $h$ and the radius $r$ that minimizes the surface area of a cylindric tin containing a given volume $V(h,r) = v_0 = 330$. We can formulate this problem by $V(h,r) = \pi r^2 h$ and $S(h,r)=2 \pi r^2 + 2 \pi r h$, with $r,h \in \mathbb{R}_+$. To make this a unconstrained problem we could substitute $h = {v_0 / r^2}$ and we get the objective function $f(r) = \pi r^2 + 2 \pi v_0 /r \rightarrow \min$. Determine the optimum of this function by determining the point where the gradient (first derivative) is zero and the Hessian matrix (second derivative) is positive definite. 
    \exercise{Gradient-based unconstrained optimization.}{
    %
        Determine the optimum of the function $\frac{1}{4}\left((x_1)^2 + (x_2)^2\right) + \frac{3}{4} \left((x_1-1)^2 + (x_2-1)^2\right)$. First establish an expression of the gradient and of the Hessian matrix and second determine the minimizer analytically. 
    }%
    {ex:unconstrained-multivariate}
    \exercise{Lagrange Multiplier Rule}{%
        Use the Lagrange multiplier method to find the closest point on Earth to a satellite. Assume Earth is a sphere centered at the origin with radius \( r = 6000 \), and the satellite's coordinates are \( (x_s, y_s, z_s) \). The point \( (x,y,z) \) on Earth satisfies the constraint:

\[
x^2 + y^2 + z^2 = r^2
\]

Minimize the Euclidean distance:

\[
d(x,y,z) = \sqrt{(x - x_s)^2 + (y - y_s)^2 + (z - z_s)^2}
\]

subject to the constraint. Since squaring preserves order, solve the equivalent problem:

\[
(x - x_s)^2 + (y - y_s)^2 + (z - z_s)^2 \to \min.
\]

Apply the Lagrange multiplier theorem to derive the solution. Assume \( (x_s, y_s, z_s) \geq (0,0,0) \).

    }%
    {ex:lagrange-multiplier}
    
    \exercise{KKT Conditions for Inequality Constraints}{%
    }%
    {ex:kkt-inequality}
       Consider the search space $\mathcal{S} = [0,2] \times [0,3]$ and the objectives
        \[
        f(x_1, x_2) = - \frac{1}{2} x_2 - \frac{1}{2} x_1 \rightarrow \min.
        \]
        The following constraints must be satisfied:
        \[
        g_1(x_1,x_2) = -2 + \frac{2}{3} x_1 + x_2 \leq 0, \quad g_2(x_1,x_2) = -4 + 2 x_1 + x_2 \leq 0.
        \]
        and $x_1 \geq 0, x_2 \geq 1-x_1$.
        To solve this problem, mark the constrained region graphically (as in Figure \ref{fig:bilinprog}). Now formulate the KKT conditions for the points $(1,1)$ and for $(1.5,1.0)$. Check if the conditions of the KKT theorem are satisfied by solving the equations for the Lagrange multipliers and investigating whether or not the differentiability and constraint qualifications are met.
        
    \exercise{KKT Conditions for Multiobjective Optimization}{%
        Assume the unconstrained multiobjective optimization problem:
        $f_1(x_1,x_2)= x_1^2 + x_2^2 \rightarrow \min$, $f_1(x_1,x_2)= (x_1-1)^2 + (x_2-1)^2 \rightarrow \min$. 
        Show that the points on the line segment $t * (1,1)^T, t\in [0,1]$ are efficient by using the KKT conditions for unconstrained optimization.
    }%
    {ex:kkt-multiobj}

\end{enumerate}

\chapter{Scalarization Methods}
A straightforward idea to recast a multiobjective problem as a
single objective problem is to sum up the objectives in an weighted
sum and then to maximize/minimize the weighted sum of objectives.
More general is the approach to aggregate the objectives to a single
objective by a so-called utility function, which does not have to be
a linear sum but usually meets certain monotonicity criteria.
Techniques that sum up multiple objectives into a single one by means  
of an aggregate function are termed {\em scalarization techniques}.
A couple of questions arise when applying such techniques:
\begin{itemize}
\item Does the global optimization of the aggregate function always (or in certain cases) result in an
efficient point?
\item Can all solutions on the Pareto front be obtained by varying
the (weight) parameters of the aggregate function?
\item Given that the optimization of the aggregate leads to an efficient point,
how does the choice of the weights control the position of the
obtained solution on the Pareto front?
\end{itemize}

Section \ref{sec:lin} starts with linear aggregation (weighted sum)
and answers the aforementioned questions for it. The insights we
gain from the linear case prepare us for the generalization to
nonlinear aggregation in Section \ref{sec:nonlin}. The expression or
modeling of preferences by means of aggregate functions is a
broad field of study called Multi-attribute utility theory (MAUT).
An overview and examples are given in Section \ref{sec:maut}. A
common approach to solve multicriteria optimization problems is the
distance to a reference point method. Here the decision pointer
defines a desired 'utopia' point and minimizes the distance to it.
In Section \ref{sec:drp} we will discuss this method as a special
case of a scalarization technique.

\section{Linear Aggregation}
\label{sec:lin}
Linear weighting is a straightforward way to summarize objectives.
Formally, the problem:

\begin{equation}
\label{eq:probgen} f_1(x) \rightarrow \min, \dots, f_m(x)
\rightarrow \min
\end{equation}
is replaced by:
\begin{equation}
\label{eq:linprob} \sum_{i=1}^m w_i f_i(x) \rightarrow \min, w_1,
\dots, w_m >0
\end{equation}
A first question that may arise is whether the solution of problem
\ref{eq:linprob} is an efficient solution of problem
\ref{eq:probgen}. This is indeed the case as points that are
non-dominated w.r.t. problem \ref{eq:probgen} are also non-dominated
w.r.t. problem \ref{eq:linprob}, which follows from:
\begin{equation}
\forall {\bf y}^{(1)}, {\bf y}^{(2)} \in \mathbb{R}^m:  {\bf
y}^{(1)} \prec {\bf y}^{(2)} \Rightarrow \sum_{i=1}^{m} y_i^{(1)} <
\sum_{i=1}^m y_i^{(2)}
\end{equation}
Another question that arises is whether we can find all points on
the Pareto front using linear aggregation and varying the weights or
not. The following theorem provides the answer. To state the 
theorem, we need the following definition:

\begin{definition} Proper efficiency \cite{Ehr05}\\
Given a Pareto optimization problem (Eq. \ref{eq:probgen}), then a
solution $x$ is called efficient in the Geoffrion sense or {\em
properly efficient}, iff (a) it is efficient, and (b) there exists a
number $M
> 0$ such that $\forall i = 1, \dots,m$ and $\forall x \in
\mathcal{X}$ satisfying $f_i(x) < f_i(x^*)$, there exists an index
$j$ such that $f_j(x^*) < f_j(x)$ and:
$$
\frac{ f_i(x^*) - f_i(x)} { f_j(x) - f_j(x^*)} \leq M
$$
The image of a properly efficient point we will term properly
non-dominated. The set of all proper efficient points is termed {\em
proper efficient set}, and its image {\em proper Pareto front}.
\end{definition}
Note that in the bi-criterion case, the efficient points which are
Pareto optimal in the Geoffrion sense are those points on the
Pareto-front, where the slope of the Pareto front ($f_2$ expressed
as a function of $f_1$) is finite and nonzero (see Fig.
\ref{fig:proper}). The parameter $M$ is interpreted as trade-off.
The proper Pareto optimal points can thus be viewed as points with a
bounded trade-off.

\begin{theorem} Weighted sum scalarization\\
Let us assume a Pareto optimization problem (Eq. \ref{eq:probgen})
with a Pareto front that is cone convex w.r.t. positive orthant
($\mathbb{R}^m_\geq$). Then for each properly efficient point $x \in
\mathcal{X}$ there exist weights $w_1 >0$, $\dots$, $w_m > 0$ such
that $x$ is one of the solutions of $\sum_{i=1}^m f_i(x) \rightarrow
\min$.
\end{theorem}

 \begin{figure}[t]
\begin{center}
 \includegraphics[width=7cm]{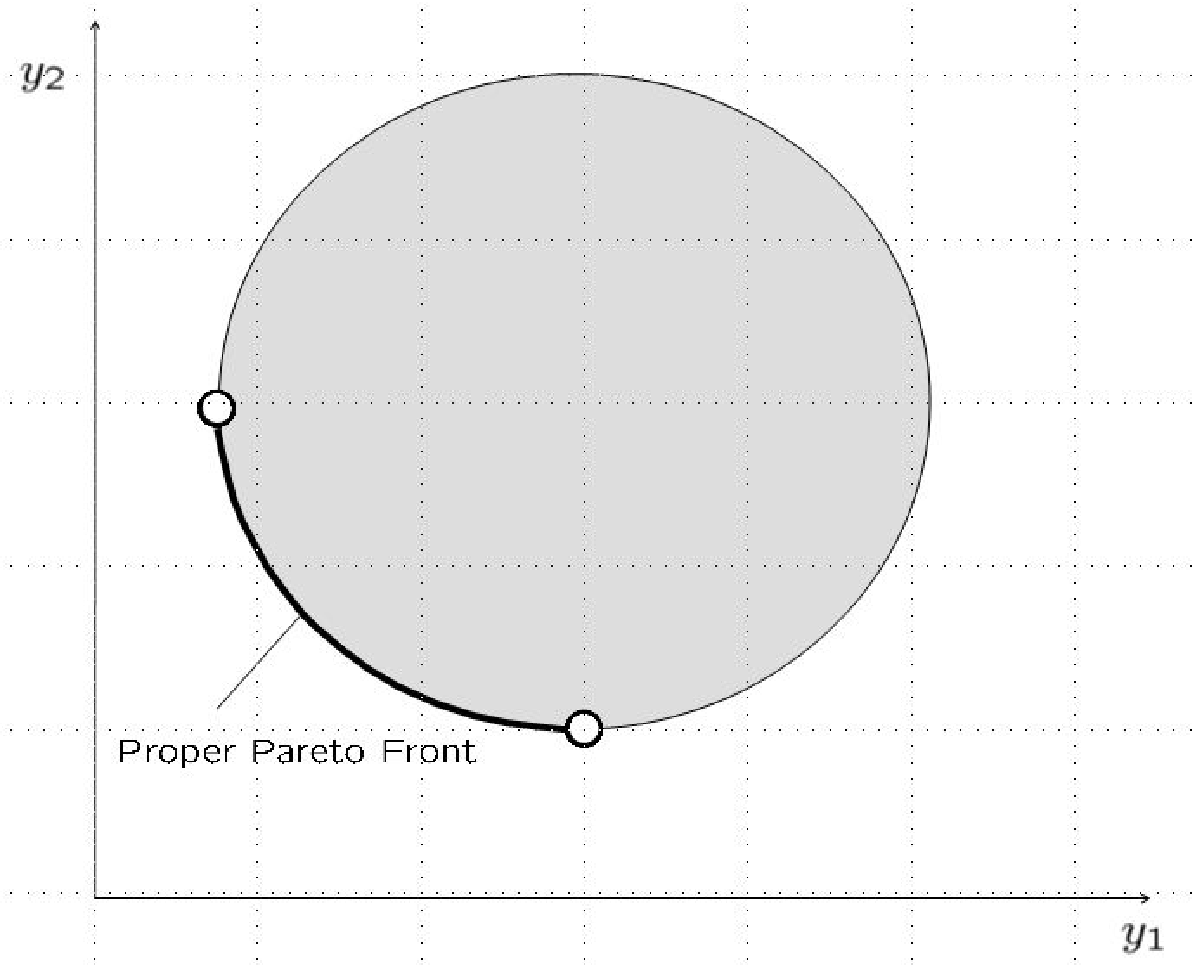}
\end{center}
\caption{\label{fig:proper} The proper Pareto front for a bicriteria
problem, for which in addition to many proper Pareto optimal
solutions there exist also two non-proper Pareto optimal
solutions.}
\end{figure}
\begin{figure}[t]
\includegraphics[width=7cm]{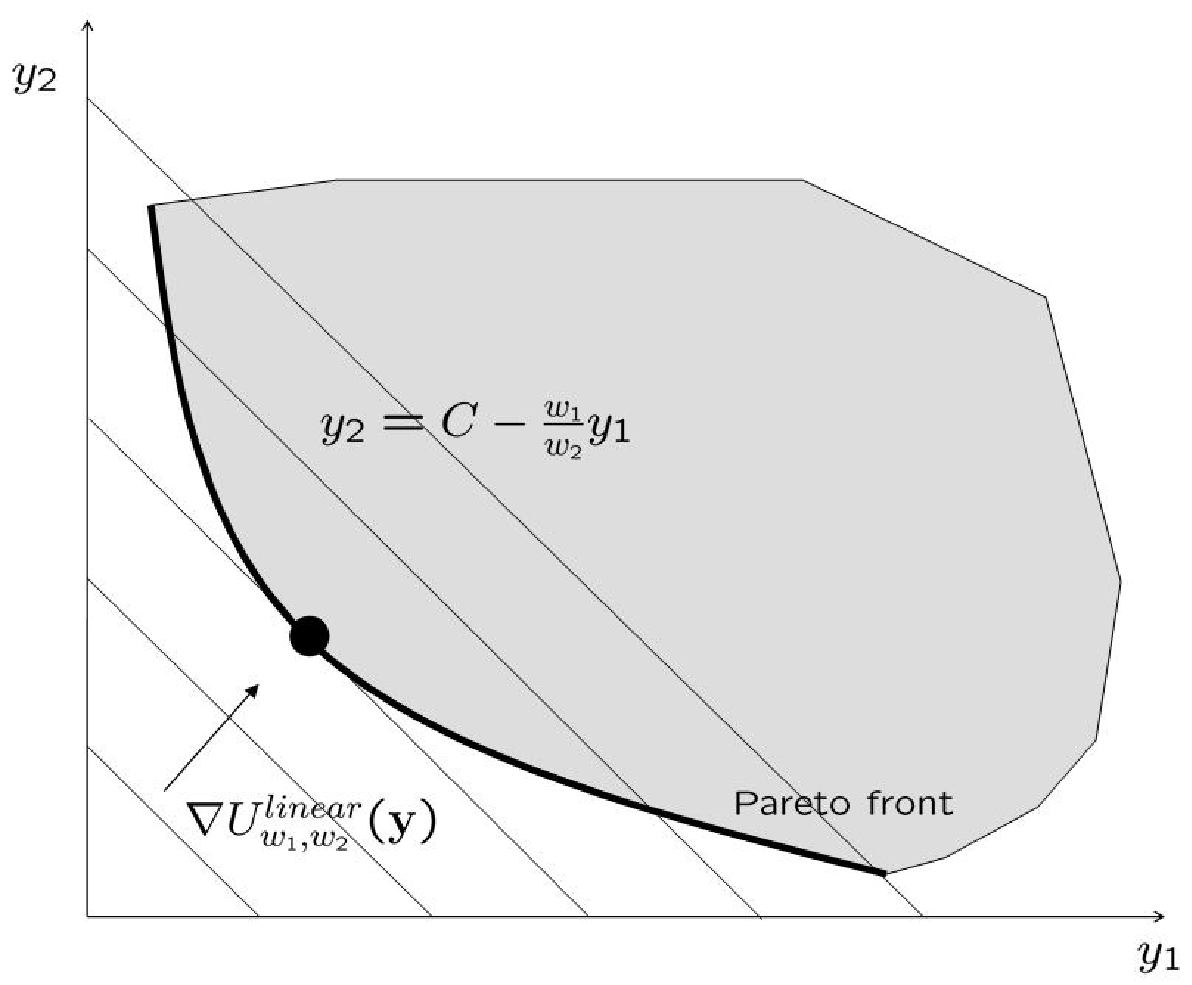}
\includegraphics[width=7cm]{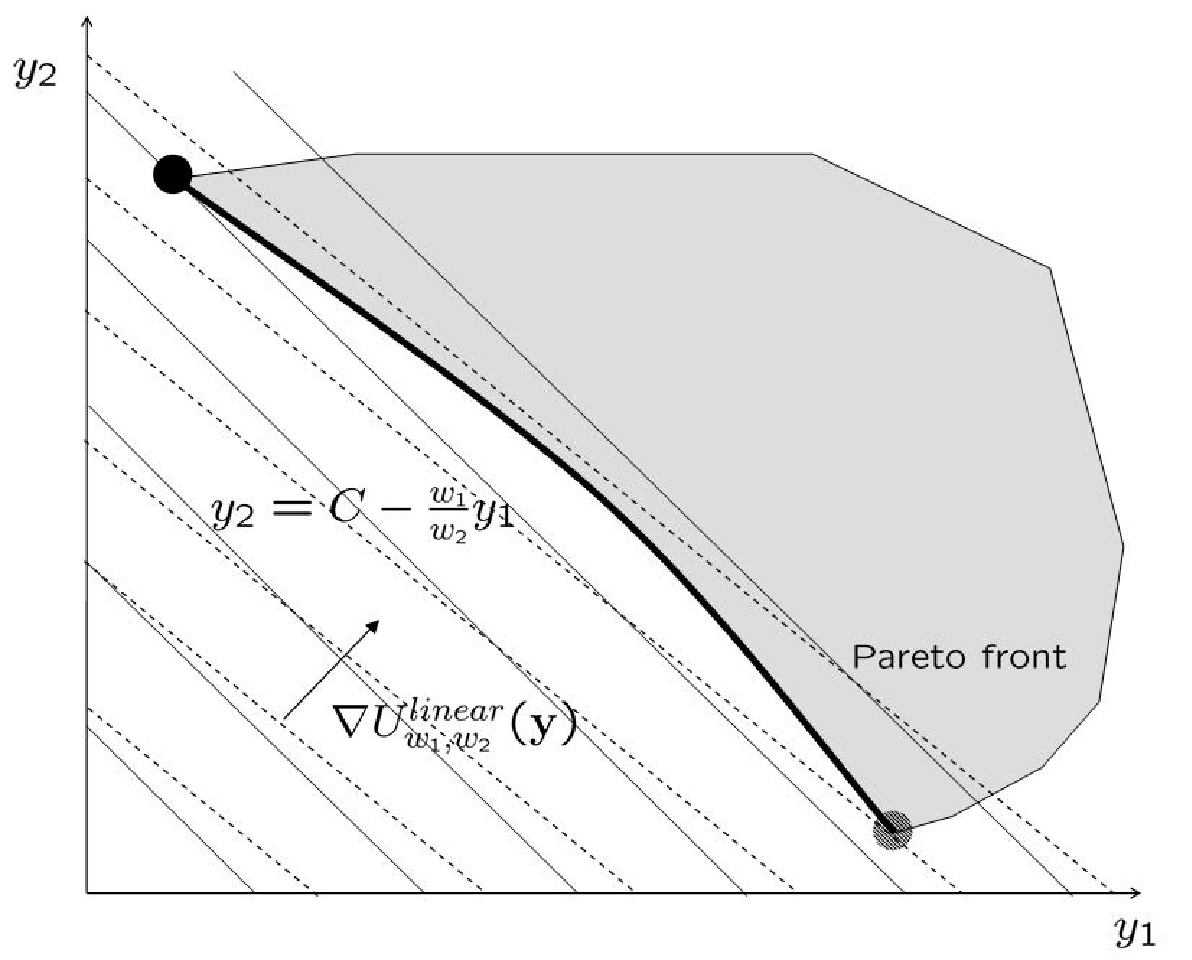}
\caption{\label{fig:concon}The concave (left) and convex Pareto
front (right).}
\end{figure}
In case of problems with a non-convex pareto front it is not always
possible to find weights for a given proper efficient point $x$ such
that $x$ is one of the solutions of $\sum_{i=1}^m f_i(x) \rightarrow
\min$. A counterexample is given in the following example:

\begin{example}
In Fig. \ref{fig:concon} the Pareto fronts of two different
bi-criterion problems are shown. The figure on the right-hand side
shows a Pareto front which is cone convex with respect to the
positive orthant. Here the tangential points of the level curves of
$w_1 y_1 + w_2 y_2$ are the solutions obtained with linear
aggregation. Obviously, by changing the slope of the level curves by
varying one (or both) of the weights,  all points on the Pareto
front can be obtained (and no other). On the other hand, for the
concave Pareto front shown on the right-hand side only the extreme
solutions at the boundary can be obtained.
\end{example}

As the example shows linear aggregation has a tendency to obtain
extreme solutions on the Pareto front, and its use is thus
problematic in cases where no a-priori knowledge of the shape of the
Pareto front is given. However, there exist aggregation functions
which have less tendency to obtain extreme solutions or even allow
to obtain all Pareto optimal solutions. They will be discussed in
the next section.
\section{Nonlinear Aggregation}
\label{sec:nonlin}

Instead of linear aggregation we can use nonlinear aggregation
approaches, e.g. compute a product of the objective function value.
The theory of utility functions can be viewed as a modeling approach
for (non)linear aggregation functions.

A utility function assigns to each combination of values that may
occur in the objective space a scalar value - the so-called utility.
This value is to be maximized. Note that the linear aggregation was
to be minimized. Level curves of the utility function are
interpreted as indifference curves (see Fig. \ref{fig:utility}).

In order to discuss a scalarization method it may be interesting to
analyze where on the Pareto front the Pareto optimal solution that
is found by maximizing the utility function is located. Similar to
the linear weighting function discussed earlier, this is the point
where the level curves of the utility (looked upon in descending
order) first intersect with the Pareto front (see Fig.
\ref{fig:utilitymeetspareto}).

\begin{figure}
\begin{center}
\includegraphics[width=10cm]{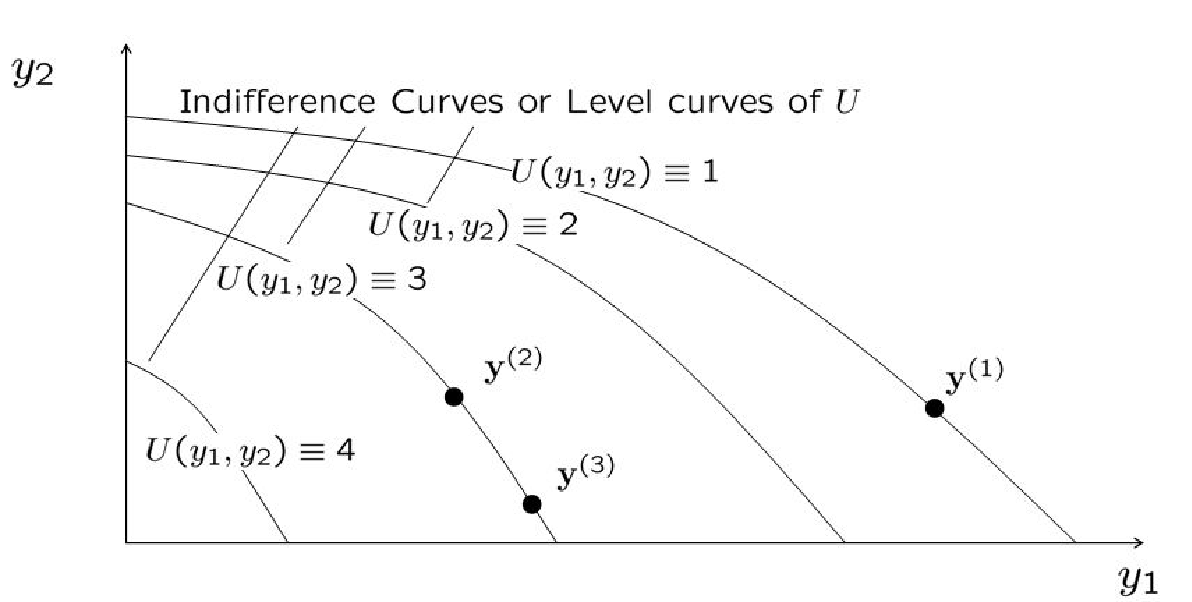}
\end{center}
\caption{\label{fig:utility} Utility function for a bi-criterion
problem. If the decision-maker has modeled this utility function in
a proper way, he/she will be indifferent whether to choose   ${\bf
y}^{(2)}$ and ${\bf y}^{(3)}$, but prefer  ${\bf y}^{(3)}$ and ${\bf
y}^{(2)}$ to ${\bf y}^{(1)}$.}
\begin{center}
\includegraphics[width=10cm]{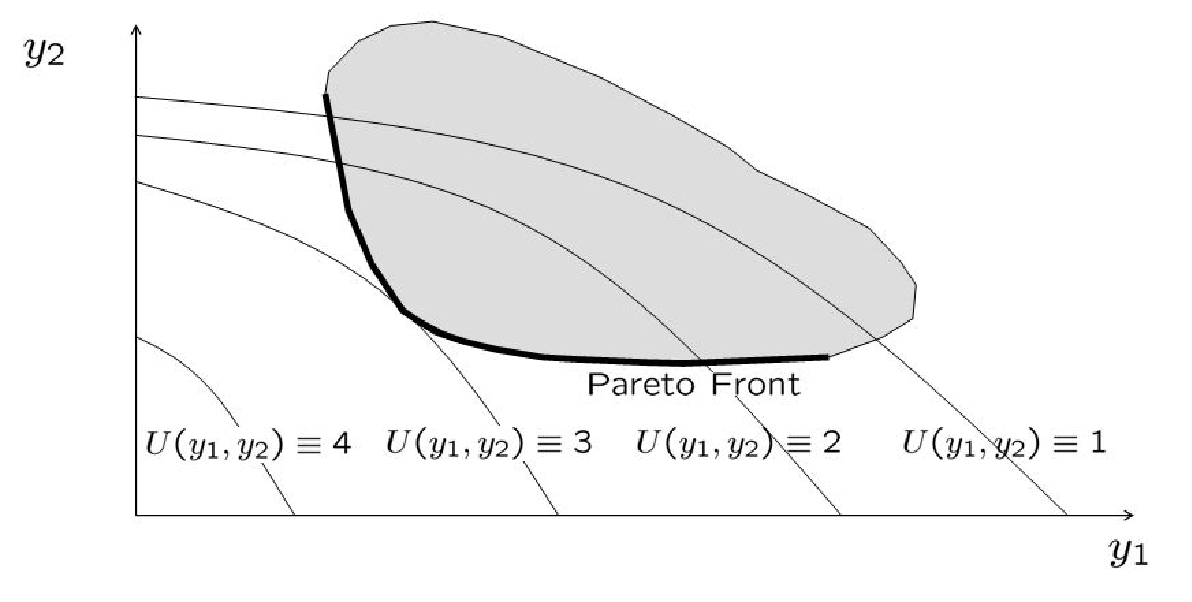}
\end{center}
\caption{\label{fig:utilitymeetspareto} The tangential point of the
Pareto front with the indifference curves of the utility function
$U$ here determines where the solution of the maximization of the
utility function lies on the Pareto front.}
\end{figure}

\section{Multi-Attribute Utility Theory}
\label{sec:maut}

Next, we will discuss a concrete example for the design of a utility
function. This example will illustrate many aspects of how to
construct utility functions in a practically useful, consistent, and
user-friendly way.
\begin{example}
Consider you want to buy a car. Then you may focus on three
objectives: speed, price, fuel-consumption. These three criteria can
be weighted. It is often not wise to measure the contribution of an
objective function to the overall utility in a linear way. A elegant
way to model it is by specifying a function that measures the degree
of satisfaction. For each possible value of the objective function
we specify the degree of satisfaction of this solution on a scale
from $0$ to $10$ by means of a so-called {\em value function}. In
case of speed, we may demand that a car is faster than $80
\mbox{m/mph}$ but beyond a speed of, say, 180 km/h the increase of
our satisfaction with the car is marginal, as we will not have many
occasions where driving at this speed gives us advantages. It can
also be the case, that the objective is to be minimized. As an
example, we consider the price of the car. The budget that we are
allowed to spend marks an upper bound for the point where the value
function obtains a value of zero. Typically, our satisfaction will
grow if the price is decreased until a critical point, where we may
no longer trust that the solution is sold for a fair price and we
may get suspicious of the offer.

The art of the game is then to sum up these objectives to a single
utility function. 

\subsection{Desirability Functions}

Desirability functions, introduced by Harrington \cite{Har65} in the context of quality assurance in industry, provide a systematic approach to multiobjective decision-making. These functions transform objective values in a way that accounts for diminishing marginal utility—once a certain level of performance in one objective is reached (e.g., a car is sufficiently fast), further improvements contribute little to overall desirability, making it more beneficial to focus on other objectives. Conversely, there exist performance thresholds below which a solution becomes entirely unacceptable, regardless of its performance in other objectives.

A modern adaptation of this approach is as follows Given value functions
$v_i: \mathbb{R} \rightarrow [0,10], i = 1, \dots, m$ mapping
objective function values to degree of satisfaction values, and
their weights $w_i, i=1, \dots, m$, we can construct the following
optimization problem with constraints:
\begin{figure}
\includegraphics[width=14cm]{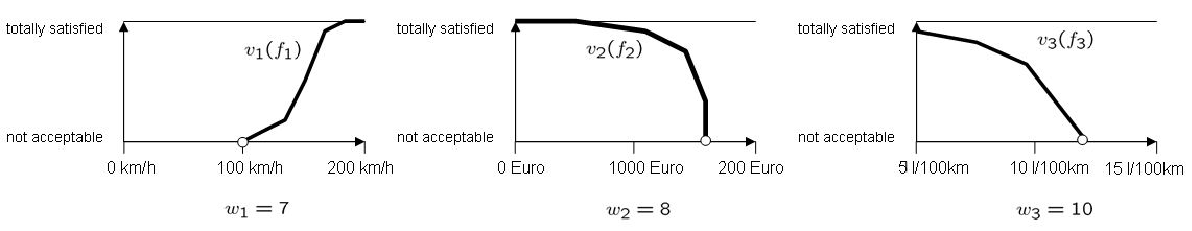}
\caption{\label{fig:maut} The components (value functions) of a
multiattribute utility function.}
\end{figure}
\begin{eqnarray}
U({\bf f}(x)) = \alpha \underbrace{\frac{1}{m} \sum_{i=1}^m w_i v_i(f_i(x))}_{\mbox{common interest }} + \beta \underbrace{\min_{i \in \{1, \dots, m\}} w_i v_i(f_i(x))}_{\mbox{minority interest }}  ,\\
( \mbox{here: } m=3)\\
 s.\,t.\  v_i(f_i(x)) > 0, i = 1, \dots, m
\end{eqnarray}
Here, we have one term that looks for the 'common interest'. This
term can be comparably high if some of the value functions have a
very high value and others a very small value. In order to enforce a
more balanced solutions w.r.t. the different value functions, we can
also consider to focus on the value function which is least
satisfied. In order to discard values from the search space,
solution candidates with a value function of zero are considered as
infeasible by introducing strict inequality constraints.
\end{example}

A very similar approach is the use of desirability indices. They have
been first proposed by Harrington \cite{Har65}
for applications in industrial quality management. Another well known
reference for this approach is \cite{DS80}.

We first give a rough sketch of the method, and then discuss its formal details.

As in the previously described approach, we map the values of the objective function
to satisfaction levels, ranging from not acceptable (0) to totally satisfied (1). The values in between 0 and
one indicate the gray areas. Piecewise defined exponential functions  are used to
describe the mappings. They can be specified by means of three parameters.
The mapped objective function values are now called desirability indices.
Harrington proprosed to aggregate these desirability indices by a product expression,
the minimization of which leads to the solution of the multiobjective problem.

The functions used for the mapping of objective function values to desirability
indices are categorized into one-sided and two-sided functions. Both have a parameter
$y^{min}_i$ (lower specification limit), $y^{max}_i$ (upper specification limit), $l_i, r_i$ (shape parameters), and $t_i$ (symmetry center).
The one-sided functions read:
\begin{equation}
D_i = \left\{
\begin{array}{ll}
0, & \  y_i < y^{min}_i\\
{\left(\frac{y_i - y^{min}_i}{t_i - y^{min}_i}\right)}^{l_i},& \ y^{min}_i < y_i < t_i\\
1, & y_i \geq t_i
\end{array}
\right.
\end{equation}
and the two-sided functions read:
\begin{equation}
D_i = \left\{
\begin{array}{ll}
0, & \  y_i < y^{min}_i\\
{\left(\frac{y_i - y^{min}_i}{t_i - y^{min}_i}\right)}^{l_i},& \ y^{min}_i \leq y_i \leq t_i\\
{\left(\frac{y_i - y^{max}_i}{t_i - y^{max}_i}\right)}^{r_i},& \  t_i < y_i \leq y^{max}_i  \\
0, & y_i > y^{max}
\end{array}
\right.
\end{equation}
The two plots in Fig. \ref{fig:dindex} visualize one-sided (l) and two-sided (r) desirability indexes.
\begin{figure}
\begin{center}
\includegraphics[width=0.8\textwidth]{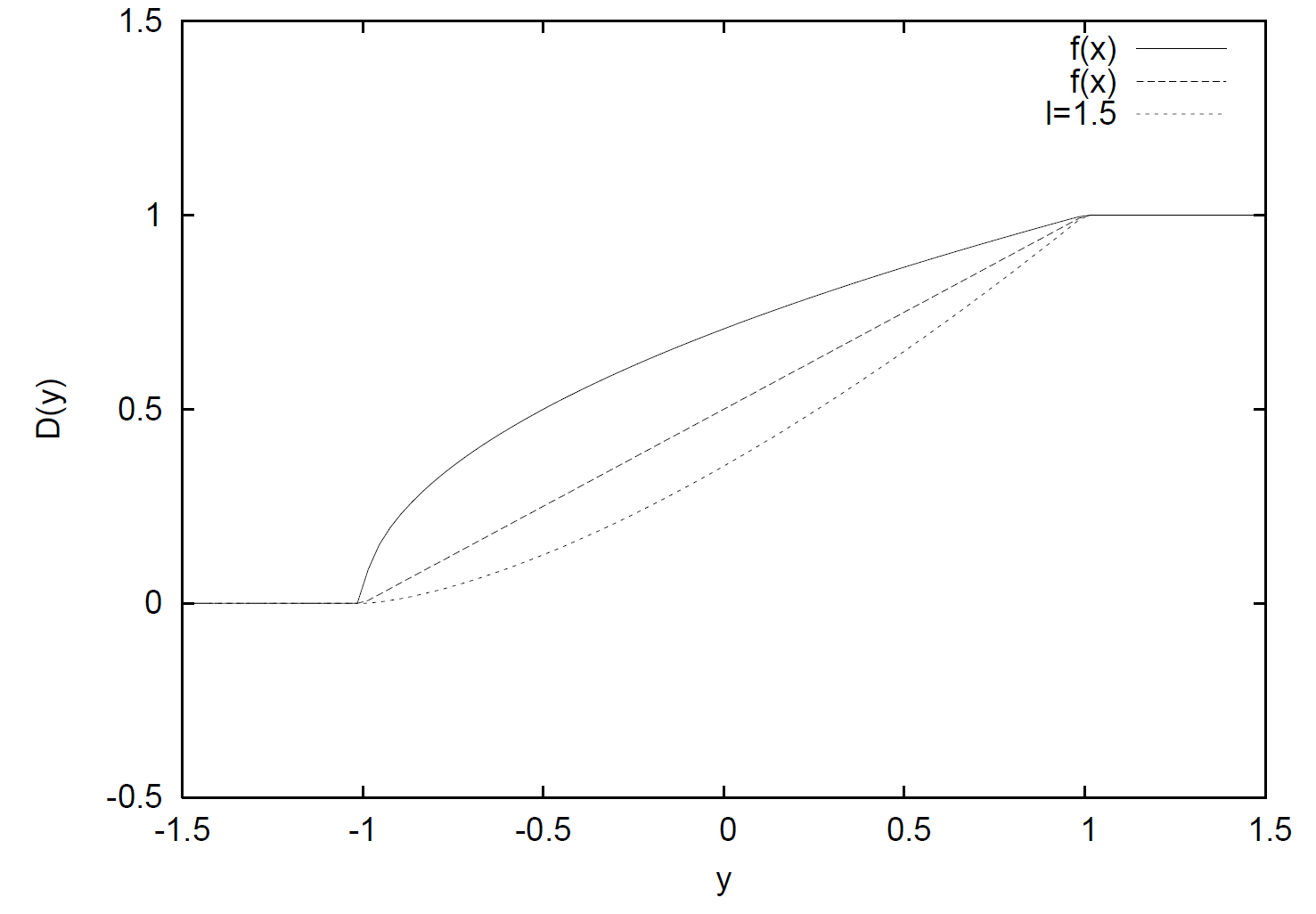} \\
\includegraphics[width=0.8\textwidth]{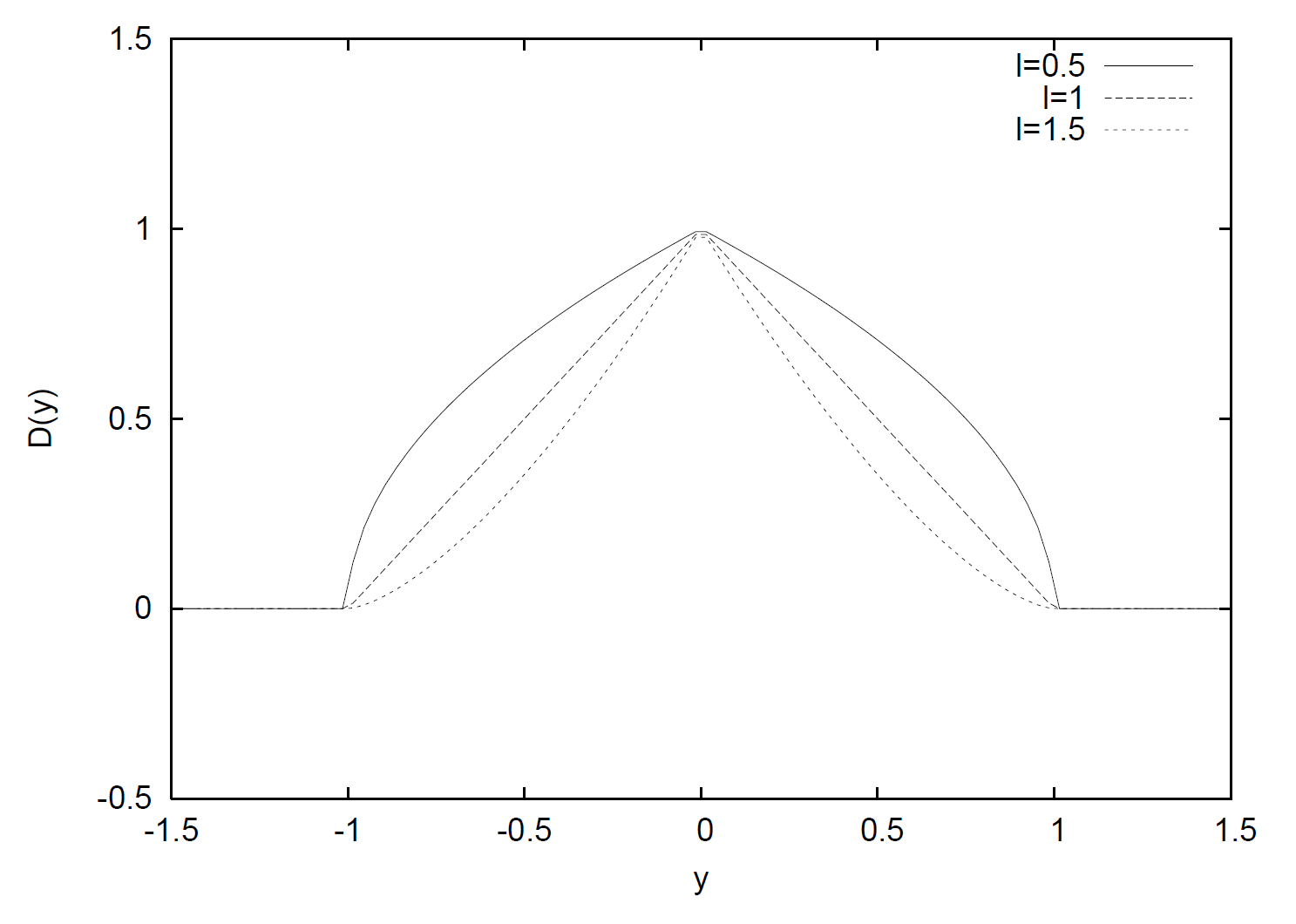}
\end{center}
\caption{\label{fig:dindex} In the top figure we see and examples for one-sided desirability function with
parameters $y^{min} = -1, y^{max} = 1, l \in \{0.5, 1, 1.5\}$.  The bottom figure displays a plot
of two sided desirability functions of the Derringer-Suich type for parameters $y^{min}= -1. y^{max}=1,$  $l \in \{0.5, 1.0, 1.5\}$, and $r$ being set to the same value than $l$.}

\end{figure}

The aggregation of the desirability indices is done by means of a product formula, that is to be maximized:
\begin{equation}
D = (\prod_{i=1}^{k} D_i(y_i))^{\frac{1}{k}}.
\end{equation}

\newpage
In literature, many approaches for constructing non-linear utility
functions are discussed.

The Cobb-Douglas utility function is widely used in economics. Let
$f_i, i = 1, \dots, m$ denote non-negative objective functions, then
the Cobb-Douglas utility function reads:
\begin{equation}
U(x) = \prod_{i=1}^m f_i(x)^{\alpha_i}
\end{equation}
It is important to note that for the Cobb-Douglas utility function
the objective function values are to be minimized, while the utility
is to be maximized. Indeed, the objective function values, the
values $\alpha_i$, and the utility have usually an economic
interpretation, such as the amount of goods: $f_i$, the utility of a
combination of goods: $U$, and the elasticities of demand:
$\alpha_i$. A useful observation is that taking the logarithm of
this function transforms it into a linear expression:
\begin{equation}
\log U(x) = \sum_{i=1}^m \alpha_i \log f_i(x)
\end{equation}
The linearity can often be exploited to solve problems related to
this utility function analytically.

Another approach to constructing utility functions in product-form is the Keeney-Raiffa utility function framework \cite{KR93}.  
Let \( f_i \) represent non-negative objective functions. The utility function is defined as:
\begin{equation}
U(x) = K \prod_{i=1}^m (w_i u_i (f_i(x)) + 1),
\end{equation}

where \( w_i \) are weight coefficients assigned to the objective functions, taking values between 0 and 1, and \( K \) is a positive scaling constant. The functions \( u_i \) are strictly increasing for positive input values, ensuring that higher values of \( f_i(x) \) correspond to greater utility.
A general remark on how to construct utility functions is, that the
optimization of these functions should lead to Pareto optimal
solutions. This can be verified by checking the monotonicity
condition for a given utility function $U$:
\begin{equation}
\forall x, x' \in \mathcal{X}: x \prec x' \Rightarrow U(x) > U(x')
\end{equation}
This condition can be easily verified for the two given utility
functions.

\section{Distance to a Reference Point Methods}
\label{sec:drp}
A special class of utility functions is the distance to the
reference point (DRP) method. Here the user specifies an ideal
solution (or: utopia point) in the objective space. Then the goal is
to get as close as possible to this ideal solution. The distance to
the ideal solution can be measured by some distance function, for
example a weighted Minkowski distance with parameter $\gamma$. This
is defined as:

\begin{equation}
\label{eq:minko} d({\bf y}, {\bf y}') =  [{\sum_{i=1}^m w_i |y_i -
y_i'|^\gamma}]^{\frac{1}{\gamma}}, \gamma \geq 1, w_1 > 0, \dots,
w_m
> 0
\end{equation}
Here, $w_i$ are positive weights that can be used to normalize the
objective function values. In order to analyze which solution is
found by means of a DRP method we can interpret the distance to the
reference point as an utility function (with the utility value to be
minimized). The indifference curves in case of $\gamma = 2$ are
spheres (or ellipsoids) around the utopia point. For $\gamma
> 2$ one obtains different super-ellipsoids as indifference curves.
Here, a super-ellipsoid around the utopia point ${\bf f}^*$ of
radius $r \geq 0$ is defined as a set:
\begin{equation}
S(r) = \{ \mathbf{y} \in \mathbb{R}^m | d({\bf y}, {\bf f}^*) = r \}
\end{equation}
with $d: \mathbb{R}^m \times \mathbb{R}^m \rightarrow
\mathbb{R}^+_0$ being a weighted distance function as defined in Eq.
\ref{eq:minko}.
\begin{example}
In Figure \ref{fig:ell} for two examples of a DRP method it is
discussed how the location of the optimum is obtained
geometrically, given the image set $\mathbf{f}(\mathcal{X})$. We
look for the super-ellipsoid with the smallest radius that still
touches the image set. If two objective functions are considered and
weighted Euclidean distance is used, i.e. $\gamma = 2$, then the
super-ellipsoids are regular ellipses (Fig. \ref{fig:ell}). If
instead a Manhattan distance ($\gamma=1$) is used with equal
weights, then we obtain diamond-shaped super-ellipsoids (Fig.
\ref{fig:ell}).
\end{example}

Not always an efficient point is found when using the DRP method.
However, in many practical cases the following sufficient condition
can be used in order to make sure that the DRP method yields an
efficient point. This condition is summarized in the following
lemma:

\begin{lemma} Let $\mathbf{f}^* \in \mathbb{R}^m$ denote an utopia
point, then
\begin{equation}
\mathbf{x}^* = \arg \min_{\mathbf{x} \in \mathcal{X}} d({\bf
f}(\mathbf{x}), {\mathbf{f}^*})\end{equation} is an efficient point,
if  for all ${\bf y} \in \mathbf{f}(\mathcal{X})$ it holds that
$\mathbf{f}^* \preceq {\bf y}$.
\end{lemma}

\begin{figure}
\begin{center}
\includegraphics[width=6cm]{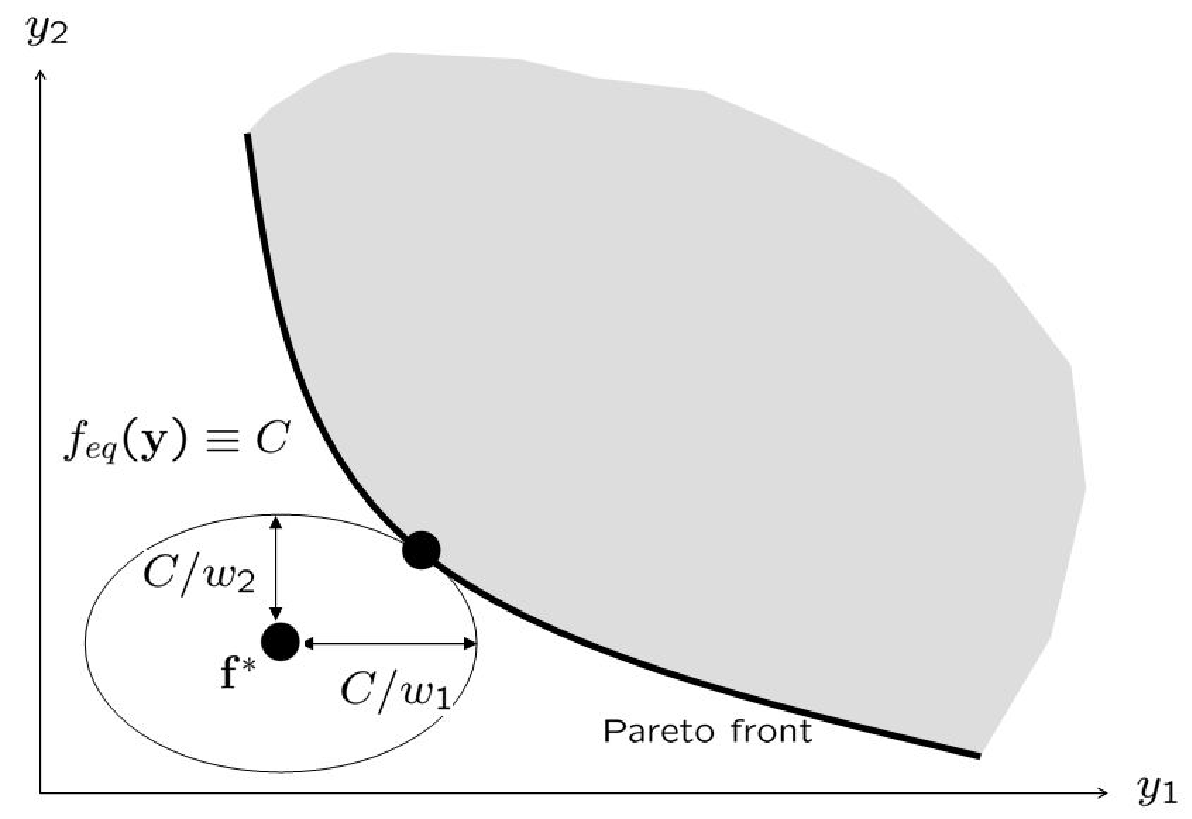} \includegraphics[width=6cm]{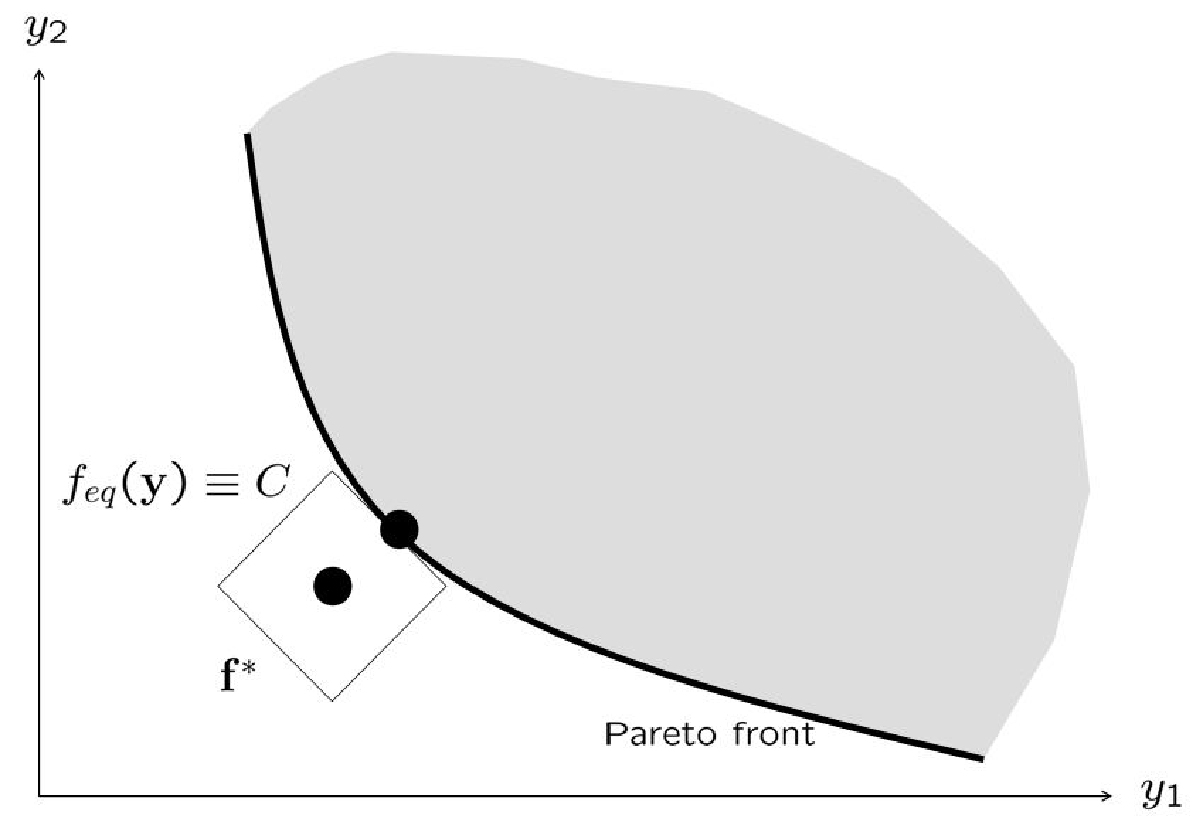}
\end{center}
\caption{\label{fig:ell}Optimal points obtained for two distance to
DRP methods, using the weighted Euclidean distance (left) and the
manhattan distance (right).}
\end{figure}
Often the utopia point is chosen to be zero (for example when the
objective functions are strictly positive). Note that it is neither
sufficient nor necessary that $\mathbf{f}^*$ is non-dominated by
$\mathbf{f}(\mathcal{X})$. The counterexamples given in Fig.
\ref{fig:sufnec} confirm this.

Another question that may arise, using the distance to a reference point method
is whether it is possible to find all points on the Pareto front, by changing
the weighting parameters $w_i$ of the metric. Even in the case that the utopia
points dominate all solutions we cannot obtain all points on the Pareto
front by minimizing the distance to the reference in case of $\gamma < \infty$.
Concave parts of the Pareto front may be overlooked, because we encounter the
problems that we discussed earlier in case of linear weighting.

\begin{figure}
\begin{center}
\includegraphics[width=6cm]{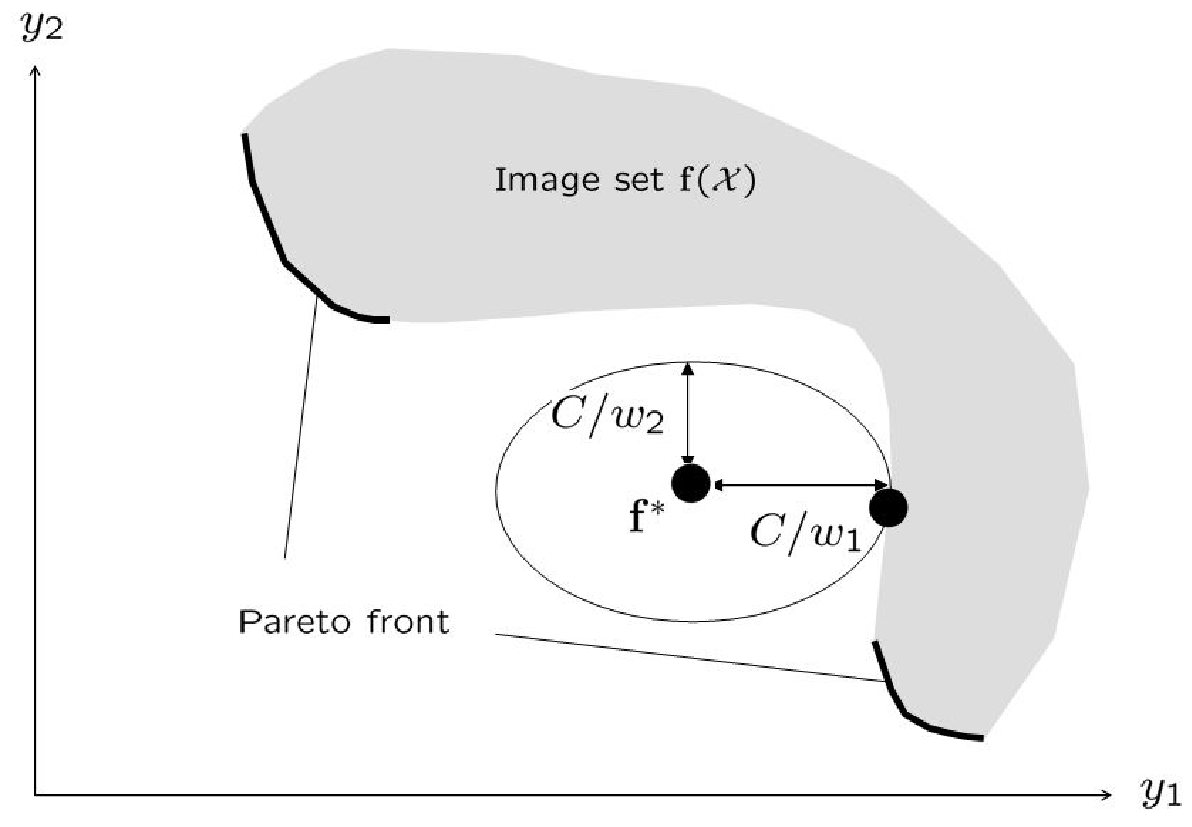} \includegraphics[width=6cm]{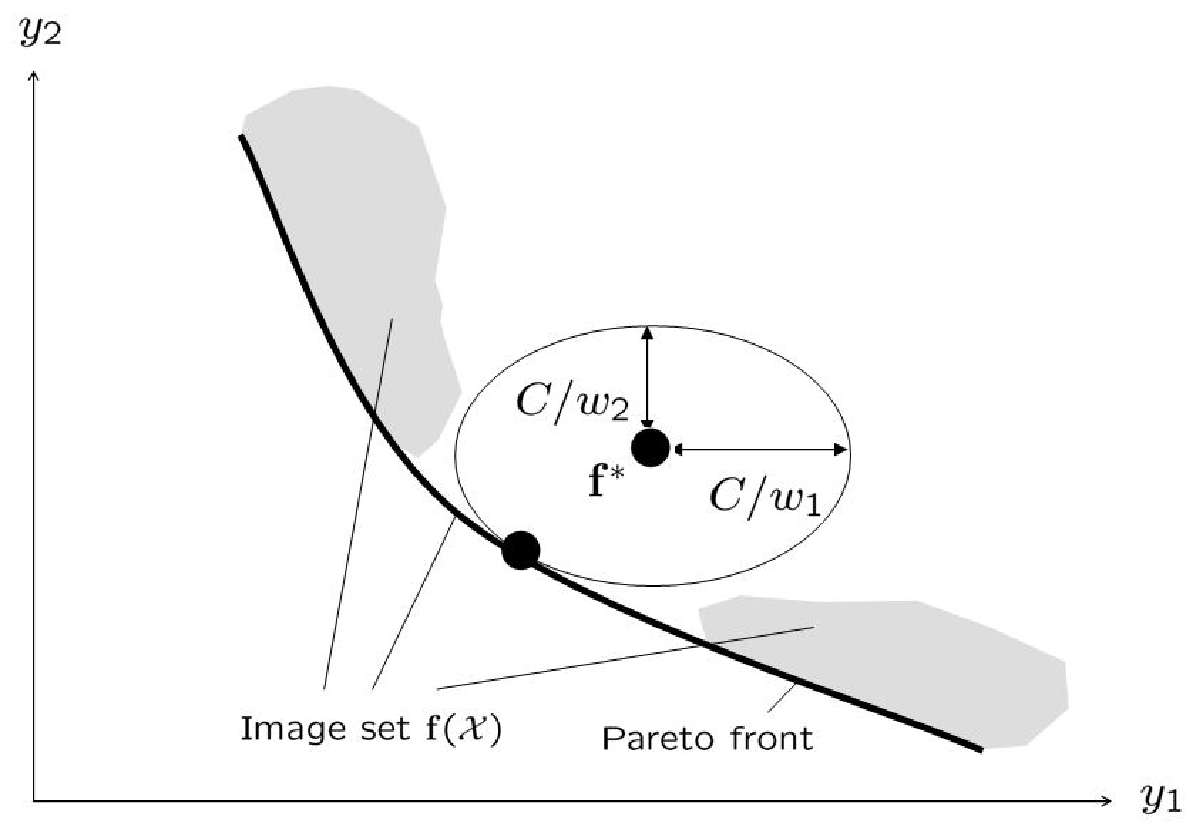}
\end{center}
\caption{\label{fig:sufnec} In the left figure we see and example
for a utopia point which is non-dominated by the image set but the
corresponding DRP method does not yield a solution on the Pareto
front. In the right figure we see an example where an utopia point
is dominated by some points of the image set, but the corresponding
DRP method yields a solution on the Pareto front.}
\end{figure}

In the case of the weighted Chebyshev distance to a reference point function (also referred to as achievement scalarizing function in the field of goal programming), all points of the Pareto front can be obtained as minimizers:
\begin{equation}
d^{\infty}_{\bf w}({\bf y}, {\bf y}') = \max_{i\in \{1, \dots, m\}} w_i |y_i - y_i'|.
\end{equation}
By optimizing this distance with different weights $w_i$, all points on the Pareto front can be found. More formally, the following condition holds:
\begin{equation}
\forall {\bf y} \in \mathcal{Y}_N, \quad \exists w_1, \dots, w_m \text{ such that } {\bf y} \in \arg \min_{{\bf y}' \in \mathcal{Y}} d_{\bf w}^\infty ({\bf y}', {\bf f}^*).
\end{equation}
\noindent
However, using the Chebyshev metric may also yield dominated points, even in cases where ${\bf f}^*$ dominates all solutions in ${\bf f}(\mathcal{X})$. These solutions are then only weakly dominated.
To mitigate this, the \textit{Augmented Chebyshev Distance to a Reference Point} is typically employed, which introduces an additional term $\rho \sum f_i(x)$, where $\rho$ is a very small positive constant (often chosen as machine epsilon).

In summary, distance-to-a-reference-point methods can be seen as an alternative scalarization approach to utility function methods with a clear interpretation of results. They require the definition of a target point (that ideally should dominate all potential solutions), and a metric needs to be specified. We note that the Euclidean metric is not always the best choice. Typically, the weighted Minkowski metric is used as a metric. The choice of weights for this metric and the choice of $\gamma$ can significantly influence the result of the method. Except for the Chebychev metric, it is not possible to obtain all points on a Pareto front by changing the weights of the different criteria. The latter metric, however, has the disadvantage that also weakly dominated points may be obtained.
\begin{figure}[ht]
\begin{center}
\includegraphics[width=6cm]{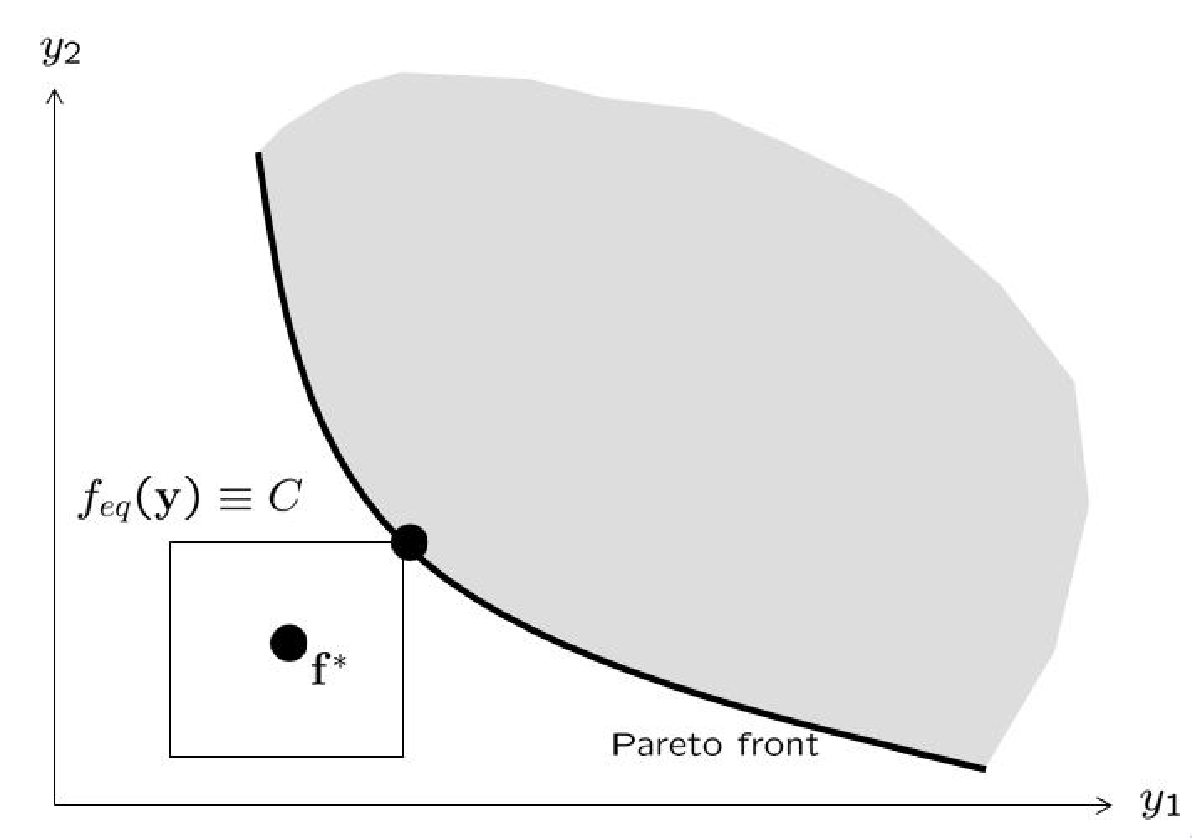} \includegraphics[width=6cm]{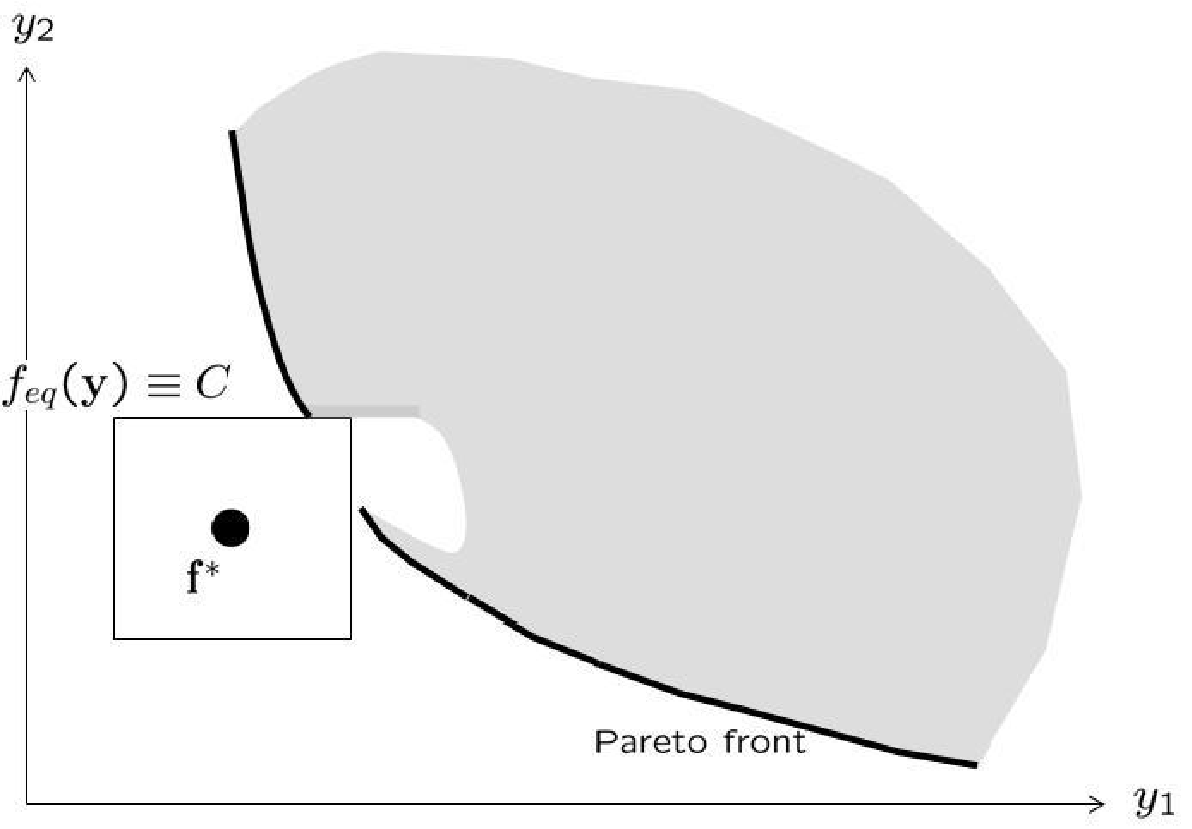}
\end{center}
\caption{\label{fig:cheb} In the left figure we see and example
where an non-dominated point is obtained using a DRP with the Chebychev distance.
In the right figure we see an example where also dominated solutions minimize
the Chebychev distance to the reference point. In these cases a non-dominated
solution may be missed by this DRP method if it returns some single
solution minimizing the distance.}
\end{figure}

\section{Goal Programming}
Goal Programming (GP) is one of the oldest and most widely used Multicriteria Decsion Making (MCDM) approaches. Its popularity stems from its flexibility in solving decision problems with multiple criteria, incomplete information, numerous decision variables, and constraints. By adhering to a realistic satisficing philosophy, GP minimizes the deviations between target goals and their achievements \cite{Ignizio1983}.

 In many aspects goal programming resembles the above distance reference point methods, but it adds particular interpretations to the results in terms of over-and underachievement of goals. In goal programming, each goal is associated with a target value that the decision maker wishes to achieve. Deviations from these targets are captured by nonnegative variables that represent underachievement and overachievement.

The basic model can be stated as:
\begin{equation}
  f_i(x) + n_i - p_i = t_i, \quad i = 1, \dots, m,
\end{equation}
where
\begin{itemize}
  \item \(f_i(x)\) is the achievement function of the \(i\)th goal,
  \item \(t_i\) is the target level for goal \(i\),
  \item \(n_i \ge 0\) represents underachievement (negative deviation),
  \item \(p_i \ge 0\) represents overachievement (positive deviation).
\end{itemize}

A common objective is to minimize a weighted sum of deviations:
\begin{equation}
  \min \sum_{i=1}^m \omega_i \left(\frac{n_i}{t_i}\right),
\end{equation}
where \(\omega_i\) reflects the importance of the \(i\)th goal. Alternative formulations include minmax or lexicographic approaches, which focus on minimizing the maximum deviation:
\begin{equation}
  \min \max_{i=1,\dots,m} \left\{ \omega_i \frac{n_i}{t_i} \right\}.
\end{equation}
Such formulations are particularly useful when ensuring a balanced achievement across all goals is critical \cite{Lee1984,Wierzbicki1980,Wierzbicki2000}.

Note that by dividing by the target value \(t_i\), we express overachievement and underachievement relative to that target. For example, consider an environmental objective \(f_1(x)\) with target \(t_1\) and an economic objective \(f_2(x)\) with target \(t_2\). If the achievement levels are given by
\[
f_i(x) + n_i - p_i = t_i, \quad i=1,2,
\]
then the normalized deviations
\[
\frac{n_i}{t_i} \quad \text{and} \quad \frac{p_i}{t_i}
\]
represent the underachievement and overachievement, respectively, as fractions of the target. Multiplying by 100 yields the deviations in percentage terms, which is easy to interpret by decision makers.

\section{Achievement Scalarizing Function}
A concept that is very similar to the weighted Chebyshev distance to a reference point method, which also seeks to balance multiple objectives by considering their scaled deviations from a reference point, is the concept of \textit{achievement scalarizing functions} was introduced by Wierzbicki \cite{Wierzbicki1980} as a means to transform a multiobjective optimization problem into a scalar-valued optimization problem while preserving Pareto-optimality. The approach is particularly useful in interactive multiobjective optimization methods, where a decision-maker provides a reference point to guide the search for preferred solutions. 

\begin{definition}
  Given a multiobjective optimization problem:
\begin{equation}
\min_{x \in X} F(x) = (f_1(x), f_2(x), \dots, f_m(x)),
\end{equation}
where $F: X \to \mathbb{R}^m$ represents $m$ objective functions, Wierzbicki's achievement scalarizing function is defined as:
\begin{equation}
S(x; z^r, \rho) = \max_{i=1,\dots,m} \left\{ \frac{f_i(x) - z_i^r}{\rho_i} \right\} + \lambda \sum_{i=1}^{m} \frac{f_i(x) - z_i^r}{\rho_i},
\end{equation}
where:
\begin{itemize}
    \item $z^r = (z_1^r, \dots, z_m^r)$ is a \textit{reference point} reflecting the decision-maker's aspirations.
    \item $\rho = (\rho_1, \dots, \rho_m)$ are positive weighting coefficients ensuring proper scalarization.
    \item $\lambda > 0$ is a small constant to ensure strict Pareto-optimality.
\end{itemize}  
\end{definition}

The function $S(x; z^r, \rho)$ ensures that minimizing it leads to solutions that are Pareto-optimal and close to the reference point $z^r$. This approach is widely applied in interactive multiobjective optimization, decision support systems, and multiobjective evolutionary algorithms.

\section{Achievement Scalarizing Functions with Reservation and Aspiration Levels}
Achievement scalarizing functions provide an alternative way to compare multiobjective solutions by converting a multi-dimensional outcome into a single scalar value. This approach typically incorporates:
\begin{itemize}
  \item \emph{Reservation levels} (\(r_i\)), which establish the minimum acceptable performance.
  \item \emph{Aspiration levels} (\(a_i\)), which represent the desired targets of performance.
\end{itemize}
The concept of a two-level approach (aspiration and reservation) is due to Wierzbicki \cite{Wierzbicki2000}.

A general form of an achievement scalarizing function is:
\begin{equation}
  s(x) = \max_{i=1,\dots,m} 
    \Bigl\{ \lambda_i \bigl( f_i(x) - [t_i + r_i] \bigr) \Bigr\}
    \;+\; \sum_{i=1}^m \mu_i \bigl( f_i(x) - a_i \bigr),
\end{equation}
where:
\begin{itemize}
  \item \(f_i(x)\) is the performance level for the \(i\)th objective,
  \item \(t_i\) is the initial target for goal \(i\),
  \item \(r_i\) is the reservation level (minimum acceptable performance) for goal \(i\),
  \item \(a_i\) is the aspiration level (desired performance),
  \item \(\lambda_i\) and \(\mu_i\) are weighting parameters reflecting the trade-offs between meeting reservation levels and pushing toward aspiration levels.
\end{itemize}

The first term penalizes deviations that fall below the sum of the target and the reservation level, whereas the second term accounts for deviations from the aspiration level. By adjusting \(\lambda_i\) and \(\mu_i\), decision makers can shape how aggressively they strive to meet these levels \cite{Perez2001,Martins2002,Wierzbicki1980,Wierzbicki2000}.

The approach was generalized for composite indicators and multiple reference points indicating 'grade'-like achievement levels (such as insufficient, good, very good, excellent) in \cite{Ruiz2020}.

%
%

\section{Transforming Multicriteria into Constrained Single-Criterion Problems}
This chapter will highlight two common approaches for transforming Multicriteria into Constrained Single-Criterion Problems. In {\em Compromise Programming} (or $\epsilon$-Constraint Method), $m-1$ of the $m$ objectives are transformed into constraints. Another approach is put forward in the so-called {\em goal attainment} and {\em goal programming method}. Here a target vector is specified (similar to the distance to a reference point methods), and a direction is specified. The method searches for the best feasible point in the given direction. For this, a constraint programming task is solved.

\subsection{Compromise Programming or \texorpdfstring{$\epsilon$}{epsilon}-Constraint Methods}
%
In compromise programming we first choose $f_1$ to be the objective function that has to be solved with highest priority and then re-state the original multicriteria optimization problem (Eq. \ref{eq:genmulopt}):
\begin{equation}
f_{1}(x) \rightarrow \min, f_{2}(x) \rightarrow \min, \dots, f_{m}(x) \rightarrow \min
\end{equation}
into the single-criterion constrained problem:
\begin{equation}
\label{eq:compromise}
f_{1}(x) \rightarrow \min, f_{2}(x) \leq \epsilon_2, \dots, f_{m}(x) \rightarrow \epsilon_{m}.
\end{equation}
In figure \ref{fig:compromiseprogram} the method is visualized for the bi-criteria case ($m=2$). Here, it can be seen that if the constraint boundary shares points with the Pareto front, these points will be the solutions to the problem in Eq. \ref{eq:compromise}. Otherwise, it is the solution that is the closest solution to the constraint boundary among all solutions on the Pareto-front.
\begin{figure}
\includegraphics[width=12cm]{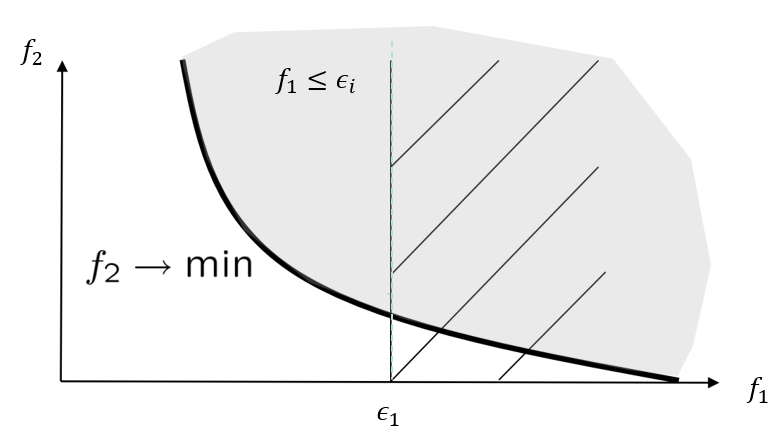}
\caption{\label{fig:compromiseprogram}Compromise Programming in the bi-criteria case. The first objective is transformed
into a constraint. }
\end{figure}
In many cases the solutions are obtained at points $x$ where all objective function values $f_i(x)$ are equal to
$\epsilon_i$ for $i =1, \dots, m$. In these cases, we can obtain optimal solutions using the Lagrange Multiplier method discussed in chapter \ref{chap:opticond}. Not in all cases the solutions obtained with the compromise programming method are Pareto optimal. The method might also find a weakly dominated point, which then has the same aggregated objective function value than some non-dominated point.
The construction of an example is left as an exercise to the reader.


The compromise programming method can be used to approximate the Pareto front. For a $m$ dimensional problem a $m-1$
dimensional grid needs to be computed that cover the $m-1$ dimensional projection of the bounding box of the Pareto front.  Due to Lemma \ref{lem:unique} given $m-1$ coordinates of a Pareto front, the $m$-th coordinate is uniquely determined as the minimum of that coordinate among all image vectors that have the $m-1$ given coordinates.
As an example, in a 3-D case (see Figure \ref{fig:comp}) we can place points on a grid stretching out from the minimal point
$(f^{min}_1, f^{max}_1)$ to the maximal point $(f^{min}_2, f^{max}_2)$ . It is obvious that, if the grid resolution is kept constant, the effort of this method grows exponentially with the number of objective functions $m$.

This method for obtaining a Pareto front approximation is easier to control than the to use weighted scalarization and
change the weights gradually. However, the knowledge of the ideal and the Nadir point is needed to compute the approximation, and the computation of the Nadir point is a difficult problem in itself.
\begin{figure}
\begin{center}
\includegraphics[width=10cm]{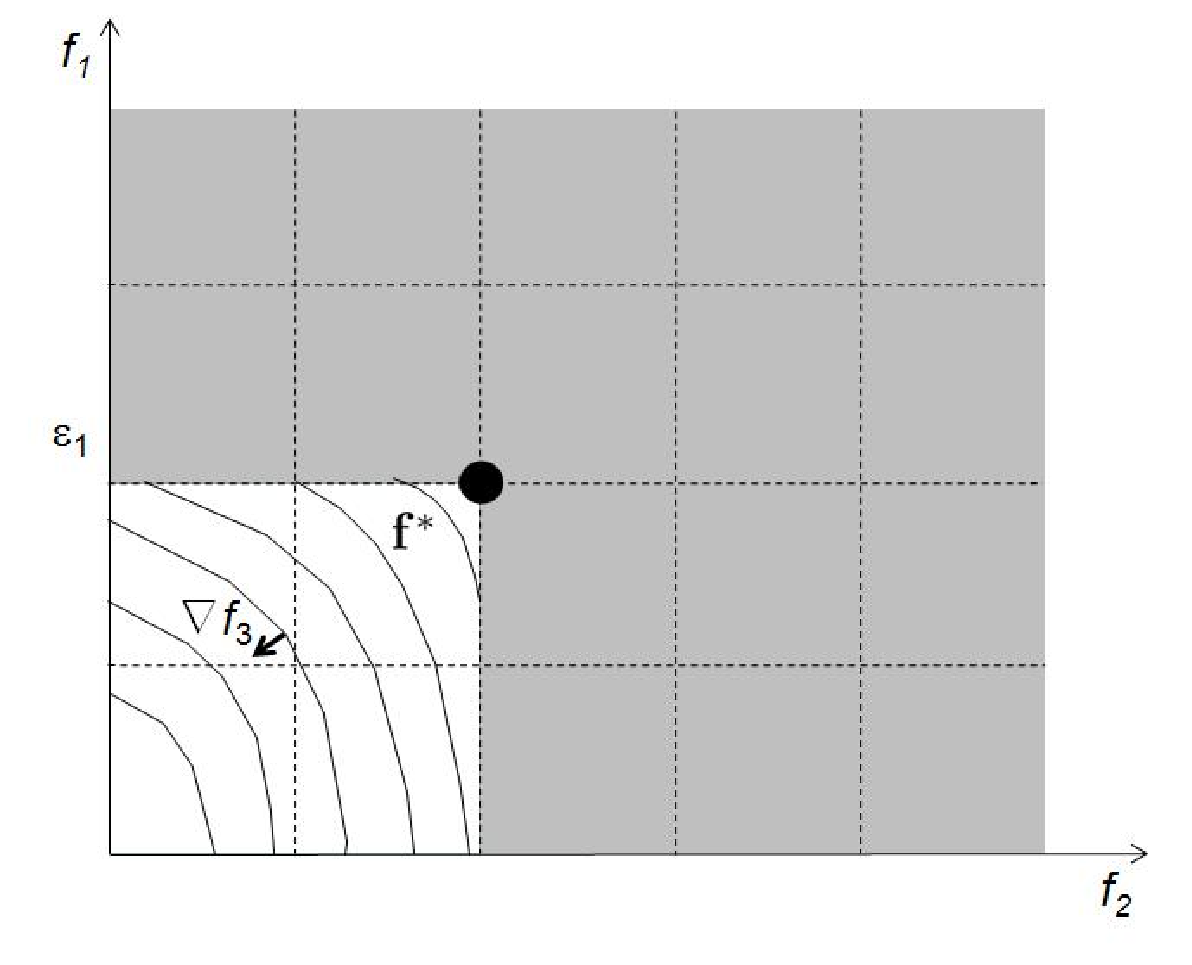}    
\end{center}
\caption{\label{fig:comp}Compromised Programming used for approximating the Pareto front with 3 objectives. }
\end{figure}

\section{Processes for Utility Function Elicitation}
Aside from the discussion of the mathematical form of utility functions, it is also important to discuss the process of eliciting utility functions in the interaction with the decision maker(s).

\subsection{Value-focused thinking}
Keeney proposed the Value-Focused Thinking framework, a systematic approach to helping decision-makers define a multi-attribute utility function \cite{Keeney96}. This approach emphasizes defining values before considering alternative decisions. The key steps in \textbf{Value-Focused Thinking} include:

\begin{enumerate}
    \item \textbf{Identify fundamental values}: Determine the core objectives that matter in the decision-making process.
    \item \textbf{Structure objectives hierarchically}: Organize objectives into fundamental and means objectives, where fundamental objectives reflect what truly matters, and means objectives help achieve them.
    \item \textbf{Define attributes and measures}: Establish criteria to quantify the objectives.
    \item \textbf{Construct a utility function}: Develop a mathematical representation of preferences based on trade-offs among attributes.
    \item \textbf{Evaluate trade-offs and explore alternatives}: Use the utility function to compare different alternatives and determine the most preferred option.
\end{enumerate}

In their 2023 paper, Afsar et al.\cite{Afsar23} emphasize the critical importance of structuring multiobjective optimization (MOO) problems and introduce a systematic approach inspired by multiple criteria decision analysis (MCDA). While MCDA typically deals with predefined alternatives characterized by specific criteria, MOO involves decision variables and constraints. To address this distinction, the authors propose an expert-driven elicitation process to identify objectives, constraints, and decision variables, ensuring accurate problem formulation and validation prior to optimization. 

\subsection{Processes for Utility Function Elicitation by Pairwise Comparison}
Often, it is easier for decision makers to rank two solutions when they see them rather than defining importance weights for objectives.
The method proposed in\cite{Greco2008} for eliciting utility functions is based on pairwise comparisons of alternatives provided by a decision-maker. This approach is particularly useful in Multi-Criteria Decision Analysis (MCDA), where explicit numerical utility functions are difficult to define directly.
The elicitation process involves the following key steps:

\begin{enumerate}
    \item \textbf{Pairwise Preference Data Collection:} The decision-maker provides pairwise comparisons of alternatives, indicating which one is preferred and by how much (if possible).
    
    \item \textbf{Construction of Constraints:} Each comparison defines constraints on the utility function $U$, ensuring that the preferred alternative has a higher utility value:
    \begin{equation}
        U(a) > U(b) \quad \text{if } a \succ b.
    \end{equation}
    If preference intensity is available, it can be incorporated as:
    \begin{equation}
        U(a) - U(b) \geq \delta \quad \text{for some } \delta > 0.
    \end{equation}

    \item \textbf{Aggregation of Criteria:} The utility function is typically assumed to be an additive model, but based on the weaknesses of additive models recently Choquet integral is preferred\cite{Branke2016}. To keep our treatise simple let us assume the additive model for now:
    \begin{equation}
        U(a) = \sum_{i=1}^{m} w_i u_i(a_i),
    \end{equation}
    where $w_i$ are the importance weights and $u_i(a_i)$ are partial utility functions.

    \item \textbf{Optimization for Utility Function Estimation:} The constraints from pairwise comparisons form a feasible region in which an appropriate utility function is determined using linear programming or regression techniques.

    \item \textbf{Consistency Check and Refinement:} The derived function is validated against additional preferences, and inconsistencies are resolved by refining the utility model.
\end{enumerate}

\section{Conclusions}

Various approaches have been discussed to reformulate a multiobjective problem
into a single-objective or a constrained single-objective problem. The methods discussed have in common that they
result in a single point, why they also are referred to as {\em single point methods}.

In addition, all single-point methods have parameters, the choice of which determines the location of the optimal
solution. Each of these methods has, as we have seen, some unique characteristics and is different to
give a global comparison of them. However, a criterion that can be assessed for all single-point methods is whether they are always resulting in Pareto optimal solutions. Moreover, we investigated whether by changing their
parameters, all points on the Pareto front can be obtained.

To express this in a more formal way we may denote a single point method by a function $A: P \times C \mapsto \mathbb{R}^m \cup \{ \Omega \}$, where
$P$ denotes the space of multi-objective optimization problems, $C$ denotes the parameters of the method (e.g. the weights in linear weighting). In order to classify a method $A$ we introduce the following two definitions:

\begin{definition} Proper single point method\\
A method $\texttt{A}$ is called {\em proper}, if and only if for all $p \in P$ and  $c \in C$ either $p$ has a solution and the point $A(p, c)$ is Pareto optimal or $p$ has no solution and $\texttt{A}(p,c) = \Omega$.
\end{definition}

\begin{definition} Exhaustive single point method\\
A method $\texttt{A}$ is called {\em exhaustive}  if for all $p \in P$: $\mathcal{Y}_N(p) \subseteq \bigcup_{c \in C} \texttt{A}(p,c)$, where $\mathcal{Y}_N(p)$ denotes the Pareto front of problem $p$.
\end{definition}

In Table \ref{tab:aggreg} a summary of the properties of methods we discussed:
\begin{table}
{\small
\begin{tabular}{|l|l|l|l|}
\hline
Single Point Method& Proper & Exhaustive & Remarks\\
\hline
\hline
Linear Weighting & Yes & No & Exhaustive for convex Pareto \\&&& fronts with only \\&&& proper Pareto optima \\
\hline
Weighted Euclidean DRP
 & No & No & Proper if reference \\&&& point dominates all \\&&& Pareto optimal points \\
\hline
Weighted Chebyshev DRP& No & Yes & Weakly non-dominated \\&&& points can be obtained, \\&&& even when reference \\&&& point dominates all \\&&& Pareto optimal points\\
\hline
Achievement Scalarizing Function& Yes & Yes & Augmentation constant \\&&& needs to be small enough. \\
\hline
Desirability index & No  & No & The classification of \\&&& proper/exhaustive is \\&&& not relevant in this case. \\
\hline
Goal programming & No  & No & For convex and concave \\&&& Pareto fronts with the method  \\&&& is proper and exhaustive \\&&&if the reference \\&&& point dominates all \\&&& Pareto optimal solutions\\
\hline
Compromise programming & No & Yes & In 2- objective spaces \\&&& the method is proper. \\&&&Weakly dominated points \\&&& may qualify as solutions  \\&&& for more than three \\&&& dimensional objective \\&&& spaces\\
\hline
\end{tabular}
}
\caption{Aggregation methods and their properties \label{tab:aggreg}.}
\end{table}
In the following chapters on algorithms for Pareto optimization, single-point methods often serve as building blocks for approaches that compute the entire Pareto front or an approximation of it.

Among scalarization techniques, the Chebyshev distance to a reference point is the only exhaustive method. However, it can be extended into a fully-fledged approach through augmentation. Importantly, scalarization methods are not solely intended for sampling the Pareto front in a posteriori approaches. In a priori and interactive methods, it is crucial that the decision maker comprehends the weight parameters and is guided through structured processes to elicit them. Methods such as desirability functions provide human-interpretable mechanisms for determining weight and shape parameters, making them particularly suitable for a priori approaches.
\section*{Exercises}
\begin{enumerate}

    \exercise{Linear Weighting Utility Function}{%
        Consider a decision problem with two objectives \( f_1(x) \) and \( f_2(x) \), where \( x \) belongs to a feasible set \( X \). A decision-maker aggregates these objectives using a linear weighting utility function:
        \[
        U(x) = w_1 f_1(x) + w_2 f_2(x),
        \]
        where \( w_1, w_2 \geq 0 \) and \( w_1 + w_2 = 1 \).  

        \begin{enumerate}
            \item[(a)] Formulate the corresponding optimization problem.  
            \item[(b)] Discuss how the choice of weights influences the optimal solution.  
            \item[(c)] What conditions on \( f_1, f_2 \) ensure that all Pareto-optimal solutions can be found by using all possible weights?  
        \end{enumerate}
    }{ex:linear-weighting}

    \exercise{Keeney-Raiffa Utility Function Graphical Solution}{%
        The Keeney-Raiffa utility function for two attributes \( x_1 \) and \( x_2 \) is given by:
        \[
        U(x_1, x_2) = k_1 u_1(x_1) + k_2 u_2(x_2),
        \]
        where \( k_1, k_2 \) are scaling constants, and \( u_1, u_2 \) are normalized attribute utility functions. Assume a decision-maker provides the following preference assessments:
        \[
        U(100, 0) = 1, \quad U(0, 100) = 1, \quad U(50, 50) = 0.6.
        \]

        \begin{enumerate}
            \item[(a)] Derive the parameters \( k_1, k_2 \) using the given information.  
            \item[(b)] Illustrate the utility function graphically in the \( (y_1, y_2) \)-plane.
            \item[(c)] Demonstrate that a logarithmic transformation linearizes the objective function, allowing a multi-objective linear program to be solved as a single-objective linear program.
            \item[(d)] Explain how changes in the utility function affect decision-making.  
        \end{enumerate}
    }{ex:keeney-raiffa}

    \exercise{Chebyshev Scalarization}{%
        Given a multiobjective optimization problem with objectives \( f_1(x), f_2(x), \dots, f_m(x) \), define the Chebyshev scalarization function:
        \[
        U(x) = \max_{i} w_i | f_i(x) - z_i^* |,
        \]
        where \( w_i > 0 \) are user-defined weights, and \( z_i^* \) are reference (ideal) values for each objective.

        \begin{enumerate}
            \item[(a)] Explain the role of the weights \( w_i \) and the reference point \( (z_1^*, \dots, z_m^*) \).   
            \item[(b)] Demonstrate that a multiobjective linear program can be converted into a single-objective linear program through Chebyshev distance scalarization using min-max linearization. 
        \end{enumerate}
    }{ex:chebyshev}

    \exercise{Utility Function from Pairwise Comparisons}{%
        A decision-maker compares four alternatives \( A, B, C, D \) pairwise and provides the following preferences:
        \[
        A \succ B, \quad B \succ C, \quad C \succ D, \quad A \succ C, \quad B \succ D.
        \]

        \begin{enumerate}
            \item[(a)] Formulate a linear programming model (constraints only) that can be used to determine the weights of an additive utility function such that \( U(A), U(B), U(C), U(D) \) are consistent with the pairwise preferences for three objective functions.    
            \item[(b)] Discuss whether these preferences allow for a unique utility function or multiple valid solutions. Can a unique solution be achieved by maximizing the robustness of constraint satisfaction? 

            \textbf{Hint:} Introduce an additional variable \( \rho \) to measure the distance of a solution from the constraint boundary and explore its role in ensuring robustness.
        \end{enumerate}
    }{ex:pairwise-utility}

    \exercise{Designing Desirability Functions for Decision Analysis}{%
        Choose a multicriteria decision problem with approximately 10 alternatives, such as selecting a dance course or purchasing a washing machine. Follow these steps to structure your analysis:

        \begin{enumerate}
            \item[(a)] Establish a preliminary intuitive ranking of the alternatives.  
            \item[(b)] Specify quantitative and qualitative criteria for evaluating the alternatives, along with any relevant constraints.  
            \item[(c)] Develop desirability functions for each objective, ensuring they accurately reflect the decision-maker’s preferences.  
            \item[(d)] Determine the relative importance of each objective and define appropriate weighting factors.  
            \item[(e)] Compute the Pareto front and visualize both feasible and infeasible solutions.  
            \item[(f)] Compare the obtained ranking with your initial intuitive ranking. Discuss whether the rankings agree, and if discrepancies arise, analyze the reasons behind them.  
        \end{enumerate}
    }{ex:desirability}

\end{enumerate}

\part{Algorithms for Pareto optimization}

\chapter{Efficient computation of the non-dominated set}
In the previous chapters, we explored ways to reformulate 
multiobjective optimization problems as single-objective 
(constrained) optimization problems. However, in many cases, it is
desirable for the decision maker to have access to the entire Pareto
front rather than a single compromise solution.

Methods for computing the Pareto front, or a finite approximation
thereof, can be broadly categorized into deterministic approaches that
guarantee convergence to Pareto-optimal sets. Some of these methods
ensure well-structured approximations, such as a uniform distribution
along the Pareto front (e.g., homotopy or continuation methods) or an
optimal trade-off in terms of hypervolume coverage (e.g., S-metric 
gradient approaches). 

Naturally, such guarantees are based on specific assumptions about
objective functions, such as convexity, smoothness, or continuity. 
When these conditions hold, one can ensure the quality of the 
computed approximation. In cases where these assumptions do not hold, 
alternative techniques, such as heuristic or metaheuristic approaches, 
may be necessary to obtain a well-distributed approximation of the 
Pareto front.

\section{Computing the non-dominated subset of a pre-ordered finite set}

It can be shown that $\Theta(n^2)$ is the time complexity for finding the minimal (or maximal) set of a partially ordered set of a \textbf{general} partially ordered set (no additional structure is given).

This algorithm finds the maximal set for every preordered set $(V, \succ)$ of size $n$ with $T(n) \in \mathcal{O}(n^2)$ pairwise comparisons:

\begin{enumerate}
    \item \textbf{INPUT} Preordered set: $(V, \succ)$
    \item \textbf{FOR ALL} $i \in \{1, \dots, n\}$
    \begin{enumerate}
        \item $t_i = \text{true}$
        \item \textbf{FOR ALL} $j \in \{1, \dots, n\}$
        \begin{enumerate}
            \item \textbf{IF} $v_j \succ v_i$ \textbf{THEN} $t_i = \text{false}$; \textbf{BREAK}
        \end{enumerate}
        \item \textbf{IF} $t_i = \text{true}$ \textbf{OUTPUT} $v_i$ (is maximal)
    \end{enumerate}
\end{enumerate}

Firstly, the naïve algorithm  shows that the problem is in $\mathcal{O}(n^2)$.

To prove a lower bound, consider that it needs to decide whether or not all elements are maximal.

To show this, we need to understand that it is necessary to check all pairs of points, say $v_i$ and $v_j$. Note that if we leave out some pair, we do not know whether $a_i$ and $a_j$ are in a dominance relation, and if $a_i$ or $a_j$ are not dominated so far, based on this information the situation might change.
One might argue that transitivity can be used to conclude the dominance before we have seen the pair. In the worst case, however, we are given an anti-chain and then we need to visit every pair to know this.

However, in geometrical settings and for the Pareto order, which is a special case of a partial order, we can do better.

\section{Kung, Luccio, and Preparata's Algorithm for the Nondominated Subset}This first chapter summarizes the algorithm of Kung, Luccio, and Preparata \cite{KLP75} to efficiently find the nondominated set of vectors with respect to the Pareto order. We use here maximization of the objectives, in order to stay close to the original article.

KLP established lower and upper bounds for the complexity of finding maximal vectors:
\begin{itemize}
    \item For \( f = 2,3 \):
    \[
    C_d(n) \in O(n \log n)
    \]
    \item For \( d \geq 4 \):
    \[
    C_d(n) \in O(n (\log n)^{d-2})
    \]
    \item Lower bound:
    \[
    C_d(n) \in \Omega(\lceil \log (n!) \rceil)
    \]
\end{itemize}

\noindent Dimension sweep and Divide-and-Conquer paradigms are used in combination to construct algorithms in 3-D. The KLP bounds still hold today and have only been improved for special cases. However, the upper bounds are not sharp—can they be improved?
KLP methods are \textbf{used in MODA Algorithms and Skyline Queries} for database applications.


\subsection{Dimension Sweep Algorithm for 2-D and 3-D}
Let $S$ be a subset of $\mathbb{R}^d$. Recall that we have defined a partial order on $\mathbb{R}^d$: for any $v, w \in \mathbb{R}^d$, $v \prec w$ iff for all $i=1,\cdots, d$,$v_i \leq w_i$ and $\exists j \in \{1,\cdots,d \}$ such that $v_j < w_j$. This partial order is inherited by the subset $S$.

For $S \subseteq \mathbb{R}^d$ the {\em maximal set} $\mbox{Max}(S)$ will be defined as
\begin{equation}
\mbox{Max}(S) = \{v \in S | \nexists u \in S: v \prec u\}.
\end{equation}
For a Pareto optimization problem, with objectives $f_1$, $\dots$, $f_d$ to be maximized, the maximal set of the image set of the search space $\mathbb{S}$ under $\mathbf{f}$ is the Pareto Front (PF) of that problem.

For all $d \geq 1$ a lower and upper bound of finding the maximal subset of a set will be given. For the proposed algorithms the time complexity will be derived as well. For $d=2 \mbox{~or~} 3$ we will describe efficient algorithms and in that case also prove their efficiency.

The number of comparisons an algorithm $A$ needs for computing the maximal set of $S$ will be denoted with $c_d(A,S)$.
The time complexity of a $d$-dimensional problem in terms of $n = |S|$ is estimated as:

\begin{equation}
\mathcal{T}_d(n) = \min_{A \in \mathcal{A}} \max_{\ S\subset_n \mathbb{R}^d} c_d(A,S)
\end{equation}

Here $\mathcal{A}$ is the set of all algorithms and $\subset_n$ is the subset operator with the restriction
that only subsets of size $n$ are considered.




For $\ybf \in \mathbb{R}^d$ we denote the vector $(y_2, \cdots, y_d )$ by  $\ybf^{*}$. In other words, the first coordinate is discarded.
%

We denote by $p_i$
the projection of $S$ to the i-th coordinate (for any fixed $i \in \{1, \cdots, d\}$).

\begin{definition}
Let $A \subseteq \mathbb{R}^d$ and $\ybf \in S$.
Then $\ybf \prec A$ iff $\exists a \in A$ such that $\ybf \prec a$
\end{definition}

\begin{lemma}
\label{lem:domination}
Let $A \subseteq \mathbb{R}^d$ and $\ybf \in S$.
Then $\ybf \prec A$ iff $ \ybf \prec \mbox{Max}(A)$. $\Box$
\end{lemma}

A prototype of the algorithm for computing the maximal elements (PF) of $S$ is as follows. In order to
present the ideas more clearly we assume $\ybf^{(i)}, i=1, \cdots, m$
\emph{to be mutually different in each coordinate.~} We shall address equality in some (or all coordinates) separately.
\begin{algorithm}[tbph]
\begin{algorithmic}[1]
\STATE {\bf input}: $S = \{ \ybf^{(1)}, \cdots, \ybf^{(k)} \} \subseteq_k \mathbb{R}^d \mbox{~and~} k \in \mathbb{N} $ \\
\COMMENT {NB \emph{We assume} the $\ybf^{(i)}$
\emph{to be mutually different in each coordinate.~}}
\STATE View the elements $S$ as a sequence and sort the elements of this sequence by the first
coordinate in descending order:
$\ybf^{(1)}, \ybf^{(2)}, \cdots, \ybf^{(k-1)}, \ybf^{(k)}$ is now such that
$$p_1(\ybf^{(1)}) >  p_1(\ybf^{(2)}) >  \cdots > p_1(\ybf^{(k-1)}) > p_1(\ybf^{(k)})$$
\STATE $i \leftarrow 1; $
\STATE $T_0 \leftarrow \emptyset;$
\COMMENT{~The $T_i$ are sets of ($d-1$)-dim vectors~}
\FORALL{$i = 1, \dots, k$}
\IF{$\ybf^{(i)*} \prec T_{i-1}$}
\STATE $T_i \leftarrow T_{i-1}$
\ELSE
\STATE $T_i \leftarrow \mbox{Max~}(T_{i-1} \cup \{\ybf^{(i)*} \})  \mbox{~and mark~} \ybf^{(i)} \mbox{~as maximal.}$
\ENDIF
\ENDFOR
\end{algorithmic}
\caption{\label{alg:prototypePFofS}{\bf Prototype Algorithm for Computing PF of a {\em finite} set S}}
\end{algorithm}
It is clear that in case $d = 2$, the sets $T_i$ contain one element and for updates only one comparison is used, so the time complexity of the algorithm in this case is $n \log n$. (The sorting requires $n \log n$ steps while the loop requires $n$ comparisons.)\\ In case $d=3$, the maintenance of the sets $T_i$ is done by a balanced binary tree with keys the second coordinate of the vectors $\ybf^{(i)}$.
A crucial step in the loop is to update the $T^{i}$ given $\ybf^{(i+1)}$. First determine the element in $T^{(i)}$ that precedes $\ybf^{i}$ in its second coordinate. Then determine the elements that are dominated by $\ybf^{(i)}$ by visiting the leaves of the search tree in descending order of the second coordinates and stop when the first visited element exceeds the third coordinate of $\ybf^{(i+1)}$ in its own third coordinate. Discard all visited points, as they are dominated in all three coordinates by $\ybf^{i+1}$.
See Figure \ref{fig:avl} for an illustration of how the tree can be used to detect dominated points efficiently.

Note that in this case, the third coordinate is also sorted.


\begin{figure}
    \centering
    \includegraphics[width=0.75\linewidth]{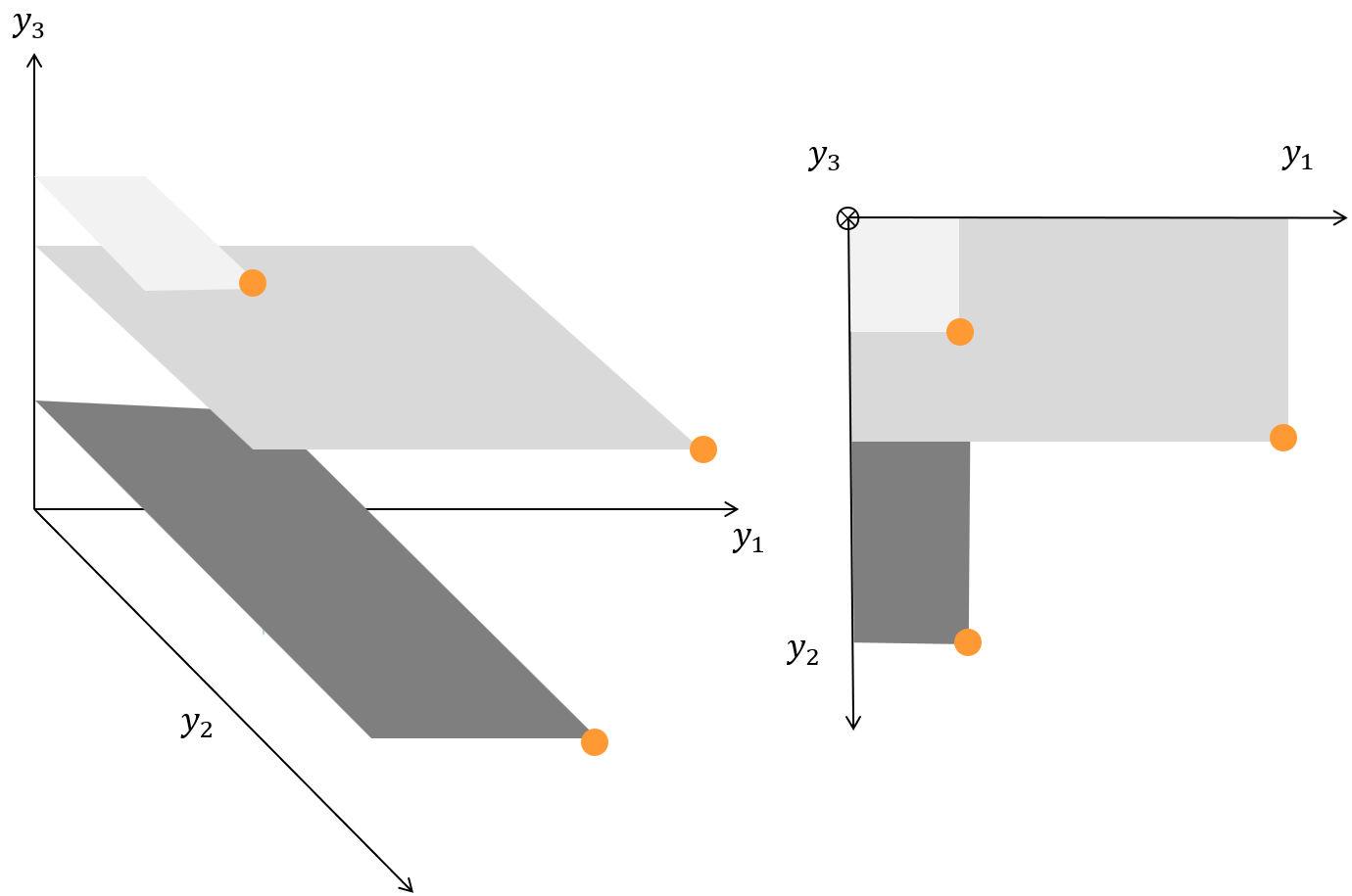}
    \caption{3-D Perspective of point sets. The darker the point the smaller the 3rd coordinate.}
    \label{fig:3d-perspective}
\end{figure}
By Lemma \ref{lem:domination} we can replace the test $\ybf^{(i)*} \prec T_{i-1}$ in step 5 of
Algorithm \ref{alg:prototypePFofS} by the test $\ybf^{(i)*} \prec \mbox{Max~}(T_{i-1})$.
A variation on Algorithm \ref{alg:prototypePFofS} is to mark an element as maximal as you
go along and work with $\mbox{Max~}(T_{i-1})$. Note that the algorithm is'monotone': once an element's star is admitted to $T_i$ it is clearly maximal (and the star will survive in the future $T_i$.
The prototype
can be specialized by describing the choices made
for the test $\ybf^{(i)*} \prec T_{i-1}$ and for the construction of the $T_i$.

\begin{figure}
    \centering
    \includegraphics[width=0.75\linewidth]{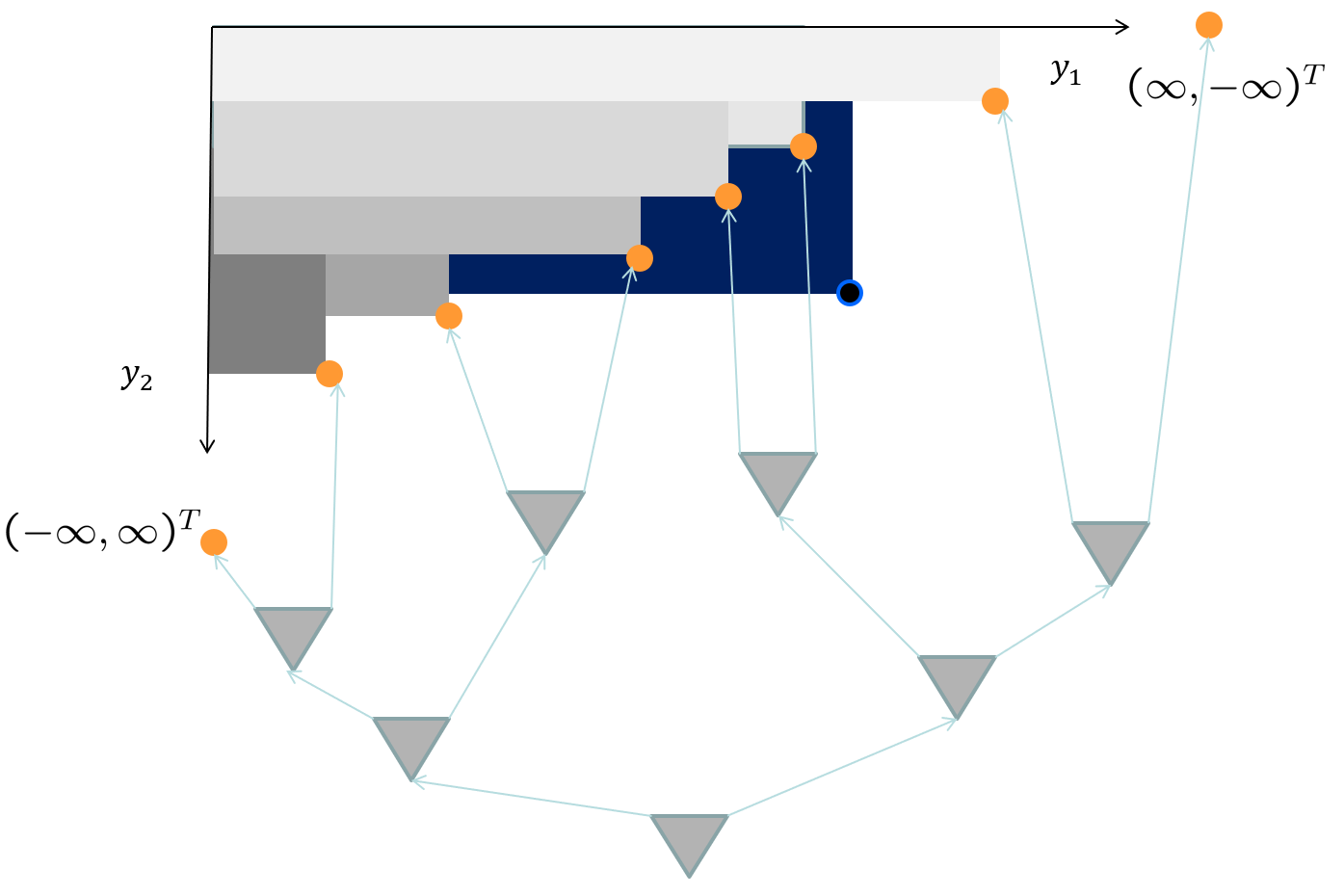}
    \caption{AVL Tree used to find points that become dominated in the $y_1-y_2$-plane. See Fig. \ref{fig:3d-perspective} for the color-encoding of the third objective.}
    \label{fig:avl}
\end{figure}

\subsection{Efficient Algorithm for the N-Dimensional Case}

For \( d \geq 4 \), KLP suggested a divide-and-conquer algorithm.

\begin{enumerate}
    \item \textbf{FUNCTION} \( M = \text{Maxima}(V, d) \)
    \item \textbf{INPUT}: Sequence of \( d \)-dimensional vectors \( V = (v_1, \dots, v_n) \) sorted by the first coordinate.
    \item Partition \( V \) into subsets \( R = (v_1, \dots, v_{n/2}) \) and \( S = (v_{n/2+1}, \dots, v_n) \).
    \item Compute:
    \begin{itemize}
        \item \( \hat{R} = \text{Maxima}(R) \)
        \item \( \hat{S} = \text{Maxima}(S) \)
        \item \( T \) is the subset of vectors in \( \hat{S} \) that is not dominated by vectors in \( \hat{R} \).
    \end{itemize}
    \item \( M \leftarrow \hat{R} \cup T \)
\end{enumerate}

The algorithm follows a \textbf{divide-and-conquer} strategy to compute the \textit{maxima} of a set of \( d \)-dimensional vectors. A vector \( v \) is said to \textbf{dominate} another vector \( w \) if all coordinates of \( v \) are greater than or equal to those of \( w \), and at least one coordinate is strictly greater. The goal is to identify the set of vectors that are \textbf{not dominated} by any other vector in the set.

\subsection*{Step-by-step Explanation}
\begin{enumerate}
    \item \textbf{Function Definition}  
    The function \( M = \text{Maxima}(V, d) \) takes as input a set \( V \) of \( d \)-dimensional vectors and returns the set of maxima.

    \item \textbf{Sorting}  
    The input sequence of vectors \( V = (v_1, \dots, v_n) \) is sorted by the \textbf{first coordinate}. Sorting aids in structuring the divide-and-conquer process efficiently.

    \item \textbf{Partitioning}  
    The sorted set \( V \) is divided into two subsets:
    \begin{align*}
        R &= (v_1, v_2, \dots, v_{n/2}) \\
        S &= (v_{n/2+1}, \dots, v_n)
    \end{align*}
    This separation ensures that all vectors in \( S \) have a first-coordinate value that is \textbf{greater than or equal to} the values in \( R \).

    \item \textbf{Recursive Computation}  
    The algorithm recursively computes the maxima for each subset:
    \begin{align*}
        \hat{R} &= \text{Maxima}(R) \quad \text{(maxima of left subset)} \\
        \hat{S} &= \text{Maxima}(S) \quad \text{(maxima of right subset)}
    \end{align*}

    \item \textbf{Filtering Non-Dominated Vectors}  
    After computing \( \hat{R} \) and \( \hat{S} \), the next step is to filter out dominated vectors:
    \begin{itemize}
        \item Any vector in \( \hat{S} \) that is dominated by at least one vector in \( \hat{R} \) is removed.
        \item Let \( T \) be the subset of vectors in \( \hat{S} \) that are \textbf{not dominated} by any vector in \( \hat{R} \).
    \end{itemize}

    \item \textbf{Final Step}  
    The final maxima set is given by:
    \[
    M \leftarrow \hat{R} \cup T
    \]
    where \( M \) contains all the non-dominated vectors from both subsets.
\end{enumerate}

\section*{Subproblem: Efficiently Determine Elements of \( S \) that are Not Dominated by \( R \)}

\begin{enumerate}
    \item Arrange elements of \( R \) as \( (u_1, \dots, u_r) \) and elements of \( S \) as \( (v_1, \dots, v_s) \) such that:
    \[
    \#1(u_1) > \dots > \#1(u_r) \quad \text{and} \quad \#1(v_1) > \dots > \#1(v_s)
    \]
    where \( \#1(v) \) represents the first component of the vector \( v \).

    \item Set boundary conditions:
    \[
    \#1(u_0) \leftarrow \infty, \quad \#1(u_{n+1}) \leftarrow -\infty
    \]

    \item Find \( k \) such that:
    \[
    \#1(u_k) \geq \#1(v_{s/2}) > \#1(u_{k+1})
    \]

    \item Partition the set \( R \) into:
    \begin{align*}
        R_1 &= (u_1, \dots, u_k) \\
        R_2 &= (u_{k+1}, \dots, u_r)
    \end{align*}

    \item Partition the set \( S \) into:
    \begin{align*}
        S_1 &= (v_1, \dots, v_{s/2}) \\
        S_2 &= (v_{s/2+1}, \dots, v_s)
    \end{align*}

    \item Compute:
    \[
    T \leftarrow \left\lfloor \frac{R_1}{S_1} \right\rfloor \cup \left\lfloor \frac{R_2}{S_1} \right\rfloor
    \cup \left\lfloor \frac{R_1}{S_2} \right\rfloor \cup \left\lfloor \frac{R_2}{S_2} \right\rfloor
    \]
    where \( \left\lfloor \frac{R}{S} \right\rfloor \) denotes the elements of \( S \) that are not dominated by any element in \( R \).
\end{enumerate}

\section*{How to Efficiently Determine \( T \)}

\noindent We define:
\[
S_1 = \left\lfloor \frac{R_2}{S_1} \right\rfloor
\]
because all elements in \( S_1 \) have better first coordinates than all elements in \( R_2 \).

\noindent Since the first coordinate in \( S_2 \) is always worse than the second coordinate in \( S_1 \), elements in \( S_2 \) can only qualify due to their other coordinates for \( T \).

\noindent Hence, it is a \( d - 1 \) dimensional problem to find:
\[
\left\lfloor \frac{R_1}{S_2} \right\rfloor
\]

\noindent The recurrence relation for computing \( T \) is:
\[
F_d(r,s) = \min_A \max_{|R|=r, |S|=s} f_d(A, R, S)
\]
where \( f_d(A, R, S) \) represents the number of comparisons required to solve:
\[
\left\lfloor \frac{R}{S} \right\rfloor
\]

\noindent We establish an upper bound:
\[
F_d(r,s) \leq F_d(k, s/2) + F_d(r-k, s/2) + F_{d-1}(k, s/2) + \frac{ds}{2}
\]
where:
- \( F_d(k, s/2) \) corresponds to \( \left\lfloor \frac{R_1}{S_1} \right\rfloor \),
- \( F_d(r-k, s/2) \) corresponds to \( \left\lfloor \frac{R_2}{S_2} \right\rfloor \),
- \( F_{d-1}(k, s/2) \) corresponds to \( \left\lfloor \frac{R_1}{S_2} \right\rfloor \),
- \( \frac{ds}{2} \) accounts for additional overhead.

\noindent {This recurrence is solved in the KLP paper, providing explicit bounds for the complexity of finding the maximal vectors.}

\section*{Exercises}

\begin{enumerate}

    \exercise{Finding Nondominated Set in 2-D}{%
        \begin{enumerate}
            \item[(a)] Explain the concept of dominance in a 2-dimensional space.
            \item[(b)] Given a set of points \( P = \{ (1,2), (3,1), (2,4), (5,2) \} \), determine the nondominated points.
            \item[(c)] Implement an algorithm to compute the nondominated set for an arbitrary 2D point set.
        \end{enumerate}
    }{ex:klp2dexercise}

    \exercise{Incremental Updates for Nondominated Sets}{%
        \begin{enumerate}
            \item[(a)] Suppose a new point is inserted into an existing nondominated set in 2-D. Describe an efficient method to update the nondominated set without recomputing from scratch.
            \item[(b)] Extend the method from (a) to 3-D and discuss how the complexity changes.
            \item[(c)] Consider a streaming scenario where points arrive one by one. How can efficient incremental updates be maintained for the nondominated set in both 2-D and 3-D?
        \end{enumerate}
    }{ex:incrementalupdates}

    \exercise{Merging Step in the N-D Divide-and-Conquer Scheme}{%
        \begin{enumerate}
            \item[(a)] Explain how the merging step works in the N-D divide-and-conquer maxima-finding algorithm. Why is it necessary to check dominance between the two partitions?
            \item[(b)] Suppose we have the following set of 8 points in 4-D:

            $P = \{
            (2, 5, 1, 7), (3, 2, 4, 6), (4, 3, 2, 8), (1, 6, 5, 2),$ \\
            $\quad (5, 4, 3, 1), (6, 1, 7, 3), (7, 8, 6, 5), (8, 7, 8, 4)
            \}$

            The points are sorted by the first coordinate. Partition the set into \( R \) and \( S \) as per the algorithm.

            \item[(c)] Compute the nondominated subsets \( \hat{R} \) and \( \hat{S} \), and then determine \( T \), the subset of \( \hat{S} \) that is not dominated by elements of \( \hat{R} \).
            
            \item[(d)] What is the final set of nondominated points after merging? Justify your answer.
        \end{enumerate}
    }{ex:ndmerging}
















\end{enumerate}


\chapter{Evolutionary Multiobjective Optimization}

Population-based metaheuristics, like Evolutionary Algorithms (EAs), optimize by evolving a set of candidate solutions, or a population, over several iterations. They approximate solutions to complex problems where classical methods fail due to nonlinearity, nonconvexity, or absence of gradient information. Although these algorithms provide high-quality upper bounds, they do not ensure optimality. Instead, they function as smart oracles, efficiently exploring the solution space using stochastic operators and leveraging computational and parallel resources to balance exploration and exploitation.

Evolutionary Algorithms (EAs) are stochastic optimization techniques inspired by biological evolution, including natural selection and genetics (see~\cite{DobzhanskyT70}). They traditionally include three subfields: Genetic Algorithms (GAs)\cite{GoldbergH88} (using binary representations), Evolution Strategies (ES)\cite{BeyerS02} (focused on continuous representations), and Evolutionary Programming (EP)\cite{FogelFF11} (supporting arbitrary representations). Over time, the lines between these subfields have blurred, and they are now collectively known as EAs\cite{BackS93}.

EAs are mainly used for optimization, especially in nonlinear, nonsmooth, and nonconvex problems where derivatives fail. Empirical studies show EAs can surpass classical methods~\cite{SchwefelHP95}. CMA-ES~\cite{HansenN2001} is a leading EA for continuous optimization.

Multi-objective evolutionary algorithms (MOEAs) aim to achieve well-distributed Pareto optimal solutions. Unlike single-objective optimization, where a single best solution exists, MOEAs require distinct selection schemes. Initially developed in the 1990s \cite{Kursawe90vector,FonsecaF93}, interest in MOEAs surged after Kalyanmoy Deb's book was published in 2001 \cite{Deb2001multi}. The main difference among MOEAs lies in their selection operators, while variation operators depend on the problem. NSGA-II \cite{KDeb_et_al_2002}, for instance, works for both continuous and combinatorial spaces, maintaining constant selection operators but requires adaptable variation operators for decision space representations.
There are currently three main paradigms for MOEA designs. These are:
\begin{enumerate}
    \item Pareto-based MOEAs use a two-tier ranking: Pareto dominance first, then diversity among equal-ranked points. Notable algorithms include NSGA-II \cite{KDeb_et_al_2002} and SPEA2 \cite{Zitzler2001spea}.
    \item Indicator-based MOEAs use performance indicators, like the hypervolume or R2 indicator, to guide selection and ranking of individuals.
    \item Decomposition-based MOEAs: The algorithm splits the problem into subproblems, each targeting different Pareto front sections. Each subproblem uses a distinct parameterization of a scalarization method. MOEA/D and NSGA-III are renowned methods here.
\end{enumerate}

In this tutorial, we will introduce typical algorithms for each of these paradigms. NSGA-II, SMS-EMOA, and MOEA/D. We will discuss important design choices and how and why other similar algorithms deviate in these choices.

\section{Pareto Based Algorithms: NSGA-II}
The basic loop of NSGA-II is given by Algorithm \ref{alg:nsga2}.
\begin{algorithm}
	\caption{\label{alg:nsga2}NSGA-II Algorithm}
	\begin{algorithmic}[1]
		\STATE initialize population $P_0 \subset \mathcal{X}^\mu$
		\WHILE{not terminate}
        \STATE \COMMENT {Begin variate}
		\STATE $Q_t \leftarrow \emptyset$
		\FORALL{$i \in \{1, \dots, \mu \}$}
		\STATE $(\mathbf{x}^{(1)}, \mathbf{x}^{(2)}) \leftarrow \mbox{select\_mates}(P_t)$ \COMMENT{select two parent individuals $\mathbf{x}^{(1)} \in P_t$ and $\mathbf{x}^{(2)} \in P_t$}
		\STATE $\mathbf{r}^{(i)}_t \leftarrow \mbox{recombine}(\mathbf{x}^{(1)}, \mathbf{x}^{(2)})$
		\STATE $\mathbf{q}^{(i)}_t \leftarrow \mbox{mutate}(\mathbf{r})$
		\STATE $Q_t \leftarrow Q_t \cup \{\mathbf{q}^{(i)}_t\}$
		\ENDFOR
        \STATE \COMMENT {End variate}
        \STATE \COMMENT {Selection step, select $\mu$-"best" out of $(P_t \cup Q_t)$ by a two step procedure:}
        \STATE  $(R_1, ..., R_\ell) \leftarrow \mbox{ non-dom\_sort}(\mathbf{f}, P_t \cup Q_t)$
        \STATE  Find the element of the partition,  $R_{i_{\mu}}$,  for which the sum of the cardinalities $|R_1|+\cdots+|R_{i_{\mu}}|$ is for the first time $\geq \mu$. If $|R_1|+\cdots+|R_{i_{\mu}}|=\mu$, $P_{t+1} \leftarrow \cup_{i=1}^{i_{\mu}}R_i$, otherwise determine set $H$ containing  $\mu-(|R_1|+\cdots+|R_{i_{\mu}-1}|)$ elements from $R_{i_{\mu}}$ with the highest crowding distance and $P_{t+1} \leftarrow (\cup_{i=1}^{i_{\mu}-1}R_i)\cup H$.
		\STATE \COMMENT {End of selection step.}
        \STATE $t \leftarrow t+1$
		\ENDWHILE
		\RETURN{$P_t$}
	\end{algorithmic}
\end{algorithm}

Initially, a population of points is created. The following process \emph{generational loop} is repeated: the population first varies, then a selection forms the new generation. This loop continues until a termination criterion is met, such as convergence (cf. \cite{Wagner2009}) or a computational budget limit. During the variation phase $\lambda$, offspring are generated by binary tournament selection, where two individuals from $P_t$ are chosen and the better-ranked is selected. Parents undergo standard recombination; for real-valued problems, SBX \emph{simulated binary crossover}\cite{Deb2001multi} is used with \emph{polynomial mutation (PM)} \cite{Deb2001multi}, producing $\lambda$ individuals from modifications or combinations of $P_t$. Finally, parents and offspring are merged into $P_t \cup Q_t$.
In the second part, the selection part, the $\mu$ best individuals of $P_t \cup Q_t$ with respect to a multiobjective ranking are selected as the new population $P_{t+1}$.

We explain the multiobjective ranking used in NSGA-II, which differentiates it from single-objective genetic algorithms. The process involves two levels: first, non-dominated sorting based on the Pareto order, and second, ranking individuals with the same initial rank using the crowding distance criterion to reflect diversity.

Let $\textrm{ND}(P)$ denote the non-dominated solutions in some population.
Non-dominated sorting partitions the populations into subsets (layers) based on Pareto non-dominance and it can
be specified through recursion as follows.
\begin{eqnarray}
	R_1&=&\mathrm{ND}(P)\\
	R_{k+1} &=& \mathrm{ND}(P \setminus \cup_{i=1}^k R_i), k = 1, 2, \dots
\end{eqnarray}
As in each step of the recursion at least one solution is removed from the population, the maximal number of layers is $|P|$. We will use the index $\ell$ to denote the highest non-empty layer. The rank of the solution after non-dominated sorting is given by the subindex $k$ of $R_k$.
It is clear that solutions in the same layer are mutually incomparable. The non-dominated sorting procedure is illustrated in Figure \ref{fig:crowd} (upper left). The solutions are ranked as follows  $R_1 = \{\mathbf{y}^{(1)}, \mathbf{y}^{(2)}, \mathbf{y}^{(3)}, \mathbf{y}^{(4)}\}$, $R_2 = \{\mathbf{y}^{(5)}, \mathbf{y}^{(6)}, \mathbf{y}^{(7)}\}$, $R_3 = \{\mathbf{y}^{(8)}, \mathbf{y}^{(9)}\}$.

Now, if there is more than one solution in a layer, say $R$, a secondary ranking procedure is used to rank solutions within that layer. This procedure applies the crowding distance criterion. The crowding distance of a solution $\mathbf{x} \in R$ is computed by a sum over contributions $c_i$ of the $i$-th objective function:
\begin{eqnarray}
	l_i(\mathbf{x}) &:=& \max( \{f_i(\mathbf{y}) | \mathbf{y} \in R \setminus \{\mathbf{x}\} \wedge f_i(\mathbf{y}) \leq f_i(\mathbf{x})  \}  \cup \{ -\infty\})\\
	u_i(\mathbf{x}) &:=& \min( \{f_i(\mathbf{y}) | \mathbf{y} \in R \setminus \{\mathbf{x}\} \wedge f_i(\mathbf{y}) \geq f_i(\mathbf{x})\}\cup\{ \infty\})\\
	c_i(\mathbf{x}) &:=& u_i - l_i, \quad i=1, \dots, m
\end{eqnarray}
The crowding distance is now given as:
\begin{equation}
	c(\mathbf{x}) := \frac{1}{m}\sum_{i=1}^m c_i(\mathbf{x}),  \mathbf{x}\in R
\end{equation}
\begin{figure}
	\begin{center}
		\begin{tikzpicture}[scale=0.8]
			\draw[thick, color=black!22, dotted, step=1cm] (-1,-1) grid (5,4);
			\draw[thick, ->] (-1,0) -- (5.5,0) node[right] {$f_1$};
			\draw[thick, ->] (0,-1) -- (0,4.25) node [above] {$f_2$};
			\foreach \i in {1, 2, 3} {
				\node[anchor=north] at (\i, 0) {$\i$};
				\node[anchor=east] at (0, \i) {$\i$};
			}
			\node[anchor=north east] at (0,0) {$(0,0)$};
						
			\fill[red] (1,3.5) circle (4pt) node[left] {$\mathbf{y}^{(1)}$};
						
			\fill[red] (2,2) circle (4pt) node[left] {$\mathbf{y}^{(2)}$};
						
			\fill[red] (2.5,1.5) circle (4pt) node[left] {$\mathbf{y}^{(3)}$};
						
			\fill[red] (3.5,1) circle (4pt) node[left] {$\mathbf{y}^{(4)}$};	
						   	
			\fill[teal] (2,3.5) circle (4pt) node[left] {$\mathbf{y}^{(5)}$};
						
			\fill[teal] (1+2,1+2) circle (4pt) node[left] {$\mathbf{y}^{(6)}$};
						
			\fill[teal] (4.5,1.5) circle (4pt) node[left] {$\mathbf{y}^{(7)}$};	
						   	
			\fill[blue] (5,3) circle (4pt) node[left] {$\mathbf{y}^{(8)}$};
						
			\fill[blue] (4,4) circle (4pt) node[left] {$\mathbf{y}^{(9)}$};

		\end{tikzpicture}
		\begin{tikzpicture}[scale=0.8]
			\draw[thick, color=black!22, dotted, step=1cm] (-1,-1) grid (5,4);
			\draw[thick, ->] (-1,0) -- (5.5,0) node[right] {$f_1$};
			\draw[thick, ->] (0,-1) -- (0,4.25) node [above] {$f_2$};
			\foreach \i in {1, 2, 3} {
				\node[anchor=north] at (\i, 0) {$\i$};
				\node[anchor=east] at (0, \i) {$\i$};
			}
			\node[anchor=north east] at (0,0) {$(0,0)$};
						
			\draw [opacity=0.3, dashed, ultra thick, red]
			(0,2) -- (0,4) -- (2,4);
			\draw [opacity=0.3, ultra thick, red]
			(2,4) -- (2,2) -- (0,2);
			\fill[red] (1,3.5) circle (4pt) node[left] {$\mathbf{y}^{(1)}$};

			\draw [opacity=0.3, ultra thick, orange]
			(1,1.5) -- (1,3.5) -- (2.5,3.5) -- (2.5,1.5) -- cycle;
			\fill[orange] (2,2) circle (4pt) node[left] {$\mathbf{y}^{(2)}$};

			\draw [opacity=0.3, ultra thick, teal]
			(2,1) -- (2,2) -- (3.5,2) -- (3.5,1) -- cycle;
			\fill[teal] (2.5,1.5) circle (4pt) node[left] {$\mathbf{y}^{(3)}$};
						
			\draw [opacity=0.3, ultra thick, blue]
			(2.5,0.5) -- (2.5,1.5) -- (4,1.5) -- (4,0.5) -- cycle;
			\fill[blue] (3.5,1) circle (4pt) node[left] {$\mathbf{y}^{(4)}$};	
						   	
			\draw [opacity=0.3, ultra thick, violet]
			(3.5,0) -- (3.5,1) -- (5,1);
			\draw [opacity=0.3, ultra thick, dashed, violet]
			(5,1) -- (5,0) -- (3.5,0);
			\fill[violet] (4,0.5) circle (4pt) node[left] {$\mathbf{y}^{(5)}$};	
		\end{tikzpicture}
	\end{center}
	\caption{\label{fig:crowd} Illustration of non-dominated sorting (left) and crowding distance (right).}
\end{figure}

For $m=2$ the crowding distances of a set of mutually nondominated points are illustrated in Figure \ref{fig:crowd} (upper right). In this particular case, they are proportional to the perimeter of a rectangle that just intersects the neighboring points (up to a factor of $\frac{1}{4}$).
Practically speaking, the value of $l_i$ is determined by the nearest neighbor of $\mathbf{x}$ to the left \emph{according to the $i$-coordinate}, and $l_i$ is equal to the $i$-th coordinate of this nearest neighbor, similarly the value of $u_i$  is determined by the nearest neighbor of $\mathbf{x}$ to the right \emph{according to the $i$-coordinate} , and $u_i$ is equal to the $i$-th coordinate of this right nearest neighbor.
The more space there is around a solution, the greater the crowding distance. Therefore, solutions with a high crowding distance should be ranked better than those with a low crowding distance to maintain diversity in the population.
This way we establish a second-order ranking. If the crowding distance is the same for two points, then it is randomly decided which point is ranked higher.

Now we explain the non-dom-sort procedure in line 13 of Algorithm 1 the role of $P$ is taken over by $P_t \cap Q_t$: In order to select the $\mu$ best members of $P_t \cup Q_t$ according to the two-level ranking described above, we proceed as follows. Create the partition $R_1, R_2, \cdots, R_{\ell}$ of $P_t \cup Q_t$ as described above. For this partition, one finds the first index $i_{\mu}$ for which the sum of the cardinalities $|R_1|+\cdots+|R_{i_{\mu}}|$ is for the first time $\geq \mu$. If $|R_1|+\cdots+|R_{i_{\mu}}|=\mu$, then set $P_{t+1}$ to $\cup_{i=1}^{i_{\mu}}R_i$, otherwise determine the set $H$ containing  $\mu-(|R_1|+\cdots+|R_{i_{\mu}-1}|)$ elements from $R_{i_{\mu}}$ with the highest crowding distance and set the next generation-population,  $P_{t+1}$, to $(\cup_{i=1}^{i_{\mu}-1}R_i)\cup H$.

Pareto-based Algorithms are probably the largest class of MOEAs. They have in common that they combine a ranking criterion based on Pareto dominance with a diversity based secondary ranking. Other common algorithms that belong to this class are as follows. The Multiobjective Genetic Algorithm (MOGA) \cite{FonsecaF93}, which was one of the first MOEAs. The PAES \cite{knowles2000paes}, which uses a grid partitioning of the objective space in order to make sure that certain regions of the objective space do not get too crowded. Within a single grid cell, only one solution is selected. The Strength Pareto Evolutionary Algorithm (SPEA2) \cite{Zitzler2001spea} uses a different criterion for ranking based on Pareto dominance. The strength of an individual depends on how many other individuals it dominates and by how many other individuals it dominates. In addition, clustering serves as a secondary ranking criterion. Both operators have been refined in SPEA2 \cite{Zitzler2001spea}, and also features a strategy to maintain an archive of non-dominated solutions. The Multiobjective Micro GA \cite{coelloCoello2001microGA} uses a small population with an archive. The Differential Evolution Multiobjective Optimization (DEMO) algorithm \cite{teaTusar2005DEMO} merges Pareto-based MOEAs with a differential evolution operator for enhanced efficiency and precision, particularly in continuous problems.

\section{Indicator-based Algorithms: SMS-EMOA}
A second algorithm that we will discuss is a classical algorithm following the paradigm of indicator-based multi-objective optimization. In the context of MOEAs, by a performance indicator (or just indicator), we denote a scalar measure of the quality of a Pareto front approximation. Indicators can be \emph{unary}, which means that they yield an absolute measure of the quality of a Pareto front approximation. They are called \emph{binary}, whenever they measure how much better one Pareto front approximation is compared to another Pareto front approximation.

 SMS-EMOA \cite{EmmerichBN05} uses the hypervolume indicator as a performance indicator. Theoretical analysis attests that this indicator has some favorable properties, as the maximization of it yields approximations of the Pareto front with points located on the Pareto front and well distributed across the Pareto front. The hypervolume indicator measures the size of the dominated space, bound from above by a reference point.

For an approximation set $A \subset \mathbb{R}^m$ it is defined as follows:
\begin{equation}
	\textrm{HI(A)} = \mathrm{Vol}(\{\mathbf{y} \in \mathbb{R}^m: \mathbf{y} \preceq_{Pareto} \mathbf{r} \wedge \exists \mathbf{a} \in A: \mathbf{a} \preceq_{Pareto} \mathbf{y}  \})
\end{equation}
Here, $\mbox{Vol}(.)$ denotes the Lebesgue measure of a set in dimension $m$. This is length for $m=1$, area for $m=2$, volume for $m=3$, and hypervolume for $m \geq 4$. Practically speaking, the hypervolume indicator of $A$ measures the size of the space that is dominated by $A$. The closer points move to the Pareto front, and the more they distribute along the Pareto front, the more space gets dominated. As the size of the dominated space is infinite, it is necessary to bound it. For this reason, the reference point $\mathbf{r}$ is introduced.

The SMS-EMOA seeks to maximize the hypervolume indicator of a population which serves as an approximation set. This is achieved by considering the contribution of points to the hypervolume indicator in the selection procedure. Algorithm \ref{alg:sms} describes the basic loop of the standard implementation of the SMS-EMOA.
\begin{algorithm}
	\caption{\label{alg:sms} SMS-EMOA}
	\begin{algorithmic}
		\STATE initialize $P_0 \subset \mathcal{X}^\mu$
		\WHILE{not terminate}
        \STATE \COMMENT{Begin variate}
		\STATE $(\mathbf{x}^{(1)}, \mathbf{x}^{(2)}) \leftarrow \mbox{select\_mates}(P_t)$
        \COMMENT{select two parent individuals $\mathbf{x}^{(1)} \in P_t$ and $\mathbf{x}^{(2)} \in P_t$}
		\STATE $\mathbf{c}_t \leftarrow \mbox{recombine}(\mathbf{x}^{(1)}, \mathbf{x}^{(2)})$
		\STATE $\mathbf{q}_t \leftarrow \mbox{mutate}(\mathbf{c}_t)$
        \STATE \COMMENT{End variate}
        \STATE \COMMENT{Begin selection}
		\STATE  $P_{t+1} \leftarrow \mbox{select}_\mathbf{f}(P_t \cup \{\mathbf{q}_t\})$ \COMMENT{Select subset of size $\mu$ with maximal hypervolume indicator from $P \cup \{\mathbf{q}_t\}$}
        \STATE \COMMENT{End selection}
		\STATE $t \leftarrow t+1$
		\ENDWHILE
		\RETURN{$P_t$}
	\end{algorithmic}
\end{algorithm}

The algorithm starts with the initialization of a population in the search space. Then it creates only \emph{one} offspring individual by recombination and mutation. This new individual enters the population, which has now size $\mu+1$. To reduce the population size again to the size of $\mu$, a subset of size $\mu$ with maximal hypervolume is selected. This way as long as the reference point for computing the hypervolume remains unchanged, the hypervolume indicator of $P_t$ can only grow or stay equal with an increasing number of iterations $t$.

Next, the details of the selection procedure will be discussed.
If all solutions in $P_t$ are non-dominated, the selection of a subset of maximal hypervolume is equivalent to deleting the point with the smallest (exclusive) hypervolume contribution. The hypervolume contribution is defined as:

$$\Delta \mathrm{HI}(\mathbf{y}, Y) = \mathrm{HI}(Y) - \mathrm{HI}(Y \setminus \{\mathbf{y}\})$$
\begin{figure}[t]
	\begin{center}
		\begin{tikzpicture}[scale=0.8]
			\draw[thick, color=black!22, dotted, step=1cm] (-1,-1) grid (5,4);
			\draw[thick, ->] (-1,0) -- (5.5,0) node[right] {$f_1$};
			\draw[thick, ->] (0,-1) -- (0,4.25) node [above] {$f_2$};
			\foreach \i in {1, 2, 3} {
				\node[anchor=north] at (\i, 0) {$\i$};
				\node[anchor=east] at (0, \i) {$\i$};
			}
			\node[anchor=north east] at (0,0) {$(0,0)$};
						
			\fill [opacity=0.3, gray]
			(1,3.5) -- (1,4) -- (5,4) -- (5,3.5) -- cycle;
			\fill[red] (1,3.5) circle (4pt) node[left] {$\mathbf{y}^{(1)}$};
						
			\fill [opacity=0.3, gray]
			(2,2) -- (2,3.5) -- (5,3.5) -- (5,2) -- cycle;
			\fill[orange] (2,2) circle (4pt) node[left] {$\mathbf{y}^{(2)}$};
						
			\fill [opacity=0.3, gray] (2.5,1.5) -- (2.5,2) -- (5,2) -- (5,1.5) -- cycle;
			\fill[teal] (2.5,1.5) circle (4pt) node[left] {$\mathbf{y}^{(3)}$};

			\fill [opacity=0.3, gray] (3.5,1) -- (3.5,1.5) -- (5,1.5) -- (5,1) -- cycle;
			\fill[blue] (3.5,1) circle (4pt) node[left] {$\mathbf{y}^{(4)}$};	
						   	
			\fill [opacity=0.3, gray] (4,0.5) -- (4,1) -- (5,1) -- (5,0.5) -- cycle;
			\fill[violet] (4,0.5) circle (4pt) node[left] {$\mathbf{y}^{(5)}$};	
						
			\fill (5,4) circle (4pt) node[above] {$r$};
			\node at (3,3) {$Y\oplus \mathbb{R}^2_{\succ \mathbf{0}}$};

		\end{tikzpicture}
		\begin{tikzpicture}[scale=0.8]
			\draw[thick, color=black!22, dotted, step=1cm] (-1,-1) grid (5,4);
			\draw[thick, ->] (-1,0) -- (5.5,0) node[right] {$f_1$};
			\draw[thick, ->] (0,-1) -- (0,4.25) node [above] {$f_2$};
			\foreach \i in {1, 2, 3} {
				\node[anchor=north] at (\i, 0) {$\i$};
				\node[anchor=east] at (0, \i) {$\i$};
			}
			\node[anchor=north east] at (0,0) {$(0,0)$};
						
			\fill [opacity=0.3, gray]
			(1,3.5) -- (1,4) -- (5,4) -- (5,3.5) -- cycle;
			\fill [opacity=0.3, red]
			(1,3.5) -- (1,4) -- (2,4) -- (2,3.5) -- cycle;
			\fill[red] (1,3.5) circle (4pt) node[left] {$\mathbf{y}^{(1)}$};
						
			\fill [opacity=0.3, gray]
			(2,2) -- (2,3.5) -- (5,3.5) -- (5,2) -- cycle;
			\fill [opacity=0.3, orange]
			(2,2) -- (2,3.5) -- (2.5,3.5) -- (2.5,2) -- cycle;
			\fill[orange] (2,2) circle (4pt) node[left] {$\mathbf{y}^{(2)}$};
						
			\fill [opacity=0.3, gray] (2.5,1.5) -- (2.5,2) -- (5,2) -- (5,1.5) -- cycle;
			\fill [opacity=0.3, teal] (2.5,1.5) -- (2.5,2) -- (3.5,2) -- (3.5,1.5) -- cycle;
			\fill[teal] (2.5,1.5) circle (4pt) node[left] {$\mathbf{y}^{(3)}$};
						
			\fill [opacity=0.3, gray] (3.5,1) -- (3.5,1.5) -- (5,1.5) -- (5,1) -- cycle;
			\fill [opacity=0.3, blue] (3.5,1) -- (3.5,1.5) -- (4,1.5) -- (4,1) -- cycle;
			\fill[blue] (3.5,1) circle (4pt) node[left] {$\mathbf{y}^{(4)}$};	
						   	
			\fill [opacity=0.3, gray] (4,0.5) -- (4,1) -- (5,1) -- (5,0.5) -- cycle;
			\fill [opacity=0.3, violet] (4,0.5) -- (4,1) -- (5,1) -- (5,0.5) -- cycle;
			\fill[violet] (4,0.5) circle (4pt) node[left] {$\mathbf{y}^{(5)}$};

			\fill (5,4) circle (4pt) node[above] {$r$};			
		\end{tikzpicture}
	\end{center}			
	\newcommand{\drawcuboid}[7]
	{
					
		\pgfmathsetmacro{\xl}{#1}
		\pgfmathsetmacro{\xu}{#4}
		
		\pgfmathsetmacro{\yl}{#2}
		\pgfmathsetmacro{\yu}{#5}
		
		\pgfmathsetmacro{\zl}{#3}
		\pgfmathsetmacro{\zu}{#6}
									
		\coordinate (orig) at (3,1.5);
						
		\coordinate (cooo) at ($(orig)+({\xl*(1)},{\xl*(-.5)})+({\yl*(-1)},{\yl*(-.5)})+({\zl*(0)},{\zl*(1)})$);
		\coordinate (cxoo) at ($(orig)+({\xu*(1)},{\xu*(-.5)})+({\yl*(-1)},{\yl*(-.5)})+({\zl*(0)},{\zl*(1)})$);
		\coordinate (cxyo) at ($(orig)+({\xu*(1)},{\xu*(-.5)})+({\yu*(-1)},{\yu*(-.5)})+({\zl*(0)},{\zl*(1)})$);
		\coordinate (coyo) at ($(orig)+({\xl*(1)},{\xl*(-.5)})+({\yu*(-1)},{\yu*(-.5)})+({\zl*(0)},{\zl*(1)})$);
		\coordinate (cooz) at ($(orig)+({\xl*(1)},{\xl*(-.5)})+({\yl*(-1)},{\yl*(-.5)})+({\zu*(0)},{\zu*(1)})$);
		\coordinate (cxoz) at ($(orig)+({\xu*(1)},{\xu*(-.5)})+({\yl*(-1)},{\yl*(-.5)})+({\zu*(0)},{\zu*(1)})$);
		\coordinate (cxyz) at ($(orig)+({\xu*(1)},{\xu*(-.5)})+({\yu*(-1)},{\yu*(-.5)})+({\zu*(0)},{\zu*(1)})$);
		\coordinate (coyz) at ($(orig)+({\xl*(1)},{\xl*(-.5)})+({\yu*(-1)},{\yu*(-.5)})+({\zu*(0)},{\zu*(1)})$);								
		draw								
		\fill[#7!10]
		(cooz) --
		(cxoz) --
		(cxyz) --
		(coyz) -- cycle;
						
		\fill[#7!20]
		(cxoo) --
		(cxoz) --
		(cxyz) --
		(cxyo) -- cycle;
						
		\fill[#7]
		(coyo) --
		(cxyo) --
		(cxyz) --
		(coyz) -- cycle;
						
	}
        \begin{center}
	\begin{tikzpicture}[scale=0.75]
		\coordinate (orig) at (0,0);
		\draw[thick, color=black!22, dotted, step=0.5cm] (-1,-2) grid (7,4);
		\draw[thick, ->] (0,0) -- (0,2) node[above] {$f_3$};
		\draw[thick, ->] (0,0) -- (3,-1.5) node[midway,below] {$f_1$};
		\draw[thick, ->] (6,0) -- (3,-1.5) node[midway,below] {$f_2$};
												
		\draw[thick, dashed] (3,1.5) -- (0,0) ;
		\draw[thick, dashed] (3,1.5) -- (6,0)  ;
			
		\draw[thick, dashed] (6,0) -- (6,2)  ;
					
		\drawcuboid{0}{0}{0}{1}{1}{1}{gray}
		\drawcuboid{0}{0}{1}{1}{1}{2}{gray}
		\drawcuboid{0}{1}{0}{1}{2}{1}{gray}
		\drawcuboid{1}{0}{0}{2}{0.5}{0.5}{gray}
		\drawcuboid{1}{0.5}{0}{2}{1.5}{0.5}{gray}
		\drawcuboid{2}{0}{0}{2.5}{0.5}{0.5}{gray}
		\drawcuboid{1}{1.5}{0} {1.5}{2}{0.25}{gray}
		\drawcuboid{0}{2}{0}{1.5}{2.5}{0.25}{gray}
		  		
		\draw[thin, dashed] (3,3.5) -- (0,2) ;
		\draw[thin, dashed] (3,3.5) -- (6,2)  ;
		\draw[thin, dashed] (0,2) -- (3,0.5) ;
		\draw[thin, dashed] (6,2) -- (3,0.5)  ;	
		\draw[thin, dashed] (3,-1.5) -- (3,0.5) ;

	      \fill[color=red]
	      	(3.5,0.25) circle (3pt) node[above] {$\mathbf{y}^{(1)}$};

	      \fill[color=orange]
	      	(5,0.5) circle (3pt) node[above] {$\mathbf{y}^{(2)}$};
		      	
	      \fill[color=teal]
	      	(2,1) circle (3pt) node[above] {$\mathbf{y}^{(3)}$};

		\fill[color=blue]
	      	(3,2.5) circle (3pt) node[above] {$\mathbf{y}^{(4)}$};

		\fill[color=violet]
			(2,-0.25) circle (3pt) node[above] {$\mathbf{y}^{(5)}$};

	\end{tikzpicture}
	\begin{tikzpicture}[scale=0.75]
		\coordinate (orig) at (0,0);
		\draw[thick, color=black!22, dotted, step=0.5cm] (-1,-2) grid (7,4);
		\draw[thick, ->] (0,0) -- (0,2) node[above] {$f_3$};
		\draw[thick, ->] (0,0) -- (3,-1.5) node[midway,below] {$f_1$};
		\draw[thick, ->] (6,0) -- (3,-1.5) node[midway,below] {$f_2$};
												
		\draw[thick, dashed] (3,1.5) -- (0,0) ;
		\draw[thick, dashed] (3,1.5) -- (6,0)  ;
			
		\draw[thick, dashed] (6,0) -- (6,2)  ;
					
		\drawcuboid{0}{0}{0}{1}{1}{1}{gray}
		\drawcuboid{0}{0}{1}{1}{1}{2}{blue}
		\drawcuboid{0}{1}{0}{1}{2}{1}{teal}
		\drawcuboid{1}{0}{0}{2}{0.5}{0.5}{gray}
		\drawcuboid{1}{0.5}{0}{2}{1.5}{0.5}{red}
		\drawcuboid{2}{0}{0}{2.5}{0.5}{0.5}{orange}
		\drawcuboid{1}{1.5}{0} {1.5}{2}{0.25}{violet}
		\drawcuboid{0}{2}{0}{1.5}{2.5}{0.25}{violet}
		  		
		\draw[thin, dashed] (3,3.5) -- (0,2) ;
		\draw[thin, dashed] (3,3.5) -- (6,2)  ;
		\draw[thin, dashed] (0,2) -- (3,0.5) ;
		\draw[thin, dashed] (6,2) -- (3,0.5)  ;	
		\draw[thin, dashed] (3,-1.5) -- (3,0.5);

	      \fill[color=red]
	      	(3.5,0.25) circle (3pt) node[above] {$\mathbf{y}^{(1)}$};

	      \fill[color=orange]
	      	(5,0.5) circle (3pt) node[above] {$\mathbf{y}^{(2)}$};
		      	
	      \fill[color=teal]
	      	(2,1) circle (3pt) node[above] {$\mathbf{y}^{(3)}$};

		\fill[color=blue]
	      	(3,2.5) circle (3pt) node[above] {$\mathbf{y}^{(4)}$};

		\fill[color=violet]
			(2,-0.25) circle (3pt) node[above] {$\mathbf{y}^{(5)}$};
		
	\end{tikzpicture}
    \end{center}
\caption{\label{fig:hyperhyper} Illustration of 2-D hypervolume (top left), 2-d hypervolume contributions (top right), 3-D hypervolume (bottom left), and 3-D hypervolume contributions (bottom right).}
\end{figure}
An illustration of the hypervolume indicator and hypervolume contributions for $m=2$ and, respectively, $m=3$ is given in Figure \ref{fig:hyperhyper}.
Efficient computation of all hypervolume contributions of a population can be achieved in time $\Theta(\mu \log \mu)$ for $m=2$ and $m=3$ \cite{Emmerich2011}. For $m=3 \mbox{ or } 4 $, fast implementations are described in \cite{Guerreiro17}. Moreover, for fast logarithmic-time incremental updates for 2-D, see  \cite{Hupkens13}. For achieving logarithmic time updates in SMS-EMOA, the non-dominated sorting procedure was replaced by a procedure, that sorts dominated solutions based on age. For more than two dimensions, fast incremental updates of the hypervolume indicator and its contributions were proposed in for more than two dimensions \cite{Guerreiro17}. 
In case dominated solutions appear the standard implementation of SMS-EMOA partitions the population into layers of equal dominance ranks, just like in NSGA-II. Subsequently, the solution with the smallest hypervolume contribution on the worst ranked layer gets discarded.

SMS-EMOA typically converges to regularly spaced Pareto front approximations. The density of these approximations depends on the local curvature of the Pareto front. For biobjective problems, it is highest at points where the slope is equal to $-45^{\circ}$\cite{Auger2009}. It is possible to influence the distribution of the points in the approximation set by using a generalized cone-based hypervolume indicator. These indicators measure the hypervolume dominated by a cone-order of a given cone, and the resulting optimal distribution becomes more uniform if the cones are acute and more concentrated when using obtuse cones (see \cite{Emmerich13cone}).

Besides the SMS-EMOA, there are a couple of other indicator-based MOEAs. Some of them also use the hypervolume indicator. The original idea to use the hypervolume indicator in an MOEA was proposed in the context of archiving methods for non-dominated points. Here the hypervolume indicator was used for keeping a bounded-size archive \cite{Knowles03}.  The term Indicator-based Evolutionary Algorithms (IBEA) \cite{Zitzler2004IBEA} was introduced in a paper that proposed an algorithm design, in which the choice of indicators is generic. The hypervolume-based IBEA was discussed as one instance of this class. Its design is however different to SMS-EMOA and makes no specific use of the characteristics of the hypervolume indicator. The Hypervolume Estimation Algorithm (HypE) \cite{BaderZ11} uses a Monte Carlo Estimation for the hypervolume in high dimensions and thus it can be used for optimization with a high number of objectives (so-called many-objective optimization problems).
MO-CMA-ES \cite{Igel2006mooCMA} is another hypervolume-based MOEA. It uses the covariance-matrix adaptation in its mutation operator, which enables it to adapt its mutation distribution to the local curvature and scaling of the objective functions.
Although the hypervolume indicator has been very prominent in IBEAs, there are some algorithms using other indicators, notably this is the R2 indicator \cite{Trautmann13}, which features an ideal point as a reference point, and the averaged Hausdorff distance ($\Delta_p$ indicator) \cite{Rudolph16Hausdorff}, which requires an aspiration set or estimation of the Pareto front which is dynamically updated and used as a reference. The idea of aspiration sets for indicators that require knowledge of the 'true' Pareto front also occurred in conjunction with the $\alpha$-indicator\cite{bringmann2013}, which generalizes the approximation ratio in numerical single-objective optimization. The Portfolio Selection Multiobjective Optimization Algorithm (POSEA) \cite{Yevseyeva14}  uses the Sharpe Index from financial portfolio theory as an indicator, which applies the hypervolume indicator of singletons as a utility function and a definition of the covariances based on their overlap. The Sharpe index combines the cumulated performance of single individuals with the covariance information (related to diversity), and it has interesting theoretical properties.

\section{Decomposition-based Algorithm: MOEA/D}
Decomposition-based algorithms split the problem into subproblems using scalarizations defined by different weights. These subproblems are solved simultaneously by dynamically assigning points and exchanging solutions with neighboring subproblems. The method establishes neighborhoods based on distances between aggregation coefficient vectors. Exchanging information with neighbors enhances search efficiency compared to solving subproblems independently. MOEA/D \cite{zhang2007moeaD} is a widely used decomposition method that builds on previous algorithms, integrating decomposition, scalarization, and local search.
%

MOEA/D typically uses Chebychev scalarizations, but the authors also propose linear weighting, which struggles with non-convex Pareto fronts, and boundary intersection methods, which need extra parameters and may produce dominated points.

MOEA/D evolves a population of individuals, each individual $\mathbf{x}^{(i)}\in P_t$ being associated with a weight vector $\mathbf{\lambda}^{(i)}$.  The weight vectors $\mathbf{\lambda}^{(i)}$ are evenly distributed in the search space, e.g., for two objectives a possible choice is:
$\lambda^{(i)} = (\frac{\lambda-i}{\lambda},\frac{i}{\lambda})^{\top},\, i\, =\, 0, ...,\mu$.

The $i$-th subproblem $g(\mathbf{x}| \mathbf{\lambda}^{i}, \mathbf{z}^*)$ is defined by the Chebychev scalarization function:
\begin{equation}
g(\mathbf{x} | \mathbf{\lambda}^{(i)}, \mathbf{z}^*) = \max_{j \in \{1, ..., m\}}\{
                                \mathbf{\lambda}^{(i)}_j |f_j(\mathbf{x}) - z^{*}_j|\} + \epsilon \sum_{j=1}^m\left( f_j(\mathbf{x}) - z^{*}_j \right)
\end{equation}
To create a new candidate solution for the $i$-th individual, consider its neighbors—incumbent solutions of subproblems with similar weight vectors. The $i$-th individual's neighborhood includes $k$ closest subproblems, based on Euclidean distance, including the $i$-th subproblem, and is denoted by $B(i)$. With these, the MOEA/D algorithm using Chebychev scalarization is detailed in Algorithm \ref{alg:moead}.

\begin{algorithm}
    \caption{\label{alg:moead} MOEA/D}
    \begin{algorithmic}
        \STATE input: $\Lambda = \{\mathbf{\lambda}^{(1)}, ..., \mathbf{\lambda}^{(\mu)}\}$ \COMMENT{weight vectors}
        \STATE input: $\mathbf{z}^*$: reference point for Chebychev distance
        \STATE initialize $P_0 \subset \mathcal{X}^\mu$
        \STATE initialize neighborhoods $B(i)$ by collecting $k$ nearest weight vectors in $\Lambda$ for each $\mathbf{\lambda}^{(i)}$
        \WHILE{not terminate}
        \FORALL{$i \in \{1, ..., \mu\}$}
          \STATE Select randomly two solutions $\mathbf{x}^{(1)}$, $\mathbf{x}^{(2)}$ in the neighborhood $B(i)$.
          \STATE $\mathbf{y} \leftarrow$ Recombine $\mathbf{x}^{(1)}$, $\mathbf{x}^{(2)}$ by a problem specific recombination operator.
          \STATE $\mathbf{y}' \leftarrow$ Local problem specific, heuristic improvement of $\mathbf{y}$, e.g. local search, based on the scalarized objective function $g(\mathbf{x} | \mathbf{\lambda}^{(i)}, \mathbf{z}^*)$ .
          \IF{$g(\mathbf{y}' | \mathbf{\lambda}^{(i)}, \mathbf{z}^*) < g(\mathbf{x}^{(i)} | \mathbf{\lambda}^{(i)}, \mathbf{z}^*)$}
            \STATE $\mathbf{x}^{(i)} \leftarrow \mathbf{y}'$
          \ENDIF
          \STATE Update $\mathbf{z}^*$, if neccessary, i.e, one of its component is larger than one of the corresponding components of $\mathbf{f}(\mathbf{x}^{(i)})$.
        \ENDFOR
        \STATE $t \leftarrow t+1$
        \ENDWHILE
        \RETURN{$P_t$}
    \end{algorithmic}
\end{algorithm}

Consider these two points about MOEA/D:
(1) The algorithm retains generic elements like recombination, implemented with standard genetic algorithm operators, and a local improvement heuristic. This heuristic seeks nearby solutions that meet constraints and perform well on objective functions.
(2) MOEA/D includes a feature to collect non-dominated solutions in an external archive during a run. Since this archive only affects the final output, it is a general feature applicable to EMOAs and could be similarly used in SMS-EMOA and NSGA-II. To simplify comparisons to these algorithms, the archiving strategy was excluded from the description.

Decomposition-based MOEAs have gained popularity due to their scalability for multi-objective problems. The NSGA-III \cite{Deb13nsgaiii} algorithm, designed for many-objective optimization, uses dynamically updated reference points. Generalized Decomposition \cite{GPF13} employs a mathematical programming solver for updates, showing good performance on continuous problems. Novel hybrid techniques like Directed Search \cite{schuetze2016directedSearch} combine mathematical programming and decomposition, leveraging the Jacobian matrix to explore promising directions in the search space.
\section{Performance Assessment}
To evaluate multiobjective evolutionary or deterministic optimizers, consider (1) computational resources used and (2) result quality.

Computation time, often gauged by counting fitness function evaluations, is a primary resource in both single and multiobjective optimization. Unlike single-objective optimization, multiobjective optimization requires not only proximity to a Pareto optimal solution but also comprehensive coverage of the Pareto front. Since results of multiobjective algorithms are finite approximation sets, it's essential to determine when one set is superior, as per Definition \ref{def:set_A_better_than_set_B} (see \cite{zitzler2003performance}).
\begin{definition}{Approximation Set on the Pareto Front. }
A finite subset $A$ of $\mathbb{R}^m$ is an approximation set to the Pareto front if and only if $A$ consists of mutually (Pareto) non-dominated points.
\end{definition}
\begin{definition}{Comparing Approximation Sets of a Pareto Front. }
\label{def:set_A_better_than_set_B}
Let $A \mbox{ and } B $ be approximation sets of a Pareto front in $\mathbb{R}^m$. We say that $A$ is better than $B$ if and only if every $b \in B$ is weakly dominated by at least one element $a \in A$ and $A \neq B$. Notation: $A \rhd B$.
\end{definition}
Figure \ref{fig:A_is_better_than_B} shows examples where one set is better than another, while Figure \ref{fig:A_is_not_better_than_B} illustrates when a set is not superior.
\begin{figure}
\begin{center}
\begin{tikzpicture}[scale=0.49] 
		\pgfmathsetmacro {\xa}{cos (90)}
		\draw[thick, color=black!22, dotted, step=1cm] (-1,-1) grid (10,8);
		\draw[thick, ->] (-1,0) -- (10.25,0) node[right] {$f_1$};
		\draw[thick, ->] (0,-1) -- (0,8.25) node [above] {$f_2$};
		\foreach \i in {1, 2} {
			\node[anchor=north] at (\i, 0) {$\i$};
			\node[anchor=east] at (0, \i) {$\i$};
		}
        \node[anchor=east] at (3.7,-0.5){$\mathbf{\dots}$};
        \node[anchor=north] at (-0.4, 4){$\mathbf{\vdots}$};
		\node[anchor=north east] at (0,0) {$(0,0)$};
\foreach \Point in {(1,6), (2,5), (5,3), (7,2), (9,1.5)}{
    \node at \Point {\color{blue} $\square$\color{black}};
}
\foreach \Point in {(3,6), (4,5), (6,4), (8,2.5)}{
    \node at \Point {\color{red} $\circ$\color{black}};
}					
    \end{tikzpicture}
\begin{tikzpicture}[scale=0.49]
		\pgfmathsetmacro {\xa}{cos (90)}
		\draw[thick, color=black!22, dotted, step=1cm] (-1,-1) grid (10,8);
		\draw[thick, ->] (-1,0) -- (10.25,0) node[right] {$f_1$};
		\draw[thick, ->] (0,-1) -- (0,8.25) node [above] {$f_2$};
		\foreach \i in {1, 2} {
			\node[anchor=north] at (\i, 0) {$\i$};
			\node[anchor=east] at (0, \i) {$\i$};
		}
        \node[anchor=east] at (3.7,-0.5){$\mathbf{\dots}$};
        \node[anchor=north] at (-0.4, 4){$\mathbf{\vdots}$};
		\node[anchor=north east] at (0,0) {$(0,0)$};
\foreach \Point in {(1,6), (2,5), (5,3), (7,2), (9,1.5)}{
    \node at \Point {\color{blue} $\square$\color{black}};
}
\foreach \Point in {(3,6), (4,5), (6,4), (7,2)}{
    \node at \Point {\color{red} $\circ$\color{black}};
}					
    \end{tikzpicture}
\end{center}
\caption{In both pictures, the blue square points outperform the red-circle points. The right picture shows a non-empty intersection between the two sets. \label{fig:A_is_better_than_B} }
\end{figure}
\begin{figure}[t]
\begin{center}	
\begin{tikzpicture}[scale=0.49] 
		\pgfmathsetmacro {\xa}{cos (90)}
		\draw[thick, color=black!22, dotted, step=1cm] (-1,-1) grid (10,8);
		\draw[thick, ->] (-1,0) -- (10.25,0) node[right] {$f_1$};
		\draw[thick, ->] (0,-1) -- (0,8.25) node [above] {$f_2$};
		\foreach \i in {1, 2} {
			\node[anchor=north] at (\i, 0) {$\i$};
			\node[anchor=east] at (0, \i) {$\i$};
		}
        \node[anchor=east] at (3.7,-0.5){$\mathbf{\dots}$};
        \node[anchor=north] at (-0.4, 4){$\mathbf{\vdots}$};
		\node[anchor=north east] at (0,0) {$(0,0)$};
\foreach \Point in {(1,6), (2,5), (5,3), (7,2), (9,1.5)}{
    \node at \Point {\color{blue} $\square$\color{black}};
}
\foreach \Point in {(3,6), (3.5,4.5), (6,4), (8,2.5)}{
    \node at \Point {\color{red} $\circ$\color{black}};
}					
    \end{tikzpicture}
\begin{tikzpicture}[scale=0.49] 
		\pgfmathsetmacro {\xa}{cos (90)}
		\draw[thick, color=black!22, dotted, step=1cm] (-1,-1) grid (10,8);
		\draw[thick, ->] (-1,0) -- (10.25,0) node[right] {$f_1$};
		\draw[thick, ->] (0,-1) -- (0,8.25) node [above] {$f_2$};
		\foreach \i in {1, 2} {
			\node[anchor=north] at (\i, 0) {$\i$};
			\node[anchor=east] at (0, \i) {$\i$};
		}
        \node[anchor=east] at (3.7,-0.5){$\mathbf{\dots}$};
        \node[anchor=north] at (-0.4, 4){$\mathbf{\vdots}$};
		\node[anchor=north east] at (0,0) {$(0,0)$};
\foreach \Point in {(1,6), (2,5), (5,3), (7,2), (9,1.5)}{
    \node at \Point {\color{red} $\circ$\color{black}};
}
\foreach \Point in {(1,6), (2,5), (5,3), (7,2)}{
    \node at \Point {\color{blue} $\square$\color{black}};
}					
    \end{tikzpicture}
\end{center}
\begin{center}
\begin{tikzpicture}[scale=0.45] 
		\pgfmathsetmacro {\xa}{cos (90)}
		\draw[thick, color=black!22, dotted, step=1cm] (-1,-1) grid (10,8);
		\draw[thick, ->] (-1,0) -- (10.25,0) node[right] {$f_1$};
		\draw[thick, ->] (0,-1) -- (0,8.25) node [above] {$f_2$};
		\foreach \i in {1, 2} {
			\node[anchor=north] at (\i, 0) {$\i$};
			\node[anchor=east] at (0, \i) {$\i$};
		}
        \node[anchor=east] at (3.7,-0.5){$\mathbf{\dots}$};
        \node[anchor=north] at (-0.4, 4){$\mathbf{\vdots}$};
		\node[anchor=north east] at (0,0) {$(0,0)$};
\foreach \Point in {(2,6),  (5,3),  (9,1.5)}{
    \node at \Point {\color{blue} $\square$\color{black}};
}
\foreach \Point in {(1,7), (3,4),(7,2)}{
    \node at \Point {\color{red} $\circ$\color{black}};
}
\end{tikzpicture}
\begin{tikzpicture}[scale=0.45] 
		\pgfmathsetmacro {\xa}{cos (90)}
		\draw[thick, color=black!22, dotted, step=1cm] (-1,-1) grid (10,8);
		\draw[thick, ->] (-1,0) -- (10.25,0) node[right] {$f_1$};
		\draw[thick, ->] (0,-1) -- (0,8.25) node [above] {$f_2$};
		\foreach \i in {1, 2} {
			\node[anchor=north] at (\i, 0) {$\i$};
			\node[anchor=east] at (0, \i) {$\i$};
		}
        \node[anchor=east] at (3.7,-0.5){$\mathbf{\dots}$};
        \node[anchor=north] at (-0.4, 4){$\mathbf{\vdots}$};
		\node[anchor=north east] at (0,0) {$(0,0)$};
\foreach \Point in {(1,6), (2,5), (5,3), (7,2), (9,1.5)}{
    \node at \Point {\color{red} $\circ$\color{black}};
}
\foreach \Point in {(3,6), (4,5), (6,4), (8,2.5)}{
    \node at \Point {\color{blue} $\square$\color{black}};
}					
\end{tikzpicture}
\end{center}
\caption{
In all the pictures, blue square points outperform red circle points, except in the two pictures on the right where red circle points outperform blue square points. \label{fig:A_is_not_better_than_B}
}
\end{figure}

The study in \cite{zitzler2003performance} employs this relation on sets to categorize performance indicators for Pareto fronts. To achieve this, they defined the concepts of completeness and compatibility of these indicators concerning the 'is better than' set relation.
\begin{definition}{Unary Set Indicator. }
A unary set indicator is a mapping from finite subsets of the objective space to the set of real numbers. It is used to compare (finite) approximations to the Pareto front.
\end{definition}

\begin{definition}{Compatibility of Unary Set Indicators concerning the 'is better than' order on Approximation Sets. }
A unary set indicator $I$ is compatible concerning the 'is better than' or $\rhd$-relation if and only if $I(A) >  I(B) \Rightarrow A \rhd B$.
A unary set indicator $I$ is complete with respect to the 'is better than' or $\rhd$-relation if and only if $ A \rhd B \Rightarrow I(A) > I(B)$. If in the last definition we replace $>$ by $\geq$ then the indicator is called weakly-complete.
\end{definition}
The hypervolume indicator and some of its variations are complete. Other indicators compared in the paper \cite{zitzler2003performance} are weakly-complete or not even weakly-complete. It has been proven in the same paper that no unary indicator exists that is complete and compatible at the same time. Moreover for the hypervolume indicator $\mbox{ HI } $ it has be shown that $\mbox{ HI } (A) > \mbox{ HI }(B) \Rightarrow \neg(B \rhd A)$. The latter we call weakly-compatible.

Discussions on the hypervolume indicator assume all approximation set points dominate a reference point. Recently, free hypervolume indicators have been introduced, eliminating the need for a reference point while being complete and weakly-compatible for all approximation sets \cite{zitzler2003performance}. 

Binary indicators, introduced alongside unary ones \cite{zitzler2003performance}, often follow unary indicators. They compare two approximation sets, outputting a real number to indicate which set is better and by how much \footnote{Conceivably one can can introduce $k$-ary indicators. To our knowledge, so far they have not been used in multiobjective optimization.}. When the true Pareto front is known, they assist in benchmarking test problems. A typical example is the binary $\epsilon$-indicator, which requires defining a binary relation on points in $\mathbb{R}^m$ for each $\delta \in \mathbb{R}$.
\begin{definition}{$\delta$-domination.}
Let $\delta \in \mathbb{R}$ and let $a \in \mathbb{R}^m$ and $b \in \mathbb{R}^m$. We say that $a$ $\delta$-dominates $b$ (notation: $a \preceq_{\delta} b$) if and only if $a_i \leq b_i + \delta, i=1,\dots, m$.
\end{definition}
Next, we can define the binary indicator $I_{\epsilon}$.
\begin{definition}{The Binary Indicator $I_{\epsilon}$.}
Given two approximation sets $A$ and $B$, then $I_{\epsilon}(A,B) := \inf_{\delta \in \mathbb{R}} \{ \forall b \in B \  \exists a \in A \mbox{ such that } a \preceq_{\delta} b \}$.
\end{definition}
For a fixed $B$, a smaller $I_{\epsilon}(A,B)$ improves how well set $A$ approximates $B$. Two key properties are: $  A \rhd B \Rightarrow I_{\epsilon}(B,A) > 0$ and $I_{\epsilon}(A,B) \leq 0 \mbox{ and } I_{\epsilon}(B, A) > 0 \Rightarrow A \rhd B$. These indicate that the binary $\epsilon$-indicator can determine if $A$ is better than $B$. However, knowing the hypervolume indicator for sets $A$ and $B$ cannot determine if $A$ is better than $B$.

Some indicators are useful when there is knowledge of the Pareto front. As suggested in \cite{Rudolph16Hausdorff}, the Hausdorff distance can be used to measure an approximation set's closeness to the Pareto front. Additionally, the binary $\epsilon$-indicator can be converted to a unary indicator when the second input is the known Pareto front, and this indicator requires minimization.

\begin{figure}   
\begin{center}
	\begin{tikzpicture}[scale=0.35] 
		\pgfmathsetmacro {\xa}{cos (90)}
		\draw[thick, color=black!22, dotted, step=1cm] (-1,-1) grid (19,15);
		\draw[thick, ->] (-1,0) -- (19.25,0) node[right] {$f_1$};
		\draw[thick, ->] (0,-1) -- (0,15.25) node [above] {$f_2$};
		\foreach \i in {1, 2} {
			\node[anchor=north] at (\i, 0) {$\i$};
			\node[anchor=east] at (0, \i) {$\i$};
		}
        \node[anchor=east] at (3.7,-0.5){$\mathbf{\dots}$};
        \node[anchor=north] at (-0.4, 4){$\mathbf{\vdots}$};
		\node[anchor=north east] at (0,0) {$(0,0)$};

%
%
\fill[opacity=1.0,gray]     
        (6.5,15) -- (6.5,13) -- (7,13) -- (7,12) -- (8.5,12) --(8.5, 9.5) -- (10.5,9.5) -- (10.5, 8) -- (13,8)-- (13,6.5) --
         (15,6.5)--(15,6) -- (17.5, 6) -- (17.5, 3.5) -- (19,3.5) -- (19, 15) -- cycle;
%
\fill[opacity=0.75, gray]  
       (5,15) --(5,13) --(6.5, 13) -- (6.5, 15) -- cycle;
\fill[opacity=0.75, gray]  
       (6.5,13) -- (6.5,12) -- (7,12) -- (7,13) -- cycle;
\fill[opacity=0.75, gray]  
       (7,12) -- (7,9.5) -- (8.5, 9.5) -- (8.5, 12) -- cycle;
\fill[opacity=0.75, gray]  
        (8.5,9.5) -- (8.5, 8) -- (10.5,8) -- (10.5,9.5) -- cycle;
\fill[opacity=0.75, gray]  
        (10.5,8) -- (10.5,6.5) -- (13,6.5) -- (13,8) -- cycle;
\fill[opacity=0.75, gray]  
       (13,6.5) -- (13,6) -- (15,6) -- (15,6.5) -- cycle;
\fill[opacity=0.75, gray]  
       (16.5,6) -- (16.5,3.5) -- (17.5,3.5) -- (17.5, 6) -- cycle;
\fill[opacity=0.75, gray]  
       (17.5,3.5) -- (17.5,2) -- (19,2) -- (19,3.5) -- cycle;
%
\fill[opacity=0.5, gray]  
       (2.5,15) -- (2.5,13) -- (4,13) -- (4,12) -- (5.5,12) -- (5.5,9) -- (7.5,9) -- (7.5,7) -- (10.5,7) --
       (10.5, 8) -- (8.5,8) -- (8.5, 9.5) --(7, 9.5) -- (7,12) -- (6.5,12) -- (6.5,13) -- (5,13) -- (5,15) -- cycle;
\fill[opacity=0.5, gray]  
       (11.5,6.5) -- (11.5,6) -- (13,6) -- (13,6.5) -- cycle;
\fill[opacity=0.5, gray]  
       (15,6) -- (15,3.5) --(16.5,3.5) -- (16.5,6) -- cycle;
\fill[opacity=0.5, gray]  
       (16.5,3.5) -- (16.5,1) -- (19,1) -- (19,2) -- (17.5,2) -- (17.5,3.5) -- cycle;

\fill[opacity=0.25, gray]  
      (1,15) -- (1,13) -- (2.5, 13) -- (2.5,15) -- cycle;
\fill[opacity=0.25, gray]  
      (2.5,13) -- (2.5,12) -- (4,12) -- (4,13) -- cycle;
\fill[opacity=0.25, gray]  
     (4,12) -- (4,9) -- (5.5,9) -- (5.5,12) -- cycle;
\fill[opacity=0.25, gray]  
     (5.5,9) -- (5.5,7) -- (7.5,7) -- (7.5,9) -- cycle;
\fill[opacity=0.25, gray]  
     (7.5,7) -- (7.5,6) -- (11.5,6) -- (11.5,6.5) -- (10.5, 6.5) -- (10.5,7) -- cycle;
\fill[opacity=0.25, gray]  
     (11.5,6) -- (11.5,3.5) -- (15,3.5) -- (15,6) -- cycle;
\fill[opacity=0.25, gray]  
     (15,3.5) -- (15,1) -- (16.5,1) -- (16.5,3.5) -- cycle;
\fill[opacity=0.25, gray]  
     (16.5,1) -- (16.5,0.5) -- (19,0.5) -- (19, 1) -- cycle;

\draw[ ultra thick,  black] 
        (2.5,15.25) -- (2.5,13) -- (4,13) -- (4,12) -- (5.5,12) -- (5.5, 9) -- (7.5,9) -- (7.5,7) -- (10.5,7) --
        (10.5,6.5) -- (11.5,6.5) -- (11.5,6) -- (15,6) -- (15,3.5) -- (16.5, 3.5) -- (16.5,1) -- (19.25,1);
\foreach \Point in {(1,13),(4,9),(7.5,6),(17.5,2)}{
    \node at \Point {\color{brown} $\triangle$\color{black}};
}
\foreach \Point in {(2.5,12),(5.5,7),(11.5,3.5),(15,1)}{
    \node at \Point {\color{blue} $\square$\color{black}};
}
\foreach \Point in {(5,13),(7,9.5),(10.5,6.5),(15,3.5)}{
    \node at \Point {\color{red} $\circ$\color{black}};
}
\foreach \Point in {(6.5,12),(8.5,8),(13,6),(16.5,0.5)}{
    \node at \Point {\color{black} $\times$\color{black}};
}
    \end{tikzpicture}
\end{center}
\caption{The median attainment curve is represented by a black polygonal line. Four approximation sets are depicted: blue squares, brown triangles, red circles, and black crosses. Darker gray areas indicate dominance by more approximation sets.} 
\label{fig:attainment_curve_III}
\end{figure}

The attainment curve approach, as discussed in \cite{FonsecaFH01attain}, helps understand evolutionary multiobjective algorithms by generalizing the cumulative distribution function. This distribution, based on finite approximation sets of the Pareto front, estimates the probability that a point in the objective space is attained or dominated by an approximation set. Formally, for a given approximation set $A = \{ a^{(1)}, a^{(2)},  \dots, a^{(k)} \}$, $P$ represents the probability that an algorithm will find a solution to reach the goal in one run. The attainment function $\alpha_A : \mathbb{R}^m \rightarrow [0,1]$ related to an approximation set $A$ can be approximated using the outcome sets $A_1, \dots, A_s$ from $s$ algorithm runs. The function assigns a value of 1 if a vector is in the approximation set or dominated by it, otherwise 0.
For $m = 2 \mbox{ or } 3$ we can plot the boundaries where this function changes its value. These are the attainment curves ($m=2)$ and attainment surfaces ($m=3$). In particular the median attainment curve/surface gives a good summary of the behavior of the optimizer.  It is the boundary where the function changes from a level below 0.5 to a level higher than 0.5. Alternatively, one can look at lower and higher levels than 0.5 in order to get an optimistic or, respectively, a pessimistic assessment of the performance.

In Figure \ref{fig:attainment_curve_III} an example of the median attainment curve is shown. We assume that the four approximation sets are provided by some algorithm.  



\section{Many-objective Optimization}
Optimization with more than three objectives is currently termed many-objective optimization (see, for instance, the survey \cite{Li15}). This is to stress the challenges one meets when dealing with more than 3 objectives. The main reasons are:
\begin{enumerate}
\item problems with many objectives have a Pareto front which cannot be visualized in conventional 2D or 3D plots instead other approaches to deal with this are needed;
\item the computation time for many indicators and selection schemes become computationally hard, for instance, time complexity of the hypervolume indicator computation grows super-polynomially with the number of objectives, under the assumption that $P \neq NP$;
\item last but not least the ratio of non-dominated points tends to increase rapidly with the number of objectives.
For instance, the probability that a point is non-dominated in a uniformly distributed set of sample points grows exponentially fast towards 1 with the number of objectives.
\end{enumerate}
In the field of many-objective optimization different techniques are used to deal with these challenges. For the first challenge, various visualization techniques are used, such as projection to lower-dimensional spaces or parallel coordinate diagrams. In practice, one can, if the dimension is only slightly bigger than 3,  express the coordinate values by colors and shape in 3D plots.

Naturally, in many-objective optimization indicators which scale well with the number of objectives (say polynomially) are very much desired. Moreover, decomposition based approaches are typically preferred to indicator based approaches.

The last issue demands major shifts in strategies. Simplifying the search space isn't enough due to too many alternatives, requiring a hierarchy stricter than Pareto and needing extra preference knowledge, aided by interactive methods. Correlated objectives can be combined into one using techniques like dimensionality reduction and community detection. Adaptive and hybrid methods using evolutionary techniques with local search or machine learning effectively tackle many-objective optimization, broadening its use in various fields such as engineering, environmental management, finance, and network optimization. See also \cite{Brockhoff2023} for a recent overview.
\section*{Exercises}
\begin{enumerate}

\exercise{Ranking, Crowding Distance, and Hypervolume Contribution}{%
    Consider a population of candidate solutions evaluated on two objectives (both to be minimized). The population is given as: 

    \[
    P = \{ (1,8),\; (2,6),\; (3,5),\; (4,4),\; (5,3),\; (6,2),\; (7,1) \}
    \]

    \begin{enumerate}
        \item[(a)] \textbf{Non-dominated Sorting:}  
        Perform non-dominated sorting on the population \(P\) and assign a ranking order (i.e., determine the front number for each solution) following the NSGA-II approach. Explain your steps and reasoning.
        
        \item[(b)] \textbf{Crowding Distance Calculation:}  
        For the best ranked subset (i.e., the first non-dominated front), compute the crowding distance for each solution. Show how boundary solutions are treated and discuss the role of the crowding distance in preserving diversity.

        \item[(c)] \textbf{Hypervolume Contribution:}  
        Assuming a reference point \( r = (8,9) \), calculate the hypervolume contribution of each solution in the first front. That is, compute the exclusive hypervolume that would be lost if each solution were removed from the set.
        
        \item[(d)] \textbf{Discussion:}  
        Discuss how the ranking order, crowding distance, and hypervolume contributions are integrated in evolutionary multi-objective algorithms to balance convergence toward the Pareto front and maintain solution diversity.
    \end{enumerate}
}{ex:ec_metrics}

\exercise{Computing the Hypervolume Indicator}{%
    In multi-objective optimization, the hypervolume indicator is a key metric that quantifies the volume (in objective space) covered by a set of solutions relative to a chosen reference point.

    \begin{enumerate}
        \item[(a)] \textbf{Hypervolume Calculation:}  
        Given a set of solutions in a 2-objective minimization problem:
        \[
        P = \{ (1,8),\; (2,6),\; (3,5),\; (4,4) \}
        \]
        and a reference point \( r = (5,9) \), compute the hypervolume indicator for \(P\). Provide a step-by-step explanation of how you partition the objective space and calculate the covered volume.

        \item[(b)] \textbf{Algorithm Outline:}  
        Outline an algorithm for computing the hypervolume indicator for a general set of solutions in \(2\)-objective space and $3$-objective space. Discuss the computational challenges and compare the complexity of this algorithm with that of non-dominated sorting.
    \end{enumerate}
}{ex:hypervolume_indicator}

\exercise{Optimizing Disc Placement with NSGA-III}{%
    Download and install the DESDEO library \cite{Misitano21} and use its implementation of the NSGA-III algorithm to solve the following multi-objective optimization problem.
    
    \textbf{Problem Statement:}  
    Place as many non-overlapping discs as possible within a unit square (of side length 1). All discs share the same radius and must be fully contained within the square. The two conflicting objectives are:
    \begin{itemize}
        \item Maximize the number of discs placed.
        \item Maximize the radius of the discs.
    \end{itemize}
    
    \begin{enumerate}
        \item[(a)] \textbf{Problem Formulation:}  
        Provide a formal mathematical formulation of the problem. Define the decision variables, the constraints (ensuring discs do not overlap and are completely contained in the unit square), and the two objective functions.
        
        \item[(b)] \textbf{Solution Encoding:}  
        Describe an appropriate encoding of a candidate solution (i.e., a potential disc arrangement) within the DESDEO framework. Explain how the positions and the common radius of the discs will be represented.
        
        \item[(c)] \textbf{Running NSGA-III:}  
        Configure and run the NSGA-III algorithm using the DESDEO library to approximate the Pareto front for this problem. Present your results, including the obtained Pareto front, and analyze the trade-offs observed between the number of discs and the disc radius.
        
        \item[(d)] \textbf{Discussion:}  
        Discuss potential challenges when applying NSGA-III to this problem, such as constraint handling and maintaining solution diversity, and suggest possible strategies to overcome these issues.
    \end{enumerate}
}{ex:desdeo_nsga3}
\end{enumerate}
\chapter{Exact Methods for Finding Pareto Optimal Sets}

Although evolutionary algorithms provide powerful approximations of Pareto optimal sets, exact methods ensure the complete identification of these solutions. Exact methods use deterministic strategies to systematically enumerate or compute Pareto optimal points, leveraging mathematical structures and optimization techniques. In this chapter, we offer only a brief overview of several exact approaches—including homotopy and continuation methods, multiobjective linear programming, Newton-Raphson hypervolume-based methods, and Bayesian global optimization using expected hypervolume improvement—with the understanding that a more detailed treatment may be presented in future editions. 

\section{Homotopy and Continuation Methods}

Homotopy methods systematically trace solution paths by deforming an initial problem into the target optimization problem~\cite{Hillermeier01,Schuetze2005continuation}. Continuation methods extend this idea by following solution trajectories across a parametric space, ensuring comprehensive Pareto front exploration. These methods are particularly useful for solving nonlinear multiobjective problems where standard optimization techniques may struggle with non-convexity or disconnected Pareto sets. They have been successfully applied to engineering and economic optimization problems where smooth trade-offs exist between objectives, for instance in efficient power station management~\cite{Hillermeier01}.

\section{Multiobjective Linear Programming (MOLP)}

Multiobjective Linear Programming (MOLP) techniques leverage linear programming formulations to compute exact Pareto optimal solutions in problems with linear objective functions and constraints~\cite{Ehr05}. The widely used weighted sum method and the -constraint method transform the multiobjective problem into a sequence of single-objective optimizations, systematically uncovering Pareto optimal solutions. Benson's method~\cite{Benson98} and dual-based algorithms ensure efficient Pareto front generation, especially for large-scale linear problems encountered in logistics, finance, and operations research.

\section{Hypervolume-Based Newton Method}

The Newton-Raphson method can be adapted to multiobjective optimization by optimizing hypervolume directly~\cite{Emmerich2007grad,Wang17}. This method iteratively refines a solution by computing the hypervolume gradient and Hessian, making precise adjustments to converge to Pareto optimality. Hypervolume-based Newton-Raphson methods provide fast convergence and are particularly effective for continuous multiobjective problems with smooth trade-off surfaces \cite{Hernandez18}. By incorporating second-order information, they improve convergence speed compared to gradient-based methods. Recently, these methods have been extended to handle constraints \cite{Wang23constraint}

\section{Bayesian Multicriteria Global Optimization}

Bayesian optimization is an advanced approach to optimizing expensive black-box functions. By modeling the objective functions probabilistically, it efficiently balances exploration and exploitation. The ParEGO algorithm efficiently uses Gaussian process regression to predict Chebyshev scalarizing functions from past evaluation data \cite{Knowles06}. The Hypervolume Expected Improvement (HVEI) criterion \cite{Emmerich05}  extends Bayesian optimization to multiobjective problems by selecting new evaluation points that maximize expected improvements in hypervolume and the R2-indicator \cite{Deutz12}, a concept that was later refined in~\cite{Couckuyt2014}. This approach is particularly useful for computationally expensive applications such as drug discovery, aerodynamic design, and hyperparameter tuning in machine learning.

\section{Conclusion}

Exact methods for computing Pareto optimal sets provide rigorous solutions to multiobjective optimization problems, complementing heuristic approaches like evolutionary algorithms. While homotopy and continuation methods offer systematic front exploration, MOLP and hypervolume-based Newton-Raphson methods enable precise optimization. Bayesian global optimization further enhances efficiency in scenarios with expensive function evaluations. Combining these methods with evolutionary strategies can yield robust hybrid approaches for solving complex multiobjective optimization problems.

\chapter*{Exercises}

\begin{enumerate}[leftmargin=*, label={\arabic*.}]
    \exercise{Hypervolume Gradient Calculation and Interpretation}{%
        Consider a finite set of 2-D points \(P = \{(0,4), (2,1), (1,2), (3,0)\}\) in the objective space that defines a non-dominated front for a multiobjective optimization problem. The hypervolume indicator \(H(P)\) measures the volume of the objective space dominated by these points relative to a given reference point \(r = (5,5)\).

        \begin{enumerate}[label=(\alph*)]
            \item \textbf{Derivation:} \\
            Derive the gradient of the hypervolume indicator, \(\nabla H(P)\), with respect to the coordinates of the points in \(P\). 

            \item \textbf{Interpretation:} \\
            Discuss the behavior of the gradient components corresponding to points that are \emph{dominated} by other points in \(P\) (i.e., points that do not contribute to the hypervolume).
        \end{enumerate}
    }{ex:hypervolume}

    \exercise{Geometry of the Pareto Front and Efficient Set in MOLP}{%
        Multiobjective linear programming (MOLP) problems, with linear objective functions and constraints, yield Pareto optimal sets that exhibit distinctive geometrical features.

        \begin{enumerate}[label=(\alph*)]
            \item \textbf{Geometric Characterization:} \\
            Describe the typical geometric structure of the Pareto front in the objective space and the efficient set in the decision space for MOLP problems. In your answer, address whether these sets are usually convex, connected, or possess any other notable geometrical properties.

            \item \textbf{Continuation Methods:} \\
            How can these insights help construct near-optimal Pareto-efficient solutions close to known efficient ones when all objective functions and constraints are differentiable? For answering this, try to model mathematically the tangent at an efficient point of an unconstrained biobjective problem using the gradient.
        \end{enumerate}
    }{ex:geometry}
\end{enumerate}

%

\end{document}